\documentclass[a4paper,fleqn,leqno,10pt,oneside]{article}
\usepackage[utf8]{inputenc}
\usepackage[english]{babel}
\usepackage{amsmath}
\usepackage{amsthm}
\usepackage[dvipsnames]{xcolor}
\usepackage{latexsym,amssymb}
\usepackage{mathtools}

\usepackage{subdepth}           % sub- and superscripts are always
\usepackage{graphicx}
\usepackage{bbm}
\usepackage{fancyhdr}

\usepackage{geometry}
\usepackage{pdflscape}
\usepackage{hyperref}

\newcommand{\Levy}{L{\'e}vy} % always use as \Levy{}, see
\DeclareMathOperator*{\wlim}{w^\sharp-lim}
\DeclareMathOperator{\anz}{\#}
\newcommand{\cadlag}{c\`adl\`ag}
\newcommand{\dx}{{\textup{d}}}               % the d for differentials
\newcommand{\eqd}{\overset{\textup{d}}{=}}   % *e*qual *i*n *d*istribution

\newcommand{\eps}{\varepsilon}

\newcommand{\abs}[1]{\lvert#1\rvert} %Definition Betragsstriche

\newcommand{\norm}[1]{\lVert#1\rVert} % Definition Normstriche

\newcommand{\ind}[1]{\mathbbm{1}_{\{#1\}}} %Definition Indikatorfunktion
\newcommand{\indset}[1]{\mathbbm{1}_{#1}}  %Definition Indikatorfunktion

\newcommand{\wt}{\widetilde}

\newcommand{\dr}{\underline{\underline{r}}}

\newcommand{\hu}{h}
\newcommand{\wh}{\widehat}

\newcommand{\suml}{\sum\limits}
\newcommand{\ve}{\varepsilon}
\newcommand{\uuxi}{\underline{\underline{\xi}}}

\newtheorem{proposition}{Proposition}[section]
\newtheorem{corollary}[proposition]{Corollary}
\newtheorem{lemma}[proposition]{Lemma}
\newtheorem{theorem}{Theorem}%[section]

\newtheorem{conjecture}{Conjecture}[section]

\newtheorem{xx}{\bf xxx}

\newtheorem{zz}{\bf zzz}

\makeatletter                                           % This is to ensure that remark titles are bold face (as in theorem
\def\th@newremark{\th@remark\thm@headfont{\bfseries}}   % style) but the body is in normal font instead of italics (as
\makeatletter                                           % in theorem style). The standard remark style of amsthm produces
\theoremstyle{definition}
\newtheorem{definitionx}[proposition]{Definition}%[section]

\theoremstyle{newremark}
\newtheorem{remarkx}[proposition]{Remark}

\newenvironment{remark}
  {\pushQED{\qed}\remarkx}
  {\popQED\endremarkx}
\newenvironment{definition}
  {\pushQED{\qed}\definitionx}
  {\popQED\enddefinitionx}

\usepackage{enumerate}

\DeclareMathOperator{\supp}{supp}              % support
\DeclareMathOperator{\Pois}{Poiss}              % Poisson distribution
\DeclareMathOperator{\Exp}{Exp}                % Exponential distribution

\allowdisplaybreaks[4]

\newcommand{\mcB}{\mathcal{B}}
\newcommand{\mcC}{\mathcal{C}}
\newcommand{\mcD}{\mathcal{D}}
\newcommand{\mcE}{\mathcal{E}}
\newcommand{\mcG}{\mathcal{G}}
\newcommand{\mcH}{\mathcal{H}}
\newcommand{\mcI}{\mathcal{I}}
\newcommand{\mcJ}{\mathcal{J}}
\newcommand{\mcL}{\mathcal{L}}
\newcommand{\mcM}{\mathcal{M}}
\newcommand{\mcS}{\mathcal{S}}

\newcommand{\mfC}{\mathfrak{C}}
\newcommand{\mfe}{\mathfrak{e}}

\newcommand{\mfK}{\mathfrak{K}}

\newcommand{\mfl}{\mathfrak{l}}
\newcommand{\mfM}{\mathfrak{M}}

\newcommand{\ntree}{\mathfrak{0}}
\newcommand{\mfp}{\mathfrak{p}}
\newcommand{\mfP}{\mathfrak{P}}

\newcommand{\mfu}{\mathfrak{u}}
\newcommand{\mfU}{\mathfrak{U}}
\newcommand{\mfv}{\mathfrak{v}}
\newcommand{\mfV}{\mathfrak{V}}
\newcommand{\mfw}{\mathfrak{w}}
\newcommand{\mfW}{\mathfrak{W}}

\newcommand{\mfY}{\mathfrak{Y}}

\newcommand{\E}{\mathbbm{E}}

\newcommand{\K}{\mathbbm{K}}
\newcommand{\N}{\mathbbm{N}}
\newcommand{\R}{\mathbbm{R}}
\newcommand{\U}{\mathbbm{U}}
\newcommand{\Z}{\mathbbm{Z}}
\newcommand{\M}{\mathbbm{M}}

\renewcommand{\P}{\mathbbm{P}}
\newcommand{\1}{\mathbbm{1}}
\newcommand{\GP}{\mathrm{GP}}
\newcommand{\Pro}{\mathrm{P}}

\newcommand{\bbS}{\mathbb{S}}

\newcommand{\bbK}{\mathbbm{K}}
\newcommand{\bbN}{\mathbbm{N}}
\newcommand{\bbU}{\mathbbm{U}}

\newcommand{\uur}{\underline{\underline{r}}}

\newcommand{\uuu}{\underline{\underline{u}}}

\numberwithin{equation}{section}

\begin{document}

\title{Genealogy-valued Feller diffusion}

\author{A. Depperschmidt$^1$, A. Greven$^2$} \date{\today}
\maketitle

\begin{abstract}
  We consider the evolution of the genealogy of the population
  currently alive in a Feller branching diffusion. In contrast to the
  approach via labeled trees in the continuum random tree world
  \cite{Ald1991a,LeGall93}, following \cite{GPWmp13}, the genealogies
  are modelled as elements of a Polish space $\U$ which consists of
  all equivalence classes of ultrametric measure spaces. This space
  equipped with an operation called concatenation, denoted by
  $(\U,\sqcup)$ has a rich algebraic (semigroup) structure,
  \cite{infdiv,ggr_GeneralBranching}, which is used effectively to
  study branching processes. We focus on the evolution of the
  genealogy in time and the large time asymptotics conditioned on
  survival up to present time and on survival forever. We develop the
  calculus in such a way that it can be applied in the future to more
  complicated systems, such as logistic branching or state dependent
  branching. Furthermore the approach we take carries over very
  smoothly to spatial models with infinitely many components.

  We prove existence, uniqueness, continuity of paths and a
  generalized Feller property of solutions of the martingale problem
  for this genealogy-valued, i.e.\ $\U$-valued Feller diffusion. The
  uniqueness is shown via Feynman-Kac duality with the distance matrix
  augmented Kingman coalescent. By conditioning on the entire
  population size process and then observing the genealogy part we
  obtain the precise relation to a specific time-inhomogeneous
  $\U_1$-valued Fleming-Viot process with varying resampling rate,
  $\U_1$ being the set of all equivalence classes of ultrametric
  probability measure spaces. This relation gives the so-called skew
  martingale representation of the $\U$-valued Feller diffusion.

  Via the Feynman-Kac duality we deduce the generalized branching
  property of the $\U$-valued Feller diffusion. Using a semigroup
  operation through concatenations on $\U$, \cite{infdiv}, together
  with the generalized branching property,
  \cite{ggr_GeneralBranching}, we obtain a \Levy{}-Khintchine formula
  for the $\U$-valued Feller diffusion and determine explicitly the
  \Levy{} measure which has a special form, allowing us to obtain for
  $h>0$ a decomposition into depth-$h$ subfamilies which leads to a
  representation in terms of a Cox point process of genealogies where
  ``points'' correspond to single ancestor subfamilies.

  We determine the $\U$-valued process conditioned to survive until a
  finite time $T$ correcting a result from the $\R_+$-valued
  literature in the computation of the diffusion coefficient. This is
  the key ingredient of the excursion law of the $\U$-valued Feller
  diffusion. Next we study asymptotics of the $\U$-valued Feller
  diffusion conditioned to survive forever and obtain its
  Kolmogorov-Yaglom limit and show that the limiting processes solve
  well-posed $\U$-valued martingale problems.

  Using infinite divisibility and skew martingale problems we obtain
  various representations of the long time limits: $\U$-valued
  backbone construction of the Palm distribution, the $\U$-valued
  version of the Kallenberg tree, the $\U$-valued version of Feller's
  branching diffusion with immigration from an immortal line \`a la
  Evans \cite{Evans93}. On the level of $\U$-valued processes we still
  have equality (in law) of the $Q$-process, i.e., the process
  conditioned to survive up to time $T$ in the limit $T \to \infty$,
  the size-biased process and Evans' branching process with
  immigration from an immortal line. The $\U$-valued generalized
  quasi-equilibrium is a size-biased version of the Kolmogorov-Yaglom
  limit law.

  The above results are key tools for analyzing genealogies in spatial
  branching populations. We construct the genealogy of the interacting
  Feller diffusion on a countable group (super random walk) and obtain
  results on a martingale problem characterization, duality,
  generalized branching property and the long time behavior for this
  object. As an application we give a two scale analysis of the super
  random walk genealogy with strongly recurrent migration providing
  the asymptotic genealogy of clusters via the $\U^\R$-valued version
  of the Dawson-Watanabe process. We indicate the situation in other
  dimensions.

  Finally we enrich the $\U$-valued Feller process further, encoding
  the information on the whole population ever alive before the
  present time $t$ and describe its evolution. This leads to the so
  called fossil process and we relate its limit for $t \to \infty$ to
  the continuum random tree.
\end{abstract}
\bigskip


% \medskip
% \noindent
% \textbf{AMS Subject Classification:}  60K35 (Primary); 60J60,
  % 92D25

\medskip
\footnoterule

\noindent {\footnotesize $1)$ Universität Hamburg, Fachbereich
  Mathematik, Bundesstr.\ 55,
  20146 Hamburg, andrej.depperschmidt@uni-hamburg.de, \\ $2)$
  Universit\"at Erlangen-N\"urnberg, Department Mathematik, Cauerstr.
  11, 91058 Erlangen, greven@math.fau.de}

\newpage

\tableofcontents

\newpage

\section{Introduction}
\label{s.introduc}

\paragraph{Background}
Diffusion limits of Galton-Watson processes and the construction of
Feller's branching diffusion date back to the last century, the study
of measure valued branching processes to the seventies. Nevertheless,
important features of these processes are still being discovered; see
\cite{Lamb07,Li2011} and their extensive source of references.
Furthermore \emph{spatial} versions such as super random walks and
super processes have been explored; see \cite{D93,Eth00,DG96,DG03}.
Also the \emph{genealogies} associated with such an evolution of
\emph{all individuals ever alive} have been studied via \emph{labeled
  trees} starting with the work of Neveu \cite{Neveu86} and subsequent
work by Aldous \cite{Aldous90,Ald1991,Ald1991a,Ald1993}, Le Jan
\cite{LJ91} and Le Gall \cite{LeGall93} led to the description in
terms of the \emph{continuum random tree}, which is encoded by
excursions of Brownian motion. These constructions have been extended
to branching processes with fat tail offspring distributions which
lead to jump processes in the limit, \cite{DLG05,DuquesneLeGall2002}.
The genealogies in processes with immigration have also been
systematically studied, see \cite{Lamb02} for example. All of the
above references use a coding of the genealogy as a \emph{labeled
  tree} or a \emph{labeled $\R$-tree} that represent all individuals
\emph{ever} alive. This coding is more difficult to handle in spatial
population models, compare here \cite{DuquesneLeGall2002} where
special branching features are used to tackle this.

We are interested in describing the \emph{evolution} of the genealogy
of the currently alive population in time as \emph{solution to
  well-posed martingale problems}. For this we will use a coding with
\emph{equivalence classes of ultrametric measure spaces} elements of a
\emph{Polish space $\U$}, which seems better suited to treat
\emph{evolutions} of genealogies in time. This approach is also quite
flexible for extensions to multi-type models. In particular the form
of the description is open for much more general forms of the change
of generation in variable size populations.

But even much more crucial is that the approach taken allows very
naturally to pass to spatial models on \emph{infinite geographic
  spaces}, a point distinguishing it from other approaches, possible.
See \cite{DG18evolution} for a survey of our approach. This approach
extends the concept of historical processes of Dawson and Perkins
\cite{DP91} suited for super processes on $\R^d$ and allows to tackle
general population processes which was the first attempt to overcome
the difficulties arising in the \emph{spatial context} from methods
using representations of population dynamics in terms of countable
particle systems via lookdown constructions which allow to work even
with \emph{stochastic equations}, as in the case of Donnelly and Kurtz
\cite{DonnellyKurtz1999b,DonnellyKurtz1999a} or more recently in
\cite{Gufler2018a} and \cite{BGKTW2021}, but is not possible on
infinite geographic spaces. Current research also suggests that the
approach also has applications in stochastic processes with values in
large or countable graphs.

Different from the Fleming-Viot world studied so far we can use an
algebraic structure $(\U,\sqcup)$ very effectively based on recent
work \cite{infdiv,ggr_GeneralBranching} to study the depth-$h$
subfamily structure by determining the \Levy{} measure, the excursion
law, Evans branching with immigration from an immortal line and many
other objects relevant for the \emph{genealogy-valued}
Feller-diffusion. These structures are much more complicated to deal
with for Fleming-Viot models.

To lift results for the classical branching processes, $\R$-valued or
measure-valued, to the level of genealogies we have to deal with
processes whose state space is Polish, but not \emph{locally compact}
or $\sigma$-\emph{compact}. Even though we can overcome the problems
arising from the more complicated state space, the techniques
exploited use special features as dualities so that it is not easy to
say in general which properties of known population models can be
lifted to genealogy-valued, i.e., $\U$-valued Markov versions of these
processes.

As a general program one could look at the genealogies and their
evolution of individual based population models where one can
immediately write down the $\U$-valued process which describes the
evolution of the present day population as time passes by. Then the
issue would be if the associated infinite population limit models in
which, if one considers only the population sizes in space and type
spaces one can obtain a limiting dynamic, also the process of the
corresponding genealogies have a limiting dynamics, respectively under
what conditions on the rates can we lift theorems on the level of
genealogies. Since the state spaces for genealogies are Polish spaces
but not compact, locally compact or $\sigma$-compact, it is difficult
to do this in full generality and in fact this is not true in general
that this lifting works. (Here even $\Lambda$-Canning's models, which
are not dust free are such models as it turns out.) For that reason we
focus here on a specific model, where we can use duality and
Feynman-Kac duality as tools to compensate for lacking compactness
properties of the state spaces.

\paragraph{Perspectives}
The challenge for the future is to analyze the $\U$-valued version of
the diffusion $\dx X_t=\sqrt{b(X_t)} \; \dx W_t$, instead of the
Feller case $b\cdot x$ (for a constant $b>0$), for locally Lipschitz
function $b$ with $b(0)=0$, $b(x)>0$ for $x>0$ and $b(x)=O(x)$ as $x
\to \infty$, where the independence of the evolution of subfamilies
has to be replaced by exchangeability. The solutions can be
constructed and are well-posed by the skew martingale problem
approach, but to analyze the finite structure this poses great
challenges. However, the developed calculus puts this in a better
perspective, but clearly new concepts and techniques are also needed.
Also sub- and supercritical case should be replaced by logistic terms
as drift such as $a(x(x-K))$ for example. The important point in all
these cases is to also pass to spatial models and in particular to
infinite geographic spaces.

\paragraph{Basic questions}
In this framework we treat \emph{three topics} concerning the
evolution of genealogies in continuous state branching populations,
more specifically the \emph{$\U$-valued Feller diffusion} model and
the corresponding super-random walk.
\medskip

\noindent
\textbf{\emph{(1) Basics of the $\U$-valued Feller diffusion}} In this
work we begin with a foundational part and we look at the
\emph{evolving} genealogy of the Feller diffusion model from a new
perspective by rigorously defining a \emph{process} $\mfU$ with a
Polish state space $\U$ which we recall later and which captures the
\emph{evolution in time of the genealogy} of the individuals
\emph{currently alive} as an \emph{evolving Markov process} defined
via a \emph{well-posed} martingale problem.

The main tool for this approach is the description of genealogies by
\emph{equivalence classes of ultrametric measure spaces} leading to
the state space $\U$; see \cite{GPW09} furthermore see \cite{Gl12} or
\cite{DGP11,GSW} for generalizations. In the world of Fleming-Viot
processes with \emph{fixed} population sizes, this martingale problem
approach has been used extensively on the space of ultrametric
probability measure spaces, a closed subspace of $\U$, denoted by
$\U_1$; see \cite{GPWmp13,DGP12,GSW}. Varying population sizes are
considered in \cite{Gl12,infdiv} and \cite{ggr_GeneralBranching}.

We study the process $(\mfU_t)_{t\ge 0}$ which is defined as the
unique solution of a \emph{well-posed martingale problem} with values
in $\U$. Since $\U$ is not anywhere locally compact, we need some
version of the Feller property to overcome the lack of a Weierstraß
theorem for the test functions of the martingale problem, which would
give a nice countable dense set in $C_b(\U)$ to obtain continuous
paths and the Markov property. Here comes as a key tool that the
operators are \emph{second order operators} and that we have a
\emph{duality} based on polynomials. Polynomials are \emph{not} dense
but contain a countable and \emph{measure determining set} of test
functions.

For the \emph{existence} of a solution we use a classical particle
approximation where the details can be found in \cite{Gl12}. This also
allows to state some kind of universality law in the sense that the
genealogy processes of any kind of approximating individual based
branching particle systems with critical offspring distribution and
finite variance converges to the $\U$-valued Feller diffusion. The
\emph{uniqueness} of the solution is a consequence of a
\emph{Feynman-Kac duality} introduced in this paper. This Feynman-Kac
duality with the $\U$-valued Kingman coalescent is the main tool of
analysis. It shows that the law of the genealogy, i.e.\ the
\emph{$\U$-valued} Feller diffusion, can be seen as a reweighted law
of the $\U$-valued genealogy associated with the Kingman coalescent.
The law of the Feller diffusion genealogy at a fixed time $t$ puts
higher mass on those Kingman coalescent paths which have late mergers,
i.e., the Feller genealogy favors smaller distances than the Kingman
coalescent does.

A finer view of this aspect can be investigated via an $\U_1$-valued
process $\wh \mfU$ obtained by normalizing $\mfU$ by the population
size and \emph{conditioning} on the complete total mass process. This
allows to relate the Feller genealogy to a \emph{Fleming-Viot
  genealogy}, namely to a $\U_1$-valued \emph{time-inhomogeneous
  Fleming-Viot process}. This is similar to the structure of the
spatial multi-type branching processes in \cite{DG03}. Instead of
getting existence by particle approximation as in \cite{Gl12}, one can
also use alternatively the $\U_1$-valued Fleming-Viot process from
\cite{GPWmp13} and the $\R_+$-valued Feller diffusion to construct the
$\U$-valued one as functional of this pair.

This $\U_1$-valued Fleming-Viot process is for suitable paths of the
total population size process in \emph{duality} with a
\emph{time-changed coalescent}, giving an alternative proof for the
well-posedness of the martingale problem for the process $\mfU$. This
allows also for a \emph{strong} dual representation of the conditional
law of the genealogy given the population size process at least in the
\emph{critical} case.

We prove that the $\U$-valued Feller diffusion is in fact a
\emph{branching} process in a \emph{generalized sense}; see
\cite{ggr_GeneralBranching} for the concept and an alternative proof.
Roughly speaking it is shown that given the genealogy at some time $t$
and dividing the population in sub-populations the $s$-tops of the
genealogies of the sub-populations at time $t+s$ have evolved
independently according to the same mechanisms. The $s$-tops are the
recent genealogies seen from time $t+s$ backwards up to time $s$. This
leads to a \Levy{}-Khintchine formula (see \cite{infdiv}) whose
\emph{\Levy{} measure on $\U \setminus \{\ntree\}$} we identify here
explicitly. These results allow then to prove some properties of the
genealogies. In particular, we obtain explicit decompositions into
depth $h$ single ancestor subfamilies leading to a \emph{Cox point
  process representation} on specific semigroups in $\U$ namely in the
subspace induced by diameter-$2h$ ultrametric measure spaces. The main
focus for us is on identifying explicitly the \emph{genealogical
  \Levy{} measure} on $\U$ and identifying the law of a \emph{single
  ancestor subfamily} (which are the points which the \Levy{} measure
selects), for which we give three different representations: (1) via a
Yule tree, (2) via a time-inhomogeneous $\U$-valued coalescent, (3)
via an entrance law for the $\U$-valued process which we obtain from
the \emph{excursion law} of genealogies which itself will be
constructed. The analysis of these issues is more difficult in the
resampling world and has not been attacked so far successfully due to
the dependence between subfamilies.

In the theory of $\R_+$ or measure-valued branching processes the
property of infinite divisibility plays an important role allowing to
decompose the state in independent components and to derive Poisson
point process decompositions of the states; see e.g.\ Chapter~IV in
\cite{LG99} or Section~II.7 in \cite{zbMATH01827859}. This happens
also on the level of genealogy valued, i.e.\ $\U$-valued, objects as
well as we shall show in our framework (and already mentioned above).
Therefore it is possible to make use of the large number of results in
the $\R_+$ or measure-valued case to obtain information about the
subfamily decompositions as we call it, to study more model specific
questions. This has to be explored further in the future.

Since a critical branching process dies out almost surely we want to
obtain more detailed information on the genealogy conditioned on the
event of \emph{survival up to the present time}. This also gives the
structure of a single ancestor subfamily in the above description via
$\U$ appearing in the Cox point process representation.

For that purpose we first identify the \emph{$\U$-valued process
  conditioned to survive until time $T$} as solution of a well-posed
martingale problem for times $[0,T]$, \emph{correcting on the way an
  error in the $\R_+$-valued literature} \cite{LN68} where the
diffusion coefficient was not correct. We obtain here a
time-inhomogeneous super-critical state-dependent (generalized)
branching process with values in $\U$. This provides also tools to
investigate next the properties of genealogies of populations
surviving up to large times $T$, i.e.\ $T \to \infty$.

\medskip

\noindent
\textbf{\emph{(2) $\U$-valued Feller conditioned on long survival}}
The second part of this work is concerned with qualitative properties
of the genealogy for \emph{large} times and conditioned on the event
of \emph{survival for large times} or survival \emph{forever}. In the
case of Fleming-Viot models it is much easier to establish convergence
since we obtain as $t\to\infty$ a $\U_1$-valued equilibrium. For the
$\U$-valued Feller diffusion this is more subtle and we need to study
conditional law under survival forever as was done for the $\R$-valued
case.

To include rigorously the genealogies as \emph{$\U$-valued object}
into the analysis, we need to generalize concepts from the
$\R_+$-valued versions to $\U$-valued ones; see e.g.\
\cite{Overbeck93,Lamb02,Lamb07} for the former. Related ideas for
\emph{labeled} trees for \emph{individual based} models have been
considered by Chauvin, Rouault and Wakolbinger in \cite{CRW91} and by
Kallenberg in \cite{Kall77} and for labeled marked trees in the
Brownian snake construction of Le Gall in \cite{LG99}. Here the
$\R_+$-valued and individual based versions suggest conjectures for
$\U$-valued processes which can indeed be verified. This gives also
rise to a better understanding of \emph{spatial} models despite some
new features arising.

In order to study the behavior of the genealogy for $t \to \infty$
conditioned on survival we define further $\U$-valued Markov processes
related to the evolution of the genealogy of the Feller diffusion. In
particular we identify the evolving genealogy as $\U$-valued process
\emph{conditioned} to \emph{survive forever} ($\U$-valued $Q$-process)
respectively its \emph{size-biased}, i.e.\ $h$-transformed version
(Palm law for processes on $\U$). For both we identify its rescaling
limit as time tends to infinity to get a \emph{generalized
  quasi-equilibrium} on $\U$, where we represent these objects via
solutions of \emph{well-posed martingale problems} on $\U$ and via
$\U$-valued \emph{backbone representations}.

We also relate these process to the genealogical, that is, $\U$-valued
version of \emph{Evans' branching with immigration from an immortal
  line} \cite{Evans93}. This requires forms of the $\U$-valued
martingale problem with different features than treated so far. In
particular, we consider this process via a version with values in the
$[0,\infty)$-\emph{marked} ultrametric measure spaces
$\U^{[0,\infty)}$. The point here is to give the appropriate
\emph{analog of the Cox point process representation} of the Feller
diffusion for the $\U$-valued process conditioned to survive forever.

We will see that the macroscopic structure of the limit genealogy is
different when conditioning on survival \emph{forever} and
conditioning on survival up to a \emph{finite time horizon} $t$ in the
limit $t$ tending to infinity. While the latter conditioning is more
appropriate for a single population, the conditioning on survival
forever is appropriate for studying spatial model since it describes
the family of a \emph{typical individual} of the entire spatial
population; see Theorem~\ref{T:KOLMOGOROVLIMIT} and
Theorem~\ref{T.CLUMP}. \medskip

\noindent
\textbf{\emph{(3) Spatial and fossil populations}}

\noindent
\textbf{(i)} We finally discuss genealogies of spatial branching
models. In particular, for the \emph{super random walk} on a
geographic space $V$ we establish well-posedness of the martingale
problem of the corresponding process with values in the
\emph{$V$-marked} ultrametric measure spaces (the space $\U^V$ which
we recall later) which is the basis for analysis in future work.
Thanks to the branching property of the $\U^V$-valued Feller
diffusions the results we have for this process are important elements
in the analysis of \emph{spatial} branching models. These are the key
objects of interest for us. The branching property allows to view the
spatial models as a superpositions of \emph{independent} copies of
$\U$-valued processes which makes the analysis significantly easier
and more explicit than in spatial Fleming-Viot models, which remain to
be treated in detail.

As an example for the application of the analysis of the law of the
$\U$-valued Feller diffusion conditioned on survival forever we use it
to address for a spatial model, the super-random walk, the question of
the interplay between genealogy and spatial distribution of the mass
in the limit $t \to \infty$ and how dimension and the properties of
migration come into play.

For that we analyze asymptotically as time tends to infinity the
formation of clumps of high population size at rare spots in space in
the $\U^\Z$-valued \emph{strongly recurrent} super random walk. Here
we can obtain via \emph{two-scale} analysis a precise and explicit
\emph{asymptotic description}, first of genealogies of the population
in the rare spots of high population size and then in second step
scale the structure of the genealogy of one of the rare \emph{clumps}.
We describe the evolving genealogy of the clumps population as the
\emph{$\U^\R$-valued Dawson-Watanabe} (super) \emph{process}, which we
construct here. This analysis combines our results with results on the
$\mcM(\R)$-valued super random walk by Dawson and Fleischmann
\cite{DF88}. We indicate the situation in the cases $d=2$ and $d\geq
3$.

\noindent
\textbf{(ii)} Furthermore we connect our results to the literature
mentioned in the beginning and give the precise relation to the
\emph{continuum random tree} mentioned above as the generalized Yaglom
limit of the genealogy of the population alive before or at time $t$
(an $\M$-valued process, see \eqref{ag2} and the sequel), the so
called metric measure space valued \emph{fossil process}; see
\cite{GSWfoss}.

\paragraph{Outline}
The paper is organized as follows. In Section~\ref{ss.ummspace} we
give preparations, in particular we recall the suitable \emph{Polish
  state space} $\U$ for genealogies and recall a polar representation
of elements $\mfu \in \U$ which gives a decomposition of $\mfu$ into
its mass and genealogy components. Furthermore we introduce the
collection of consistent \emph{concatenation semigroup} structures on
$\U$, namely $\{(\U,\sqcup^h) : h>0\}$ and recall the notion of
\emph{infinite divisibility} on $\U$.

In Section~\ref{sec:results} we present the main definitions and
\emph{results} of this paper. More precisely, in
Section~\ref{ss.mainres} we study the \emph{basic structural
  properties} of the $\U$-valued Feller diffusion and its dual, namely
in Section~\ref{sss.martrep} its characterization via a martingale
problem and main properties, in ~\ref{FKdual} the duality theory, in
~\ref{branchmarkcox} the branching property and Cox cluster
representation and ~\ref{sss.entrance} excursion and entrance law and
the process conditioned to survive up to time $T$. In
Section~\ref{ss.longuval} we study the \emph{long time behavior} in
three parts. In Section~\ref{sss.longlim} the long time behavior is
considered via conditioning on survival up to time $T$ with
$T \to \infty$ and via size-biasing together with the decomposition in
an independent sum of the Kallenberg tree and a copy of the
untransformed process. In both cases we consider scaling limits. This
is further refined in Section~\ref{sss.1616} where we give a dynamical
representation of the Kallenberg tree via the $\U$-valued version of
Evans' process with immigration from an immortal line and in
Section~\ref{sss.backboncon} the backbone representation via
concatenation of a $\U$-valued Cox point process is obtained.
Section~\ref{ss.extension} applies and generalizes $\U$-valued Feller
diffusion model. Section~\ref{sss.genalspat} introduces the spatial
version, the \emph{genealogical super random walk} and
Section~\ref{sss.contrand} relates it to the \emph{continuum random
  tree}.

Sections~\ref{sec:duality}-\ref{s.proofext} are devoted to the
\emph{proofs} of these results. Section~\ref{sec:duality} proves
Theorem~\ref{T:DUALITY} on duality, Section~\ref{sec:Fel:ex} proves
Theorem~\ref{THM:MGP:WELL-POSED}. Section~\ref{sec:branching} proves
finer properties of the processes, such as Theorem~\ref{T:BRANCHING}
(branching property). Finally Section~\ref{sec:Kolmogorovlimit}
contains the proof of Theorem~\ref{T:KOLMOGOROVLIMIT} and other
results on the large time limit (on the quasi equilibrium and Yaglom
limit). Section~\ref{s.proofext} gives proofs for the extensions to
spatial models and to the fossil process allowing to exhibit the
relation to the CRT.

Some more technical points are collected in the appendix in
Sections~\ref{sec:comp-diff-coeff}, \ref{app.facts},
\ref{sec:infin-divis-xxx} and~\ref{sec:appr-solut-mart}. In
Section~\ref{sec:comp-diff-coeff} we give the calculation correcting
the scaling limit result in \cite{LN68}. In Sections~\ref{app.facts}
and \ref{sec:infin-divis-xxx} we collect some consequences of infinite
divisibility on $\U$ as studied in \cite{infdiv} and
\cite{ggr_GeneralBranching}. Finally, in
Section~\ref{sec:appr-solut-mart} we briefly discuss approximation of
solutions of certain martingale problems.

\section[Ultrametric measure spaces and concatenation
semigroups]{Ultrametric measure spaces and
  concatenation semigroups}
\label{ss.ummspace}

In this section we introduce the \emph{state space} $\U$ whose
elements can be interpreted as genealogies of population processes and
recall the \emph{topological semigroup structure} of that space
developed in \cite{infdiv} and \cite{ggr_GeneralBranching}.
Furthermore we present tools and objects which we use to deal with
random variables on these spaces. In particular we introduce
\emph{polynomials} as the basic test functions, the \emph{polar
  decomposition} of the states and \emph{concatenation semigroups}.
For details we refer the reader to \cite{GPW09} and \cite{infdiv}.

\subsection{State spaces: the topological spaces
  \texorpdfstring{$\U$}{U} and \texorpdfstring{$\U_1$}{UV}}
\label{sss.topspa}

We describe a population by a set $U$ of its currently alive
individuals together with its genealogy by giving the
\emph{genealogical distances} $r(\cdot,\cdot)$ of pairs of individuals
in $U$, and by a finite measure $\mu$ on the Borel sets of $U$. Here,
the genealogical distance means the tree distance on the genealogical
tree of a branching population, i.e.\ twice the time to the \emph{most
  recent common ancestor} of a given pair of individuals.

\begin{definition}[Ultrametric measure spaces]
  A \label{D.ummsp} triple $(U,r,\mu)$ is called an \emph{ultrametric
    measure space}, if $(U,r)$ is a complete and separable ultrametric
  space and $\mu$ is a finite measure on its Borel-$\sigma$-algebra.

  Ultrametric measure spaces $(U,r,\mu)$ and $(U',r',\mu')$ are called
  \emph{equivalent} if there is an isometry $\varphi$ between the
  supports of $\mu$ and $\mu'$ that satisfies
  $\mu' = \varphi_\ast \mu$. Here $\varphi_\ast \mu$ denotes the image
  measure of $\mu$ under $\varphi$. The equivalence class of
  $(U,r,\mu)$ is denoted by $[U,r,\mu]$. The sets of \emph{equivalence
    classes} of \emph{ultrametric measure spaces} and more
  specifically \emph{ultrametric probability measure} spaces are
  denoted by
  \begin{align}
    \label{tv1}
    \bbU & \coloneqq \bigl\{[U,r,\mu]: (U,r,\mu) \text{ an ultrametric
           measure space with finite measure $\mu$}\bigr\},\\
    \label{tvFV}
    \U_1 & \coloneqq \bigl\{[U,r,\mu]: (U,r,\mu) \text{ an
           ultrametric
           measure space with probability measure $\mu$}\bigr\}.
  \end{align}
  We refer to the elements of $\bbU$ and $\bbU_1$ as genealogies.
\end{definition}

For $a \in \R_+$ and $\mfu = [U,r,\mu] \in \U$ we write
\begin{align}
  \label{eq:multaU}
  a \cdot \mfu \coloneqq [U,r, a \mu].
\end{align}
The null measure on a metric space will be usually denoted by $0$. On
$\U$ we define the following specific elements
\begin{align}
  \label{eq:null-one-tree}
  \ntree \coloneqq [\{1\},r,0] \quad \text{and} \quad \mfe \coloneqq
  [\{1\},r,\delta_1].
\end{align}
Here $r$ is necessarily the zero metric on $\{1\}$. We refer to
$\ntree$ and $\mfe$ as the \emph{null element} respectively the
\emph{unit element}. Note that with the operation defined in
\eqref{eq:multaU} we have $\ntree = 0 \cdot \mfe$. Furthermore note
that for any $x$ in a set $X$ equipped with a metric $r$ we also have
$\mfe = [X,r,\delta_x]$.

\begin{remark}[Ultrametric spaces and $\R$-trees]\label{r.ultsp}
  Any ultrametric space $(U,r)$ can be embedded isometrically into an
  $\R$-tree such that the \emph{leaves} of the $\R$-tree correspond to
  the elements of $U$; cf.\ Remark~2.2 in \cite{GPWmp13}. We call this
  object the associated ancestral tree. Then the distance of two
  leaves is given by the sum of the distances to the most common
  ancestor.
\end{remark}

Now we introduce some objects needed for the definition of the
topology on $\U$ which turns it into a Polish space.

\begin{definition}[ultrametric distance matrices, distance matrix
  distributions]\label{D.umdm}
  % \leavevmode\\
  Define the set of \emph{ultrametric distance matrices} of order
  $n\ge2$ by
  \begin{align}
    \label{tv2}
    [0,\infty)^{\binom{n}{2}} \coloneqq \bigl\{(r_{ij})_{1\le i<j\le n} :
    r_{ij}\ge 0\; \,\forall 1\le i<j \le n \text{ and } r_{ij}\le r_{ik}
    \vee r_{kj}\; \,\forall 1\le i<k<j\le n\bigr\}.
  \end{align}
  For elements $u_1,\dots,u_n$ of a metric space $(U,r)$, writing
  $\underline u = (u_1,\dots,u_n)$ we define
  $\dr(\underline{u}) \in [0,\infty)^{\binom{n}{2}}$ by
  \begin{align}
    \label{eq:dr}
    \dr(\underline{u}) = \bigl(r(u_i,u_j)_{1\le i < j \le n} \bigr)
  \end{align}
  and define the mapping
  \begin{align}
    \label{e638}
    R^{(n)}: U^n \rightarrow [0,\infty)^{\binom{n}{2}}, \quad R^{(n)}
    (\underline{u})= \dr(\underline{u}).
  \end{align}
  For $\mfu = [U,r,\mu] \in \U$ and integers $n \ge 2$ we define the
  \emph{distance matrix measure} $\nu^{n,\mfu}$ by
  \begin{align}
    \label{eq:e638dmd}
    \nu^{n,\mfu} \coloneqq (R^{(n)})_\ast \mu^{\otimes n},
  \end{align}
  that is, $\nu^{n,\mfu}$ is the image measure of $\mu^{\otimes n}$
  under $R^{(n)}$.
\end{definition}

An important set of functions on $\U$ is the set of polynomials. For a
set $A$ we denote by $\mathrm{b}\mcB(A)$ the set of bounded measurable
real-valued functions on $A$.

\begin{definition}[Polynomials]
  \label{D.polyn}
  For an integer $n\ge 0$ and
  $\varphi\in \mathrm{b}\mcB([0,\infty)^{\binom{n}{2}})$ (for $n=0,1$ the
  function $\varphi$ is assumed to be constant) we define the
  function $\Phi=\Phi^{n,\varphi}: \bbU\to\R$ as follows
  \begin{align}
    \label{e:0606131101}
    \mfu =
    [U,r,\mu] \mapsto
    \begin{cases}
      \Phi^{n,\varphi} (\mfu)  \coloneqq
      \langle\varphi,\nu^{n,\mfu}\rangle,
      & n \ge 2,\\
      \Phi^{1,\varphi}(\mfu)  \coloneqq c \mu(U), & n = 1,\; \varphi
      \equiv c,  \\
      \Phi^{0,\varphi}(\mfu)  \coloneqq c, & n=0,\; \varphi \equiv c.
    \end{cases}
  \end{align}
  The smallest non-negative integer $m$, for which there is $\varphi$
  satisfying \eqref{e:0606131101} with $n=m$ is called the
  \emph{degree} of the polynomial $\Phi$. Whenever we need to stress
  the dependence on the degree $m$ and $\varphi$ we write
  $\Phi^{m,\varphi}$. The set of polynomials on $\U$ of degree $m$ is
  denoted by $\Pi_m$. Furthermore we set
  \begin{align}
    \label{eq:Pi-total}
    \Pi \coloneqq \bigcup_{m \in \N_0} \Pi_m \quad \text{and} \quad
    \wh\Pi \coloneqq \Pi|_{\U_1},
  \end{align}
  i.e.\ $\Pi$ is the set of polynomials of all degrees and $\wh \Pi$
  is the set of polynomials restricted from $\U$ to $\U_1$. The
  elements of $\wh\Pi$ will be denoted by $\wh\Phi$ with the same
  notational conventions concerning the degree $m$ and the function
  $\varphi$.

  For a subclass $\mcC$ of bounded measurable functions on the space
  of distance matrices we write
  \begin{align}
    \label{tv3}
    \Pi(\mcC) \coloneqq \bigcup_{n \in \N_0} \bigl\{\Phi^{n,\varphi} :
    \varphi\in\mcC \cap \mathrm{b}\mcB([0,\infty)^{\binom{n}{2}})\bigr\}.
  \end{align}
  Here, again for $n\in \{0,1\}$ the set
  $\mathrm{b}\mcB([0,\infty)^{\binom{n}{2}})$ consists of constant
  functions.
\end{definition}

We note that the set of polynomials is a linear space. Furthermore
every polynomial can be viewed as a monomial. We denote by $\mathcal
C_b = \mathcal C_b([0,\infty)^{\binom{n}{2}})$ the set of bounded
continuous functions on $[0,\infty)^{\binom{n}{2}}$.

\begin{definition}[Topology]
  \label{D.topo}
  The topology on $\U$ induced by $\Pi(\mathcal C_b)$ is called the
  \emph{Gromov weak topology}. By induced we mean for convergence
  $\mfu_n \to \mfu$ on $\U$ we require $\Phi(\mfu_n) \to \Phi(\mfu)$
  on $\R$ for all $\Phi \in \Pi(\mathcal C_b)$.
\end{definition}
\begin{remark}[Polish metrizable space]
  \label{r.d.topo}
  In \cite{GPW09} it is shown that $\U_1$ equipped with the Gromov
  weak topology is a Polish space which is metrizable by the
  Gromov-Prohorov metric. Here we do \emph{not} restrict to
  probability measures but the results from \cite{GPW09} can be
  extended to $\U$. For \emph{two} extensions of the Gromov-Prohorov
  metric to $\U$ we refer to Section~2.4 in \cite{Gl12}. We use here
  the topology and the metric which are called in \cite{Gl12} extended
  Gromov weak topology respectively extended Gromov-Prohorov metric.

  Let us briefly recall this metric here. Let $(Z,r_Z)$ be a complete
  and separable metric space. For a Borel subset $A$ of $Z$ set
  $A^{\varepsilon} = \{z \in Z: \inf_{y \in A} r_{Z}(z,y)
  <\varepsilon\}$. The Prohorov distance of finite measures $\mu$ and
  $\nu$ on the Borel-$\sigma$-Algebra on $Z$ is defined by
  \begin{align}
    \label{eq:32}
    d_{\Pro} (\mu,\nu) = \inf\{\varepsilon>0 : \mu(A)\le
    \nu(A^\varepsilon) + \varepsilon \text{ and } \nu(A)\le
    \mu(A^\varepsilon) + \varepsilon, \text{ for all closed $A \subset Z$}\}.
  \end{align}
  Then the extended Gromov-Prohorov metric is a direct extension of
  the metric introduced in \cite{GPW09}. More precisely for
  $\mfu_1=[U_1,r_1,\mu_1], \mfu_2=[U_2,r_2,\mu_2]\in \U$ it is defined
  by
  \begin{align}
    \label{eq:31}
    d_{\GP} (\mfu_1,\mfu_2) = \inf d_{\Pro} (\mu_1 \circ
    (\varphi_1)^{-1}, \mu_2 \circ (\varphi_2)^{-1}),
  \end{align}
  where the infimum is taken over all isometric embeddings
  $\varphi_1: U_1 \to Z$ and $\varphi_2: U_2 \to Z$ into a common
  complete and separable metric space $(Z,r_Z)$. By Proposition~2.4.12
  in \cite{Gl12} the metric space $(\U, d_{\GP})$ is complete and
  separable. Note that in this metric space there is only one null
  space which is given by the null element from
  \eqref{eq:null-one-tree}.

  For brevity we will refer throughout the paper to $d_{\GP}$ as the
  Gromov-Prohorov metric on $\U$ and call the corresponding topology
  the Gromov-weak topology.
\end{remark}

\begin{remark}[Non-local compactness]
  \label{rem:Unloc.comp}
  The spaces $\U$ and $\U_1$ are \emph{Polish} spaces when equipped
  with the \emph{Gromov weak topology}; see Remark~\ref{r.d.topo}.
  However, the spaces are neither locally compact nor
  $\sigma$-compact. In particular the Stone-Weierstraß theorem does
  not apply and in fact the set of polynomials is a countable but not
  a dense subset of $\mathcal{C}_b (\U)$; see Remark~2.6 in
  \cite{Loehr2013} for an argument that $\wh \Pi$ is not dense in
  $\mathcal{C}_b(\U_1)$. (Here $\mathcal{C}_b (\cdot)$ denotes again
  the continuous bounded functions on the particular set.) For that
  reason the Feller property (see Definition~\ref{rk:Fellert}) and the
  strong Markov property are somewhat subtle to establish.
\end{remark}

The following lemma can be shown combining Proposition~2.6 from
\cite{GPW09}, Proposition~4.6 of Chapter~3 in \cite{EK86} and the
discussion around equation (4.21) after that proposition in
\cite{EK86}.
\begin{lemma}[Convergence criterion]
  \label{lem:Pi:separating}
  The algebra generated by $\Pi$ is separating, on
  \begin{align}
    \label{tv4}
    \wt{\mcM} = \Bigl\{P \in \mcM_1(\bbU):
    \limsup_{K\to \infty} \frac{1}{K} \Bigl(\int \bar{\mfu}^{K} \,
    P (\dx\mfu) \Bigr)^{1/K}<\infty\Bigr\}.
  \end{align}
  Furthermore, this algebra is also convergence determining, whenever
  the limiting point is in $\wt{\mcM}$.
\end{lemma}
Note that $\wt{\mcM}$ is the set of all distributions on $\bbU$ which
are uniquely characterized by all moments of the total mass; see
Theorem~3.2.9 and Corollary~3.2.10 in \cite{D93}.

\begin{remark}
  If \label{rem:feller_momens} $(Z_t)_{t \ge 0}$ is a an $\R_+$-valued
  Feller diffusion, i.e.\ a solution of the stochastic differential
  equation $\dx Z_t = \sqrt{b Z_t}\, \dx B_t$, with $Z_0 >0$ and $b>0$
  and $B=(B_t)_{t\ge 0}$ a Brownian motion (cf.\ Remark~\ref{r.697}),
  then from the form of its Laplace transform (see \eqref{e1260}) one
  can easily see that for any $t \ge 0$ it is not only defined for
  positive $\lambda$ but is actually analytic in $\lambda$ in a
  neighborhood of $0$. Thus, for any $t \ge 0$ the law of $Z_t$ is
  determined by its moments.
\end{remark}

Thus, the above lemma and the fact that the set $\wh\Pi$ is measure
and convergence determining on $\U_1$; see Corollary~3.1 in
\cite{GPW09}, implies the following corollary.

\begin{corollary}[Law determining test functions]\label{cor.681}
  Let $\mfU=(\mfU_t)_{t\ge 0}$ be a stochastic process on $\U$ whose
  total mass process is given by Feller's continuous state branching
  diffusion starting in some $x_0 \in [0,\infty)$. Then, for any
  $t \ge 0$ the distribution of $\mfU_t$ is contained in $\wt\mcM$.

  Furthermore the set $\Pi$ of polynomials is law determining and
  convergence determining on $\mcM_1(\U)$. In particular there exists
  a countable measure determining set of test functions.
\end{corollary}

\subsection{Polar decomposition of elements of
  \texorpdfstring{$\U$}{U}}
\label{sec:polar-repr-stat}

A useful point of view on $\U$ is via the \emph{polar decomposition}
of its elements. More precisely, one can decompose a state
$\mfu = [U, r, \mu] \in \U$ in its \emph{total mass} and its
\emph{pure genealogy} parts: In the case $\mu(U)>0$ we set
\begin{align}
  \label{e998}
  \bigl(\bar \mfu,\hat \mfu \bigr) \coloneqq
  \bigl(\bar \mu,\left[U,r,\hat \mu \right]\bigr),
  \quad \text{where} \quad
  \bar\mu \coloneqq \mu(U) \quad \text{and} \quad \hat \mu
  \coloneqq \frac{\mu}{\bar \mu}.
\end{align}
There is obviously a bijection between $(0,\infty) \times \U_1$ and
$\U_{>0} = \{\mfu = [U,r,\mu] \in \U: \mu(U) >0\}$. The space
$\U_{>0}$ equipped with the metric induced by the product metric on
$(0,\infty) \times \U_1$ is not complete. Its completion is introduced
in Section~2.4.2 in \cite{Gl12}. In the completion the elements
$\mfu \in \U$ with $\bar \mu=0$ are ``identified'' with some
$(0,\tilde \mfu) \in [0,\infty) \times \U_1$. In particular in this
completion there are uncountably many elements with total mass $0$ and
convergence of the distance to $0$ requires both, the convergence of
the distance of the total masses to $0$ and the convergence of the
distances of the genealogies equipped with probability measures on
$\U_1$ to $0$.

In the present paper we will work with the Gromov-weak topology on
$\U$ induced by the metric $d_\GP$ from \eqref{eq:31}. This is weaker
than the topology on the completion mentioned above in
Remark~\ref{r.d.topo} because for a sequence
$(\bar\mfu_n,\hat\mfu_n)_{n=0,1,\dots}$ we do not require that
$\bar\mfu_{n} \to 0$ implies the convergence of $\hat\mfu_{n}$ to some
limit $\hat\mfu \in \U_1$. On $\U_{>0}$ both topologies do
\emph{coincide}. For details we refer the reader to Section~2.4 in
\cite{Gl12}. In many cases we can also say what happens in the
stronger topology, however in spatial models this gets intricate and
technical.

Using the polar decomposition we can also define the \emph{normalized
  distance matrix distribution} of elements of
$\U \setminus \{\ntree\}$ as
\begin{align}
  \label{dmd-normalized}
  \hat\nu^{n,\hat\mfu} \coloneqq (R^{(n)})_\ast \hat \mu^{\otimes n},
\end{align}
where $R^{(n)}$ is as defined in \eqref{e638}. Of course we then have
$\nu^{n,\mfu} = \bar \mfu^n \hat\nu^{n,\hat\mfu}$.

\begin{remark}[Polynomials on product spaces and polar decomposition
  of elements of $\U$]\label{rem:poldeg}
  % \leavevmode \\
  We have stated that polynomials are separating and convergence
  determining on $\U$. Polynomials satisfy these properties on $\R_+$
  as well. Using standard arguments for measure determining functions
  on product spaces one would consider monomials of the form
  $\Phi^{m,\varphi}$, with $\varphi \in \mathcal
  C_b([0,\infty)^{\binom{n}{2}},\R)$ so that
  $\Phi^{m,\varphi}(\mfu)=\bar\mu^m \int \varphi \,
  \dx\hat\mu^{\otimes n}$ for $m,n \in \N_0$ and expect this to be a
  measure determining set on $\R_+ \times \U_1$ restricted to $\wt
  \mcM$. Lemma~\ref{lem:Pi:separating} shows that we do not need all
  combinations of $m, n \in \N_0$ to separate points. The reason is
  that $\U \subset \R_+ \times \U_1$, but contains only elements of
  the form $\bar \mfu \cdot \wh \mfu$, whish is a subset of $\R_+
  \times \U_1$.
\end{remark}

\subsection[Concatenation semigroup
  \texorpdfstring{$(\U,\sqcup^h)$}{(U,sqcup)},
  \texorpdfstring{$h$}{h}-truncation and infinite divisibility on
  \texorpdfstring{$\U$}{U}]{Concatenation semigroup
  \texorpdfstring{$(\U,\sqcup^h)$}{(U,sqcup)},
  \texorpdfstring{$h$}{h}-truncation and \\infinite divisibility on
  \texorpdfstring{$\U$}{U}}
\label{sss.concat}

It is well known that thanks to the branching property the laws of
$\R_+$-valued branching processes are infinitely divisible. It turns
out that the same is true for the $\U$-valued Feller diffusion. To
this end we need a semigroup structure on $\U$ (see \cite{infdiv})
with respect to an operation which we call concatenation.

\paragraph{Concatenation semigroup $(\U,\sqcup)$}
Consider a representative of $\mfu = [U,r,\mu] \in \U$. For a given
$h>0$ we want to decompose the population represented by $U$ into
\emph{subfamilies} in which the \emph{time to the most recent common
  ancestor} is less than $h$, i.e., the genealogical distance between
pairs of individuals inside each of the subfamilies is smaller than
$2h$. Since we work with ultrametric spaces we obtain a disjoint
decomposition of the whole space in a collection of subspaces with
diameters strictly less than $2h$. We call the equivalence classes of
such spaces \emph{$h$-trees}. The \emph{$h$-trees} themselves can be
connected with each other to form new spaces whose equivalence classes
we call \emph{$h$-forests}. Both these objects are elements of $\U$.

Using the pairwise distance matrix distribution we can formally define
the objects which we just described in words as follows. The subset of
\emph{$h$-trees} in $\U$ is defined by
\begin{align}
  \label{e.tr48}
  \U(h) \coloneqq \{\mfu \in \U: \nu^{2,\mfu} ([2h,\infty)) = 0 \},
\end{align}
and the subset of \emph{$h$-forests} in $\U$ is defined by
\begin{align}
  \label{e.tr46}
  \U(h)^\sqcup \coloneqq
  \{\mfu \in \U : \nu^{2,\mfu}((2h,\infty)) = 0\}.
\end{align}
Obviously we have $\U(h) \subset \U(h)^\sqcup$.

\medskip
For $\mfu_1, \mfu_2 \in \U(h)^\sqcup$ with $\mfu_i = [U_i,r_i,\mu_i]$,
$i=1,2$ we define the \emph{$h$-concatenation} of $\mfu_1$ and
$\mfu_2$ by
\begin{align}
  \label{e.tr47}
  \mfu_1 \sqcup^h \mfu_2 \coloneqq
  [U_1 \uplus U_2, r_1 \sqcup^h r_2, \mu_1 + \mu_2].
\end{align}
Here $\uplus$ denotes the disjoint union of sets and
$r_1 \sqcup^h r_2$ is a metric on $U_1 \uplus U_2$ defined by
\begin{align}
  \label{grx53}
  r_1 \sqcup^h r_2|_{U_1 \times U_1} = r_1, \quad r_1 \sqcup^h
  r_2|_{U_2 \times U_2}
  = r_2, \quad r_1 \sqcup^h r_2|_{U_1 \times U_2} \equiv 2h.
\end{align}
Finally, $\mu_1 + \mu_2$ should be interpreted as
$\wt \mu_1 + \wt \mu_2$ on $U_1 \uplus U_2$ where $\wt \mu_i$,
$i \in\{1,2\}$ denotes the extension of $\mu_i$ to $U_1 \uplus U_2$.
We equip $\U(h)^\sqcup$ and $\U(h)$ with the relative topology from
$\U$. In particular, $\U(h)^\sqcup$ is a \emph{Polish space}.

Note that \emph{$h$-concatenation} $\sqcup^h$ is an associative and
commutative operation acting on elements of $\U(h)^\sqcup$. Thus, for
every $h > 0$, $(\U(h)^\sqcup, \sqcup^h)$ is a \emph{topological
  semigroup} with the neutral element $\ntree$.

\medskip
We define the \emph{$h$-top} of $\mfu =[U,r,\mu]\in \U$ as
\begin{align}
  \label{e781}
  \lfloor \mfu\rfloor (h) \coloneqq [U,r \wedge 2h,\mu] \in
  \U(h)^\sqcup.
\end{align}
In Theorem~1.13
% \ref{Inf-p.delphic}
in \cite{infdiv} it is shown that for any $\mfu \in \U$ and any $h>0$
there is a unique (up to order) sequence $(\mfu_i : i \in I_h)$
indexed by a (possibly finite) set $I_h$, such that
\begin{align}
  \label{tv14}
  \lfloor \mfu\rfloor (h) = \mathop{{\bigsqcup}^h}_{i\in I_h}
  \mfu_i,\; \text{ with $\mfu_i \in \U(h) \setminus\{\ntree\}$, $i \in I_h$}.
\end{align}

Besides the $h$-top of $\mfu \in \U$ we will also need the
\emph{$h$-trunk} denoted by $\lceil \mfu\rceil(h)$. For
$\mfu = [U,r,\mu]$ consider the decomposition of
$\lfloor \mfu\rfloor (h)$ as in \eqref{tv14} with
$\mfu_i = [U_i,r_i,\mu_i] \in \U(h) \setminus \{\ntree\}$. Then the
$h$-trunk $\lceil \mfu\rceil(h)$ of $\mfu$ is defined as the
equivalence class of the ultrametric measure space $(I_h,r^*,\mu^*)$,
i.e.\
\begin{align}
  \label{def:trunk}
  \lceil \mfu\rceil(h) \coloneqq [I_h,r^*,\mu^*],
\end{align}
where $I_h$ is as above and
\begin{align}
  \label{def:trunk-metr-meas}
  \begin{split}
    r^*(i,j) & = \inf\{(r(u,v) - 2h) \vee 0 : u \in U_i, v\in
    U_j\}, \; i,j \in I_h, \\
    \mu^*(\{i\}) & = \mu_i(U_i), \; i \in I_h.
  \end{split}
\end{align}
We call the mapping $T_h: \U \to \U(h)^\sqcup$ defined by
\begin{align}
  \label{eq:hTrunc}
  T_h (\mfu) = \lfloor \mfu \rfloor (h)
\end{align}
the \emph{$h$-truncation}. This allows us to turn $\U$ into a
semigroup w.r.t.\ a collection of operations $\{\sqcup^h : h>0\}$. We
extend the operation $\sqcup^h$ to all of $\U$ by setting
\begin{align}
  \label{e750}
  \mfu_1 \sqcup^h \mfu_2\coloneqq  T_h(\mfu_1) \sqcup^h T_h(\mfu_2).
\end{align}
This way we obtain a collection of \emph{topological semigroups}
$\{(\U, \sqcup^h), h \ge 0\}$, which is \emph{consistent} under
$T_h$, i.e.\ for
$h>h': T_{h'}(\mfu_1 \sqcup^h \mfu_2)=T_{h'}(\mfu_1)
\sqcup^{h'}T_{h'}(\mfu_2); \mfu_1,\mfu_2 \in
\U(h)^\sqcup$.

The polynomials which we introduced in \eqref{e:0606131101} fit to
this structure if for $\Phi = \Phi^{m,\varphi} \in \Pi$ and $h>0$ we
introduce the corresponding $h$-\emph{truncated polynomial} by
\begin{align}
  \label{eq:pol:trunc}
  \Phi_h\coloneqq \Phi^{m,\varphi_h} \quad \text{with} \quad \varphi_h
  (\dr) \coloneqq \varphi(\dr) \prod_{1\le i < j \le m} \ind{r_{ij} < 2h}.
\end{align}
With this notation we have
\begin{align}
  \label{e755}
  \Phi_h(\mfu_1 \sqcup^h \mfu_2) = \Phi_h (\mfu_1) + \Phi_h (\mfu_2)
  \text{ for all } \mfu_1,\mfu_2 \in \U(h)^\sqcup.
\end{align}
Similar identity holds for all $\Phi^{n,\varphi}$ with $\varphi$ which
have support in $[0,2h)^{\binom n 2}$; see Theorem~1.27 in
\cite{infdiv}.

\paragraph{Infinite divisibility}
Using the structures $\{(\U(h)^\sqcup, \sqcup^h),T_h: h>0\}$
introduced above one can obtain the \emph{\Levy{}-Khintchine
  representation} of \emph{infinitely divisible} $\U$-valued random
variables. Here infinite divisibility means that for all $h>0$ and
$n\in \N$ the $h$-truncations can be represented as $h$-concatenations
of $n$ i.i.d.\ $\U(h)^\sqcup$-valued random variables. This notion was
introduced in \cite[Section~1.5]{infdiv}.

On a Polish space $E$, where we have defined bounded sets together
with a point infinitely far away $\mathcal M^\# (E)$ denotes the set
of boundedly finite measures on $E$. Here we will consider
$E = \U\setminus \{\ntree\}$ with the point $\ntree$ infinitely far
away.

According to Theorem~1.37 in \cite{infdiv} an infinitely divisible
random ultrametric measure space $\mfU$ has a \Levy{}-Khintchine
representation of its Laplace functional. More precisely, there exists
a unique measure
$\Lambda_\infty \in \mcM^\#(\U \setminus \{\ntree\})$, called
\emph{\Levy{} measure} (also often referred to as \emph{canonical
  measure}) with
$\int (\bar{\mfu} \wedge 1) \Lambda_\infty (\dx \mfu) < \infty$ so
that for any $h \in (0,\infty)$ we have
\begin{align}
  \label{ag2inf}
  % -\log L_{\mfU}(\Phi_{h}) \coloneqq
  -\log \E\bigl[
  \exp\bigl(-\Phi_{h}(\mfU)\bigr)\bigr]
  & = \int_{\U(h)^\sqcup \setminus
    \{\ntree\}} \bigl(1-e^{-\Phi_{h}(\mfu)} \bigr)\, \Lambda_h (\dx \mfu) \quad
    \forall \, \Phi \in \Pi_+,\\
  \intertext{where $\Lambda_{h} \in
  \mcM^\#(\U(h)^\sqcup \setminus \{\ntree\})$ is defined by  }
  \label{e.ar1inf}
  \Lambda_{h}(\dx \mfu)
  & = \int_{\U \setminus \{\ntree\}}
    \Lambda_\infty(\dx\mfv) \, \ind{\lfloor \mfv\rfloor (h) \in \dx\mfu}.
\end{align}
We say that $\mfU$ is \emph{$t$-infinitely divisible}, if for all
$h \le t$ there is a unique
$\Lambda_t \in \mcM^\#(\U(t)^\sqcup \setminus \{\ntree\})$ so that
$\mfu \mapsto (\bar{\mfu}\wedge 1)$ is integrable with respect to
$\Lambda_t$, and for all $h\in (0,t]$ equations \eqref{ag2inf} and
\eqref{e.ar1inf} hold with $\Lambda_t$ instead of $\Lambda_\infty$. In
either case, for any $h$ in the suitable range, we have
\begin{align}
  \label{ag2binf}
  \Lambda_h \bigl(\U(h)^\sqcup \setminus \{\ntree\}\bigr)
  = - \log \P(\bar \mfU =0) \in [0,\infty].
\end{align}
We refer to $\Lambda_h$ as the \emph{$h$-\Levy{} measure} and to
$\Lambda_\infty$ as the \emph{\Levy{} measure} of $\mfU$. Formula
\eqref{ag2inf} means that we can represent $\lfloor\mfU\rfloor(h)$ via
an inhomogeneous \emph{Poisson point process} $N(\Lambda_h)$ on
$\U(h)^\sqcup \setminus \{\ntree\}$ with intensity measure $\Lambda_h$
as
\begin{align}
  \label{e902}
  \lfloor \mfU\rfloor (h) =
  {\mathop{\bigsqcup\nolimits^h}\limits_{\mfu \in N(\Lambda_h)}} \mfu.
\end{align}
That means that for each $h>0$ the state of the $\U$-valued Feller
diffusion can be decomposed into depth-$h$ \emph{single ancestor}
subfamilies where the subfamilies are given by i.i.d.\ $\U(h)$-valued
random variables.

\section{Concepts and main results}
\label{sec:results}

In this section we formulate in three subsections the \emph{main
  results} on the $\U$-valued Feller diffusion. In
Subsection~\ref{ss.mainres} we present the \emph{martingale
  characterization} and discuss the relation to the $\U_1$-valued
\emph{Fleming-Viot process} from \cite{GPWmp13} and various \emph{dual
  representations} via $\U$-valued coalescents. Here we also give
descriptions of structural properties such as branching property in a
version for $\U$-valued processes and we use the \emph{semigroup
  structure} $(\U,\sqcup)$ to establish the \emph{\Levy{}-Khintchine
  formula} on $\U$ allowing to describe the depth-$h$ subfamily
decompositions via the \emph{Cox cluster representation}. In
Subsection~\ref{ss.longuval} we study the \emph{long time behavior}
and we focus on populations \emph{conditioned on survival} for long
times or \emph{size-biased} populations. Finally, in
Subsection~\ref{ss.extension} we consider extensions of the results to
the spatial case, i.e.\ to the \emph{super random walk}, and discuss
versions of our processes with \emph{fossils} which for any $t \ge 0$
include \emph{all individuals} alive at times $s \le t$ and the
relation to the \emph{continuum random tree}.

\subsection{Results 1: Genealogies and the
  \texorpdfstring{$\U$}{U}-valued Feller diffusion}
\label{ss.mainres}

The first group of results
(Theorems~\ref{THM:MGP:WELL-POSED}-~\ref{TH.MARTU}) includes the
martingale problem characterization of the $\U$-valued Feller
diffusion and of its polar decomposition, the Feynman-Kac and
conditional dualities, and the generalized branching property with the
corresponding Cox cluster representation. Furthermore we study the
entrance law, the excursion law and identify the $\U$-valued genealogy
process of the population conditioned to survive up to the present
time.

\subsubsection{Martingale problem and identification of population
  size and genealogy processes}
\label{sss.martrep}

Here we introduce the $\U$-valued Feller diffusion as solution of a
martingale problem and characterize the population size process and
genealogy (conditioned on the population size) process as two $\R_+$
respectively $\U_1$-valued Markov processes in their own right. First
we recall the classical Feller diffusion.

\begin{remark}[Feller's branching diffusion on $\R_+$]
  \label{r.697}
  \emph{Feller's branching diffusion with immigration} is an
  $\R_+$-valued stochastic process which solves the stochastic
  differential equation
  \begin{align}
    \label{eq:FelBDiff-ab}
    \dx Z_t = c\,\dx t + aZ_t \,\dx t + \sqrt{b Z_t}\, \dx B_t, \quad
    \text{with } Z_0 >0.
  \end{align}
  Here $b >0$ the \emph{diffusion coefficient} arising from the
  \emph{individual branching rate} of the corresponding particle
  approximation, $a \in \R$ is the \emph{sub-/super-criticality
    coefficient}, and $c \ge 0$ is the \emph{immigration} rate. In the
  cases $a<0$, $a=0$ and $a>0$ the branching diffusion is called
  \emph{sub-critical}, \emph{critical} respectively
  \emph{super-critical}. We will refer to the case $a \ne 0$ as the
  \emph{non-critical} case and call $a$ the \emph{non-criticality}
  coefficient if the sign of $a$ is not important. In the case $c =0$
  the process is referred to as \emph{Feller's branching diffusion}
  and this is the process (together with its genealogy) we will mostly
  deal with in this paper. The process with immigration will appear in
  a special form when we condition Feller's branching diffusion on
  survival forever. This is well known from the classical
  $\R_+$-valued branching processes theory.

  The process $Z=(Z_t)_{t\ge 0}$ can be obtained as the many
  individuals -- small mass -- fast branching limit of sequences,
  called $Z^{(N)}=(Z_t^{(N)})_{t \ge 0}$ below, of individual based
  models. For instance, consider a sequence of continuous time
  Galton-Watson processes $X^{(N)}=(X^{(N)}_t)_{t\ge 0}$ with
  branching at rate $b$, and offspring distribution with mean
  $1+\frac{a}{bN}+o(\frac{1}{N})$ and variance $1+o(1)$ as
  $N\to\infty$. Furthermore assume that at rate $Nc$ a new individual
  immigrates into the population. Rescale mass, and speed up time by
  setting $Z_t^{(N)} = \frac{1}{N} X_{Nt}^{(N)}$. Then, provided that
  the initial conditions $Z_0^{(N)}$ converge weakly to $Z_0$ as
  $N\to\infty$, the sequence $Z^{(N)}=(Z_t^{(N)})_{t \ge 0}$ converges
  as $N\to\infty$ in distribution to the solution of
  \eqref{eq:FelBDiff-ab}.

  In the binary branching case one could choose for $X^{(N)}$ the
  offspring distribution with branching in two individuals with
  probability $p_N = \frac12 + \frac{a}{2bN}$ and death with
  probability $1-p_N$. For immigration, at rate $Nc$ new particles are
  added to the population. Rescaling $X^{(N)}$ as above the limiting
  process is a solution of the SDE \eqref{eq:FelBDiff-ab} including
  the immigration term.
\end{remark}

For a critical binary Galton-Watson process with branching rate $b$
starting with one individual the probability of survival up to time
$N$ is approximately $2/(Nb)$. Thus, if the initial number of
individuals is $N$ then the number of families that survive up to time
$Nt$, for instance, is approximately Poisson with mean $2/(bt)$.

From this it easy to deduce the following result, which we state here
for future reference.

\begin{lemma}
  Consider \label{FellDiffDecomp} the critical $\R_+$-valued Feller
  diffusion $(Z_t)_{t \ge 0}$ defined in \eqref{eq:FelBDiff-ab} with
  $a=c=0$. For $0<s<t$, conditioned on $Z_{t-s}$ the random variable
  $Z_t$ be decomposed in a $\Pois\bigl(2 Z_{t-s}/(b(t-s))\bigr)$
  distributed number of i.i.d.\ random variables.

  Conditioned on $Z_{t-s}$, the vector of family sizes at time $t$
  going back to one ancestor at time $t-s$, arises as the limit
  $N\to\infty$ of rescaled critical binary Galton-Watson processes
  $(Z^{(N)}_t)_{t\ge 0}$ (as in above remark with $a=0$) starting
  initially with $N$ individuals.
\end{lemma}

\begin{remark}[Time-inhomogeneous Feller branching diffusion on
  $\R_+$] \label{r.inh.FR}
  A time inhomogeneous version of the solution of the stochastic
  differential equation \eqref{eq:FelBDiff-ab} arises if the
  coefficients $b$, $a$ and $c$ are measurable functions of time which
  are bounded on finite time intervals.
\end{remark}

\begin{remark}[Particle models and operators of $\U$-valued
  diffusions]
  The \label{r.697a} individual based processes from the previous
  remark naturally allow to read off $\U$-valued versions of the
  processes. Define the set $U_t$ as $\{1,2,\dots,n_t\}$, where $n_t$
  is the population size at time $t$. Furthermore, define
  \emph{ancestors} and \emph{descendants} in the obvious way, and the
  \emph{genealogical distance} $r_t$ as the ``usual'' graph distance,
  i.e.\ distance of two individuals from the population at time $t$ is
  twice the time back to their most recent common ancestor. Obviously,
  this defines an ultrametric space and letting $\mu_t$ be the
  \emph{counting measure} on $U_t$ we can encode the branching process
  together with its genealogy at time $t$ by an ultrametric measure
  space and taking its equivalence class we obtain an element of
  $\U$. The evolution is of course Markovian. In the case of an
  immigration event at time $s$ we set the distance between the new
  individual and every other individual alive at that time to be $2s$.

  We do not focus on approximation results and prove here limit
  theorems for individual based $\U$-valued processes only in the
  critical case. For the Feller diffusion the arguments are
  essentially based on corresponding results from \cite{Gl12}. Let us
  note however, that approximation results concerning generator
  convergence allow us to \emph{determine} the correct operators of
  our diffusion processes in various cases and we will use this
  approach often throughout the paper; see for instance the derivation
  of the branching operator in \eqref{eq:4}.

  For other processes which we consider here, such as processes
  conditioned to live forever or $h$-transformed processes (these
  processes lead to branching with immigration), there are some issues
  concerning \emph{path convergence} in $D(\U,\R)$ which we will not
  discuss in the present paper. We will rely on the fact that particle
  approximations determine the operators also in these cases.
\end{remark}

\paragraph{Martingale problem}
\label{martprob}

For any distribution on the state space $\U$ we want to construct a
stochastic process which has the given distribution as the initial
distribution, satisfies the Markov property and whose paths satisfy
some regularity conditions. The processes will be defined as solutions
of \emph{well-posed martingale problems}. We use here the following
notion of a martingale problem.

\begin{definition}[(local) martingale problem]
  Let $E$ be a Polish space, $\nu$ a probability measure on its
  Borel-$\sigma$-algebra, $\mathcal F$ a subspace of bounded
  measurable functions on $E$ and $\Omega$ a linear operator whose
  domain $\mcD$ is contained in $\mathcal F$. The distribution $P$ of
  an $E$-valued stochastic process $X=(X_t)_{t\ge 0}$ is called a
  \emph{solution of the $(\nu,\Omega,\mathcal F)$-martingale problem}
  if $X_0$ has distribution $\nu$, the paths of $X$ are almost surely
  in the Skorohod space $D([0,\infty),E)$, and for all
  $F \in \mathcal F$,
  \begin{align}
    \label{eq:4mp}
    \Bigl(F(X_t)- F(X_0) - \int_0^t \Omega F(X_s)\,\dx s \Bigr)_{t\ge
    0}
  \end{align}
  is a $P$-martingale with respect to the canonical filtration. If the
  solution $P$ is unique, then $(\nu,\Omega,\mathcal F)$-martingale
  problem is said to be \emph{well-posed}. If the processes in
  \eqref{eq:4mp} are only local martingales then we refer to $P$ as
  the solution of the $(\nu,\Omega,\mathcal F)$-\emph{local}
  martingale problem.
\end{definition}

Since the state space $\U$ is not locally compact we use on that state
space the following notion of (generalized) Feller property.

\begin{definition}[Generalized Feller property]
  \label{rk:Fellert}
  Let $\mfU=(\mfU_{t})_{t\ge 0}$ be a $\U$-valued Markov process
  starting under the law $\P_\nu$ in the initial law
  $\nu \in \mathcal{M}_1(\U)$, where $\nu$ is arbitrary. Denote by
  $\E_\nu$ the expectation w.r.t.\ to $\P_\nu$. We say that $\mfU$
  satisfies the \emph{generalized Feller property} if for every
  $\Phi \in \mathcal{C}_b(\U)$ the function
  \begin{align}
    \label{eq:35}
    \nu \mapsto \E_\nu [ \Phi(\mfU_t)] \quad \text{is continuous in the weak
    topology on $\mcM_1(\U)$.}
  \end{align}
\end{definition}

Note that the generalized Feller property implies that for all $t >0$
the function $\nu \mapsto \mathcal L_\nu(\mfU_t)$ is continuous on
$\mathcal M_1(\U)$, a property which in turn implies our defining
condition of the Feller property. To check the Feller property of our
processes we will proceed as follows: The set
$\Pi(\mathcal{C}_b) \subset \mathcal C(\U)$ is measure and convergence
determining. As a consequence it suffices to verify the condition
\eqref{eq:35} for all $\Phi \in \Pi(\mathcal C_b)$ which is easier
since we have a duality based on functions from $\Pi(\mathcal C_b)$.

If $\mfU$ is a Markov process satisfying the Feller property from the
Definition~\ref{rk:Fellert} then $T=(T_t)_{t\ge 0}$, defined by
$T_t F (\mfu) = \E_\mfu[F(\mfU_t)]$ is a semigroup on
$\mathcal C_b (\U)$. In the classical case, i.e.\ if the state space
is compact or locally compact the bounded continuous functions resp.\
the continuous functions vanishing at infinity are uniformly
continuous. Hence, for semigroups $T=(T_t)_{t\ge 0}$ satisfying the
Feller property in the locally compact case we have
$\norm{T_t -\operatorname{Id}} \to 0$ in the operator norm as
$t \downarrow 0$, i.e.\ $T$ is a strongly continuous semigroup. For
non-locally compact spaces this argument does not work.

\medskip

We consider first the critical case, which corresponds to $a=0$ in
\eqref{eq:FelBDiff-ab}, and introduce the following operator on
polynomials $\Phi \in \Pi(\mathcal C_b^1)$. For
$\mfu \in \U \setminus \{\ntree\}$ motivated by the individual based
approximation we define
\begin{align}
  \label{tv5}
  \Omega^{\uparrow} \Phi^{n,\varphi}(\mfu)
  & \coloneqq \Omega^{\uparrow,\mathrm{grow}} \Phi^{n,\varphi}(\mfu) +
  \Omega^{\uparrow,\mathrm{bran}} \Phi^{n,\varphi}(\mfu)\\
  \intertext{with}
  \label{mr3}
  \Omega^{\uparrow,\mathrm{grow}}\Phi^{n,\varphi}(\mfu)
  & \coloneqq \Phi^{n,\overline{\nabla} \varphi} (\mfu) , \quad
    \overline{\nabla} \varphi =
    2 \sum_{1\le i<j \le n} \frac{\partial \varphi}{\partial r_{i,j}}
    \; \text{for } n \ge 2 \text{ and }  0 \text{ otherwise}, \\
  \label{mr3a}
  \Omega^{\uparrow,\mathrm{bran}}\Phi^{n,\varphi}(\mfu)
  & \coloneqq \frac{b}{\bar{\mfu}} \sum_{1\le k < \ell \le n}
    \Phi^{n,\varphi\circ \theta_{k,\ell}} (\mfu) \; \text{for } n \ge 2
    \text{ and }  0 \text{ otherwise},\\
  \intertext{where the replacement of the $\ell$-th sample point by
  the $k$-th one is described by the  following operator acting on  the distance matrix
  of the sample}
  \label{eq:theta}
  \left(\theta_{k,\ell} (\dr) \right)_{i,j}
  & \coloneqq  r_{i,j}\ind{i\neq \ell, j\neq \ell}
  + r_{j\vee k, j\wedge k} \ind{i=\ell}
  + r_{i\vee k, i\wedge k} \ind{j=\ell} ,\quad 1\le i < j.
\end{align}
We extend $\Omega^{\uparrow} \Phi^{n,\varphi}$ to $\U$ by setting
$\Omega^{\uparrow} \Phi^{n,\varphi}(\ntree) = 0$. We see that
$\Omega^{\uparrow}$ maps $\Pi$ into $\Pi$ and hence is a linear
operator on this algebra of polynomials.

The growth operator is ``standard'', see e.g.\ \cite{GPWmp13} or
\cite{DGP12}. For some intuition behind the form of the branching
operator consider the following computation for the approximating
particle system from Remark~\ref{r.697} and Remark~\ref{r.697a} (set
there $a=0$ and $c=0$ for the critical case without immigration):
denoting by $\Omega^{\uparrow,\mathrm{bran},N}$ the branching operator
of the $N$-th system we have
\begin{align}
  \label{eq:4}
  \begin{split}
    \Omega^{\uparrow,\mathrm{bran},N}
    & \Phi^{n,\varphi}(\mfu) = bN^2 \int \mu(dx) \langle \varphi,\frac{1}{2}
    (\mu+\frac{1}{N}\delta_x)^{\otimes n} +
    \frac{1}{2}(\mu-\frac{1}{N}\delta_x)^{\otimes n} -\mu^{\otimes n} \rangle\\
    & = b \sum_{1 \le k < \ell \le n} \int \mu(dx) \langle \varphi,
    \mu^{\otimes (k-1)} \otimes \delta_x \otimes \mu^{\otimes
      (\ell-k-1)}
    \otimes \delta_x \otimes \mu^{\otimes (n-\ell)} \rangle + O(1/N)\\
    & = \frac{b}{\bar\mfu} \sum_{1 \le k < \ell \le n} \langle \varphi
    \circ \theta_{k,\ell}, \mu^{\otimes n} \rangle + O(1/N).
  \end{split}
\end{align}

Our first main result states that the operator $\Omega^\uparrow$
defines a ``good'' Markov process on the state space $\U$. In
particular for every initial law on $\U$ we have a unique Markov
process which solves the martingale problem for $\Omega^\uparrow$ and
has a.s.\ regular paths. Recall from Corollary~\ref{cor.681} that
there is a set of measure determining test functions. Together with
the Feller property these test functions will be used to obtain the
strong Markov property on this not locally compact state space.
Furthermore we shall show that $\Omega^{\uparrow}$ is a second order
operator, see e.g.\ Section~4.1 in \cite{DGP12} for more on this
concept.

\begin{theorem}[Well-posedness of the martingale problem]
  \label{THM:MGP:WELL-POSED}
  % \leavevmode\\
  For any $\mfu \in \bbU$ the following assertions hold.
  \begin{enumerate}
  \item The $(\delta_\mfu,\Omega^{\uparrow}, \Pi(\mathcal
    C_b^1))$-martingale problem in $D([0,\infty),\U)$ is well-posed.
  \item The unique solution $P_\mfu$ of the
    $(\delta_\mfu, \Omega^{\uparrow}, \Pi(\mathcal C_b^1))$-martingale problem
    hat the property that $\mfu \mapsto P_\mfu$ is continuous,
    % has the Feller property,
    satisfies the strong Markov property and $P_\mfu$ is concentrated
    on continuous paths.
  \item For every $\nu \in \mathcal M(\U)$ the law $P_\nu$ defined by
    \begin{align}
      \label{eq:3pnu}
      P_\nu \coloneqq \int \nu(\dx\mfu) P_\mfu
    \end{align}
    solves the \emph{local}
    $(\nu, \Omega^\uparrow,\Pi(C^1_b))$-martingale problem (recall
    that the initial law need not have finite moments of the total
    masses) and is the unique solution of that martingale problem. In
    particular the solution $P_\nu$ satisfies the Feller property.
  \end{enumerate}
\end{theorem}

Analogous generalizations of initial conditions from point masses to
measures as in the step from 2.\ to 3.\ in the above theorem will also
hold for most other processes on potentially different state spaces
that we will consider later.

Even though using duality (cf.\ Remark~\ref{r.tv45}) one can show that
$\E[\Phi(\mfU_t)]\rightarrow \E[\Phi(\mfU_0)]$ as $t \downarrow 0$
holds uniformly in $\Phi$ on $\Pi(C^1_b)$ it is \emph{not} true on all
of $\mathcal C_b(\U)$. The reason is that for each $t \ge 0$ and $\Phi
\in \Pi(C^1_b)$ the mapping $\mfu \mapsto \E_{\mfu}[\Phi(\mfU_t)]$ is
uniformly continuous, but we do not know this for functions in
$\mathcal C_b(\U) \setminus \Pi(C^1_b)$, since the Weierstraß
approximation theorem is not applicable on the state spaces $\U$ and
$\U_1$ and in fact does not hold here.

\begin{definition}[$\U$-valued Feller diffusion]
  The \label{D.tvF} solution $\mfU=(\mfU_t)_{t\ge 0}$ of the
  $(\delta_\mfu,\Omega^{\uparrow},\Pi(\mathcal C_b^1))$-martingale problem with
  continuous paths is called the \emph{$\U$-valued Feller diffusion}
  with diffusion coefficient $b$ and initial condition $\mfu$. The
  process for initial law $\nu$ on $\U$ arises as in \eqref{eq:3pnu}.
  Using the polar decomposition from Section~\ref{sec:polar-repr-stat}
  we often write $\mfU_t = (\bar\mfU_t,\wh\mfU_t)$ for $t\ge 0$ and
  refer to $\bar\mfU=(\bar\mfU_t)_{t\ge 0}$ as the \emph{population
    size process} and to $\wh\mfU=(\wh\mfU_t)_{t\ge 0}$ as the
  \emph{pure genealogy part} of $\mfU$.
\end{definition}

\begin{remark}[Non-critical case]
  \label{r.offspring}
  An analogue of Theorem~\ref{THM:MGP:WELL-POSED} holds also for
  \emph{non-critical branching} with \emph{non-criticality
    coefficient} $a \ne 0$ and $c=0$ in \eqref{eq:FelBDiff-ab}; see
  also the particle model in Remark~\ref{r.697a}. In this case the
  total mass process $(\bar{\mfU}_t)_{t \ge 0}$ is given in law by a
  solution of \eqref{eq:FelBDiff-ab} with $c=0$. Using particle
  approximation which is explained in Remark~\ref{r.697} and
  Remark~\ref{r.697a} with a similar calculation as in \eqref{eq:4}
  one can see that the corresponding operator (recall
  \eqref{tv5}--\eqref{mr3a}) has the form
  \begin{align}
    \label{eq:3omab}
    \Omega^{\uparrow,(a,b)} \Phi^{n,\varphi}(\mfu)
    & \coloneqq \Omega^{\uparrow,\mathrm{grow}} \Phi^{n,\varphi}(\mfu) +
      \Omega^{\uparrow,\mathrm{bran}} \Phi^{n,\varphi}(\mfu) +
      \Omega^{\uparrow,\mathrm{sbran}} \Phi^{n,\varphi}(\mfu),
  \end{align}
  where
  \begin{align}
    \label{e709}
    \Omega^{\uparrow, \mathrm{sbran}} \Phi^{n, \varphi} =  a n \;
    \Phi^{n,\varphi}
  \end{align}
  is the additional branching term for $n \ge 1$. In particular,
  $\Omega^{\uparrow,(a,b)}$ is a linear operator on $\Pi$. The
  corresponding process describes the $\U$-valued Feller diffusion
  with parameters $a$ and $b$.

  It is remarkable that the non-criticality leads to the addition of a
  ``first order term'' and the effect of this will become clearer once
  we have obtained the duality relation also in this case; see
  Remark~\ref{r.concrit} for the change in the duality relation in
  particular for the change of the potential. In
  Section~\ref{sec:Fel:ex}, in particular in Remark~\ref{r.5522}, we
  will explain how to rigorously prove \emph{existence} and
  \emph{uniqueness} of the corresponding martingale problems based on
  the skew martingale problems. There is however also a third method
  to treat non-criticality, namely Girsanov's transform. We do not give
  the details here, but refer to \cite{DGP12} where this is carried
  out for the Fleming-Viot model with selection.
\end{remark}

\begin{remark}[Inhomogeneous case]
  \label{r.inh.FU}
  As in the case of $\R_+$-valued Feller diffusion, see
  Remark~\ref{r.inh.FR} for assumptions, we can also assume in the
  case of $\U$-valued Feller diffusion that the coefficients $b$ and
  $a$ are functions of space and time in \eqref{mr3a} respectively \eqref{e709}
  (and therefore in \eqref{eq:3omab}).

  In the sequel we focus on the time-homogeneous case even though our
  Theorems~\ref{THM:MGP:WELL-POSED}--\ref{T:BACKBONE} %\ref{TH.KOLMO}
  can be generalized to the time inhomogeneous case.
\end{remark}

\paragraph{Properties: relation to \texorpdfstring{$\R_+$}{R}-valued
  Feller diffusion and \texorpdfstring{$\U_1$}{U1}-valued Fleming-Viot
  process}
\label{sec:prop-relat-texorpdfs}
In order to obtain the process and to study its structure better, it
is useful to split the information on the behavior of the population
size and the behavior of the genealogy, that is to consider the
processes
\begin{align}
  \label{eq:22}
  \bar\mfU = (\bar\mfU_t)_{t\ge 0} \quad \text{and} \quad \wh\mfU =(\wh\mfU_t)_{t\ge 0}
\end{align}
and to identify their \emph{dynamics} and \emph{path properties}
before and at respectively after the extinction of the population. The
precise statement follows in Propositions~\ref{prop.tvF} and
\ref{prop.834} below. We will use the notation $\bar \Phi$ and $\wh
\Phi$ for polynomials in the total population mass respectively the
pure genealogy part of the process.

We can characterize $\bar\mfU$ and $\wh \mfU$ conditioned on
$\bar\mfU$ in points (i) and (ii) below by a well-posed martingale
problems and this way identify the two components as nice processes
about each of which we have much information.

\medskip
\noindent
\textbf{\emph{(i)\, Population size process}} For the total mass
process we have the following result. For it's proof we will need some
other results and notation. A sketch of the proof  will be discussed
in Remark~\ref{r.p.prop.tvF}.

\begin{proposition}[Total mass process and classical Feller
  diffusion]\label{prop.tvF}
  If $(\mfU_t)_{t \ge 0}$ is a \emph{$\U$-valued} Feller diffusion
  then $(\bar \mfU_t)_{t \ge 0}$ is an autonomous Markov process given
  by the classical $\R_+$-valued Feller diffusion, given by solution
  of the SDE \eqref{eq:FelBDiff-ab} with $a=0$ and $c=0$.
\end{proposition}

The following well known path properties of $(\bar\mfU_t)_{t\ge 0}$
from Proposition~\ref{prop.tvF} restrict the set of functions that we
need to consider as possible total mass paths. Such properties have
been studied in the literature and are based on results of specific
classes of diffusions. See Chapter~9, Lemma~1.6 in \cite{EK86} and
Section~1(a) in \cite{DG03} which also contains some facts which we
need later for spatial models.

\begin{proposition}[Path properties of $\R_+$-valued Feller diffusion]
  Let \label{prop:ppFD} $Z=(Z_t)_{t \ge 0}$ be an $\R_+$-valued Feller
  diffusion, i.e.\ a solution of the SDE \eqref{eq:FelBDiff-ab} with
  $c=0$. The paths of $Z$ are almost surely elements of
  $C \left([0,\infty),[0,\infty)\right)$ and if $Z_0>0$ then there is
  $T_{\mathrm{ext}} \in (0,\infty]$, so that $Z_t>0$ for all
  $t \in [0,T_{\mathrm{ext}})$ and $Z_t = 0$ for all
  $t \ge T_{\mathrm{ext}}$. In the case $T_{\mathrm{ext}}<\infty$ we
  have
  \begin{align}
    \label{eq:ppFD}
    \int_r^{T_{\mathrm{ext}}} Z_s^{-1} \, \dx s =\infty \quad
    \text{for all} \; r \in [0,T_{\mathrm{ext}}).
  \end{align}
  We refer to $T_{\mathrm{ext}}$ as the \emph{extinction time} of the
  total mass process.
\end{proposition}
\begin{proof}[Proof sketch.]
  Similar assertions to \eqref{eq:ppFD} have been shown in
  \cite{DG03}, pages 24-25, the argument uses the representation of
  the process $(Z_t)_{t \ge 0}$ by Brownian motion. This in turn
  allows to rewrite the integral asymptotically close to a zero of the
  process in terms of an excursion involving $(B_t)^{-2}$. Then using
  facts about these processes and asymptotically equivalent Bessel
  processes the result follows. For details we refer to \cite{DG03}.
\end{proof}

Motivated by Proposition~\ref{prop:ppFD} we make the following
definition.
\begin{definition}[Admissible total mass paths I: Positive initial
  conditions]
  We \label{def:adm_u} call a function
  \begin{align}
    \label{eq:37}
    \bar \mfu=(\bar \mfu_t)_{t \ge 0} \in C
    \left([0,\infty),[0,\infty)\right)
  \end{align}
  \emph{admissible} as a total mass path of a $\U$-valued Feller
  diffusion if it satisfies the properties listed in
  Proposition~\ref{prop:ppFD}.
\end{definition}

\medskip
\noindent
\textbf{\emph{(ii)\, The pure genealogy process}} It is a well known
that the total mass of a Dawson-Watanabe superprocess with finite
non-trivial initial measure is given by Feller's branching diffusion
and that by Perkins' disintegration theorem the processes normalized
by the total mass and conditioned on the total mass process is a
(classical) Fleming-Viot superprpocess with time inhomogeneous
resampling rate given up to a constant factor by the reciprocal of the
total mass; see \cite{Perkins1992} or \cite[Section~4.4]{Eth00}.
Analogous result does hold in the $\U$-valued case and will be
discussed in the following.

To analyze the \emph{genealogy part} $\wh \mfU$ we will use the fact
that there is a close relationship between the $\U$-valued
\emph{Feller diffusion} and the $\U_1$-valued \emph{Fleming-Viot}
processes. In this context $\U_1$-valued \emph{Fleming-Viot} processes
with time-inhomogeneous resampling rates will arise. The latter is a
$\U_1$-valued Markov process with continuous paths arising from the
Fisher-Wright or more generally Fleming-Viot models, which themselves
arise as infinite population limits of the well known individual based
Moran models, similarly as the rescaling of Galton-Watson gave rise to
Feller's branching diffusion as described in Remark~\ref{r.697}. For
the rescaling of a Moran model with mutation leading to a Fleming-Viot
superprocess we refer the reader for instance to \cite{Eth00}. For a
definition of a $\U_1$-valued Moran model whose rescaling leads to
$\U_1$-valued Fleming-Viot process see e.g.\ Definition~2.5 in
\cite{DGP12}.

Recalling Proposition~\ref{prop.tvF}, we can \emph{condition} on the
autonomous $\R_+$-valued Feller diffusion $(\bar \mfU_t)_{t \ge 0}$ to
get the conditioned genealogy part of the process. Then the
$\U$-valued Feller diffusion $\mfU$ decomposed as
$(\bar \mfU, \wh \mfU)$ has the following key property.

\begin{proposition}[Genealogy conditioned on total mass path is
  $\U_1$-valued Fleming-Viot]\label{prop.834}
  % \leavevmode\\
  Consider the process $(\wh \mfU_t)_{t \ge 0}$ conditioned on a
  realization $\bar \mfu = (\bar \mfu_t)_{t \ge 0}$ of
  $(\bar \mfU_t)_{t \ge 0}$, denoted by
  $(\wh \mfU_t(\bar\mfu))_{t \ge 0}$. For almost surely all $\bar\mfu$ the
  process $(\wh \mfU_t(\bar\mfu))_{t \ge 0}$ is a
  \begin{align}
    \label{810}
    \begin{split}
      \text{\emph{time-inhomogeneous} } \U_1 \text{-valued
        Fleming-Viot diffusion
        $(\wh \mfU_t^{\mathrm{FV}} (\bar \mfu))_{t\ge 0}$
      } \\
      \text{ with resampling rate } d(t)=b/\bar \mfu_t \text{ at time
      } t  \text{ if } \bar\mfu_t>0.
    \end{split}
  \end{align}
  At the extinction point $T_{\mathrm{ext}}$ of the total mass path
  $\bar\mfu$ the process $(\wh \mfU_t(\bar\mfu))_{t \ge 0}$ converges
  to the null tree $\ntree$ on $\U_1$ in the Gromov weak topology and
  remains in this state for $t \ge T_{\mathrm{ext}}$. In particular
  the conditioned process has continuous $\U_1$-valued paths.

  Consequently, $(\mfU_t)_{t \geq 0}$ has paths with values in
  $\U_{\mathrm{comp}}$, the equivalence classes of compact ultrametric
  measure spaces.
\end{proposition}

\begin{remark}[Time change]\label{r.1540}
  The corollary above could also be read in terms of a time change.
  More precisely, for $T_t = b \int_0^t 1/\bar \mfu_s \,\dx s$, the
  process $(\wh \mfU_{T_t})_{t \geq 0}$ is $\U_1$-valued Fleming-Viot
  process with resampling rate-$1$ and ultrametrics scaled by
  $(T_t^{-1})$. Though this view is useful sometimes, the formulation
  above in terms of the generator is more suitable for generalizations
  to spatial or multi-type models. We shall see that a version of
  Proposition~\ref{prop.834} holds in these models, but no formulation
  via time change is possible.
\end{remark}

\begin{remark}[Notation]
  \label{rem:Pnots}
  In the sequel we will often work with our process as well as its
  conditioned variants or functionals.
  \begin{itemize}
  \item The $\U$-valued Feller diffusion possibly with superscripts
    depending on its variant will be usually denoted by
    $\mfU=(\mfU_t)_{t\ge 0}$. Its law is denoted by $P$, $P_{\mfu_0}$,
    $P_\nu$ etc. Here $\mfu_0$ is a fixed element in $\U$ and $\nu$ a
    probability measure on $\U$.

  \item The autonomous total mass process and its law are denoted by
    $\bar\mfU=(\bar\mfU_t)_{t\ge 0}$ respectively $\bar P$. Again,
    depending on initial conditions etc.\ we may have additional sub-
    or superscripts.
  \item If $\bar\mfu =(\bar\mfu_t)_{t\ge 0}$ is a realization of the
    total mass path process $\bar\mfU=(\bar\mfU_t)_{t\ge 0}$ then we
    can consider the full process or the pure genealogy part
    conditioned on the total mass. The corresponding processes are
    denoted by $\mfU(\bar\mfu) = (\mfU_t(\bar\mfu))_{t\ge 0}$ and
    $\wh\mfU(\bar\mfu) = (\wh\mfU_t(\bar\mfu))_{t\ge 0}$. The laws,
    with possible additional sub- or superscripts, are denoted by
    $P^{\bar \mfu}$ respectively $\wh P^{\bar\mfu}$.
    Corollary~\ref{prop.834} says that $\wh P^{\bar\mfu}$ is also the
    law of
    $\wh \mfU^{\mathrm{FV}} (\bar \mfu) = (\wh \mfU_t^{\mathrm{FV}}
    (\bar \mfu))_{t\ge 0}$. Furthermore we have
    \begin{align}
      \label{eq:coneq}
      \mcL \bigl[(\mfU_t(\bar\mfu))_{t\ge 0}\bigr] =
      \mcL \bigl[(\bar\mfu_t \cdot \wh\mfU_t^{\mathrm{FV}}(\bar\mfu))_{t\ge
      0}\bigr].
    \end{align}
  \end{itemize}
  Similar results will hold in many other situations where we will
  also use similar convention concerning the notation.
\end{remark}

\paragraph{Extension of the operator
  \texorpdfstring{$\Omega^{\uparrow}$}{Omega}}

To understand the material in the paragraph starting on
page~\pageref{sec:prop-relat-texorpdfs} which relates the
process $\mfU$ to familiar objects (recall in particular
Corollaries~\ref{prop.tvF} and \ref{prop.834}), as well as for some
calculations in proof sections it is useful to extend the validity of
the martingale problem to \emph{larger domains of test functions},
even if this is not needed for uniqueness.

We obtain more martingales for our process if we write the elements of
$\U$ in a particular form and then use the particle approximation from
\cite{Gl12} to extend the operator $\Omega^\uparrow$ to larger classes
of functions, which we will use heavily in the sequel to explore the
structure of $(\mfU_t)_{t \geq 0}$. We obtain this class by adding two
domains of test functions on which $\Omega^\uparrow$ can be defined.

The polar decomposition of elements of $\U$ discussed in
Section~\ref{sec:polar-repr-stat} and in particular in equation
\eqref{e998} suggests considering test functions of the form
\begin{align}
  \label{e1216}
  \Phi^{n,\varphi}(\mfu)=\bar \Phi (\bar \mfu)
  \wh \Phi^{n,\varphi} (\hat \mfu),
\end{align}
where $\wh \Phi^{n,\varphi}$ acts on the genealogy component
$\hat\mfu$, and $\bar\Phi$ acts on the mass component $\bar\mfu$ of
$\mfu$, i.e.\ recalling \eqref{eq:Pi-total}we have
$\wh \Phi^{n,\varphi}\in \wh\Pi$ and $\bar \Phi$ is a function on
$\R_+$.

Note that the set of functions of the form \eqref{e1216} is indeed an
extension of the polynomials from Definition~\ref{D.polyn}, because
for instance for $n \ge 2$ the polynomial $\Phi^{n,\varphi}$ applied to
$\mfu = [U,r,\mu] \in \U\setminus\{\ntree\}$ can be written as
\begin{align}
  \label{pol-ext}
  \Phi^{n,\varphi}(\mfu)
  = \langle \varphi, \nu^{n,\mfu} \rangle =
  \bar \mfu^n \langle \varphi, \hat\nu^{n,\hat\mfu}\rangle
  = \bar\mfu^n \wh
  \Phi^{n,\varphi}(\hat \mfu),
\end{align}
i.e., we have
$\Phi^{n,\varphi}(\mfu) = \bar \Phi (\bar \mfu) \wh \Phi^{n,\varphi}
(\hat \mfu)$ with $\bar \Phi (\bar \mfu) = \bar\mfu^n$.

\begin{remark}[Definition of the operator $\Omega^{\uparrow}$ on new
  test functions]
  \label{r.1207}
  An important subspace of \emph{bounded continuous} functions
  constitute those functions which are mapped by the operator
  $\Omega^{\uparrow}$ again onto \emph{bounded} continuous functions.
  On this subspace we can apply the various standard results on
  transition semigroups. We define it by
  \begin{align}
    \label{e1212}
    \mcD_1 \coloneqq \bigl\{\Phi = \bar\Phi\wh\Phi : \bar\Phi \in
    C_b^2([0,\infty),\R),
    x^{-1} \bar\Phi(x) \leq
    \mathrm{Const} \text{ as } x \to 0 \; \text{and} \; \wh\Phi \in
    \wh\Pi \bigr\}.
  \end{align}

  Another important subspace is the set of functions $\Phi$ on which
  $\Omega^\uparrow$ can be applied to $S_t \Phi$, where
  $(S_t)_{t\ge 0 }$ denotes the semigroup of our process. This vector
  space can be chosen as
  \begin{equation}
    \label{e1189}
    \mcD_2\coloneqq \bigl\{\Phi = \bar\Phi\wh\Phi : \bar \Phi \in C^2
    ([0,\infty),\R),\, \limsup_{\bar \mfu \to \infty} |\bar \Phi''
    (\bar \mfu)|
    /\bar \mfu^n < \infty \; \text{ for some}\; n \in \N  \;
    \text{and} \; \wh\Phi \in \wh\Pi \bigr\}.
  \end{equation}
  We use here that $\wt\mcM$ (recall from \eqref{tv4}) is preserved
  under $S_t$ (see Remark~\ref{rem:feller_momens}) and hence $GS_t
  \Phi = S_t (G\Phi)$ exists for all $t \ge 0$, where $G$ is any
  combination of the operators we use here for the martingale
  problems.

  For general $\Phi^{n,\varphi} = \bar\Phi \wh\Phi^{n,\varphi}\in
  \mcD_i$, $i =1,2$ the generator $\Omega^\uparrow$ can be written in
  the following form (for the proof we refer to \cite{Gl12}
  Section~5.4 and in particular equations (5.68-69) and discussion
  around equation (1.85) with the choice $\alpha=0$ there, mind
  however a typo in (1.87) where it should read $(m^\mfu)^{\alpha-1}$)
  \begin{align}
    \label{eq:r.742.1}
    \Omega^{\uparrow} \Phi^{n,\varphi} (\mfu) = \wh \Phi^{n,\varphi}(\hat \mfu)\,
    \Omega^{\mathrm{mass}} \bar \Phi (\bar{\mfu}) + \bar \Phi (\bar{\mfu})\,
    \Omega^{\mathrm{gen}}_{\bar\mfu} \wh \Phi^{n,\varphi}(\hat \mfu).
  \end{align}
  Here the generator parts are given by
  \begin{align}
    \label{eq:12a}
    \Omega^{\mathrm{mass}} \bar \Phi(\bar{\mfu})
    & = \frac{b \bar{\mfu}}{2} \frac{\partial^2}{\partial
      \bar{\mfu}^2} \bar \Phi (\bar{\mfu}), \\
    \label{eq:12b}
    \Omega_{\bar\mfu}^{\mathrm{gen}} \wh \Phi^{n,\varphi}(\hat \mfu)
    & = \frac{b}{\bar{\mfu}} \; \Omega^{\uparrow,\mathrm{res}} \wh
      \Phi^{n,\varphi}(\hat \mfu) + \Omega^{\uparrow, \mathrm{grow}}
      \wh \Phi^{n,\varphi}(\hat \mfu),\\
    \intertext{with}
    \label{eq:12c}
    \Omega^{\uparrow,\mathrm{res}} \wh \Phi^{n,\varphi} (\hat \mfu)
    & = \sum_{1\le k < \ell \le n}
      \langle \varphi \circ \theta_{k,\ell} - \varphi, \wh
      \nu^{n,\hat\mfu}\rangle .
  \end{align}
  Indeed we have an extension of $\Omega^\uparrow$, because using
  $\Phi^{n,\varphi}(\mfu)=\bar\mfu^n \wh \Phi^{n,\varphi}(\hat \mfu)$
  as in \eqref{pol-ext} we can rewrite the branching part from
  \eqref{mr3a} for $\mfu$ with $\bar\mfu>0$ as follows
  \begin{align}
    \label{mr3aexta}
    \begin{split}
      \Omega^{\uparrow,\mathrm{bran}}\Phi^{n,\varphi}(\mfu)
      & = \frac{b}{\bar{\mfu}} \bar\mfu^n \sum_{1\le k < \ell \le n}
      \langle  \varphi \circ \theta_{k,\ell}, \hat\nu^{n,\hat\mfu}\rangle\\
      & = \frac{b}{\bar{\mfu}} \bar\mfu^n \sum_{1\le k < \ell \le n}
      \bigl(\langle \varphi \circ \theta_{k,\ell} - \varphi, \wh
      \nu^{n,\hat\mfu}\rangle + \langle \varphi, \wh
      \nu^{n,\hat\mfu}\rangle\bigr) \\
      & = \frac{b}{\bar{\mfu}} \bar\mfu^n
      \Bigl(\Omega^{\uparrow,\mathrm{res}} \wh \Phi^{n,\varphi} (\hat\mfu) +
      \frac{n(n-1)}{2} \wh\Phi^{n,\varphi} (\hat \mfu) \Bigr).
    \end{split}
  \end{align}
  Thus, in this case the generator defined in \eqref{tv5} can be
  written as
  \begin{align}
    \label{e748}
    \Omega^{\uparrow} \Phi^{n,\varphi} (\mfu) = \frac{n(n-1)}{2}
    \frac{b}{\bar{\mfu}} \bar \mfu^n \wh \Phi^{n,\varphi}(\hat \mfu) +
    \frac{b}{\bar{\mfu}} \bar\mfu^n \Omega^{\uparrow,\mathrm{res}}
    \wh \Phi^{n,\varphi} (\hat \mfu)  + \bar\mfu^n \; \Omega^{\uparrow,
    \mathrm{grow}} \; \wh \Phi^{n,\varphi} (\hat \mfu).
  \end{align}
  Choosing $\bar \Phi \equiv 1$, which is in $\mcD_2$, in
  \eqref{eq:r.742.1} we obtain
  \begin{align}
    \label{e746}
    \Omega^{\uparrow} \Phi^{n,\varphi} (\mfu) =
    \Omega_{\bar\mfu}^{\mathrm{gen}} \wh \Phi^{n,\varphi}(\hat \mfu),
  \end{align}
  i.e.\ when $\bar \mfu$ is the present value of the mass then it
  plays the role of a parameter for the genealogy component of the
  process.
\end{remark}

We summarize the discussion above in the following result.
\begin{proposition}[Modifications and extensions of the martingale
  problem]
  The \label{prop.1183} process from Theorem~\ref{THM:MGP:WELL-POSED}
  solves the martingale problem also with $\mcD_1$, $\mcD_2$ instead
  of $\Pi(\mathcal C_b^1)$ in its modified form and is in particular a
  solution of the extended $(\Pi(\mathcal C_b^2)\cup \mathcal D_1 \cup
  \mathcal D_2,\Omega^{\uparrow})$-martingale problem.
\end{proposition}

Propositions~\ref{prop.tvF} and \ref{prop.834} are consequences of the
representations \eqref{eq:r.742.1} and \eqref{e746}. We sketch their
proofs in the following remarks.

\begin{remark}[Proof sketch of Proposition~\ref{prop.tvF}]
  From \label{r.p.prop.tvF} equation \eqref{eq:r.742.1} in Remark~\ref{r.1207} it follows
  that the operator $\Omega^\uparrow$ acts on $\bar \Phi \wh \Phi$
  according to the product rule. This follows also from the fact that
  $\Omega^\uparrow$ is a second order operator; see \cite{DGP12}.
  Both, the total mass martingale problem and the genealogy part
  martingale problem are well-posed. In particular, combining
  \eqref{eq:r.742.1} and \eqref{eq:12a} the assertion of
  Propositions~\ref{prop.tvF} follows.
\end{remark}

\begin{remark}[Proof sketch of Proposition~\ref{prop.834}]
  The operator of the well-posed Fleming-Viot martingale problem (see
  \cite{GPWmp13}) acts on functions $\wh \Phi \in \wh\Pi$ and is given
  by
  \begin{align}
    \label{e939}
    \Omega^{\mathrm{FV}} \wh \Phi (\hat\mfu) = \Omega^{\uparrow,
    \mathrm{grow}} \wh \Phi (\hat\mfu) +
    d \, \Omega^{\uparrow,\mathrm{res}}\wh \Phi (\hat\mfu),
  \end{align}
  where $\Omega^{\uparrow, \mathrm{grow}}$ is the growth part of the
  generator defined in \eqref{mr3}, $d>0$ is the resampling rate and
  $\Omega^{\uparrow,\mathrm{res}}$ is the resampling part of the
  generator defined in \eqref{eq:12c}. We can allow here a
  \emph{time-dependent} continuous resampling rate $d$ based on
  admissible functions in our case and still the arguments of the
  proof of the existence and the uniqueness from \cite{GPWmp13} go
  through without complications. The reason is that we can approximate
  the resampling rate by ones which are piecewise constant, so that
  the theorems from \cite{GPWmp13} apply on the corresponding time
  intervals. Using the dual process it is easy to show convergence as
  we shrink the time intervals of constancy periods to zero.
\end{remark}
%% C

\paragraph{Skew martingale problems, an excursion}

The \label{p:skew} skew martingale problem allows in many cases to
study (two) aspects of a stochastic process $Z$, by isolation of
functionals for each of those aspects, given by processes $X$ and $Y$,
and following their evolution. Typically one functional solves an
autonomous martingale problem, the other conditioned on the
realization of the first solves a time-inhomogeneous martingale
problem (the conditional martingale problem) and both should be well
posed. The evolution of the functionals can on the other hand be used
to \emph{construct} the process $Z$.

An example is the spatial multitype branching process where the
branching rate is determined by the total masses at the geographic
sites. Then one studies the occupation numbers at the site and the
relative frequencies of the types by writing down a skew martingale
problem; see \cite{DG03}. In the present work we look at the total
mass and the genealogies of a Feller diffusion or in a spatial setting
at the super random walk.

A skew martingale problem for a process $(Z_t)_{t\ge 0}$ with values
in a Polish space $E$ arises from a
$(\mu^Z,G^Z,\mathcal{F}^Z)$-martingale problem, where $\mathcal{F}^Z$
is a measure determining domain $G^Z$ on the state space $E$. Namely,
we have two functionals $X=F_1(Z)$ with values in a Polish space $E_1$
and $Y=F_2(Z)$ with values in a Polish space $E_2$, such that there is
a bijection $F: E_1 \times E_2 \to E$ such that $Z=F(X,Y)$ on
$E_1\times E_2$. In particular $F(F_1(Z),F_2(Y))=Z$, so that we can
say that $(X,Y)$ is a coding of $Z$. The maps $F_1$, $F_2$ and $F$ are
assumed to be measurable.

The key point is now that we assume that
\begin{align}
  \label{eq:7}
  \begin{split}
    \text{$(X_t)_{t \ge 0}$ is a solution of a well-posed
      $(\mu^X,G^X,\mathcal{F}^X)$-martingale problem}\\
    \text{and is $E_1$-valued Borel-Markov process,}
  \end{split}
\end{align}
where for $f_1 \in \mathcal F^X$, $f_1 \circ F_1 (\cdot) \in \mathcal
F^Z$ and
\begin{align}
  \label{eq:11}
  (G^Z f_1 \circ F_1) (z) = (G^X f_1)(F_1(z)), \quad z \in E.
\end{align}
Then $(X_t)_{t \ge 0}$ solves the $(\mu^X,G^X,\mathcal
F^X)$-martingale problem with $\mu^X = \mu_{F_1}$. For a solution
$P^Z$ of the $(\mu^Z,G^Z,\mathcal F^Z)$-martingale problem we then
have a solution $P^X$ of the $(\mu^X,G^X,\mathcal F^X)$-martingale
problem.

We now consider the law $P^{Y | X}$ of $Y=(Y_t)_{t\ge 0}$ with
$Y_t = F_2(Z_t)$ conditioned on the realization of the (whole path) of
the process $X=(X_t)_{t\ge 0}$. Here a second assumption is needed.
For every $t \ge 0$ we define
\begin{align}
  \label{eq:12}
  G^{Y | X}_t f_2(y) = G^Z(f_2 \circ F_2)(F((y,X_t))), \quad f_2 \in
  \mathcal F^Y,
\end{align}
where $\mathcal F^Y$ is an algebra of test functions $f$ on $E_2$
satisfying $f \circ F \in \mathcal F^Z$. We assume that for almost
surely every realization of $(X_t)_{t\ge 0}$
\begin{align}
  \label{eq:14}
  \text{the $(\mu^Y,(G_t^{Y|X})_{t \ge 0}, \mathcal{F}^Y)$-martingale
  problem is well-posed $P^X$-a.s.}
\end{align}
The corresponding solution is a \emph{time-inhomogeneous} martingale
problem.

We note that for a solution of the
$(\mu^Z,G^Z,\mathcal{F}^Z)$-martingale problem $P^{Y|X}$ is
$P^X$-a.s.\ a solution of \eqref{eq:14}. If this solution is unique,
then with $P^Y = \mathcal L\bigl((Y_t)_{t\ge 0}\bigr)$
\begin{align}
  \label{eq:15}
  P^Y = \int_{D([0,\infty),E_1)} P^{Y|X=x} \, P^X(dx).
\end{align}
The pair $(P^X,P^{Y|X})$ determines uniquely the joint law $P^{(X,Y)}$
of the processes $X$ and $Y$ and then automatically the law of the
solution $Z$ is given by the push-forward law under $F$:
\begin{align}
  \label{eq:16}
  P^Z = F_\ast P^{(X,Y)}.
\end{align}
Therefore we obtain a unique solution for the
$(\mu^Z,G^Z,\mathcal F^Z)$-martingale problem by solving the
\emph{skew martingale problem} consisting of the \emph{$X$-martingale
  problem} and the \emph{conditional $Y|X$-martingale problem}.

In the case of the $\U$-valued Feller diffusion $(\mfU_t)_{t\ge
  0}$ we have
\begin{align}
  \label{eq:17}
  E & = \U, \quad E_1 = \R_+, \quad E_2 = \U_1, \\
  \label{eq:18}
  X & = \bar \mfU, \quad Y = \widehat \mfU, \quad F(\bar\mfu, \hat
  \mfu) = [(u,r,\bar\mfu \hat\mu)].
\end{align}
We equip $E_1 \times E_2$ with the Gromov-weak topology and not with
the product topology. This way all elements $(0,\wh\mfu)$ for $\mfu
\in \U_1$ are identified and $F$ is indeed a (continuous) bijection.

The test functions are chosen as follows: $\mathcal F^X$ as the
algebra of the polynomials generated by $x \mapsto x^n$, $x \in \R_+$,
$n \in \N_0$, $\mathcal{F}^Y$ is the algebra of polynomials on $\U_1$
and $\mathcal{F}^Z$ is the algebra of polynomials on $\U$.

Then $X$ is the $\R_+$-valued Feller-diffusion and $Y |X$ is the
$\U_1$-valued Fleming-Viot diffusion with resampling rate at time $t$
given by $b(\bar\mfU_t)^{-1}$ before the extinction and as $\infty$ at
times beyond the extinction time, i.e.\
$\wh \mfU_t =[\{1\},\underline 0, \delta_1]$ for
$t \ge T_{\mathrm{ext}}$, where $T_{\mathrm{ext}}$ is the extinction
time of $(\bar\mfU_t)_{t\ge 0}$.

\subsubsection{Feynman-Kac duality, conditional duality and strong
  dual representation}\label{FKdual}

Here we first discuss a \emph{Feynman-Kac duality} and a
\emph{conditioned (on the total population size process) duality}
which allows for a \emph{strong conditioned dual representation}. The
dualities facilitate the study of finer properties of genealogies and
allow to prove uniqueness of solutions of martingale problems.

\paragraph{(1) Feynman-Kac duality}
There is a \emph{Feynman-Kac (FK) duality relation} for the
$\U$-valued Feller diffusion. The corresponding dual process is based
on a \emph{partition-valued} process and partition elements which we
associate with a certain position in the population, and in addition
the process is enriched by the \emph{distance matrix} .

Let $\bbS$ be the set of partitions of $\bbN$ with finite number of
partition elements, that is $p \in \bbS$ can be written as
$p=(\pi_1,\dots,\pi_n)$ for $n\in \N$ with pairwise disjoint
$\pi_\cdot$ and $\N = \pi_1 \cup \dots \cup \pi_n$. The $\pi_i$,
$i=1,\dots,n$ are called \emph{partition elements}. We define a
partial order on $\bbN$ induced by a partition by stipulating that
$i<j$ implies $\min \pi_i < \min \pi_j$, i.e., partition elements are
ordered according to their \emph{minimal element}. We denote by
$\abs{p}$ the number of partition elements of $p$.

For $\pi, \pi' \in p$ we define the partition obtained after their
coalescence in $p \in \bbS$ by
\begin{align}
  \label{tv6}
  \kappa_p(\pi,\pi') \coloneqq
  \left(p \setminus \{\pi,\pi'\}\right) \cup \{\pi \cup \pi'\}.
\end{align}
For $i \in \N$ and $p \in \bbS$ we set $p(i) = \pi_k$ if
$i \in \pi_k$, i.e.\ $p(i)$ is the partition element containing $i$.
We write $i \sim_p j$ if $p(i)=p(j)$, that is if $i$ and $j$ are in
the same partition element of $p$.

For $\dr \in [0,\infty)^{\binom{\N}{2}}$ and
$p = (\pi_1,\dotsc,\pi_n) \in \bbS$ we define the distance matrix
$\dr^p \in [0,\infty)^{\binom{\N}{2}}$ by
\begin{align}
  \label{tv7}
  (\dr^p)_{ij} \coloneqq r_{\min p(i), \min p(j) } , \quad 1 \le i < j.
\end{align}

Now we consider a sample and the associated distance matrix. For
$\dr = \dr (u_1,u_2,\dots)$, recall \eqref{eq:dr}, and
$p = (\pi_1,\dotsc,\pi_n) \in \bbS$ we define
$\dr^p=\dr^p(\underline u)$ by
\begin{align}
  \label{tv8}
  \dr^p(\underline u):
  \begin{cases}
    U^{|p|} & \to [0,\infty)^{\binom{\N}{2}},\\
    \underline{u}_p & \mapsto \dr^p (\underline{u}_p),
  \end{cases}
\end{align}
where $\underline{u}_p = (u_{\min \pi_1}, \dots, u_{\min \pi_{|p|}})$
and $(\dr^p(\underline{u}_p))_{ij} = r(u_{\min p(i)}, u_{\min p(j)})$.
Thus, it suffices to know $u_{\min \pi_1},\dots,u_{\min \pi_{|p|}}$ to
construct $\dr^p(\underline{u})$.

\medskip
The state space of the FK-dual will be
\begin{align}
  \label{tv10}
  \bbK = \bbS \times [0,\infty)^{\binom{\N}{2}},
\end{align}
where $[0,\infty)^{\binom{\N}{2}}$ is the space of
countably-dimensional distance matrices; cf.\ \eqref{tv2} for the
space of finitely-dimensional distance matrices. The space $\bbK$
equipped with the product topology is Polish (see \cite{GPWmp13}). We
note that every state of the coalescent is associated uniquely with a
finite ultrametric space.

The dual process $\mfK = (\mfK_t)_{t \ge 0}$ is a \emph{Kingman
  coalescent} enriched with an evolving distance matrix. The evolution
of $\mfK$ is as follows
\begin{itemize}
\item each pair of partition elements \emph{coalesces} independently
  at rate $b$,
\item distances between the elements $i,j \in \N$ are \emph{initially
    $0$} and \emph{grow} at speed $2$ as long as $i$ and $j$ are in
  different partition elements and then we define the distance of two
  partition elements as twice the time until the coalescence and $2t$
  if no such coalescence event occurs.
\end{itemize}

We define a set
$\mathcal B_{\mathrm{fc}} \coloneqq \mathcal B_{\mathrm{fc}}$ by
\begin{align}
  \label{eq:20}
  \begin{split}
    \mathcal B_{\mathrm{fc}} \coloneqq \bigl\{\varphi: &
    [0,\infty)^{\binom{\N}{2}}
    \to \R : \\
    & \qquad \qquad \varphi \text{ is bounded, measurable and depends on
      finitely many coordinates}\}
  \end{split}
\end{align}
For any $\phi \in \mathcal B_{\mathrm{fc}}$ and $(p,\dr') \in \bbK$ we
define analogously to \cite{GPWmp13} the polynomial $H^\varphi$ by
\begin{align}
  \label{tv11}
  H^{\varphi} :
  \begin{cases}
    \bbU \times (\bbS \times [0,\infty)^{\binom{\N}{2}}) & \to \R \\
    \bigl(\mfu, (p,\dr')\bigr) & \mapsto
    H^{\varphi}\bigl(\mfu,(p,\dr')\bigr) =
    \int \mu^{\otimes |p|}(\dx \underline{u}_p)\, \varphi \left(
      \dr^p(\underline{u}_p) + \dr' \right).
  \end{cases}
\end{align}
We define on $\U \times (\bbK \times \mathcal B_{\mathrm{fc}})$ the
\emph{duality function} $H$ by
\begin{align}
  \label{eq:19}
  H(\mfu,((p,\dr'),\varphi)) = H^\varphi (\mfu,(p,\dr')),
\end{align}
where the dual component $\varphi$ is set to be constant in time.

The set of functions
$\bigl\{H^\varphi\bigl(\cdot,(p,\dr')\bigl) : (p,\dr') \in \bbS,\,
\varphi: \mathcal B_{\mathrm{fc}}\bigr\}$ is \emph{separating} and
\emph{convergence determining} on $\wt \mcM$ (recall \eqref{tv4} and
Corollary ~\ref{cor.681}) and hence $H$ is a good duality function for
laws supported on $\wt \mcM$.

Next, we relate the enriched Kingman coalescent $\mfK$ with the
$\U$-valued Feller diffusion.
\begin{theorem}[Feynman-Kac duality for genealogies: Kingman and
  Feller]
  \label{T:DUALITY}
  % \leavevmode \\
  For $\mfu_0 \in \U$ let $\mfU = (\mfU_t)_{t\ge 0}$ be the solution
  to the $(\delta_{\mfu_0}, \Omega^{\uparrow}, \Pi(\mathcal C_b^1))$-martingale
  problem. For $(p_0,\dr'_0) \in \bbK$ let
  $\mfK=(\mfK_t)_{t\ge 0}=((p_t,\dr_t'))_{t\ge 0}$ be as defined above
  with initial condition $(p_0,\dr'_0)$. Then, for all $\varphi$
  depending on finitely many coordinates we have for all
  $\mfu_0 \in \U,(p_0,\uur'_0) \in \K$:
  \begin{align}
    \label{tv12}
    \E_{\mfu_0} \left[ H \bigl(\mfU_t,((p_0,\dr'_0),\varphi)\bigr) \right]
    = \E_{(p_0,\dr'_0)} \left[ H\bigl(\mfu_0,((p_t,\dr_t'),\varphi)\bigr)
    \exp\left(\int_0^t b {\binom{|p_s|}{2}} \, \dx s \right) \right],
    \quad \forall \; t \ge 0.
  \end{align}
\end{theorem}

\begin{remark}[Non-critical case]
  \label{r.concrit}
  The FK-duality holds also for non-critical $\U$-valued branching
  diffusions based on the operators \eqref{tv5} and in addition the
  operator from \eqref{e709}. The only modification in the case of the
  Feynman-Kac potential, namely the exponent on the right hand side of
  \eqref{tv12} carries as in the critical case the occupation time of
  the coalescent, but in the noncritical case an additional term in
  the potential is added, meaning that the potential for the critical
  case $a=0$ is replaced for $a\ne 0$ by
  \begin{align}
    \label{e.ggr1}
    \int^t_0 \biggl( b \binom{|p_s|}{2} + a |p_s|\biggr) \, \dx s.
  \end{align}
  Here $a \in \R$ is the non-criticality coefficient; recall
  \eqref{eq:FelBDiff-ab}. Note, that the non-criticality does not add
  ``splitting'' to the ancestral tree (recall Remark~\ref{r.ultsp})
  because the \emph{coalescence rate remains} the same. However, there
  is a \emph{reweighting} of the tree by the changed exponential term
  giving more respectively less weight, depending on the sign of $a$,
  to coalescent paths with later mergers than the Kingman coalescent
  and therefore bigger distances.
\end{remark}

\begin{remark}[Kingman coalescent genealogy]
  For \label{r.coaltre} the $\U_1$-valued Fleming-Viot diffusion the
  duality allows to relate the ultrametric probability measure space
  which is associated with the \emph{entrance law} of the Kingman
  coalescent from a countable population at time zero evolved for
  infinite time; see \cite{GPW09}. In the case of the Feller diffusion
  for every finite population of size $n$ a \emph{reweighting} takes
  place through the Feynman-Kac term. However, this does \emph{not}
  lead to a \emph{consistent family of laws} of a process, in which
  all the finite ultrametric spaces are embedded.
\end{remark}

The issue described in the above remark says in particular that the
Kingman coalescent cannot be used for a strong duality. However, there
is another possibility.

\paragraph{(2) Conditional duality}
Here we introduce \emph{conditional duality} for the pure genealogy
part, using the fact that $(\bar \mfU_t)_{t \ge 0}$ is an
\emph{autonomous Markov process} on $\R_+$. More precisely, we
introduce a duality for the $\U$-valued Feller diffusion
\emph{conditioned} on the \emph{complete} total mass process
$\bar\mfU= (\bar \mfU_t)_{t \ge 0}$ and observe only the process
$\wh\mfU$ for which we want the dual representation.

The dual of $\wh \mfU$ conditioned on
$\bar\mfU=(\bar\mfU_t)_{t\ge 0}$, denoted by
$\mfC(\bar\mfU) = (\mfC_t(\bar\mfU))_{t\ge 0}$, is as before an
\emph{enriched coalescent} but with a \emph{time dependent}
coalescence rate at time $t$ given by
\begin{align}
  \label{e3566}
  b \cdot \bar \mfU_t^{-1}
\end{align}
if $\mfU_t>0$ and $\infty$ after the extinction of $\bar\mfU$. At the
time when $\bar \mfU_t$ hits zero the process $\mfC(\bar\mfU)$
coalesces to a single partition in the time before, so that from the
extinction time on it is constant equal to a single partition. This
object was considered in \cite{DG03}. See also
Section~\ref{ss.uniqumar}. The duality functions $H(\cdot,\cdot)$ are
the same as before, i.e.\ as defined in \eqref{tv11}, with the
difference that $\mu$ is replaced by $\hat \mu$. More precisely the
duality function is now a function on $\U_1$ in its first variable and
for the process $\wh \mfU$ \emph{conditioned} on $\bar \mfU$ the
duality relation is the same as in \eqref{tv12} but \emph{without} the
reweighting through the FK-term.

\begin{theorem}[Conditioned duality for $\U$-valued Feller]
  \label{PROP.1007}
  % \leavevmode\\
  Let $\bar \mfu = (\bar \mfu_t)_{t \ge 0}$ be an admissible total
  mass path as defined in Definition~\ref{def:adm_u} and $\mfU$ the
  $\U$-valued Feller diffusion starting in $\mfu \in \U$. The following
  duality relations holds for all $\hat\mfu_0 \in \U_1$ resp.\
  $\mfu_0 \in \U$,
  $(p,\dr') \in \bbK$,
  $\varphi \in C_b \bigl([0,\infty)^{\binom{\N}{2}},\R)$ and for the
  following two processes.
  \begin{enumerate}[(a)]
  \item For processes $\wh \mfU^{\mathrm{FV}}(\bar \mfu)$ and the
    coalescent $\mfC(\bar \mfu)$:
    \begin{align}
      \label{e1015}
      \E_{\hat \mfu_0} \bigl[H\bigl(\wh \mfU^{\mathrm{FV}}_t(\bar \mfu),
      ((p,\uur'),\varphi)\bigr) \bigr]
      = \E_{(p,\uur')} \bigl[H
      \bigl(\hat\mfu_0,((p_t(\bar \mfu),\uur_t'(\bar
      \mfu)),\varphi)\bigr)\bigr].
    \end{align}
  \item For processes $\wh \mfU (\bar\mfu)$, a functional of $\mfU$ and the coalescent
    $\mfC(\bar \mfu)$:
    \begin{align}
      \label{e1015a}
      \E_{\hat \mfu_0} \bigl[H\bigl(\wh \mfU_t(\bar \mfu),
      ((p,\uur'),\varphi)\bigr) \big\vert \bar\mfU = \bar\mfu \bigr]
      = \E_{(p,\uur')} \bigl[H \bigl(\hat\mfu_0,((p_t(\bar \mfu),\uur_t'(\bar
      \mfu)),\varphi)\bigr)\bigr].
      % \quad \hat\mfu_0 \in \U_1, \; (p,\dr') \in \bbK,
    \end{align}
    almost surely w.r.t.\ the distribution of $\bar\mfU$.
  \end{enumerate}
\end{theorem}

\begin{remark}[State dependent branching]\label{r.1022}
  The conditioned duality holds also in the case of total mass
  dependent branching. Suppose that the underlying total mass process
  $Z$ solves the SDE $\dx Z_t=\sqrt{b(Z_t)} \, \dx B_t$ for a locally
  Lipschitz function $b$ with $b(0)=0$, $b(z) > 0$ for $z > 0$ and
  $g(z)= O(z^2)$ as $z \to \infty$. Then the individuals total mass
  dependent branching rate is $h(z)=b(z)/z$, and the rate of
  coalescence in the conditional duality is $h(z)/z$ when the total
  mass is $z$.
\end{remark}

To understand better the genealogies in the \emph{non-critical} case,
recall Remark~\ref{r.offspring}, we consider them from the point of
view of the \emph{conditioned duality}. In this case we have
Feynman-Kac duality with potential $a |p_t|$.

In the non-critical case the conditional duality is as follows.
Consider test functions $\Phi = \bar\Phi \wh\Phi \in \mcD_1$ with
$\bar\Phi = const$ and assume that $\wh\Phi = \wh \Phi^{n,\varphi}$
with $\varphi$ depending on $n$ coordinates. On these test functions
we know the generator.

We have to recall that the duality condition w.r.t.\ the function $H$
for a process with generator $G$ and dual process with generator
$G_{\text{dual}}$ reads
\begin{align}
  \label{eq:6}
  GH(\cdot,\mfK)(\mfu) = G_{\text{dual}} H (\mfu,\cdot)(\mfK), \quad
  \mfK \in \bbK, \, \mfu \in \U.
\end{align}
We calculate both sides. Calculations similar to those in
Remark~\ref{r.1207} show that in the non-critical case only a
Feynman-Kac term to the conditional dual for $\wh\mfU$ because the
effect of the non-criticality resides in the total masses. This term
is given by
\begin{align}
  \label{e925}
  a n \, \wh \Phi (\hat \mfu).
\end{align}
Thus, there is a \emph{reweighting} of the critical finite sampled
trees which arise as time-inhomogeneous coalescent trees, longer
branches are more or less favored depending on the sign of $a$.

\paragraph{(3) Strong conditioned dual representation}
Above observation raises the question of a, as we call it,
\emph{strong} duality where the whole state is represented by the
dual. Indeed the duality relation above allows to give a \emph{strong
  dual representation} in terms of the path of the \emph{autonomous}
total mass process and the \emph{probability entrance law} of the
\emph{conditioned dual process} associated with our process for times
$t \in [0,T_{\mathrm{ext}})$, where $T_{\mathrm{ext}}$ is the
extinction time of $\mfU$. Namely we associate with the
\emph{enriched} coalescent an ultrametric measure space based on the
ultrametric from \eqref{tv7} and the uniform distribution. Thus, we
obtain a process $(\mfC_t(\bar \mfu))_{t \in [0,T_{\mathrm{ext}})}$
for every finite number of basic individuals for the coalescent. Here
for the duality at time $t$ the coalescence rate at time $s$ for
$s \in [0,t)$ is $(\mfu_{t-s})^{-1}$ at backward time $s \in [0,t]$.
It has been shown that one can construct the $\U_1$-valued
\emph{probability entrance law} starting with countably many
individuals denoted $\mfC^\infty(\bar \mfu)$; see \cite{GPW09} for the
time-homogeneous case where the existence of the entrance law is
shown. Then we can strengthen the Theorem~\ref{PROP.1007} to a
stronger statement about the state of the genealogy process
$\wh \mfU_t(\bar \mfu)$ conditioned on the total mass process
$\bar \mfu$. Namely it has the following form.
\begin{corollary}[Strong conditioned duality]
  \label{cor.1226}
  \begin{align}
    \label{e892}
    \mathcal{L}[\wh \mfU_t(\bar \mfu)] =
    \mathcal{L}[\mathfrak{C}_t^\infty (\bar \mfu)], \quad
    \mcL \left[(\bar \mfU)_{t \geq 0}\right] \text{ - a.s. in } \bar
    \mfu \quad \text{for } t \in [0,T_{\mathrm{ext}}).
  \end{align}
\end{corollary}
\begin{proof}
  The point is, that if we know indeed that the entrance law exists we
  can argue as follows: From the conditional duality relation we know
  that every finite $n$-subcoalescent of the entrance law equals in
  law the element of $\U_1$ given by a sample of $n$ points from the
  process $\widehat{\mathcal U}$. Since the entrance law specifies a
  skeleton of $\mfC_t^\infty$ this implies the identity of the two
  objects.
\end{proof}

The presence of the term \eqref{e925} in the non-critical case means
that we can \emph{not} obtain a \emph{strong} duality for
$\mfu \neq \ntree$ since then we have a \emph{Feynman-Kac} duality in
which case the laws for different sizes of the coalescent, i.e.\
number of individuals in the basic set which we partition \emph{not}
form a \emph{consistent family} of laws on $\U$.

\subsubsection[Generalized branching, Markov branching tree
and Cox cluster representation]{Structural properties: Generalized
  branching property, Markov branching tree and Cox cluster
  representation}
\label{branchmarkcox}
We turn now to three \emph{different} structural properties of the
genealogy of the population currently alive summarized in
Theorem~\ref{T:BRANCHING}. We want to decompose the population and
identify the law of the \emph{number} of depth-$h$ single ancestor
subfamilies (open 2$h$-balls), to find their \emph{law} $\varrho_h^t$
if we consider each of them as a random element in $\U(h)$ and to
identify the joint law of this set of $h$-subfamilies. This gives the
state at time $t$ as \emph{concatenation of the families of the
  distinct depth-$h$ founding fathers}. The key point is to show that
they are \emph{i.i.d.} and their number is \emph{Poisson} distributed
if we \emph{condition on the total mass} of the whole population at
time $t-h$. This identification is the so called \emph{Cox cluster
  representation} of the genealogy in terms of depth-$2h$ single
ancestor subfamilies extending the corresponding measure-valued
notion; see \cite{D93}.

Indeed the semigroup structures $\{(\U,\sqcup^h): h >0\}$ allows to
define the concept of the \emph{generalized branching property} of a
semigroup $(Q_t)_{t \ge 0}$ for a $\U$-valued Markov process (this is
recalled in Definition~\ref{d.mbt}(a) below in a form useful here)
which has been studied in detail in \cite{ggr_GeneralBranching}.
Furthermore, in \cite{infdiv} for a random element in $\U$ the concept
of a \emph{Markov branching tree} was introduced to describe this
particular $h$-subfamily structures. The Cox cluster representation
then follows from the generalized \emph{branching property} since this
implies \emph{infinite divisibility} of the marginal distribution and
with a special form of the \Levy{}-Khintchine formula on
$(\U(h),\sqcup^h)$ which is based on concepts and results in
\cite{infdiv}.

To formulate our result for the $\U$-valued Feller process below in
Theorem~\ref{T:BRANCHING} we need first three groups of concepts and
ingredients, \emph{(Markov) branching property, Yule tree} with leaf
law, and the (autonomous) total mass process $\bar\mfU$ which is given
by $\R_+$-valued Feller diffusion which we label (1), (2), (3). The
key result is Theorem~\ref{T:BRANCHING} below.

Finally in a separate paragraph we apply the Feynman-Kac
\emph{duality} to understand better Cox cluster representation from
the previous subsection.

\paragraph{(1)}
We want to show that the $\U$-valued Feller diffusion has states
$\mfU_t$ with the generalized branching property and hence produces
Markov branching trees. Then we use this to obtain a \emph{Cox point
  process representation} of the genealogy $\mfU_t$ as concatenation
over a Cox point process on genealogies generating a random depth-$h$
subfamilies, this number, which we denote by $M_h$, is the number of
\emph{depth-$h$ single ancestor} subfamilies which are elements of
$\U(h)$ and which we denote by $\mfY^{(h)}$. We shall ``explicitly''
determine these ingredients of the CPP in Theorem~\ref{T:BRANCHING}
below. Explicitly means here to characterize the law of $M_h$ and
$\mfY^{(h)}$.

We next introduce rigorously the needed concepts. Recall here
\eqref{eq:pol:trunc} for $\Phi_t$ and \eqref{tv14} for the $h$-tops
$\lfloor\mfu\rfloor (h)$.

\begin{definition}[Generalized branching, Markov branching tree, Cox
  cluster representation]\label{d.mbt}
  \leavevmode
  \begin{enumerate}[(a)]
  \item We say that a semigroup $Q_t$ (or an associated Markov
    process) on $\mcB(\U)$ has the \emph{generalized branching
      property} if for every $\mfu_1,\mfu_2 \in \U$ and for every
    $\Phi \in \Pi$:
    \begin{align}
      \label{tv15}
      \int Q_t(\mfu_1 \sqcup^{t} \mfu_2, \dx\mfu) \Phi_t(\mfu) =
      \int Q_t(\mfu_1, \dx \mfu) \Phi_t(\mfu)
      + \int Q_t(\mfu_2, \dx \mfu) \Phi_t(\mfu).
    \end{align}
  \item We say that the random ultrametric space $\mfU$ is a
    \emph{$t$-Markov branching tree} if for every $h \in [0,t)$, for the
    $h$-tops $\lfloor\mfU\rfloor (h)$ there exist
    \begin{align}
      \label{e872}
      m_h \in \mcM([0,\infty)), \quad \text{and} \quad \varrho_h
      \in \mcM_1 (\U) \; \text{ with full measure on}\;
      \U(h) \setminus \{\ntree\},
    \end{align}
    such that we have a \emph{Cox point process representation} of the
    $\U(h)^\sqcup$-valued $h$-top, i.e.\ the $h$-top is the
    concatenation of a \emph{mixed Poisson number $M_h$ of i.i.d.\
      $\U(h)\setminus\{\ntree\}$ valued random variables $\mfU_i$ with law
      $\varrho_h$}:
    \begin{align}
      \label{tv16}
      \lfloor \mfU\rfloor (h) =
      \mathop{\bigsqcup\nolimits^h}_{i=1,\dots,M_h} \mfU_i.
    \end{align}
    Here, the empty concatenation is the zero element of $\U$ and the
    \emph{mixing measure} (or Cox measure) $m_h$ \emph{for $M_h$ is
      infinitely divisible} and its \Levy{} measure will be denoted by
    $\lambda^{m_h}$.
  \end{enumerate}
\end{definition}

\begin{remark}[Equivalent definition of generalized branching
  property]
  \label{r.962}
  The generalized branching property can equivalently be described in
  terms of separating multiplicative functions or by requiring that
  on $\mcB(\U(t+h))$ we have
  \begin{align}
    \label{e1129}
    Q_t(u_1 \sqcup^h u_2, \cdot)= (Q_t(u_2,\cdot)
    \ast^h \; Q_t(u_2,\cdot))(\cdot),\quad h>0, t \ge 0.
  \end{align}
  Here, $\ast^h$ denotes the convolution with respect to $\sqcup^h$ on
  $\U(h)^\sqcup$ extended to $\U$.
\end{remark}

An example of a Markov branching tree is provided by \emph{compound
  Poisson forests} on $\U(t)^\sqcup$. For the definition recall
$\U(t)$, $\U(t)^\sqcup$ and the concatenation operation $\sqcup^t$
from \eqref{e.tr48} -- \eqref{e.tr47}.

\begin{definition}[Compound Poisson forest]
  Let \label{d.compoi} $\theta>0$ and
  $\upsilon\in\mcM_1(\bbU(t)^\sqcup \backslash \{\ntree\})$. Let $M$ be
  a Poisson random variable with parameter $\theta$. Let $\mfU_i$,
  $i\in\bbN$, be an i.i.d.\ sequence of random $t$-forests with
  $\mcL[\mfU_1]= \upsilon$. Assume that $(\mfU_i)_{i\in\bbN}$ and $M$
  are independent. We call the $t$-concatenation of
  $(\mfU_i)_{i\in\{1,\dots,M\}}$ defined by
  \begin{align}
    \label{rg12}
    \mfP_t\coloneqq \sideset{}{^{t}}\bigsqcup_{i=1,\dots,M} \mfU_i,
  \end{align}
  a \emph{compound Poisson $t$-forest} with parameters $\theta$ and
  $\upsilon$, a $\textup{CPF}_t(\theta,\upsilon)$ for short.
\end{definition}
Note that every $\textup{CPF}_t(\theta,\upsilon)$ is a random
$t$-forest, i.e.\ an element of $\bbU(t)^\sqcup$. The corresponding
\emph{$t$-\Levy{} measure} is given by $\theta \cdot \upsilon$. If
$\upsilon$ puts full measure on $\U(t)$ then the $\mfU_i$ are actually
``trees'', i.e.\ \emph{single ancestor elements}.

\begin{remark}[\Levy{} measure of a Markov branching tree]
  \label{r.1394}
  A \emph{Markov branching tree} has an \emph{infinitely divisible
    law} (see Proposition~\ref{l.ac7289}) whose \emph{\Levy{} measure}
  has a particular form that allows for a Cox point process
  representation by a concatenation of elements in $\U(h)$, i.e.\ the
  \emph{prime elements} of $\U$ describing \emph{single ancestor
    subfamilies}. In the general infinitely divisible case based on
  the \Levy{}-Khintchine representation of the Laplace functional (see
  \eqref{ag2inf} in Section~\ref{sss.concat}) one would expect
  ``only'' a Poisson point process representation by concatenation of
  elements from $\U(h)^\sqcup$.

  The \Levy{} measure $\Lambda_h^{\mfU_t}$ of the $t$-Markov branching
  tree $\mfU_t$ is of the following form (see
  \eqref{n.e.tilde.lambda}, where we denote by $\P$ the probability
  law of the PPP\, $N(y \varrho^t_h(\cdot))$):
  \begin{align}
    \label{e.tr12}
    \Lambda_{h}^{\mfU_t}(\dx \mfu) = \int_{\R_+}
    \lambda^{m_h^t}(\dx y) \, \mathbb{P}\Bigl(\bigsqcup_{\mfw \in
    N(y \varrho_h^t(\cdot))} \mfw \in \dx \mfu \Bigr), \quad h\in (0,t].
  \end{align}
  Here, similarly to notation of Definition~\ref{d.mbt} (but adding an
  additional superscript $t$ on $m_h$ and $\varrho_h$), $m_h^t$ is an
  infinitely divisible law on $[0,\infty)$. In our context $m_h^t$
  will be the law of the random total population size
  $\bar\mfU_{t-h}$. Furthermore, $\lambda^{m_h^t}$ is the \Levy{}
  measure of the law $m_h^t$, $N(y \varrho_h^t(\cdot))$ is a PPP on
  $\U(h)$ with \emph{intensity measure} $y \varrho_h^t$. The intensity
  measure
  \begin{align}
    \label{eq:38}
    \begin{split}
      \text{$\varrho^t_h$}
      & \text{ is the \emph{\Levy{} measure} of the
        $h$-truncation from the \Levy{}-Khintchine} \\
      & \text{representation of $\U$-valued random variables $\mfU_t$}.
    \end{split}
  \end{align}
  We obtain it by \emph{fixing} $y$ in \eqref{e.tr12} and taking as
  $\mfU_t$ the concatenation over the PPP $N(y \varrho_h^t)$, cf.\
  \eqref{rg12} and see Theorem~1.37 in combination with
  Corollary~1.40 in \cite{infdiv}.
\end{remark}

\paragraph{(2)}
\label{tv847}
For a Markov branching tree our goal is to identify first $M_h$ as
mixed Poisson, i.e.\ identifying the mixing measure $m_h$ which is the
law of $M_h$ and, second determine the law $\varrho_h$ of $\mfY^{(h)}$
which provides the law of the summands $\mfU_i$ from of the
concatenation producing $\mfU_t$ via \eqref{e872}. We aim at giving a
device which generates $(\mfY^{(h)}_s)_{s \in [0,h)}$ by giving the
$s$-\emph{truncations as stochastic process in $s$}. To this end, we
need the \emph{Yule processes with leaf laws} and \emph{compound
  Poisson point processes} on $\U \setminus \{\ntree\}$ we define in
the following definitions. The above mentioned device will be derived
in the proof of Theorem~\ref{T:BRANCHING} in
Section~\ref{sec:branching}.

\begin{definition}[Genealogical Yule tree]
  %\leavevmode\\
  Fix $t>0$. \label{Yule} A $\U(t)$-valued random variable $\mfU$ is
  called a \emph{Yule tree} with \emph{splitting rate}
  $(\beta_s)_{s\in [0,t)} \in [0,\infty)^{[0,t)}$ and \emph{leaf law}
  $\nu_t \in \mcM_1([0,\infty))$, denoted by
  \begin{align}
    \label{grx63}
    \operatorname{Yule}\bigl((\beta_s)_{s\in [0,t)}, \nu_t\bigr)
    = \Bigl[\wt M= \{1,\dots, M\}, r, \mu = \sum_{i=1}^M \bar m_i
    \delta_{\{i\}}\Bigr],
  \end{align}
  if the metric space $(\wt M,r)$ is generated by a Yule tree with
  \emph{splitting rate} $(\beta_s)_{s\in [0,t)}$ independent of the
  masses of the leaves and $r$ being the genealogical distance. The
  latter are given by the i.i.d.\ nonnegative random variables
  $\bar m_1, \bar m_2, \dots$ distributed according to the \emph{leaf
    law} $\nu_t$ and give the sizes of masses at the leaves
  $\{1,\dots, M\}$.

  In other words, the \emph{generator} of the \emph{driving} Yule tree
  process $([\{1,\ldots, M_s\}, r_s,\mu_s])_{s \in [0,t)}$ acts on
  functions $\Phi^{m,\varphi} \in \Pi(\mathcal C_b^1)$ at time $s$ as follows
  (recall \eqref{mr3})
  \begin{align}
    \label{grx64}
    A_s \Phi^{m,\varphi} (\mfu)
    = \Phi^{m,\overline{\nabla} \varphi}(\mfu) + \beta_s \sum_{j=1}^{\bar{\mfu}}
    \left(\int(\mu+\delta_{\{j\}})^{\otimes m}(\dx \underline{u})
    \varphi(\dr(\underline{u}))
    - \int \mu^{\otimes m}(\dx \underline{u}) \varphi(\dr(\underline{u}))
    \right).
  \end{align}
  Here, $M_s$ is the number of leaves at time $s$ and $\mu_s$ is the
  counting measure on $\{1,\dots,M_s\}$. Furthermore $\mfu =
  [\{1,\dots, \bar{\mfu}\}, r, \sum_{j=1}^{\bar{\mfu}}
  \delta_{\{j\}}]$. Note that here the mass is a positive integer.
\end{definition}

\medskip
\noindent
\textbf{\emph{Construction of a Yule tree with prescribed splitting
    rate and leaf law}} \; To obtain $\mfY^{(h)}$ we need a specific
Yule tree. Consider an elementary individual based Yule process on
time interval $[0,h)$ starting with one individual at time $0$ and
splitting at time $s \in [0,h)$ at rate (compare also \eqref{e1401})
\begin{align}
  \label{e1407}
  2(h-s)^{-1}.
\end{align}
For every $s \in (0,h)$ the Yule tree gives rise to an ultrametric
space whose ultrametric is given by the genealogical distance. Next,
equip the space at time $s$ with the \emph{leaf law} $\nu^{(h)}_{s}$
given as the exponential distribution with parameter
$2(b(h-s))^{-1}$. We obtain a collection of processes
\begin{align}
  \label{e934}
  (\mfY_s^{(h)})_{s \in [0,h)} \text{ with values in  } \U, \quad
  \mfY^{(h)}_s= \operatorname{Yule} \Big( \big( 2(h-s')^{-1} \big)_{s' \in
  [0,s]}, \Exp\big( 2(b(h-s))^{-1} \big)\Big).
\end{align}
Then we can define the limiting forest $\mfY^{(h)}$ at time $h$. The
proof of the following lemma can be found in Section~\ref{ss.th3} on
p.~\pageref{p.l.e1408}.
\begin{lemma}[Existence of $\mfY^{(h)}$] \label{l.e1408}
  % \leavevmode \\
  For each $h>0$ there is a $\U$-valued random variable $\mfY^{(h)}$
  so that
  \begin{align}
    \label{e1408}
    \mcL \bigl[\mfY^{(h)} \bigr] = \lim_{t \uparrow h} \; \mcL
    \bigl[\mfY_t^{(h)} \bigr].
  \end{align}
\end{lemma}

\paragraph{(3)}
Furthermore we need to be able to define later on the \Levy{} measure
and need the Feller diffusion  $(Z_t)_{t \ge 0}$, which is the
solution of
\begin{align}
  \label{e948}
  \dx Z_t = \sqrt{bZ_t} \; \dx B_t, \;Z _0=\mu_0(U_0),
\end{align}
where $B$ is standard Brownian motion. This is needed in order to be
able to condition on the total mass process and then being able to
represent the condition as an autonomous stochastic process.

\medskip

Now with the points (1)-(3) we have all the needed ingredients and can
state our theorem. Recall the notation and concepts introduced in
Section~\ref{sss.concat}. The proof of the following theorem is given
in Section~\ref{ss.th3}.
\begin{theorem}[Branching property, Markov branching tree, \Levy{}
  measure, conditioned genealogy process]
  \label{T:BRANCHING}
  % \leavevmode \\
  Consider the initial state $\bar u \mfe$ for some $\bar u >0$. For
  general initial state the result holds for
  $\lfloor \mfU_t\rfloor(t)$ instead of $\mfU_t$.
  \begin{enumerate}[(a)]
  \item The \emph{$\U$-valued Feller diffusion} $(\mfU_t)_{t\ge 0}$
    has the generalized branching property.
  \item If $\mfU_0$ is in $\U$, then for each $t>0$ the random
    variable $\mfU_t$ is a $t$-Markov branching tree and
    $t$-infinitely divisible. The parameters of the corresponding Cox
    point process representation on $\U(h)^\sqcup$ from \eqref{tv16}
    and ingredients \eqref{e872} and \eqref{e948} are as follows. For each
    $h \in (0,t]$, $t>0$ we have
    \begin{align}
      \label{e1402}
      m_h^t & = \mcL[2(bh)^{-1}Z_{t-h}],\\
      \label{e1403}
      \varrho_h^t & = \mcL[\mfY^{(h)}],
    \end{align}
    where $\mfY^{(h)}$ is the random variable from \eqref{e1408}, and
    here $\varrho_h^t$ does not depend on $t$.
  \end{enumerate}
\end{theorem}

Note that in (b) the law $\varrho_h^t$ would depend on $t$ if the
diffusion coefficient $b$ would be inhomogeneous. From \eqref{e1403}
and the construction of $\mfY^{(h)}$ we can conclude that the law
$\varrho_h^t$ is in fact concentrated on $\U(h)$, that is, on
\emph{open} $2h$-balls which correspond to ``depth at most $h$''
single ancestor subfamilies and $\varrho^{t'}_h=\varrho^t_h$ for
all $t' \geq t \geq h$. In particular if we consider the
\emph{path of decompositions}
\begin{equation}
  \label{e2082}
  h \mapsto {\mathop{\bigsqcup\nolimits^h}\limits_{i \in I^h}}
  \left[U_i^h,r_i^h,\mu_i^h \right],
  \quad
  |I^h|= \Pois \left(2(bh)^{-1} Z_{t-h}\right),
\end{equation}
then for all $h \in (0,t]$, \emph{given} $Z_{t-h}$ \emph{and} $I_h$
the elements $[U_i^h,r_i^h,\mu_i^h]$, $i \in I_h$ are \emph{i.i.d.}
with law $\varrho_h^t$.

\medskip
Two further different characterizations of $\varrho_h^t$ are
given below in \eqref{e3732} in terms of \emph{coalescents} and in
\eqref{e1611} in terms of \emph{entrance laws} and excursion laws.

\begin{remark}[Cluster representation]
  Theorem~\ref{T:BRANCHING}, part (b) gives the \emph{Cox cluster
    representation} (see \cite{D93} page 45/46 for that concept),
  i.e.\ a unique decomposition into \emph{depth-$h$ single ancestor
    subfamilies} of the time $t$ population and its state $\mfU_t$ can
  be represented accordingly as a concatenation over a Cox point
  process on $\U(h)$ via \eqref{tv16}. Here the Cox measure and the
  \emph{single ancestor subfamily law} are given in \eqref{e1402} and
  \eqref{e1403}. More precisely we can represent the $t$-top of
  $\mfU_t$ as a concatenation of a $\Pois(\bar \mfU_0)$ number of
  random elements in $\U(t)$ chosen at random according to
  $\varrho_t^t$. This corresponds to the decomposition in the
  \emph{families of founding fathers}. Moreover, for $h \in (0,t)$ the
  $h$-tops have a \emph{Cox cluster representation} with Cox measure
  $m_h^t$ given via the total mass $\bar \mfU_{t-h}$ at time $t-h$. In
  particular, for given $\bar \mfU_{t-h}=u$ we have the representation
  as a concatenation of $\Pois((bh)^{-1}u)$ distributed number of
  independent random variables with distribution $\varrho^t_h$.
\end{remark}

\begin{remark}[The associated path of subfamily
  decompositions identification]\label{r.21099}
  Since $\mfU_t$ is a state in a stochastic branching process we can
  consider the \emph{whole path of a family decompositions in $h$},
  which will give us the complete geometric structure of $\mfU_t$ if
  we vary $h$ in $(0,t)$ and in particular $\mfU_t$ as limit
  $h\uparrow T$ where the balls are successively partitioned further
  and further.

  The \emph{\Levy{} measure} of the Cox measure $m_h^t$, denoted by
  $\lambda^{m_h^t}$, is explicitly known to be given by
  \begin{align}
    \label{e1790}
    \lambda^{m_h^t} (\dx z) = \frac{1}{\left((t-h)b/2\right)^{2}}
    \exp\Bigl(-\frac{z}{(t-h)b/2}\Bigr) \,\dx z.
  \end{align}
  We insert this in \eqref{e.tr12} applied to $\U$-valued Feller
  diffusion and obtain a decomposition in \emph{depth-$h$ subfamilies}
  corresponding in a representative of the state at time $t$ to
  \emph{decomposition in open $2h$-balls grouped in open $2t$-balls}.
  More precisely, we obtain a decomposition in $M_h^t$ different open
  $2t$-balls each of which is decomposed in $N^{t,(i)}_{h}$ many open
  $2h$-balls $\mfU^i_k$, where for $k=1,\dots, N^{t,(i)}_{h}$,
  $i=1,2,\dots,M_h^t$ the $\U(h)$-valued random variables $\mfU^i_k$
  are \emph{independent} of $N^{t,(i)}_{h}$ and $M_h^t$ and are
  i.i.d.\ distributed according to $\varrho^t_h$.

  Let $(Y_i)_{i \in \N}$ be i.i.d.\ $\Exp((t-h)b/2)$-distributed, and
  let $N^{t,(i)}_{h}$ be independent $\Pois(\frac{2}{bh} Y_i)$
  distributed random variables. The number of $i$ with
  $N^{t,(i)}_{h}\geq 1$ is given by $M_h^t$ and this can be thought of
  as considering $2h$-balls in distance less than $2t$ and group them
  in the $2t$-balls.
\end{remark}

\begin{remark}[Relation to Cox cluster representation of $\R_+$-valued
  Feller diffusion]\label{r.1810} %\leavevmode \\
  Projection onto the total mass component of the state $\mfU$ results
  in a Cox point process representation of the total mass
  corresponding for each $h \in [0,t]$ to a different depth-$h$ single
  ancestor subfamily decomposition of $\bar \mfU_t$ and each time we
  get a sum of i.i.d.\ masses, which has its own Cox-measure $m_h^t$,
  i.e.
  \begin{align}
  \label{e1812}
    \bar \mfU_t=\sum_{\bar \mfw \in N(Y \cdot \varrho^t_h)} \bar
    \mfw, \qquad \mcL[Z]=m_h^t.
  \end{align}
  This is the \emph{Cox cluster representation} of the $\R_+$-valued
  Feller diffusion at the depth $h$. The problem is now to identify
  $\varrho^t_h$ projected on the component $\bar \mfu$ as a measure on
  $\R_+$. This would give $\mcL[\bar \mfw]$. This will be identified
  in the next section on entrance laws of the $\U$-valued Feller
  diffusion as the \emph{entrance law} of an $\R_+$-valued Feller
  diffusion from state $0$ at time $h$. This means that via the
  projection we obtain the i.i.d.\ decomposition
  \begin{align}
  \label{e1818}
    \bar \mfU_t=\sum_{i=1}^{M_h^t} \bar \mfw_i, \;\; \text{where} \;\;
    \mathcal L[M^t_h]=\Pois(\bar \mfU_{t-h}), \;\; \mcL[\bar
    \mfw]=\mcL[Z_h^0],
  \end{align}
  where $Z_h^0$ is the time $h$ state of an $\R_+$-valued Feller
  diffusion starting from state $0$.
\end{remark}

\paragraph{Representation of $\varrho^t_h$ via conditioned
  duality}
We can use the conditioned duality to represent the \emph{\Levy{}
  measure} of $\mfU$; see \eqref{e872},\eqref{e1403}. Namely the
$h$-\Levy{} measure of the process has as one ingredient (recall
\eqref{e1403}) $\varrho^t_h$, which generates the $h$-tops of $\mfU_t$
by concatenation of a Cox point process and this we want to relate
this representation to one in terms of to a coalescent genealogy via
the conditional duality if we condition $\mfU$ on the total mass path.
Therefore our $\varrho^t_h$ arise as mixture over the path law given
to have at time $t-h$ a particular value of $\bar \mfu_t$.

Let $P_{t-h,u}$ be the law of the $\R_+$-valued Feller diffusion given
the value at time $t-h$ is $u \in (0, \infty)$ and then restrict to
path from $t-h$ to $t$. Then the following corollary is a consequence
of part (b) of the Theorem~\ref{T:BRANCHING} and
Corollary~\ref{cor.1226} together with the fact that the time in the
coalescent runs backward.
\begin{corollary}[Genealogy: \Levy{} measure via coalescent]
  \label{cor.3640}
  % \leavevmode \\
  For a given total mass path $\bar \mfu$ consider the coalescent
  entrance law $\mfC^\infty_t(\bar \mfu)$. We have
  \begin{align}
    \label{e3732}
    \varrho_h^t (\cdot)=
    \int\limits_{C([0,t],\R_+)} \eta_h^t(\bar \mfu) (\cdot)
    P_{t-h,u}\; (\dx \bar\mfu), \quad h \in [0,t],
  \end{align}
  where $\eta_h^t(\bar \mfu)$ is the law on $\U(h)^\sqcup$
  concentrated on $\U(h)$ arising from the $h$-top of the state
  $\mfC_t^\infty(\bar \mfu)$ as follows. Decomposing
  $\mfC^\infty(\bar\mfu)$ at time $h$ in
  partition elements, we obtain $\{\hat \mfu_i^{(h)}, i \in I\}$ and
  get
  \begin{align}
    \label{e3649}
    \lfloor \mfC_t^\infty (\bar \mfu)\rfloor (h)
    = {\mathop{\bigsqcup_{i \in I}}}^h \hat \mfu_i^{(h)}, \quad
    \hat \mfu_i^{(h)} \in \U_1(h).
  \end{align}
  Here $I$ is indexing the partition elements according to their
  smallest elements.

  Then we pick $i$ uniformly distributed in $I$ and set (with $M_h^t$
  and $\mfw_i$ from \eqref{e1818}):
  \begin{align}
    \label{e3653}
    \eta_h^t(\bar \mfu) \coloneqq \mcL[\mfu_i^{(h)}], \;
    \mfu_i^{(h)} = \bar\mfu_i^{(h)} \hat\mfu_i^{(h)}
    \quad
    \text{with} \; \bar\mfu_i = \mfw_i, \; i=1,\dots, M_h^t.
  \end{align}
  The measure $m_h$ is then the law under which a particular $u$ in
  $P_{t-h,u} (\cdot)$ is chosen.
\end{corollary}

\begin{proof}\label{pr.3655}
  Observe that we can decompose
  $\lfloor \mfU^{\mathrm{Fel}}\rfloor (h)$ uniquely in subfamilies
  corresponding in a representative of $\mfU^{\mathrm{Fel}}$ to
  disjoint $2h$-balls in $\U(h)$, i.e.\ write
  $\lfloor \mfU^{\mathrm{Fel}}\rfloor(h)={\mathop{\sqcup}\limits_{i
      \in I}}^h \mfu_i$. Then we know from the \Levy{}-Khintchine
  representation and the fact that the state is a Markov branching
  tree that these are independent identically distributed random
  elements and their number is $\Pois(b \bar\mfu_{t-h})$-distributed
  and the law of one is
  $\varrho_h^t \in \mcM_1(\U(h)^\sqcup \setminus \{\ntree\})$ which
  gives full measure to $\U(h)$, which we decompose in masses and
  state of the genealogies in $\U_1$.

  First we decompose the mass into the pieces associated with the open
  $2h$-balls. This is generated autonomously and gives the
  $\bar\mfu_i^{(h)}$ for our decomposition. On the other hand we can
  decompose the $h$-top of $\mfC_t^\infty(\bar \mfu)$ in disjoint
  $2h$-balls (uniquely since the coalescence times have a continuous
  distribution). By definition of the metric for
  $\mfC_t^\infty(\bar \mfu)$, in a representative these $2h$-balls
  correspond to the partition elements at time $h$. Therefore
  conditioning on $\bar \mfu_{t-h}=u$ is the law of the subspace
  spanned by a partition element at running time $h$ of the
  coalescent.

  By the uniqueness of the decomposition of $\mfU^{\mathrm{Fel}}_t$ in
  $\U(h)$-elements up to permutations the claim follows.
\end{proof}

The theme of a representation via Cox point processes we will take up
again below in the form of backbone construction for the conditioned
processes.

\subsubsection{Excursion law, entrance law and process conditioned on
  survival}
\label{sss.entrance}
Here we discuss first the key ingredients for the better understanding
of depth-$h$ single ancestor family at the time $t$ populations, i.e.\
of the \emph{cluster law $\varrho_h^t$} from Theorem~\ref{T:BRANCHING}
and for the further discussion of the longtime behavior, namely the
\textit{excursion law} or the \textit{entrance law} of the $\U$-valued
Feller diffusion starting from the zero element $\ntree$. This will give
the description of the typical \emph{founding fathers family}.

The second important object here is the process $\mfU^T$ arising from
$\mfU$ \emph{conditioned} on the event to survive until a fixed time
$T$. Note that $\mfU$ is a process that goes \emph{extinct} in an
a.s.\ finite time. The process $\mfU^T$ is in close relationship with
the entrance and excursion laws. The conditioned process $\mfU^T$ will
be characterized in the main Theorem~\ref{TH.MARTU}, and in its
Corollary~\ref{cor:elcl} we relate it to the excursion law of the
$\U$-valued Feller diffusion.

All these topics will also be crucial later studying the population
surviving for \emph{long} time. A good summary of notion and results
on entrance laws, excursion laws and related concepts in the context
of branching processes is found in the monograph \cite{Li2011}; see in
particular Chapter~8 and Section~A.5.

\paragraph{Excursion law and entrance law of
  \texorpdfstring{$\mfU$}{U} from the zero element \texorpdfstring{$\ntree$}{0}}
We want to study here the measure $\varrho_h^t$ and relate it to the
\emph{entrance law from the zero tree $\ntree$}. Let us first extend
the notion of admissible paths from Definition~\ref{def:adm_u}.

\begin{definition}[Admissible total mass paths II: Excursions and
  conditioning]
  \label{def:adm_uII}
  % \leavevmode \\
  We call a function
  $\bar \mfu=(\bar \mfu_t)_{t \ge 0} \in C
  \left([0,\infty),[0,\infty)\right)$ \emph{admissible} as an
  excursion of a total mass path of a $\U$-valued Feller diffusion if
  $\bar\mfu_0=0$ and there are
  $0\le T_{\mathrm{ent}} < T_{\mathrm{ext}} \le \infty$, so that
  $\bar \mfu_t>0$ for all $t \in (T_{\mathrm{ent}},T_{\mathrm{ext}})$
  and $\bar \mfu_t = 0$ otherwise. Furthermore, for all
  $r \in (T_{\mathrm{ent}},T_{\mathrm{ext}})$ we assume
  $\int_{T_{\mathrm{ent}}}^{r} 1/\bar\mfu_t \, \dx t =\infty$ and in the
  case $T_{\mathrm{ext}} <\infty$ we also assume
  $\int_r^{T_{\mathrm{ext}}} 1/\bar\mfu_t \, \dx t =\infty$.
\end{definition}

We start with the \emph{excursion law} for the \emph{mass process}.
For the total mass process the excursion law from $0$, denoted by
$\bar P_0$, exists and is well known. If we denote by
$\bar P_\varepsilon$ the total mass process starting with mass
$\varepsilon>0$ then (for the topology see below)
\begin{align}
  \label{e1369tm}
  \bar P_0 = \lim_{\varepsilon \to 0} \frac{1}{\varepsilon} \;
  \bar P_{\varepsilon}.
\end{align}
This is assertion (3a) in \cite{PY82}; see also Theorem~1 in
\cite{hutzenthaler2009} for a rigorous proof. We claim that the paths
are also under the entrance law admissible, i.e.\ we have
\begin{align}
  \label{eq:39}
  \bar P_0 (A)=1,
\end{align}
where $A$ is the set of all admissible paths. As before we use here
\cite[Lemma~1.6 in Ch. 9]{EK86} and \cite{DG03}. In Proof sketch of
Proposition~\ref{prop:ppFD} we explain how to use these references.
Our goal here is to lift this result to the $\U$-valued setting.

To this end, for $\ve>0$ we consider the $\U$-valued Feller diffusion
$\mfU = (\mfU_t)_{t \ge 0}$ with initial state
$\mfU_0 = \ve \cdot \mfe$, where $\mfe$ is the unit element from
\eqref{eq:null-one-tree}. We set
\begin{align}
  \label{e1435a}
  P_{\ve \cdot \mfe} \coloneqq \mcL [(\mfU_t)_{t \ge 0} | \mfU_0 = \ve
  \cdot \mfe].
\end{align}
Since survival of the $\U$-valued Feller diffusion depends only on the
total mass Markov process on $\R_+$ which is autonomous, we can
condition the $\U$-valued Feller diffusion on the total mass process
$\bar\mfU = (\bar \mfU_t)_{t \ge 0}$ which starts in $\varepsilon$.
For a realization $\bar\mfu=(\bar\mfu_t)_{t\ge 0}$ of this process we
define
\begin{align}
  \label{e1435n}
  P^{\bar \mfu}_{\ve \cdot\mfe}
  & \coloneqq \mcL \bigl[\mfU(\bar\mfu)| \bar\mfu_0 =\varepsilon \bigr],\\
  \label{e1435}
  \wh P^{\bar \mfu,\varepsilon}_{\mfe}
  & \coloneqq \mcL\bigl[\wh\mfU(\bar\mfu)| \bar\mfu_0 = \varepsilon\bigr]
    = \mcL \bigl[\wh\mfU^{\mathrm{FV}}(\bar\mfu)) | \bar\mfu_0 =\varepsilon\bigr].
\end{align}
Here $\mfU(\bar\mfu)= (\mfU_t(\bar\mfu))_{t\ge 0}$ is the $\U$-valued
diffusion conditioned on the total mass path and
$\wh \mfU^{\mathrm{FV}}(\bar\mfu) = (\wh
\mfU^{\mathrm{FV}}_t(\bar\mfu))_{t\ge 0}$ is the time-inhomogeneous
$\U_1$-\emph{valued Fleming-Viot diffusion} obtained by taking the
resampling rate $d(t) = b/\bar \mfu_t$ at time $t$ and initial state
$\wh \mfU^{\mathrm{FV}}_0 = \mfe$. In particular, the generator is as
given in \eqref{e746} and the process is the one from
Corollary~\ref{prop.834}. Recall also relation \eqref{eq:coneq} between
the full and the pure genealogy part of the conditioned process.

Since we deal here with excursion measures, which are typically
\emph{$\sigma$-finite}, we need the following generalization of the
concept of weak convergence of probability measures on the path spaces
$C([0,T],\U)$ or $C([0,T],\R_+)$. A special role is played by the path
equal to the zero element of $\U$ corresponding to starting in the
trap. We here consider the open $\ve$-neighborhood of
$\underline{\ntree}$ and the closed complements (which are Polish
spaces), where we want to have the restriction of our measures to
converge weakly as finite measures. This is usually formalized as
follows (see\cite{DVJ08,LR16} for this object).

For $\sigma$-finite measures on $C([0,T),\U\setminus \{\ntree\})$,
where \emph{$\underline{\ntree}$ is the constant path equal to
  $\ntree$} we introduce the \emph{weak$^\sharp$-topology} with
respect to the \emph{point $\underline{\ntree}$ as infinity point}.
Roughly speaking sequences of $\sigma$-finite measures converge if
their restrictions to complements of open $\ve$-neighborhoods of
$\underline{\ntree}$ converge as finite measures weakly.

The following result is the announced generalization of
\eqref{e1369tm} to the $\U$-valued setting.

\begin{proposition}[$\U$-valued Feller excursion law and entrance law
  from $\ntree$]\label{prop.1382el} %\leavevmode\\
  For $\varepsilon>0$ let $P_{\varepsilon \cdot \mfe}$ be the law on
  $C([0,\infty),\U)$ defined in \eqref{e1435a} and
  $\wh P^{\bar\mfu,\varepsilon}_\mfe$ in \eqref{e1435}. Then we have
  \begin{align}
    \label{e1377}
    P_{\varepsilon \cdot \mfe} = \int \wh P^{\bar \mfu,\varepsilon}_\mfe \; \bar
    P_\varepsilon (\dx \bar \mfu).
  \end{align}
  Furthermore the following limit exists (w.r.t.\ $\ntree$)
  \begin{align}
    \label{e1369}
    P_\ntree  \coloneqq \wlim_{\varepsilon \to 0}
    \frac{1}{\varepsilon} \; P_{\ve \cdot \mfe},
  \end{align}
  and \eqref{e1377} holds for $\ve=0$. The corresponding entrance law
  of the $\U$-valued Feller diffusion from the null element $\ntree$
  is given by $\{P_\ntree(\mfU_t \in \cdot) : t >0\}$.
\end{proposition}
\begin{proof}
  The equation \eqref{e1377} is clear. Next we note that \eqref{e1377}
  holds for $\bar P_0$ by \eqref{e1369tm}. We shall see that for any
  $\ve \ge 0$ the function
  $\bar\mfu \mapsto \wh P^{\bar \mfu,\ve}_\mfe$ is a continuous
  function on the set of admissible functions and in particular
  $\wh P^{\bar \mfu,0}_\mfe$ exists. This is clear for every $\ve>0$.
  By duality we see that the value at $t=0$ must be $\mfe$ and the law
  of the path converges as $\ve \to 0$ and $\wh P_\mfe^{\bar \mfu,\ve}$
  depends conditionally on $\bar\mfu$ in $A$ for $\ve\ge 0$. Then the
  assertion \eqref{e1369} follows from \eqref{e1377} and the quoted
  $\R_+$-valued result.

  For the last assertion concerning the entrance law, we combine
  Corollaries~\ref{prop.tvF} and \ref{prop.834} and the fact that the
  resampling rate $d(t)$ depends only on the current state, so that
  evolving a measure means \emph{first} evolving the total mass with
  its transition kernel and then based on the new piece of the path of
  the mass, an then \emph{second} evolving the pure genealogy path
  with the time-inhomogeneous Fleming-Viot kernel on $\U_1$.
\end{proof}

We call $P_\ntree$ the \emph{excursion law} of the $\U$-valued Feller
diffusion \emph{from the null element $\ntree$}. If we want to
consider excursions starting at time $\alpha$ (instead of time $0$) we
would have to include $\alpha$ in the notation.

In general an excursion law is a $\sigma$-finite measure as is the
case here. However, in our context $P_{\ntree} (\bar\mfU_t>0)$ is
finite for any $t>0$ and we focus now on this restricted and
normalized version of the excursion law. We relate now the law
$\rho_t^h$ of the depth-$t$ single ancestor subfamily at the time $t$
in the $\U$-valued Feller diffusion to the entrance law from
Proposition~\ref{prop.1382el}. For the proof of the following result
see p.~\pageref{pr.5155}.
\begin{proposition}[$\U$-valued Feller excursion law and the measure
  $\varrho^t_t$]\label{cor.1382}
  % \leavevmode\\
  Let $P_{\ntree;t}^{\mathrm{prob}}$ denote the time $t$ marginal of
  the excursion law $P_\ntree$ conditioned on survival beyond time
  $t$, i.e.\ normalized by $P_{\ntree} (\bar\mfU_t>0)$. Let
  $\varrho^t_t$ be as defined in \eqref{eq:38}. Then we have
  \begin{align}
    \label{e1611}
    P_{\ntree;t}^{\mathrm{prob}} = \varrho_t^t.
  \end{align}
  Note that the r.h.s.\ equals $\mcL \bigl[\mfY^{(t)} \bigr]$; see
  \eqref{e1403}. Furthermore $\mfY^{(t)}$ equals also the limiting
  Yule tree from \eqref{e1408} (with $h=t$ in both cases).
\end{proposition}

This suggests to not only look at time-$t$ marginals of the excursion
law. Namely there is one more \emph{family} of excursion laws and
associated entrance laws which will be relevant in the following.
Namely, we start from the \emph{normalized excursion law} of $\mfU$
from above and restrict it to paths on $[0,T]$ which means replacing
the excursion law by a law of paths escaping $\ntree$ till time $T$.
It will be introduced below in the next paragraph in the context of
processes conditioned to survive until some fixed time $T>0$. We
denote the corresponding \emph{probability law} by
\begin{align}
  \label{e1424}
  P^{\mathrm{prob}}_{\ntree;0,T} (\, \cdot\, )
  = P_{\ntree} (\, \cdot \, \cap \{\bar\mfU_T>0\})/ P_{\ntree} (\bar\mfU_T>0).
\end{align}
As we shall see below, for each $T>0$ this induces a collection of
entrance laws indexed by $t \in (0,T]$. This family of excursion laws
is related to the \Levy{} measure of the $\U$-valued Feller diffusion
$\mfU=(\mfU_t)_{t\ge 0}$ and a dynamical representation will be given
in Theorem~\ref{TH.MARTU}(c) below, i.e.\ we specify the process for
which the marginals of $P^{\mathrm{prob}}_{\ntree;0,T}$ form an
entrance law from the $\ntree$-element.

\begin{remark}
  \label{r.2007}
  It is easy to see by explicit calculation using the normal FK-dual
  that the Feynman-Kac dual of the entrance law from the zero element
  of the $\U$-valued Feller diffusion can be represented via the
  coalescent conditioned to coalesce by time $t$. Indeed, via
  conditional duality the law $\wh P_{\bar \mfu}$ has a
  time-inhomogeneous coalescent as a dual which coalesces to one
  lineage by time zero.
\end{remark}

\paragraph{The \texorpdfstring{$\U$}{U}-valued Feller diffusion
  conditioned to survive until time \texorpdfstring{$T$}{T}}

The next object is the process $\mfU$ \emph{conditioned on
  non-extinction by time $T$}, i.e.\ on $\bar \mfU_T > 0$, for a fixed
$T>0$ which is a.s.\ the event $\bar\mfU_t>0$ for $t \in [0,T]$. This
process will be denoted by
\begin{align}
  \label{e1011}
  \mfU^T = \bigl(\mfU^T_t \bigr)_{t \in [0,T]}.
\end{align}

For the total masses of $\mfU$, i.e.\ the Feller diffusion, it is well
known that conditioned to survive till time $ T $ we get a
\emph{time-inhomogeneous generalized (i.e.\ state-dependent)
  super-critical branching diffusion} with ($T$-dependent) time
inhomogeneous drift and the original volatility coefficient which was
explicitly calculated in special cases (see \cite{LN68}) but in
addition to a generalization this needs a \emph{correction}.

We proceed here differently with the calculation and get the following
state and time dependent coefficients for the total mass process
\begin{align}
  \label{e984}
  \wt a_T(t,x) = \frac{2x/(T-t)}{\exp (2x/(b(T-t)))-1}, \quad \wt
  a_T(t,0)= b\; \text{ and } \; \wt b_T(t,x) = bx, \quad t \in [0,T].
\end{align}
In \cite{LN68} the formula $\wt b_T(t,x)=x(2+\wt a_T(t,x))$ appears,
which is \emph{not correct}. The computation is carried out in
Section~\ref{sec:comp-diff-coeff}.

We get the state- and time-dependent (positive)
\emph{super-criticality coefficient}
\begin{align}
  \label{e1445}
  a_T(t,x) = \frac{2/(T-t)}{\exp (2x/(b(T-t)))-1}
\end{align}
in the individual rate of branching at time $s$ and state $x$. Note
that in particular $a_T(t,x)$ replaces the coefficient $a$ in
\eqref{eq:FelBDiff-ab}, $b$ stays the same and $c=0$.
%see also
% Remarks~\ref{r.697} and \ref{r.offspring}.
Note also that as $x \downarrow 0$ we have $\wt a_T (t,x) \sim b$ but
$a_T(t,x) \sim b/x$. This is what we see in the generator on the level
of the genealogies.

\begin{remark}[Scaling property]
  \label{r.1627}
  Note that under a time-mass scaling $x \mapsto a^{-1}x$,
  $s \mapsto a^{-1}s$ and the time-horizon scaling
  $T \mapsto a^{-1} T$ the term $\wt a_T(s,x)$ remains invariant.
  Therefore under this rescaling the mass process $\bar\mfU^T$ is
  invariant, if the initial state is zero.
\end{remark}

With the rate $a_T(\cdot,\cdot)$ of super-criticality, and rate
$b_T(\cdot,\cdot)=b$ of critical branching we can run a
time-inhomogeneous \emph{generalized} super-critical $\U$-valued
Feller diffusion $(\mfU_t^T)_{t \ge 0}$. This means that we define the
generator of $\mfU^T$ denoted by
\begin{align}
  \label{e1336}
 \Omega^{\uparrow,(a_T,b_T)}
\end{align}
as the generalization of the operator $\Omega^{\uparrow,(a,b)}$
described in \eqref{eq:3omab} by replacing in \eqref{e709} $a$ and $b$
at time $t$ by $a_T(t,\bar\mfu)$ respectively $b_T(t,\bar\mfu)$. One
can make this process time-homogeneous with state space $\R \times \U$
by passing to the \emph{time-space process} $(t,\mfU_t)_{t \ge 0}$.
Here we need to extend the domain of test functions to achieve that
this domain is mapped under $\Omega^{\uparrow,(a_T,b_T)}$ into itself.
Here we use $\mcD_1$ and $\mcD_2$ from \eqref{e1212} respectively
\eqref{e1189} to obtain linear operators.

\begin{remark}
  \label{r.1095}
  Note that these processes do \emph{not} have the branching property
  as previously defined because the super-criticality coefficient
  $a_T$ is \emph{state-dependent}. This requires using different
  techniques to show uniqueness. Since the state dependence is only
  via the total mass process we will use a \emph{conditional} duality.
\end{remark}

\begin{remark}
  \label{r.143}
  Note that $a_T(t,x)$ converges to $b/x$ as $T \to \infty $, i.e.\ in
  the limit the super-criticality coefficient has a \emph{pole} at
  $x=0$. This means that the total mass process has a constant drift
  $b$ and that at small mass the super-criticality rate of the
  individuals diverges.
\end{remark}

Before we state the well-posedness of $\mfU^T$, we point out that
there is another complication we have to handle for $\mfU^T$. Namely
we will need this process starting at time $s$ with zero-mass. One
immediate consequence is that the process is a one ancestor
ultrametric space, i.e.\ there exists exactly one open $2u$-ball at
time $u$, here $s < u < t$. (This follows from the conditional
duality, which we develop in the proof.) Therefore we construct this
process below in part (b) by starting from time $s + \varepsilon$ and
considering then the limit $\ve \to 0$ yielding the state
$\ntree = (0,\mfe)$ for the pair as limit of
$(\bar \mfU_t, \wh \mfU_t)$.

\begin{theorem}[Martingale problem: $\U$-valued Feller conditioned on
  survival till time $T$]
  \label{TH.MARTU}
  % \leavevmode \\
  Fix $T>0$ and consider functions $a$ and $b$ on $[0,\infty)$ with
  $a\coloneqq \wt a_T(\cdot,\cdot)$ and $b\coloneqq \wt b_T(s,x)=bx$
  as in \eqref{e984}.
  \begin{enumerate}[(a)]
  \item For every $T>0$ the process $\mfU^T=(\mfU^T_t)_{t \in [0,T]}$,
    recall \eqref{e1011} for a definition, is a time-inhomogeneous
    Markov process, with values in $\U$.
  \item For any $\mfu \in \U\setminus \{\ntree\}$ the
    $(\delta_\mfu, \Omega^{\uparrow, (a,b)}, \Pi(C^1_b))$-martingale
    problem (recall \eqref{e1336}) is well-posed. The solution is given by
    $\mfU^T=(\mfU^T_t)_{t \in [0,T]}$ with $\mfU^T_0 =\mfu$. If we
    start this process at time $s$, we replace $T$ by $T-s$ in
    formulas \eqref{e984} and \eqref{e1445}.
  \item For the initial state $\ntree = (0,\mfe)$ we can construct an
    entrance law $\mcL[(\mfU_t^{T,\mathrm{entr}})_{t \in (0,T]}]$ of
    $(\mfU_t^T)_{t \in (0,T]}$ so that as $t \downarrow 0$,
    $(\bar \mfU^T_t,[U_t^T,r_t^T, \hat \mu_t^T])$ converges (weakly) to
    $\ntree = (0,\mfe)$ on $\R_+ \times \U_1$.

    Furthermore, for $P^{\mathrm{prob}}_{\ntree;0,T}$ from
    \eqref{e1424} we have
    \begin{align}
      \label{e2488}
      \mcL\Bigl[\bigl(\mfU^{T,\mathrm{entr}}_t\bigr)_{t \in [0,T]}
      \Bigr]=P_{\ntree;0,T}^{\mathrm{prob}}.
    \end{align}
  \item For general initial (random) states $\mfU_0^T$ with law
    supported on $\U\setminus \{\ntree\}$ the process $\mfU^T$ is
    defined by a martingale problem similarly to \eqref{eq:3pnu}.

  \item Consider the initial non-random state
    $\mfu_0 = (\bar\mfu_0,\hat\mfu_0) \ne \ntree$. The process
    $\bar \mfU^T=(\bar \mfU_t^T)_{t \in [0,T]}$ starting in
    $\bar \mfu_0$ is a state-dependent super-critical $\R_+$-valued
    ``branching'' diffusion with super-criticality coefficient
    $a_T(t,\bar \mfU_t)$ from \eqref{e1445} and volatility
    $b\bar\mfu$.

    Conditioned on $\bar \mfU^T$, the process
    $\wh \mfU^T =(\wh \mfU^T_t)_{t \in [0,T]} = ([U_t,r_t,\wh
    \mu_t])_{t \in [0,T]}$ is a time-inhomogeneous $\U_1$-valued
    Fleming-Viot diffusion with (finite) resampling rate
    $b/\bar \mfU_t^T$ at time $t \in [0,T]$ and starting in
    $\hat \mfu_0$.
  \end{enumerate}
\end{theorem}

The above result identifies the process defined as the $\U$-valued
Feller diffusion conditioned to survive up to time $T$ as a
time-inhomogeneous Markov process which is a \emph{super-critical,
  state-dependent branching process} and the state dependence and
time-inhomogeneity are present only in the super-criticality per
individual. We can also use $\wh\mfu^T$ to better understand the
excursions as described in the following result.

\begin{corollary}[Entrance law and conditional law]
  Denoting \label{cor:elcl} $P_\ntree^T$ the excursion law of the
  $\U$-valued Feller diffusion $\mfU$ normalized by
  $P_\ntree (\bar \mfU_T>0)$ and restricted to paths on $[0,T]$ we
  have $P_\ntree^T = \mathcal L[(\mfU_t^T)_{t \in [0,T]}]$.
\end{corollary}

The result in part (d) of Theorem~\ref{TH.MARTU} can be used to give
another conditional dual representation of $\varrho_h^t$.
\begin{corollary}[Coalescent representation of $\varrho_h^t$]
  \label{cor:CRr_h-t}
  Consider the $\U_1$-valued coalescent with time-inhomo\-geneous rates
  given by $(b/\bar\mfu_{t-s}^t)_{s\in [0,h]}$. Then, recalling from
  \eqref{eq:multaU} the notation $a[U,r,\mu]=[U,r,a\mu]$, we have
  \begin{align}
    \label{eq:CRr_h-t}
    \varrho_h^t = \mathcal L[\bar\mfu_h^t \cdot \mfC_h^\infty(\mfu^t)].
  \end{align}
\end{corollary}

\subsection{Results 2: Long surviving \texorpdfstring{$\U$}{U}-valued
  Feller diffusion}
\label{ss.longuval}

The second group of results on the $\U$-valued Feller diffusion
(Theorem~\ref{T:KOLMOGOROVLIMIT}-~\ref{T:BACKBONE}) consists of the
$t \rightarrow \infty$ asymptotics of the population
\emph{conditioned} on different forms of \emph{long time survival} or
\emph{survival forever}. We recall that conditioning to survive for a
fixed time $T$ gives a state-dependent branching process, where the
super-criticality is state dependent. This destroys the branching
property. We shall see that this is different if we condition to
survive forever, where we have a branching property appearing again.
We proceed as follows.

\medskip
\noindent
(1) We consider first conditioning on survival with finite time
horizon $T$ and then letting $T$ tend to infinity \textit{the
  $Q$-limit} process $\mfU^\dagger$, since generally this is the name
in the $\R$-valued version. For this process
$(\mfU_t^\dagger)_{t\ge 0}$ we consider then the limit for
$t \to\infty$ (rescale distances-mass by $t$), which is called
\textit{generalized quasi-equilibria} of the genealogies of $\mfU$.

\medskip
\noindent
(2) In addition we consider a different order of limits than in (1),
namely the extension of the \emph{Kolmogorov-Yaglom exponential limit
  law} (KY-limit) for the $Q$-process, where we condition the process
to survive until time $t$, \emph{rescale the total mass and distances
  at time $t$} by multiplying them with $t^{-1}$ and then take
$t\to\infty$.

\medskip
\noindent
(3) Complementary we
consider size-biasing and representations of this object in particular
represent it via the $\U$-valued version of \emph{Evans branching with
  immigration} from an \emph{immortal line} and show it equals the
$Q$-process.

\medskip
\noindent
(4) Finally, the goal is to bring together the above three groups of
results via various representations of the limit genealogies by
concatenations over \emph{Cox point processes} which uses the
\textit{backbone constructions}.

The results on the long time behavior come in \emph{four pieces} with
six theorems: first, addressing (1)-(3) above, we have
\emph{$Q$-process}, Palm process and the Kallenberg tree together with
the KY-limit; second point is the construction of the $\U$-valued
version of the \emph{Evans' branching process with immigration} from
an immortal line and the connection to the $\U$-valued Kallenberg tree;
third, the \emph{$\U$-valued backbone} construction; fourth, the
\emph{KY limit laws} and their relations for all appearing processes
we have constructed in the preceding subsections.

\subsubsection[\texorpdfstring{$Q$}{Q}-process with Kolmogorov-Yaglom
limit, Palm measure and Kallenberg tree]{Long surviving Feller
  diffusion 1: \texorpdfstring{$Q$}{Q}-process with Kolmogorov-Yaglom
  limit, Palm measure and Kallenberg tree}
\label{sss.longlim}

Since the critical Feller diffusion becomes \emph{extinct} almost
surely in finite time, interesting questions arise by considering
conditioned genealogies in various regimes of conditioning on
survival. As in the case of critical (discrete) branching processes
conditioned on survival further rescaling is needed in some cases in
order to obtain interesting limits. Here we analyze the behavior of
the \emph{$\U$-valued Feller diffusion} $\mfU = (\mfU_t)_{t \ge 0}$
for $t \to \infty$ via \emph{scaling limits} of processes arising from
\emph{conditioning} of $\mfU$ on survival in various ways. Namely, we
first condition on \emph{survival forever}. The second construction is
based on \emph{size-biasing} (Palm measure) with the Kallenberg tree
as the main ingredient. These objects will be important once we come
to \emph{spatial} populations on infinite geographic spaces like
$\Z^d$ for example. Then the rare events of survival of a Feller
diffusion become visible, since in this case we deal with many
independent such processes corresponding to the sites where time-$0$
individuals have at large time $t$ surviving descendants somewhere in
space but with large total mass of order $t$.

\paragraph{\texorpdfstring{$Q$}{Q}-process and Kolmogorov-Yaglom limit
  law for $\U$-valued Feller diffusion}
\label{Kolmoglim}

One possibility of conditioning is to condition the process to
\emph{survive forever}. Here, survival forever means that we consider
the process $\mfU$ at time $t$, condition it to survive until time
$T \ge t$, and let $T \to \infty$. The existence of the limit is
stated below in Theorem~\ref{T:KOLMOGOROVLIMIT}(a). The limiting
process is referred to as the \emph{$Q$-process} and will be denoted
by
\begin{align}
  \label{e977Q}
  \mfU^\dagger = (\mfU^\dagger_t)_{t \ge 0}.
\end{align}
In Lemma~\ref{c.1310} it is shown that the limiting $Q$-process can be
obtained as solution of a \emph{well-posed martingale problem}. There
more precisely we can consider the process
$\mfU^T=(\mfU_t^T)_{t\in[0,T]}$ from Theorem~\ref{TH.MARTU}(a) for
$T \to \infty$.

If the law of the $Q$-process at time $t$ converges as $t\to\infty$
then the limiting object is referred to as the \emph{Yaglom limit}. If
the $Q$-process has an equilibrium then this is referred to as
\emph{quasi-equilibrium} of $\mfU$; see e.g.\ \cite{Lamb07,MV12}. It
is known that in the \emph{critical} case the $\U$-valued Feller
diffusion has no quasi-equilibrium or Yaglom limit because in this
case already the total mass diverges. Therefor in particular this
holds for the $Q$-process an we have to rescale mass and distances to
get a generalized quasi-equilibrium in both cases. First we look at
the Feller diffusion, then after a suitable \emph{rescaling} we do get
a limiting law, which we refer to as the \emph{generalized Yaglom
  limit}. A classical result due to Yaglom (in the case of discrete
branching processes) says that the conditioned law
$\mcL [T^{-1} \bar \mfU_T |\bar \mfU_T>0]$ of the process converges
weakly as $T \to \infty$ to exponential distribution $\Exp(2/b)$.
Rescaling the $Q$-process for the $\R$-valued Feller diffusion we
obtain with this rescaling the size-biased exponential law.

Turning now to $\U$-valued objects we consider therefore for $T>0$ the
rescaled process
\begin{align}
  \label{ag1-split}
  \breve \mfU^T = (\breve \mfU^T_t)_{t \in [0,T]} = \bigl(\bigl[
  U_t,t^{-1} r_t,t^{-1} \mu_t\bigr]\bigr)_{t \in [0,T]} \quad
  \text{conditioned on $\bar \mfU_T>0$}.
\end{align}
We denote the process $\mfU^\dagger$ from \eqref{e977Q} rescaled as in
\eqref{ag1-split} by
\begin{align}
  \label{e977}
  \breve \mfU^\dagger = (\breve \mfU^\dagger_t)_{t \ge 0}.
\end{align}

Recall the operator $\Omega^{\uparrow,(a,b)}$ from \eqref{eq:3omab}
and its generalization in \eqref{e1336}. In the following lemma we
show that $\mfU^\dagger$ is a \emph{state-dependent super-critical
  branching process} with coefficients which are limits for
$T\to\infty$ of the corresponding coefficients of the process
$\mfU^T=(\mfU_t^T)_{t\in[0,T]}$ from Theorem~\ref{TH.MARTU}. More
precisely we show that $\mfU^\dagger$ is characterized by a well-posed
martingale problem, the lemma is proven in Section~\ref{ss.prtkol}

\begin{lemma}[Well-posedness of martingale problem of $Q$-process]
  % \leavevmode\\
  For \label{c.1310} any point $\mfu \in \U \setminus \{\ntree\}$ the
  $\bigl(\delta_\mfu,\Omega^{\uparrow,(a,b)}, \Pi(C^1_b)\bigr)$
  martingale problem (recall \eqref{eq:3omab}) with $a$ and $b$ given
  by functions
  \begin{align}
    \label{e1459}
    a(t,\bar \mfu)=b/\bar \mfu \, \text{ and } \, b(t,\bar \mfu)=b,
  \end{align}
  is well-posed and defines a Markov process denoted by $\mfU^\dagger$
  and referred to as $Q$-process (of $\mfU$).
\end{lemma}
We give in Corollary~\ref{c.1609} in connection with
Proposition~\ref{pr.palm1} an alternative description of the dynamics
of $\mfU^\dagger$.

\begin{remark}[Comparison with super-critical Feller]
  \label{r.1726}
  We obtain a branching process whose total mass process has drift
  $x \frac{b}{x}=b$ in state $x$ because $a_t(s,x)$ converges to $b/x$
  for $T \to \infty$ and every $s,x$, which looks like immigration.
  Indeed we will make this more precise in Section~\ref{sss.1616}.
\end{remark}

We extend the process $\mfU^T=(\mfU_t^T)_{t\in[0,T]}$ from
Theorem~\ref{TH.MARTU}(a) beyond time $T$ to a process
$\mfU^T = (\mfU_t^T)_{t \ge 0}$ by setting $\mfU^T_t=\mfU^T_T$ for
$t \ge T$. In the following theorem we show that the $\U$-valued
Feller diffusion has a $Q$-process, a generalized \emph{Yaglom} limit
as well as a generalized \emph{quasi-equilibrium distribution}. Recall
here the rescaling we will consider in b) and c) below
\eqref{ag1-split}.

\begin{theorem}[$Q$-process, KY-limit and generalized
  quasi-equilibrium for genealogies]\label{T:KOLMOGOROVLIMIT}
  % \leavevmode\\
  Let $\mfu \in \U\setminus \{\ntree\}$ be an arbitrary initial condition
  of the $\U$-valued Feller diffusion.  Then the following assertions
  hold.
  \begin{enumerate}[(a)]
  \item For $\mfU^T = (\mfU_t^T)_{t \ge 0}$ and the $Q$-process
    $\mfU^\dagger=(\mfU^\dagger_t)_{t \ge 0}$ from Lemma~\ref{c.1310}
    (and \eqref{e977Q}) we have
    \begin{align}
      \label{e1026}
      \mcL[(\mfU^T_t)_{t \ge 0}]
      \xRightarrow{T \to \infty} \mcL [(\mfU^\dagger_t)_{t \ge 0}].
    \end{align}
  \item The scaled process
    $\breve \mfU^\dagger = (\breve \mfU^\dagger_t)_{t \ge 0}$ from
    \eqref{e977} has a (generalized) quasi-equilibrium, i.e.\ there is
    a $\U$-valued variable $\breve \mfU^\dagger_\infty$ such that
    \begin{align}
      \label{tv17b}
      \mcL[\breve \mfU_t^\dagger] \xRightarrow{t \to \infty}
      \mcL[\breve \mfU_\infty^\dagger].
    \end{align}
  \item The KY-limit of the $\U$-valued critical Feller diffusion
    exists and is different from the generalized quasi-equilibrium of
    the $Q$-process, i.e.\ there is a $\U$-valued variable
    $\breve \mfU^\infty_\infty$ such that
    \begin{align}
      \label{tv17}
      \mcL[\breve \mfU_T^T] \xRightarrow{T \to \infty}
      \mcL[\breve \mfU_\infty^\infty],
    \end{align}
    but $\mcL[\breve \mfU^\dagger_\infty]$ is a size-biased version of
    $\mcL[\breve{\mfU}_\infty^\infty]$.
  \end{enumerate}
\end{theorem}

This means that the macroscopic time-space view on the surviving
population gives \emph{different} pictures in the cases of
conditioning on survival forever and conditioning on survival up to a
finite but diverging time-horizon. In particular also the genealogies
look different in these cases.

The conclusion might be however that looking at non-spatial population
one should work with the concept leading to
$\breve{\mfU}^\infty_\infty$ even though the genealogy has no
transparent decomposition into subfamilies since those remain even in
the limit $t\to\infty$ dependent. All we can do in that case is using
the conditional duality, where again the subfamilies are described via
the enriched partitions of the coalescent. This needs further
exploration.

However in \emph{spatial} situations the generalized quasi-equilibrium
is important since it describes the family of a \emph{typical}, i.e.\
randomly chosen individual from the overall population, a key object.
See also Section~\ref{sss.genalspat} for a discussion of spatial
models.

\paragraph{The \texorpdfstring{$\U$}{U}-valued Feller diffusion under
  the Palm measure}
Another method to study a process going to extinction is to consider
its \emph{size-biased} law, which is also known as the \emph{Palm
  distribution}. Recall that the Palm distribution of a $\U$-valued
process $\mfU$ at time $t$ is the law
$\bar{\mfU}_t \cdot \dx P_{[0,t]}$ if $P_{[0,t]}$ is the path law
$\mcL[(\mfU_s)_{s \in [0,t]}]$. Define $h:\U \to \R_+$ by
$h(\mfU)=\bar\mfU$ and note that this is a positive harmonic function
on $\U\setminus\{\ntree\}$. Therefore we can define an $h$-transform.
For our process this law, the Palm distribution, is an
\textit{$h$-transform} of $\mfU$ with previously mentioned $h$ and
hence is again Markovian and its generator can be calculated from the
one of $\mfU$ with the help of $h$. We denote the process realizing
the Palm distribution by
\begin{align}
  \label{e1607}
  \mfU^{\mathrm{Palm}} = (\mfU^{\mathrm{Palm}}_t)_{t \ge 0}.
\end{align}
Here we consider a construction of the \emph{Palm distribution}
w.r.t.\ the total mass process $\bar\mfU$ which arises as the
\emph{$h$-transformed process} and can be consistently defined for all
$t > 0$ and we use the fact of being an $h$-transform to construct a
corresponding $\U$-valued process via a new martingale problem in
Proposition~\ref{pr.palm1} below.

What is known for the $\R_+$-valued process? The Palm of the
$\R_+$-valued Feller diffusion can be described alternatively by two
processes which we recall below in \eqref{e1116}, \eqref{e1125}; see
\cite{Evans93}. To this end, consider the $\R_+$-valued Feller
diffusion $Z=(Z_t)_{t \ge 0}$ satisfying
\begin{align}
  \label{e1111}
  \dx Z_t = \sqrt{bZ_t} \, \dx B_t, \text{ starting in } Z_0=z_0.
\end{align}
The \emph{size-biased} Feller diffusion and the one \emph{conditioned
  to survive forever} can be represented in two ways, namely as
\emph{Feller branching diffusion with immigration}
$\tilde Z=(\tilde Z_t)_{t \ge 0}$ satisfying
\begin{align}
  \label{e1116}
  \dx \tilde Z_t= b \, \dx t + \sqrt{b \tilde Z_t} \, \dx B_t, \enspace
  \tilde Z_0 = z_0,
\end{align}
or alternatively the Palm law is given by
$Z^{\mathrm{Palm}}=(Z_s^{\mathrm{Palm}})_{s \in [0,t]}$ via the
\emph{Kallenberg tree}, given by
$Z^{\mathrm{Kal}} = (Z^{\mathrm{Kal}}_t)_{t\ge 0}$, as the process
\begin{align}
  \label{e1125}
  Z^{\mathrm{Palm}}= Z+Z^{\mathrm{Kal}} = (Z_t + Z^{\mathrm{Kal}}_t)_{t
  \ge 0}, \enspace
\end{align}
where $Z^{\mathrm{Kal}}$ is a version of $\tilde Z$ with
$Z^{\mathrm{Kal}}_0= 0$, independent of $Z$. For details we refer to
\cite{Evans93} and \cite{PY82}.

These facts can later even be lifted to the spatial case of
\emph{super random walk}. We will show later in Theorem~\ref{T.1205},
that even the Palm of the $\U$-valued Feller diffusion
$\mfU^{\mathrm{Palm}}$ allows a similar decomposition as a
concatenation of $\mfU$ and $\mfU^{\mathrm{Kal}}$.

\begin{remark}[Moments are measure determining]
  We \label{rem:mom_md} note that since the moments of a Feller
  diffusion are finite and measure determining for all $t \ge 0$, this
  immediately holds for the Palm measure, where for each $m$ up to a
  constant the $m$-th moment is given by the $(m+1)$-st moment of the
  original one. In particular the statement of the
  Corollary~\ref{cor.681} holds also for the Palm measure.
\end{remark}

First, we have to \emph{establish} these alternative representations
of the Palm measure also for the \emph{$\U$-valued Feller branching}.
Since we know that $\mfU^{\mathrm{Palm}}$ is an $h$-transform of a
Markov process we want to know the operator of the martingale problem
acting on polynomials. We obtain by explicit calculation with the
$h$-transform property (see Subsection~\ref{ss.prp1p2} for proofs) the
following result.

\begin{proposition}[Representation Palm $1$]
  \label{pr.palm1}
  Consider the polynomials $\Pi(C^1_b)$ as test functions and for
  $\Phi^{n,\varphi}$ we set
  \begin{align}
    \label{e1372}
    \Omega^{\uparrow,\mathrm{Palm}} \; \Phi^{n,\varphi} (\mfu)
    = \frac{n b}{\bar\mfu} \Phi^{n,\varphi}(\mfu) +
    \Omega^{\uparrow, \mathrm{bran}} \Phi^{n,\varphi}(\mfu)
    + \Omega^{\uparrow, \mathrm{grow}} \Phi^{n,\varphi}(\mfu).
    % \Omega^{\uparrow,\mathrm{Palm}} \; \Phi^{n,\varphi} (\mfu)
    % = \Omega^{\uparrow, \mathrm{grow}} \; \Phi^{n,\varphi}(\mfu) +
    % \frac{b}{\bar
    % \mfu} \frac{n(n+1)}{2} \Phi^{n,\varphi}(\mfu) + \frac{b}{\bar
    % \mfu} \sum_{1\le k < \ell \le n} \Phi^{n,\varphi \circ
    % \theta_{k,\ell}}(\mfu).
  \end{align}
  This operator maps $\Pi(\mcC_b^1)$ into $\Pi$ and is a linear
  operator on $\Pi$. In particular this operator specifies a
  well-posed martingale problem.
\end{proposition}

Recall the $Q$-process $\mfU^\dagger$ from \eqref{e977Q}. We have the
following corollary.
\begin{corollary}[Equality of $\mfU^\dagger$ and
  $\mfU^{\mathrm{Palm}}$]
  \label{c.1609}
  \begin{align}
    \label{e1610}
    \mcL [\mfU^\dagger]=\mcL[\mfU^{\mathrm{Palm}}].
  \end{align}
\end{corollary}
\begin{proof}
  It is easy to check that the generators of the $Q$-process from
  Lemma~\ref{c.1310} and the size-biased process (Palm process) agree
  on polynomials and hence these two processes \emph{agree}.
\end{proof}

Due to the form of the generator in \eqref{e1372} we can represent the
distribution of $\mfU^{\mathrm{Palm}}$ also in the following form
since $\bar \mfu$ in this law is $>0$ for $t \ge 0$ based on the
standard criterion for positive paths of diffusions on $\R_+$, which
can be found in \cite{RoWi00}, and leads in our situation (see
\cite{DG03} pages 21-22) to positive paths $(\bar\mfU^T_t)_{t\ge 0}$
and $\mfU^\dagger_t)_{t\ge 0}$.

\begin{proposition}[Representation Palm $2$]
  Abbreviating \label{pr.palm2}
  $\mcL [(\mfU^{\mathrm{Palm}})_{t \ge 0}]$ by $P^{\mathrm{Palm}}$ we
  have
  \begin{align}
    \label{e1386}
    P^{\mathrm{Palm}} = \int P^{\mathrm{Palm},\bar \mfu} \;
    \bar P^{\mathrm{Palm}} (\dx \bar \mfu),
  \end{align}
  where $\bar P^{\mathrm{Palm}}$ is the Palm law of the $\R_+$-valued
  Feller diffusion which is supported on admissible paths
  $\mfu=(\mfu_t)_{t\ge 0}$ (see Definition~\ref{def:adm_u}), and
  $P^{\mathrm{Palm},\bar \mfu}$ is the regular version of
  $\mcL[\mfU^{\mathrm{Palm}} |\bar \mfU^{\mathrm{Palm}} = \bar\mfu]$.

  We denote (recall \eqref{e939} for the Fleming-Viot process) by
  $\wh \mfU^{\mathrm{FV}} (\bar \mfu) =(\wh \mfU^{\mathrm{FV}}_t (\bar
  \mfu))_{t\ge 0}$ the time-inhomogeneous Fleming-Viot process
  $\wh \mfU^{\mathrm{FV}} (\bar \mfu)$ with immigration where the
  resampling and immigration rates at time $t$ are given by
  $d(t)=b/\bar \mfu_t$ respectively $c(t)=b/\bar \mfu_t$; recall that
  $\bar\mfu_t >0$ for $t >0$. Then we have for $\bar\mfu$ a
  realization of $\bar P^{\mathrm{Palm}}$:
  \begin{align}
    \label{eq:33}
    P^{\mathrm{Palm},\bar \mfu} = \mathcal L[(\bar \mfu_t \; \wh
    \mfU_t^{\mathrm{FV}}(\bar\mfu))_{t \geq 0}] \quad \bar P^{\mathrm{Palm}}\text{-a.s.}.
  \end{align}
\end{proposition}

\paragraph{Kallenberg decomposition of the $\U$-valued Feller
  diffusion under the Palm measure}
In order to understand the Palm law better we return to the
representation of the state of the Feller diffusion at time $t$ as a
concatenation of a Cox point process on
$\U(t)^\sqcup \setminus \{\ntree\}$ following from the
\Levy{}-Khintchine formula. In fact we know in our case that the
points of the Cox process are elements of $\U(t)$, i.e. elements with
radius less than $t$ representing the depth-$t$ subfamilies.
Size-biasing yields here in addition to a version of $\mfU_t$ simply
in addition to a version of $\mfU_t$ one additional independent
depth-$t$ subfamily independent of the rest, this is the so called the
so called $\U$-valued \emph{Kallenberg tree}.

We claim now that also the $\U$-valued size-biased process is the
concatenation of the $\U$-valued Feller diffusion and of the entrance
law from $\ntree$ of the size-biased Feller diffusion the so called
\emph{Kallenberg tree}, which plays this role also in the $\R$-valued
case, recall \eqref{e1125}.

More precisely we can decompose the size-biased $\U$-valued Feller
diffusions in two \emph{independent} sub-trees, which if
$t$-concatenated result in the $t$-top of the full tree. The first is
a copy in law of the original $\U$-valued Feller diffusion at time $t$
the second is what we call the \emph{($\U$-valued) version of the
  Kallenberg tree}. We will prove below that this is in law the
\emph{$\U$-valued Feller diffusion size-biased and observed at time
  $t$ which however is started at time $0$ in the zero-tree}, this law
is called $\mfU^{\mathrm{Palm},0}_t$ and was constructed using the
\emph{\Levy{}-Khintchine representation}, via
Proposition~\ref{cor.1382} and formula \eqref{e1424}. Another question
would be whether we can make this decomposition consistent in $t$ such
that we can decompose in fact into two processes for all $t>0$ as we
do above.

The next theorem establishes the existence of a Kallenberg tree in the
sense of the discussion around \eqref{e1125} by identifying it as
entrance law. Later in Theorem~\ref{TH.IDENTKALL} we will see that the
$\U$-valued Kallenberg tree agrees also with the object given by the
\emph{immortal line and its descendants} at time $t$.

\begin{theorem}[Kallenberg decomposition of the Palm of $\U$-valued
  Feller diffusion] \label{T.1205}
  % \leavevmode\\
  The entrance law of $\mfU^{\mathrm{Palm}}$ from the zero element
  exists and is denoted $P^{\mathrm{Palm},0}$. It is given by the
  size-biased normalized entrance law of the Feller diffusion
  restricted to $\bar\mfU_t >0$, which is given in \eqref{e1369} and
  has various representations, see Proposition~\ref{cor.1382}.

  We have for a Feller diffusion $\mfU$ (recall \eqref{e1607}):
  \begin{align}
    \label{e12066}
    \mcL \bigl[\mfU_t^{\mathrm{Palm}}(t)\bigr] = \mcL \bigl[\mfU_t
    \sqcup^{t} \mfU_t^{\mathrm{Kal}}\bigr],
  \end{align}
  with $\mfU_t $ and $\mfU_t^{\mathrm{Kal}}$ independent and
  \begin{align}
    \label{e1683}
    \mcL \bigl[\mfU^{\mathrm{Kal}}_t \bigr]
    = \mcL \bigl[\mfU_t^{\mathrm{Palm},0} \bigr].
  \end{align}
\end{theorem}
The next question is how we can better characterize
$ \mfU^{\mathrm{Kal}}$ using the branching property and how we can use
this to obtain information on the asymptotics as $t\to\infty$. This
question we address in three following subsections.

\subsubsection[Evans' infinite horizon dynamical representation of
\texorpdfstring{$\U$}{U}-valued Kallenberg tree]{Longtime behavior of
  Feller diffusion 2: Evans' infinite horizon dynamical representation
  of the \texorpdfstring{$\U$}{U}-valued Kallenberg tree}
\label{sss.1616}

For the $\U$-valued Feller diffusion we obtained detailed information
about the genealogy through a Cox point process representation (called
\emph{Cox cluster representation}) as concatenation of independent
\emph{single ancestor subfamilies}. The question now is whether for
the processes $\mfU^\dagger$ and $\mfU^{\mathrm{Palm}}$, which are
arising by \emph{conditioning} the original $\mfU$ on surviving
forever, we can obtain a similar representation. Here the Kallenberg
decomposition \eqref{e12066} shows it suffices to do this for the
Kallenberg tree $\mfU^{\mathrm{Kal}}_t$ since we have already treated
the $\U$-valued Feller case which added independently to the latter.

How can the genealogy of $\mfU_t^{\mathrm{Kal}}$ be decomposed into
independent subfamilies? The underlying structure was revealed nicely
by Evans for the $\R_+$-valued case and subsequently in \cite{Evans93}
formulated in great generality in particular covering measure-valued
branching processes.

Motivated by describing the Palm measure of the critical Feller
diffusion, in particular the component given by the Kallenberg tree
above, Evans introduced in Theorems~2.7-2.9 in \cite{Evans93} in a
very general context of superprocesses on general state spaces (in
particular Polish spaces) a new process which we refer here to as
\emph{Evans branching} with \emph{immigration from an immortal line}.

In the context of $\U$-valued processes the analogue of this process
will be of the form of a branching process, where from an
\emph{immortal line} (which we can think of as an ``invisible''
process identical to $\ntree = 0 \cdot \mfe$) at a constant rate $b$,
\emph{$\U$-valued critical Feller diffusions break off} and is
formally defined by a $\log$-Laplace equation. Indeed he showed this
way that the \emph{total mass process} of the Palm of the $\U$-valued
Feller critical branching diffusion has this form. In fact on the
level of \emph{individual based} critical branching processes this
phenomenon is discussed in different words (the concept of immortal
line is missing) in Example~2.1 in \cite{KW1971}.

We will obtain the \emph{genealogical}, i.e.\ $\U$-valued version of
the Kallenberg tree, denoted by $\mfU^{\mathrm{Kal}}$ in
\eqref{e12066} also via a process we construct by the same Markovian
dynamic as suggested by Evans and we denote by $\mfU^\ast$, in
Theorem~\ref{TH.1061}. This $\mfU^*$ we refer to as the $\U$-valued
version, of the \emph{Evans immortal line process} and we prove in
Proposition~\ref{prop.2034} that indeed $\mfU^\ast$ is \emph{Markov}.
Then the key result of this subsection is Theorem~\ref{TH.IDENTKALL}
at the end which says that this immortal line process $\mfU^*$, the
Kallenberg tree process gives independently concatenated to $\mfU$
another very interesting \emph{representation} for the $Q$-process
$\mfU^\dagger$ respectively the Palm process $\mfU^{\mathrm{Palm}}$ as
of course \eqref{e1116} for total masses might already suggest. This
allows also to relate this to the this to the size-biased single
ancestor subfamily from the \Levy{}-Khintchine representation.

\paragraph{Strategy}
Our goal is to give a \emph{$\U$-valued formulation} of Evans' ideas,
i.e.\ we will have an immortal line from which independent $\U$-valued
Feller diffusions split off at rate $b$ from $\ntree$ (recall
Proposition~\ref{prop.1382el}). These are \emph{concatenated} via the
immortal line to a new $\U$-valued process. The description induces a
certain structure of the ultrametric in the state at time $t$. From
the description we shall derive the \emph{generator} and a martingale
problem and then establish uniqueness to get a characterization.

The way to do this is to construct first a \emph{richer process},
which we call $\mfU^{\ast,+}$. In this process the Feller populations
break off from the immortal line at time $s$ and all its individuals
obtain $s$ as an \emph{inheritable mark}. This way we can identify the
subfamily of the descendants of an immigrant arriving at times $s$. We
shall show that forgetting the marks yields indeed again a Markov
process which gives the desired object.

We will need some new concepts to proceed, in particular a state space
and a corresponding martingale problem which we introduce now.
\begin{itemize}
\item We recall the marked metric measure spaces $\U^V$.
\item We construct from the Evans' recipe the state of the
  $\U^V$-valued process at time $t$.
\item We derive from the marginal laws an operator for the
  corresponding martingale problem.
\item We show well-posedness of the martingale problem.
\item We show that forgetting the marks, i.e.\ projecting from $\U^V$
  onto $\U$ we obtain a $\U$-valued \emph{Markov} process.
\end{itemize}

\paragraph{$V$-marked genealogies and $\U^V$}
Here and later we will need \emph{marked metric measure spaces} to
model populations with types, locations etc.\ taken from some \emph{a
  priori fixed} complete and separable metric space $(V,r_V)$. Then
the basic objects are equivalence classes of $V$-\emph{marked}
ultrametric measure spaces of the form
\begin{align}
  \label{e940}
  [U \times V, r_U, \nu],
\end{align}
where $(U,r_U)$ is the population equipped with the genealogical
distance and $\nu$ is a measure on the Borel-$\sigma$-algebra of
$(U,r_U) \otimes (V,r_V)$. Note that $r_V$ is then automatically
fixed. Think here first of a finite or even a probability measure
space (later we shall also consider $\sigma$-finite measure spaces).
The projection of $\nu$ on $U$ will be denoted by $\mu$. Often, in
fact in all cases we consider here (see \cite{KliemLoehr2015}), there
exists a measurable \emph{mark function} $\kappa: U \longrightarrow V$
so that $\nu$ is of the form
\begin{align}
  \label{e941}
  \nu (\dx u, \dx v)  = \mu (\dx u) \otimes \delta_{\kappa(u)} (\dx v).
\end{align}

As in the case without marks, the symbol $[ \; \cdot \;]$ in
\eqref{e940} denotes an \emph{equivalence class} of $V$-marked metric
measure spaces. Here, in the case with mark functions, two spaces
$(U \times V, r_U, \nu)$ and $(U' \times V, r_{U'}, \nu')$ are called
\emph{equivalent} if there is a measure and a mark preserving isometry
$\varphi$, more precisely if there is
$\varphi: \supp\mu \rightarrow \supp\mu'$ with
$\mu' = \varphi_\ast \mu$ and
$\kappa' (\varphi(u), \cdot)= \kappa(u,\cdot)$ for $\mu$ - almost all
$u \in U$. The space of all equivalence classes of $V$-marked metric
measure spaces is denoted by
\begin{align}
  \label{e691}
  \U^V.
\end{align}
Suitable test functions in this setting are again polynomials which
are here of the following form and now based on two functions
$\varphi$ on distances and $g$ on marks, i.e.\ for $n \ge 2$ the
function $\varphi$ is a bounded and measurable function on
$[0,\infty)^{\binom{n}{2}}$ and a constant function in the case
$n \in \{0,1\}$. The function $g$ is a bounded and measurable function
on $V^n$ for $n \ge 1$ and a constant for $n=0$.

Furthermore, we define
\begin{align}
  \label{e942}
  \begin{split}
    \Phi^{n,\varphi,g} \bigl([U \times V,r,\nu]\bigr) = \int_{(U
      \times V)^n} \;
    & \varphi \big((r(u_i,u_j))_{1 \le i < j \le n}  \big)\\
    & \qquad \cdot g \big((v_i)_{i=1, \dots, n} \big)
    \nu \big(\dx(u_1,v_1) \big) \dots \nu \big(\dx(u_n,v_n) \big).
  \end{split}
\end{align}

The algebra generated by all polynomials is \emph{measure determining}
and will be denoted by
\begin{align}
  \label{e2506}
  \Pi^{V}.
\end{align}

For $n \ge 2$ let $C_b=C_b([0,\infty)^{\binom{n}{2}},\R)$ be the set of
bounded continuous functions on $[0,\infty)^{\binom{n}{2}}$ and for
$n \in \{0,1\}$ we identify this space with constant functions.
Similarly $C_b^1=C_b^1([0,\infty)^{\binom{n}{2}},\R)$ denotes the subset of
continuously differentiable functions in $C_b$. Furthermore, for
$n\ge 1$ we denote by $C_{bb}=C_{bb}(V^n,\R)$ the set of bounded and
boundedly supported continuous functions on $V^n$. For $n=0$ we
identify this set with constant functions. Note that
$\Phi^{0,\varphi,g}$ is just the product of the two constants. In
agreement with our previous notation we denote by
$\Pi^{V}(C_b \times C_{bb})$ the algebra generated by all polynomials
of the form \eqref{e942} with functions $\varphi$ and $g$ from the
corresponding spaces.

The \emph{Gromov weak topology} on $\U^V$ is generated by polynomials
by requiring
\begin{align}
  \label{e947}
  \mfu_n \xrightarrow{n \rightarrow \infty}  \mfu \; \text{ on $\U^V$
  if and only if } \;
  \Phi (\mfu_n) \xrightarrow{n \rightarrow \infty} \Phi (\mfu), \;
  \text{ for all $\Phi \in \Pi^{V} (C_b \times C_{bb})$}.
\end{align}
For more details on $V$-marked metric measure spaces we refer to
\cite{DGP11}, \cite{KliemLoehr2015} and \cite{DG18evolution}. We need
here the extension to the case of measures $\nu$ to be finite and
later $\sigma$-finite. This is well-known; see for instance
\cite{GSW}.

\begin{remark}[Generalized Feller property for $\U^V$-valued
  processes]
  The \label{rem:gFUV} generalized Feller property for $\U^V$-valued
  processes is defined analogously to Definition~\ref{rk:Fellert}
  where we replace $\U$ by $\U^V$. By the discussion below
  Definition~\ref{rk:Fellert} it suffices to check the continuity
  property analogous to \eqref{eq:35} in the neutral case for all
  $\Phi \in \Pi^{V}(C^1_b \times C_{bb})$.
\end{remark}

\paragraph{The $\U^{[0,\infty)}$-valued Feller diffusion with
  immigration from an immortal line:
  \texorpdfstring{$\mfU^{\ast, +}$}{frU-ast-plus}}

It is in the description of this process with immigration that we use
\emph{marked} genealogies. The process $\mfU^{\ast, +}$ is a Feller
diffusion with \emph{constant immigration} from an \emph{immortal
  individual} with the consequence that if the immigrants enter the
evolution at time $t$, their distance to the remaining population is
$2t$, since this is what the $t$-concatenation does. Therefore at
every time $s$ we have subfamilies arising which split off from the
immortal line to evolve as $\U$-valued Feller diffusions and in
addition carrying the time of splitting, i.e.\ the \emph{immigration
  time}, as an \emph{inheritable mark}. Immigration at, say time $s$,
from the zero state means that an independent Feller tree starts
growing at time $s$ according to the entrance law and this tree is
concatenated with the rest of the tree by giving the ancestor of this
immigrant family distance $2s$ to everybody else. We call this
operation of merging the \emph{sliding concatenation}, which we
formally define below in \eqref{e1963}. If we observe the resulting
population alive at time $t$ we obtain the state $\mfU_t^{*,+}$ and
varying $t\in [0,\infty)$ we obtain the process
$(\mfU^{\ast,+}_t)_{t \geq 0}$.

We have to make this construction rigorous and show that we get a nice
stochastic process. The starting point of the construction is the
measure valued process on the marks which is well-defined via
Theorems~2.7-2.9 in \cite{Evans93}. Given the measure valued process
we can construct below rigorously the process of genealogies of the
family which immigrated at time $s$ for every $s \in [0,t]$ by giving
this sub-population the mark $s$, called color and define the family
$^s\mfU$ of color $s$ processes corresponding to the population
immigrating at times $s$ and concatenate them as described above to
obtain a marked copy of the entrance law of the $s$-marked version of
the $\U$-valued Feller diffusion we have constructed previously in
Theorem~\ref{THM:MGP:WELL-POSED}.

\medskip
\noindent
\textbf{\emph{(1) The state space, description of the process.}} Here
we choose $\U^V$ as state space with $V=[0,\infty)$. Evans' idea
lifted to the level of $\U^{[0,\infty)}$-valued \emph{processes}
requires that
\begin{itemize}
\item $\mfU^{\ast,+}$ restricted to color $s$, denoted by $^s\mfU$, is
  the state of a copy of an $\U$-valued Feller diffusion starting at
  time $s$ from $\ntree$, (recall Proposition~\ref{prop.1382el} on the
  entrance law from $\ntree$),
\item is marked by one mark $s$,
\item is conditioned to survive till time $T$, where $T$ is the
  time at which we observe $\mfu^{*,+}$,
\item for different $s$ these pieces evolve independently.
\end{itemize}

Distances between elements of different pieces are defined formally
below in \eqref{e2000-r}. From this description one can construct the
transition kernel for a Markov process if we define how to run the
dynamic from a general element of $\U$ which can occur as state
starting from the zero space. This will be defined below.

\begin{remark}[Construction of Evans process by sliding concatenation]
  \label{rem:eva-sl-con}
  To construct the \emph{time $T$ state} of the Evans process, let
  $I_T$ be the countable set of starting times of excursions of the
  measure-valued Evans process on $V$ that start between $[0,T]$ and
  survive up to time $T$ and consider the corresponding $\U^V$-valued
  state, namely let
  \begin{align}
    \label{e2000-mfUs}
    \prescript{s}{}\mfU
    & = [\prescript{s \mkern-2mu}{}U \times [0,\infty),
      \prescript{s \mkern-2mu}{} r,
      \prescript{s \mkern-2mu}{}\mu \otimes \delta_s],\quad s \in I_T,
  \end{align}
  be the corresponding $\U^{[0,\infty)}$-valued Feller diffusions
  marked with $s$ starting from the zero element and conditioned to
  survive till time $T$. This process can be defined via the
  time-inhomogeneous $\U^{[0,\infty)}$-valued Fleming-Viot process
  associated with the total mass path of the corresponding type $s$ as
  $\mfU^{\mathrm{FV}}((\bar\mfu_t(s))_{t\ge 0})$ by multiplying its
  mass by $\bar\mfu_t(s)$. Here we take the martingale problem from
  \eqref{e939} lifted to the case where we have an identifiable color
  which we introduce formally in \eqref{e2581}-\eqref{e2589}.

  We define the $\U^{[0,\infty)}$-valued random variable by
  \emph{sliding} concatenation $\sqcup^{\mathrm{sli}}$ as follows:
  \begin{align}
    \label{e1963}
    \mfU^{\ast,+}_T \coloneqq \bigl[U \times [0,\infty),r,\mu\bigr]
    \coloneqq \mathop{\bigsqcup\nolimits^{\mathrm{sli}}}_{s \in I_T}
    \prescript{s}{}\mfU,
  \end{align}
  where we set (we suppress for brevity the dependence of the particular
  elements on $T$)
  \begin{align}
    \label{e2000-U}
    U & = \bigcup_{s \in I_T} \prescript{s \mkern-2mu}{}U,\\
    \label{e2000-r}
    \begin{split}
      r(i,i') & =2(T-s) \text{ for } i \in \prescript{s
        \mkern-2mu}{}U, \; i' \in \prescript{s' \mkern-3mu}{}U,
      \text{ with } s < s' < T, \\
      r(i,i') & = \prescript{s}{}r (i,i') \text{ for } i,i' \in
      \prescript{s \mkern-2mu}{}U,
    \end{split}\\
  \intertext{and}
    \label{e2000-mu}
    \mu & =\sum_{s \in I_T} \prescript{s \mkern-2mu}{}\mu \otimes \delta_s.
  \end{align}
  If we start on a state which has already evolved for some time, say
  $s$, we have to concatenate it with a piece based on a finite set
  $I_{s,T} \subset [s,0]$ and on intensity $1/(T-s)$, independent of
  everything else. Then we use the sliding concatenation with
  $I_{s,T} \cup I_T$.
\end{remark}
\begin{lemma}
  \label{l2000rmu}
  Every process fitting our description in the first paragraph of the
  point (1) must have states which are equal in law to the
  $\U^V$-valued random variable in \eqref{e1963}.
\end{lemma}

We describe the state of the $\U^V$-valued process by considering the
population size (total mass), the measure giving the frequency of
colors from a subset of $A \subseteq V$,
\begin{align}
  \label{e2535}
  \bar \mfU=\nu(U \times V) \in \R_+,\quad \bar \mfU^{\rm rel}
  (\cdot)=\nu(U \times \cdot)/\bar \mfU \in \mcM_1(V)
\end{align}
and the sampling measure from the marked population and the induced
measure on $U$:
\begin{align}
  \label{e2539}
  \wh \mfU(\cdot)=\nu(\cdot)/\bar \mfU \in \mcM_1(U \times V),\quad
  \wh \mfU^{\rm gen}=\nu(\cdot \times V)/\bar \mfU \in \mcM_1(U).
\end{align}

\medskip
\noindent
\textbf{\emph{(2) Derivation of the operator}} In the following we
will define the $\U^{[0,\infty)}$-valued Evans process rigorously via
a \emph{martingale problem} in Theorem~\ref{TH.1061}. However, first
below we \emph{derive} the operator from the description and
construction of the state at time $T$ above and which gives us, if
such a process exists (this will be addressed in
Section~\ref{ss.pr.th1061}) the transition probability. This is stated
later on below in Corollary~\ref{l.1913}. We denote the process by
\begin{align}
  \label{e1315}
  \mfU^{\ast,+}
\end{align}
and call it \emph{$\U^{[0,\infty)}$-valued Feller diffusion with
  immigration from an immortal line}; recall the general $V$-marked
metric measure space $\U^V$ from \eqref{e691}. We set here
$V=[0,\infty)$ and then $\mfU^{\ast,+}$ is a
\begin{align}
  \label{e1101}
  \U^{V} \text{-valued process}.
\end{align}

Next we derive a formula for the \emph{operator} for the martingale
problem characterizing $\mfU^{\ast,+}$. As domain of our generator we
take polynomials $\Phi^{n,\varphi,g}$ of degree $n$ and of the form as
in Section~\ref{sss.topspa} but now even with $g \in C_b^1(V^n,\R)$,
i.e.\ bounded, continuously differentiable with bounded derivative.
However, here we have time-inhomogeneous dynamics. For this purpose we
write the polynomial in the form
\begin{align}
  \label{1076}
  \Phi^{n,\varphi,g}(\mfu)
  = \bar \Phi(\bar \mfu) \wh \Phi^{n,\varphi,g}(\hat \mfu),
\end{align}
with $\bar \Phi(\bar\mfu)=\bar\mfu^n$ and
$\wh \Phi^{n,\varphi,g}(\hat \mfu) = \int_{(U \times V)^n}
\dx\hat\nu^{\otimes n}(\varphi \cdot g)$ if $\bar\mfu\neq 0$ and
otherwise equal to zero.

\begin{remark}
  \label{r.2530}
  One can turn this time-inhomogeneous process into a homogeneous one
  by passing to the state space $\R_+ \times \U^V$ and replacing test
  functions $\Phi^{n,\varphi,g}=\bar\Phi\cdot \wh\Phi^{n,\varphi,g}$
  and the time-inhomogeneous operator $\Omega^{\uparrow,+}_{V,t}$ by
  \begin{align}
    \label{e2532}
    (t,\mfu) \mapsto \Psi(t) \cdot \Phi^{n,\varphi,g} (\mfu) \text{
    respectively }
    \frac{\partial}{\partial t} + \Omega^{\uparrow,+}_{V,t}.
  \end{align}
  where $\Psi \in C^1_b([0,\infty),\R)$.
\end{remark}

Because of the \emph{time-inhomogeneity} together with \emph{singular}
effects, instead of $g : V \to \R$ we need to take as building blocks
of polynomials functions of the form
\begin{align}
  \label{e2636}
  (t,\underline v) \mapsto g(t,\underline{v}) \; \text{ and}\;
  g \in C^1_b([0,\infty)^{1+n},\R).
\end{align}
On the corresponding set of test functions we define the operator
$\Omega^{\uparrow, +}_{V}$ fitting the description of the object we
gave translating the description of Evans to the framework of
$\U^V$-valued processes which we gave above in Lemma~\ref{l2000rmu}.

We have here a marked population whose genealogy evolves as a Feller
diffusion all carrying the same mark and at \emph{time $t$}
potentially a population with the \emph{mark $t$} starts according to
an \emph{entrance law} and is concatenated with the rest of the
population. This immigration of a color $s$ at time $t$ has operator
$\Omega^{\uparrow,+,s}_{V,\mathrm{imm}, t}$. This means our generator
consist of two parts
\begin{align}
  \label{e1068}
  \Omega^{\uparrow,+}_{V,t} = \Omega^{\uparrow,+}_{V} +
  \Omega^{\uparrow,+,t}_{V,\mathrm{imm}, t}.
\end{align}
We will specify these parts separately in (i) and (ii) below.

(i) The operator $\Omega^{\uparrow,+}_{V}$ is time-homogeneous and is
the \emph{extension} of $\Omega^{\uparrow}$ from $\U$ to $\U^{V}$.
Here the evolution changes the distance matrix distribution as before
by the growth of the distances. Furthermore, branching acts as before
on $\varphi$, but here the branching changes also the relative weights
of the colors already in the population before the present time. This
means that the branching part of the operator now maps
\begin{align}
  \label{e2581}
  \Phi^{n,\varphi,g} \to \Phi^{n,\wt \varphi,\wt g},
\end{align}
where with $i,j,k \in \{1,\dots,n\}$:
\begin{align}
  \label{e2585}
  \left(\wt \varphi (\uuu),\wt
    g(t,\underline{v})\right)=\sum_{1\le i < j \le n}
  \bigl(\varphi(\uuu^{i,j}),g(t,\underline{v}^{i,j})\bigr) \ind{v_i=v_j} \;
  \text{ and}
\end{align}
\begin{align}
  \label{e2589}
  \bigl(\uuu^{i,j}\bigr)_{k,\ell} =
  \begin{cases}
    u_{k,\ell}  & : \; \text{ for }\; k,\ell \notin \{i,j\}, \\
    u_{i,\ell}  & : \; \text{ for }\; k=i,j \neq \ell, \\
    u_{k,j} & : \; \text{ for }\; j=\ell,i\neq k, \\
    u_{i,j} & : \;\text{ for }\; k =i, j=\ell,
  \end{cases}
  \qquad
  \bigl(\underline{v}^{i,j}\bigr)_k =
  \begin{cases}
    v_i & :\; \text{ for }\; k=i, \\
    v_j & :\; \text{ for }\; k=j, \\
    v_k & :\; \text{ otherwise}.
  \end{cases}
\end{align}
Therefore in $\wh \Omega^{\uparrow,\rm gen}$ the operator
$\Omega^{\uparrow,\rm grow}$ acts on $\varphi$ as before and does not
touch $g$ where the branching part now is given by \eqref{e2581}.

(ii) Next we turn to the operator
$\Omega^{\uparrow,+,t}_{V, \mathrm{imm},t}$. Here we note that the
immigration operator acts at time $t$ only on the mass in the mark
$t$, however the mark determines the distances. The operator
$\Omega^{\uparrow,+,t}_{V,\mathrm{imm}, t}$ is time-inhomogeneous and
induces an inflow of total mass at rate $b$ which has type ``$t$'' at
time $t$. Since most of the colors have died out by time $t+\ve$ and
only finitely many survive for a longer time, the measure
$\nu(U \times \cdot)$ is atomic and has only the current time $t$ as a
condensation point. Here an issue is to ``decide'' about the
genealogical relationship of the \emph{new incoming individual} to the
\emph{current population}. By our convention the new individuals
coming in at time $t$ have distance $2(t-s)$ to those carrying the mark
$s$.

The assertion of the following lemma is a consequence of \eqref{e1116}
and a construction and calculation that we carry out in its proof in
Section~\ref{ss.pr.th1061} where we show that a process satisfying our
description exists.
\begin{lemma}
  \label{l.0850}
  For any process satisfying our description in the beginning of point
  (1) the following limit exists (recall \eqref{e2636} for $g$):
  \begin{align}
    \label{e1085}
    \wh \Omega^{\uparrow,+,t}_{V,\mathrm{imm},t} \wh
    \Phi^{n,\varphi,g} (t,\mfu)
    \coloneqq \lim_{\Delta \downarrow 0} \frac{1}{\Delta} \E
    \left[\wh \Phi^{n,\varphi,g} \left(\mfU_{t+\Delta} \right) -
    \wh \Phi^{n,\varphi,g} \left(\mfU_t \right) | \mfU_t = \mfu \right]
  \end{align}
  and more precisely we have for $g$ satisfying \eqref{e2636} and
  $t>s$
   \begin{align}
    \label{e1085a}
     \wh \Omega^{\uparrow,+,s}_{V,\mathrm{imm},t} \wh
     \Phi^{n,\varphi,g} (t,\mfu)
     = \sum_{i=1}^n \; \frac{b}{\bar \mfu_t^s} \; \wh
     \Phi^{n,\varphi,g_i} (\hat\mfu),
     \text{ with } g_i(\underline{s})
     \coloneqq \delta_{\underline{s},t}
     \frac{\partial g(t,\underline{s})}{\partial s_i} +
     \frac{\partial}{\partial t} g(t,\underline{s}).
  \end{align}
  With this notation we have
  \begin{align}
    \label{eq1081}
    \Omega^{\uparrow,+,t}_{V,\mathrm{imm},t}\bar \Phi \wh
    \Phi^{n,\varphi,g} (t,\mfu)
    = b n \bar \mfu^{n-1} \wh \Phi^{n,\varphi,g}(\hat \mfu)
    + \bar \mfu^n \wh \Omega^{\uparrow,+,t}_{V,\mathrm{imm},t}
    \wh \Phi^{n,\varphi,g} (\hat \mfu).
  \end{align}
\end{lemma}

Note that we can write the r.h.s.\ of \eqref{eq1081} as
\begin{align}
  \label{e2070}
  \frac{bn}{\bar \mfu} \; \Phi^{n,\varphi,g} (\mfu) + \bar \mfu^n \;
  \wh \Omega^{\uparrow,+,t}_{V,\mathrm{imm},t} \;
  \wh \Phi^{n,\varphi,g} (\hat \mfu).
\end{align}

\medskip
\noindent
\textbf{The $\U^V$-valued Evans process: results on $\mfU^{\ast,+}$}
We combine \eqref{e1068} and Remark~\ref{r.2530} with the above lemma
to obtain the following result.
\begin{corollary}[Evans' tree $\U^{V}$-valued]
  \label{l.1913}
  The generator of the dynamics of the $\U^{[0,\infty)}$-valued Evans'
  process fitting the properties of the description induces on
  $\U^{[0,\infty)}$ a process with generator acting on polynomials as
  given in \eqref{e1068} and \eqref{e1085a}.
\end{corollary}

\begin{remark}
  \label{r.2196}
  We observe that denoting by $(S_t)_{t \ge 0}$ the semigroup of the
  process $\mfU^{\ast,+}$, for $g$ constant
  $\wh \Omega_{V,\mathrm{imm},t}^{\uparrow,+} \; S_t(\wh
  \Phi_t^{\varphi,g})=0$. Hence in \eqref{eq1081} the only additional
  term is exactly the first one which is the same that we get for
  $\mfU^\dagger$ or $\mfU^{\mathrm{Palm}}$. This will allow us to
  identify a process $\mfU^\ast$ further below in
  Theorem~\ref{TH.IDENTKALL}, which relates the
  $\U^{[0,\infty)}$-valued Evans process to $\mfU$.
\end{remark}

The following theorem is proven in Section~\ref{ss.pr.th1061}.

\begin{theorem}[Genealogies of Feller diffusion with immigration from
  immortal line]\label{TH.1061}
  % \leavevmode\\
  Consider the state space $\U^\R$ and in it the closed subset of
  states $\U^\R_{\mathrm{imm}}$, defined by requiring that the marks
  satisfy $s \ge 0$ and distances of points satisfy, that elements of
  different colors have distance twice the color difference.
  Then the following assertions hold.
  \begin{enumerate}[(a)]
  \item For each $\mfu \in \U^\R_{\mathrm{imm}}$ the
    $(\delta_\mfu,\Omega_{V}^{\uparrow,+},\Pi_{V})$-martingale problem
    is well-posed.
  \item The corresponding realization of a solution, denoted
    (generically) by $\mfU^{\ast,+}$, is a Feller (recall
    Remark~\ref{rem:gFUV}) and strong Markov process with continuous
    paths. For initial laws on $\U^{\R}_{\mathrm{imm}}$ we define the
    process similarly to \eqref{eq:3pnu} obtain the martingale problem
    with random initial conditions from $\U^\R_{\mathrm{imm}}$.
  \end{enumerate}
\end{theorem}

\paragraph{The $\U$-valued Evans process: final results on the
  projection $\mfU^\ast$}
Next we return to the question whether a functional of $\mfU^{\ast,+}$
called $\mfU^\ast$ by ignoring colors (and hence its scaled version
introduced below in \eqref{e1538} $\breve{\mfU}^\ast$) itself is
\emph{Markov} which is saying that we want to define the Feller
diffusion with immigration from the immortal line and then just
observing its genealogy part, i.e.\ the projection
$[U \times V, r \otimes r_V, \mu] \mapsto [U,r,\pi_U\mu],$ so that we
have a process with values in $\U$ rather than $\U^{V}$. Then the
$\U_1$ component, i.e., the genealogy part $\wh \mfU^\ast$ and the
total mass $\bar \mfU$ gives the \emph{functional $\tau$} and the
process is denoted by
\begin{align}
  \label{e1737}
\mfU^\ast  \text{ given by the pair } (\bar \mfU,\wh \mfU^\ast).
\end{align}
We obtain here only states which decompose in balls of radii
$s_1<s_2< \dots <t$ if $t$ is the current time with $t$ as the only
accumulation point and the distance between two balls is $2(s_k-s_i)$
if $i < k$ are the corresponding indices of $s$. We call this closed
space $\U_{\mathrm{imm}}\subset \U$. We note that the orbit of
$\mfU^*$ is contained in $\U_{\mathrm{imm}}$.

The total mass process $\bar \mfU^\ast_t$ is a Markov process namely
the solution of $\dx X_t=b\dx t+\sqrt{bX_t} \; \dx B_t$. This raises the
question whether $\mfU^\ast$ is a \emph{Markov} process. The
observation in Remark~\ref{r.2196} allows us to conclude by defining
an operator for a martingale problem:
\begin{align}
  \label{e2240}
  \Omega^{\uparrow,\ast} \Phi^{n,\varphi}=n \frac{b}{\bar \mfu}
  \Phi^{n,\varphi} + \frac{b}{\bar \mfu} {\binom n 2}
  \Phi^{n,\varphi} + \frac{b}{\bar \mfu} \; \sum_{1 \le k < \ell \le
  n} \; \Phi^{n,\theta_{k,\ell} \circ \varphi}.
\end{align}
\begin{proposition}[Markov property of $\mfU^\ast$]
  \label{prop.2034}
  % \leavevmode\\
  The process $\mfU^\ast$ is the unique solution of the
  $(\delta_{\mfU_0},\Omega^{\uparrow,\ast},\Pi)$-Martingale problem,
  where $\mfU_0$ arise from $\mfU_0$ in Theorem~\ref{TH.1061} by
  removing the colors. In particular the process $\tau(\mfU^{\ast,+})$
  is Markov.
\end{proposition}

This result above allows us to define a $\U$-valued process with
continuous paths:
\begin{align}
  \label{e1329}
  \mfU^\ast = (\tau(\mfU^{\ast,+}_t))_{t \ge 0}
\end{align}
and a scaled process
\begin{align}
  \label{e1538}
  \breve \mfU^\ast = \left([U_t,bt^{-1} r_t,(bt)^{-1} \mu_t]\right)_{t \ge 0}.
\end{align}

Then we can prove:
\begin{theorem}[Identification of Kallenberg tree as functional of the
  Evans process]
  \label{TH.IDENTKALL}
  % \leavevmode\\
  The process $\mfU^\ast$ satisfies (recall
  \eqref{e1607},\eqref{e1683}):
  \begin{align}
    \label{a2082}
    \mcL_0 [\mfU^\ast]= \mcL_0[\mfU^{\mathrm{Palm}}] = \mcL[\mfU^{\mathrm{Kal}}].
  \end{align}
\end{theorem}

Next recall the explanations of and definitions of the ingredients in
the identity around \eqref{e1403} for $\varrho_h^t$, \eqref{e977Q} for
$\mfU^\dagger$, \eqref{e1607} for $\mfU^{\mathrm{Palm}}$,
\eqref{e1125}, \eqref{e12066}, \eqref{e1683} for
$\mfU^{\mathrm{Palm}}$.

\begin{corollary}[Identification of the size-biased $\U$-valued Feller
  diffusion]\label{c.idenfell}
  % \leavevmode \\
  We have the following equality of laws:
  \begin{align}
    \label{e1335}
    \mcL_0 \bigl[\mfU^{\mathrm{Palm}} \bigr] =
    \mcL_0 \bigl[\mfU^\dagger_0 \bigr] =
    \mcL_0 \bigl[\mfU^\ast \bigr] =
    \mcL_0 \bigl[\tau(\mfU^{\ast,+})\bigr] = (\varrho_t^t)^{\mathrm{Palm}}.
  \end{align}
\end{corollary}

\subsubsection[IPP-representation of \texorpdfstring{$\mfU^\dagger_t$,
  $\mfU^{\mathrm{Palm}}_t$, $\mfU^\ast_t$}{U-dagger, U-Palm, U*} via
backbone construction]{Longtime behavior of Feller diffusion 3:
  IPP-representation of \texorpdfstring{$\mfU^\dagger_t$,
    $\mfU^{\mathrm{Palm}}_t$, $\mfU^\ast_t$}{U-dagger, U-Palm, U*} via
  backbone construction}
\label{sss.backboncon}
We are now ready to return to the \emph{question of a cluster
  representation} for $\mfU^\dagger$ the $\U$-valued Feller diffusion
\emph{conditioned to survive forever}, or equivalently
$\mfU^{\rm Palm}$, which we raised at the beginning of
Section~\ref{sss.1616}.

We represent for that at a given time $t$ the state of $\mfU^\dagger$
using the identity in law with $\mfU^\ast$ as \emph{concatenation over
  an IPP (inhomogeneous PP) or a CPP on $\U$}. We shall explain why it
is better (at least for $\mfU^\dagger$, $\mfU^{\mathrm{Palm}}$,
$\mfU^\ast$) to use here an IPP, i.e.\ a representation with an
inhomogeneous Poisson point process, where we obtain a concatenation
of independent but \emph{not} identically distributed subfamilies
defined according to the most recent common ancestor after the moment
of immigration. A key feature, the ``points'' of this concatenation
will now arise here as final state of an \emph{evolving $\U$-valued
  process}. Abstractly in terms of the $\U$-valued state of $\mfU^t_t$
we decompose into the largest ball of radius $< t$, then take the
complement and take the largest ball of radius less than the previous
etc. and obtain this way a decomposition in $\U$-valued elements which
we can concatenate (sliding concatenation) to the full state. (This
will be a \emph{convergent countable} concatenation.)

We use that $\mfU^{\mathrm{Palm}}$ is equal in law to $\mfU^\ast$
which arises from $\mfU^{\ast,+}$ by ignoring the colors. This gives
us a decomposition into disjoint populations, which are independent
and of decreasing diameter according to a most recent common ancestor
before some (random) time $t$ back, which is the time of the
immigration of the founding father. Here we start our construction
with $\mfU^{\ast,+}$ and then pass to $\mfU^\ast$ to get the
decomposition of $\mfU^{\rm Palm}$.

\paragraph{The backbone construction of an IPP-representation}
We have three objectives.

\medskip
\noindent
\emph{First} we construct a concrete IPP-representation of the scaled
process $\breve \mfU^\ast$ running with time index $s \in [0,T]$ with
a \emph{fixed time horizon} $T \in [0,\infty)$ and of the scaling
limit as $T\infty \infty$, denoted $\breve \mfU_\infty^\ast$. In this
representation the state at time $T$ is obtained as time-$T$ state of
a time-inhomogeneous process which generates the state via a point
process, in this case as a \emph{concatenation} of \emph{independent
  subfamilies} immigrating at random times $s$ (at time-inhomogeneous
rate) and surviving for time $T-s$. This representation is called the
\emph{backbone construction}. Note that here we have a concatenation
of \emph{differently distributed} but \emph{independent} pieces.

\medskip
\noindent
\emph{Second} we have to show rigorously that we can obtain from the
picture of immortal lines from above a representation of the
$\U$-valued scaled distribution of the process at each time $t$ as
$t \rightarrow \infty$, which then converges in law to the
(generalized) \emph{quasi-equilibrium of the scaled $\U$-valued Feller
  process}. This later object we construct from an \emph{inhomogeneous
  Poisson point process} on the time interval $(-\infty, 0)$ and
independent copies of $\U$-valued Feller diffusions starting at the
points of the point process and conditioned to survive till time $0$.
We have to specify the \emph{intensity} of the PPP and to give the
rule how to \emph{concatenate} the i.i.d.\ copies of the $\U$-valued
Feller diffusions.

\medskip
\noindent
Our \emph{third} point is now to relate the backbone decomposition
with the \Levy{}-Khintchine decomposition. We consider a decomposition
in \emph{identically distributed} and \emph{independent pieces} and a
CPP-representation. The process $\mfU^\ast$ (and also
$\mfU^\dagger,\mfU^{\mathrm{Palm}}$) at time $t$ are infinitely
divisible since we can decompose the initial state $Y_0$ and the
immigration rate $b$ into $n$-pieces starting in $Y_0/n$ and with
immigration at rate $b/n$, for every $n$ and then the states of the
$n$-pieces are independent and identically distributed and their
concatenation is a version of the original process. Therefore the
$h$-trees can be represented via the \Levy{}-Khintchine representation
and we even have what we called a $\U$-valued Markov branching tree
structure allowing a \emph{Cox point process representation on $\U(h)$
  of all $h$-tops} via the \Levy{}-Khintchine formula. Recall in
contrast to that that the backbone representation gives an IPP
concatenation of subfamilies immigrating from the immortal line at
depths which are ordered and are in different $\U^\sqcup(s)$.

\begin{remark}
  \label{r.2660}
  This raises the question if there is an analogous backbone
  decomposition for the process $\mfU^T$. If we consider the entrance
  law from zero there is a \emph{backbone}, but the issue is the
  \emph{dependence} between the lines breaking off. This is due to
  the \emph{non-linearity} of the drift coefficient which does not
  allow a representation as for $\mfU^{\mathrm{Palm}}$ with
  independent branches breaking off.
\end{remark}

\medskip
\noindent
\textbf{\emph{Heuristics}} We use now $\U$-valued the Evans
construction but taken from a different perspective, namely instead of
looking forward we focus on the time $t$-state and its \emph{backward
  decomposition}, using its build up \emph{in time $s,0 \leq s < t$}
and considering the \emph{limit $s \uparrow t$}. Here we consider an
immortal particle which generates at rate $b$ independent copies of
the $\U$-valued Feller diffusion. In order to obtain the ones
surviving at a specified time $t$, we have to consider only those
surviving till time $t$. These are obtained generating from the
immortal line at \emph{rate $(t-s)^{-1}$ at time $s$ populations
  surviving from time $s$ up to time $t$}, more precisely from the
entrance law from the zero element, which are grafted, using a proper
$s$-concatenation, at the time of creation to the current genealogical
tree and conditioned to survive till the time horizon $t$. This builds
up the (marked) genealogy of the current population of
$\mfU^\ast(\mfU^{\ast,+})$ from a sequence of time
$(t-h_i)$-genealogies grafted at time $s_i = t-h_i$ to the
\emph{backbone} the \emph{line of descent of the time $t$-surviving
  particle}. Formally we make use of the operation
$\sqcup^{\mathrm{sli}}$ from \eqref{e1963}.

\medskip
\noindent
\textbf{\emph{Rigorous formulation}}
This population conditioned to survive for some time $t$ once started
at time $0$ we want to generate via a \emph{Markov process} evolving
from time $0$ to time $t$, which means in particular we have to work
with the dynamic producing the conditioning to survive till time $t$
we treated above in Section~\ref{sss.entrance}.

We recall here that this process can be started in \emph{mass zero} as
$\R_+$-valued process, however as $\U$-valued process a problem arises
since the branching operator involves the term $(\bar \mfu)^{-1}$.

We have to write down the grafting to the backbone on the \emph{level
  of ultrametric measure spaces} using $(\U, \sqcup^{s})$ for suitable
$s$. Before we treat the $\U$-valued case we again work with marks for
the immigrants. (see Remark~\ref{r.4013}) In formulas we want to write
\begin{align}
  \label{e2204}
  \mfU^{\ast,+}_t=\mathop{\bigsqcup\nolimits^{\mathrm{sli}}}\limits_{s
  \in N(t)} \enspace \mfW_{s,t}^{t,+},
\end{align}
where $N(t)$ is a time-inhomogeneous PPP with intensity $n_{t-s}$ in
$s \in [0,t)$ and $(\mfW_{s,u}^{t,+})_{u \in [s,t)}$ is an
$\U^{(0,\infty)}$-valued process such that they are \emph{independent}
for different $s$. This means that we can define via the concatenation
as in \eqref{e2204} processes $(\mfV_r^{t,+})_{r \in [0,t)}$ by
replacing $t$ in $N(t)$ by $r$, similarly taking $\mfW_{s,r}^{t,+}$
and then trying to characterize this process as a multi-type, with
types $s \in [0,t)$, branching process. Finally we will have to
consider the limit $r \uparrow t$.

\medskip
\noindent
\textbf{\emph{Construction of backbone}}
We will work with the rates in \eqref{e984} above to produce the
$\U^{[0,\infty)}$-valued process. Here the $\mfU_i \in \U(h)$ and are
versions of $\mfU^t_h$ and denoted by
\begin{align}
  \label{e1009}
  (\mfV^{t,+}_s)_{s \in [0,t)}.
\end{align}
We get this process by concatenating states of various $\U^V$-valued
processes together. First turn to such an element to be concatenated.

This amounts to construct first an $\U$-valued object, namely take our
generator $\Omega^{\uparrow(a,b)}$ and replacing the constants $a$,
$b$ by functions $a_t(s,\overline{\mfu})$ respectively
$b_t(s,\overline{\mfu})\equiv b$ from \eqref{e984} (where the quantity
$\bar \mfu$ refers to the piece generating \emph{one immigrant family}
which is determined by the parameters $s$ and $t$). We note that we
have to keep track here of the time of the insertion into the
population, since we need the mass of this sub-population together
with that time of insertion to \emph{determine the rates $a$ and $b$}
for this particular sub-population, which requires to introduce a
\emph{mark} which is inherited by the descendants. We then have an
increasing sequence of marks from $V=[0,t]$ with corresponding
$\U$-valued random variables which we have to \emph{mark} with the
time of appearance.

Then \emph{add the immigration} of independent copies of this
evolution to the immortal line, where the immigration is given by an
\emph{inhomogeneous Poisson point process} with intensity
\begin{align}
  \label{e1401}
  2(t-s)^{-1} \text{ for } s \in [0,t),
\end{align}
where the $r$ becomes the mark of the inserted population.

Next \emph{concatenate} them successively to the state $\mfV_s^{t,+}$
from \eqref{e1009} which is the time point of immigration for every
$s \in [0,t]$. Altogether we get a $\U^V$-valued time-inhomogeneous
processes. For this program we proceed as follows. This process is
different from the Evans process but has the same marginals at time
$t$. Recall that the pieces are conditioned to survive till time $t$
in this construction.

To concatenate we take the union of the populations with mark less
than $s$ and define the distance as the one in the sub-population for
two individuals from the same sub-population otherwise this is defined
such that the time $u$ descendants of time $s$, $s'$ immigrants have
distance $2(u-\min(s,s'))$. The measure is defined as the sum of the
sub-population measures extended to the disjoint union the obvious
way.

Due to the independence properties we can view this also as a
collection of independent $\U$-valued processes where the evolution
starts at increasing time points which follow the evolution
corresponding to $\Omega^{\uparrow,(a,b)}$ and which are then marked
with the starting time and concatenated so that we obtain an
$\U$-valued Markov process we want to relate to $\mfU^\ast$.

\begin{remark}
  \label{r.4013}
  We will see in the end that for all $\ve >0$ we have only finitely
  many colors $s \le t-\ve$ and we denote this collection by $\mcI_s$.
  Therefore, at any fixed time $s$ before time $t$ we can decompose
  the space $(U,r,\mu)$ into disjoint balls of radii decreasing in
  space and consider the state at time $s$ as a concatenation of the
  corresponding $\U$-valued random variables $\mfV_s(i)$,
  $i \in \mcI_s$. In particular we can write the generator as a sum of
  operators acting on the $i$-th term only. Since the decomposition is
  unique we obtain an operator which depends only on the state
  $\mfV_r^t$ giving an $\U$-valued Markov process. However carrying
  out the details here is a bit cumbersome and it is more convenient
  to work with $\U^{[0,\infty)}$ first and finally project on $\U$.
\end{remark}

\paragraph{Martingale problem description of backbone}
In order to identify the object $(\mfV_s^{t})_{s \in [0,t]}$ just
constructed, in particular its state at time $t$ and to relate it to
the process $\mfU^\ast$ respectively to the state $\mfU^\ast_t$, from
above we need more information. First, for every $t>0$ and
$s\in [0,t]$ we want to obtain $\mfV_s^{t}$ as the time $s$ state of a
Markov process via a martingale problem. Then we have to show that the
corresponding martingale problem is well-posed. This martingale
problem is of a somewhat different form compared to the one we had
before in $\mfY^t$ in \eqref{e934} because here instead of fixing a
leaf law we generate the masses also dynamically. We start again by
working with a $\U^{[0,\infty)}$-valued object below and later return
to the $\U$-valued situation.

Namely consider for every $t>0$, $V=[0,\infty)$, the
\emph{time-inhomogeneous $\U^V- $valued branching process
  $(\mfV^{t,+}_s)_{s \in [0,t]} $} defined as the continuous time
$\U^V$-valued branching process with mechanism as described above, the
corresponding operator is denoted
$\wt \Omega^{\uparrow,\ast,+}_{V,s}, \; s \in [0,t]$. We use now
polynomials for the marked case of the form \eqref{e942}.

First of all we need the growth operator of the distances which acts
only via $\varphi$ and is as before, then second we need the operator
acting on masses via the drift term, again as before. We need the
operator $\Omega^{\uparrow,(a,b)}$ to act separately for each of the
populations associated with a mark $s$ characterizing the time of
immigration $s$ so that $a$, $b$ are taken as $a_t(u,x)$,
$b_t(u,x)=bx$ for $u \in [s,t]$ and $x$ the time-$u$ mass of the
\emph{type $ s$-immigrant populations}, together with an explicit time
coordinate. We require that the resampling operator acts only on
variables $(u,v)$, $(u',v)$ with equal marks while otherwise we have
the zero operator, i.e.\ recalling \eqref{e942} and
\eqref{e2581}--\eqref{e2589} we have (replacing $t$ by $s$ in \eqref{e2585})
\begin{align}
  \label{e3181}
  \Omega^{\uparrow,(a,b)}_{V,s} \; \Phi^{n,\varphi,g} =
  \Phi^{n,\tilde \varphi,\tilde g}
  % \sum_{1\le k < \ell \le n} \Phi^{n,\theta_{k,\ell}\circ \varphi, \wt
  % \theta_{n,\ell} \circ g } \1\{v_k=v_\ell\}.
\end{align}
where $\varphi \in C^1_b([0,\infty)^{\binom{n}{2}}, \R)$,
$g \in C_{bb}(\R \times \R^n,\R)$.

Note however that here we need in addition to the time-inhomogeneous
evolution given by conditioned branching also \emph{immigration} at
rate $(t-s)^{-1}$ at time $s \in [0,t]$. This means we have to include
the mechanism of immigration in our martingale problem. To describe
the immigration effect with an operator we consider an evolution in a
\emph{randomly fluctuating medium}, where the medium turns a time into
an active time where then the branching operator
$\Omega^{\uparrow, (a,b)}_{V,s}$, the colored version of
$\Omega^{\uparrow,(a,b)}$, acts on the sub-population with the
corresponding mark.

The dynamics start for all colors $s$ in the zero tree (recall we can
start this conditioned dynamics in zero). This means we have an
$\N_0$-valued medium process with time-inhomogeneous jump rate
$(t-s)^{-1}$ at time $s$ to jump one up, call this process
\begin{align}
  \label{e3144}
  L=(L(s))_{s \geq 0}.
\end{align}

The medium \textit{flips} from \emph{active to passive} at the jump
times of $L$. However at the time of immigration the new population
splits off from ``the immortal path''. This means the distances of the
new immigrant at time $s$ to an individual $i$ with mark $u < s$ is
$2(s-u)$ and these distances are added at the moment of immigration
and grow according to the entrance law from the tree with total mass
$0$ and genealogy
$\mfe=[\{ \ast \}, \underline{\underline{0}}, \delta_\ast]$. We have
to describe this effect in a generator action now, together with the
active mark $s$ appearing. Hence we need an operator describing this
transition.

Recall that the generator of the $\U$-valued process degenerates for
mass $0$ for the branching term involving the (total mass)$^{-1}$ and
hence explodes at the total mass zero. Therefore starting the
$\U$-valued time-\emph{in}homogeneous diffusion starting from zero
arises itself as an \emph{entrance law}.

The complete operator for these two effects flipping of the medium and
immigration made precise below in \eqref{e1041} acts on $\Phi$ as
follows. Set $V=[0,t]$ and define the immigration operator where
at time $s$ an immigration of color $s$ occurs, which means an
entrance law from $0$ is added to the process. We have to calculate
the infinitesimal effect generated by this influx. For this purpose we
need an operator:
\begin{align}
  \Omega_{V,\mathrm{imm},s}^{\uparrow,+}
  \Phi^{n,\varphi,g}=\sum^n_{i=1}\Phi^{n,\varphi,\hat g_i},\;
  \text{ where} \; \hat g_i=\delta_{\underline{s},t}
  \frac{\partial}{\partial s_i} g + \frac{\partial}{\partial t} g
  \; \text{ (cf.\ \eqref{e1085a}}) \label{e1041}
\end{align}
which arises from the following operator by generating $L$ by an
inhomogeneous Poisson point process with intensity $2(t-r)^{-1}$:
\begin{align}
  \label{e3248}
  \wt \Omega^{\uparrow,\ast,+}_{V,r} (L)= \sum_{s \in I_t,s \leq r} \;
  \Omega^{\uparrow, (a,b)}_{V,s},
  \quad\text{where} \quad
  I_t \coloneqq I_t(L) \coloneqq \{s \in [0,t] : L(s) \ne L (s-)\}.
\end{align}

The operator of the marked process is:
\begin{equation}
  \label{e3803}
  \Omega^{\uparrow,\ast,+}_{V,s} = \Omega^{\uparrow,(a,b)}_{V,s}
  + \frac{2}{t-s} \Omega^{\uparrow,t,s}_{V,\mathrm{imm},s}.
\end{equation}

We handle the singularity by incorporating time in the state, recall
here what we did earlier in~\ref{e2636}. Then passing to the
\emph{time-space process} on $[0,\infty) \times E$ if $E$ was the
original state space (and test functions now have the form
$\Psi \Phi^{n,\varphi,g}$, where $\Psi \in C^1_b([0,\infty), \R)$
with $\Phi$ as above) gives for the fixed time horizon $T > 0$ the
operator $\bar \Omega^{\uparrow,\ast,+}_{V}$ acting as
$\frac{\partial}{\partial s} + \wt \Omega^{\uparrow,\ast,+}_{V,s}$
i.e.\ it acts as
\begin{align}
  \label{e1081}
  \bigl(\bar\Omega^{\uparrow,*,+} (\Psi \Phi)\bigr) (t,\mfu)
  = \Bigl( \frac{\partial}{\partial t} \Psi(t) \Bigr) \Phi(\mfu) +
  \Psi(t) \wt \Omega^{\uparrow,\ast,+}_V \Phi(\mfu).
\end{align}
This requires an argument why \emph{existence and uniqueness}
still hold. The precise statement is below.

\paragraph{Main results on PPP-representation via backbone}
First we need to show that the process $(\mfV^{t,+}_r)_{r \in [0,t]}$
is well-defined and since the rate in \eqref{e1401} diverges for
$r \uparrow t$ we also have to establish that $\mfV^{t,+}_r$ converges
to a limit in $\U^V$ as $r \uparrow t$. The following result is proven
in Section~\ref{ss.profbran}.

\begin{proposition}[$(\mfV_r^{t,+})_{r\in [0,t]}$ is well-defined by
  martingale problem]\label{l.branch} \hfil\\
  For the backbone construction the following assertions hold.
  \begin{enumerate}[(a)]
  \item For a given realization of the process $L$ the
    $(\delta_0,\bar\Omega^{\uparrow,\ast}_{V,t},\Pi(C^1_b))$-martingale
    problem is well-posed for times $r \in [0,t)$.
  \item The limit of $\mcL[\mfV^{t,+}_r]$ for $r \uparrow t$ exists in
    $\U^V$ and defines $\mfV_t^{t,+}$ if we require the continuity of
    paths on $[0,t]$.
  \end{enumerate}
\end{proposition}

We can shift the PPP in \eqref{e2204} by $t$ to $[-t,0]$ and then
consider $t \to \infty$. This means given a PPP, on
$(-\infty,0] \times \U$, $\U$ the equivalence classes of ultrametric
measure spaces we have to graft these points to the element, the zero
space, representing the immortal individual at the time $-h_i$, with
$-h_i$ the $i$-th component of the PPP seen from $-\infty$ which we
make precise next.

How can we now connect $\mfV^{\ast,+}$ and the Palm law at time $T$?
\begin{theorem}[Backbone decomp.\ of Palm distr., quasi-equilibrium
  and KY-limit]
  \label{T:BACKBONE}
  \leavevmode
  \begin{enumerate}[(a)]
  \item The process $(\mfV^{T,+}_t)_{t \in [0,T]}$ is a
    $\U^{[0,\infty)}$-valued process with its law at time $T$ being
    $\mcL[\mfU_T^{\ast,+}]$.
  \item Letting the process start at time $-T$, running it up to time
    $0$ and scaling mass and distances at time $0$ by $T^{-1}$,
    converges in law for $T \rightarrow \infty$ to a $\U$-valued
    random variable $\breve \mfV_0^{\infty,+}$.
  \item
  Let $\pi_U$ in \eqref{e1404} is the projection on the genealogy
    component induced by $U \times V \rightarrow U$.Define
    \begin{equation}\label{e4192}
    V_t^t =\pi_U V_t^{t,+}.
    \end{equation}
     We have with $\breve{}$ denoting the scaling introduced in
    \eqref{e1538}
    \begin{align}
      \label{e1098}
      \mcL \bigl[\breve \mfU^{\ast,+}_t \bigr]
      & = \mcL \bigl[\breve \mfV^{t,+}_t \bigr],\\
      \label{e1404}
      \mcL \bigl[ \breve \mfU^\ast_\infty \bigr]
      & = \mcL \bigl[ \breve \mfV^\infty_0 \bigr].
    \end{align}
  \end{enumerate}
\end{theorem}

We now obtain the representation of $\mfU^\ast$ via concatenation of a
subfamily decomposition.

\begin{corollary}[Decomposition in independent subfamilies and single
  ancestor $h$-subfamilies]\label{cor.3829}
  % \leavevmode \\
  We obtain, recall \eqref{e3248} conditioned on $I_T$ a decomposition
  in independent subfamilies
  \begin{equation}
    \label{e3831}
    \mfU_T^\ast =
    \mathop{\bigsqcup\nolimits^{\mathrm{sli}}}\limits_{s \in I_T} \; \mfV_{T,s}^T.
  \end{equation}
  For the path of $h$-decompositions we have
  \begin{equation}
    \label{e3835}
    \lfloor \mfU^\ast_T \rfloor
    (h)=\bigl({\mathop{\bigsqcup\nolimits^h}\limits_{i \in I_{T-h}}}
    \; \mfU^h_i \bigr)\,  {\mathop{\sqcup}}^h\, \mfV^h_h,
  \end{equation}
  where both parts are independent and conditionally on $I_{T-h}$
  \begin{equation}
    \label{e3839}
    (\mfU_i^h)_{i \in I_{T-h}} \; \text{ is a family of i.i.d.\
      $\U(h)^\sqcup$-valued random variables.}
  \end{equation}
  Furthermore there are $\N_0$-valued random variables $N_i$,
  $i \in I_{T-h}$ so that for each $i \in I_{T-h}$
  \begin{equation}
    \label{e3843}
    \mfU_i^h={\mathop{\bigsqcup}}^h \; \wt \mfU^h_{i,j} \;
    \text{ where } \; (\wt \mfU_{i,j}^h),\;  j=1,\dots,N_i \; \text{ are
      i.i.d.\ and $\U(h)$-valued}.
  \end{equation}
\end{corollary}

Since the super-criticality parameter for the Feller process
conditioned till time $T$ is \emph{non-linear}, the dynamics of
$(\mfV_s^T)_{s \in [0,T]}$ do \emph{not} have the generalized
branching property. Hence, the decomposition in \eqref{e3835} in
$\U(h)$-elements does not give for the $\mfV$-part identically and
independently distributed elements. The decomposition in elements of
$\U(h)^\sqcup$ is independent but not identically distributed.
Nevertheless, altogether we have a fairly good control over the
geometric structure of $\mfU^\ast_T$.

\paragraph{\Levy{}-Khintchine representation}
Finally we now relate the backbone representation with the
CPP-representation via the \emph{\Levy{}-Khintchine representation} of
the (as we saw) infinitely divisible $\mfU^\ast_t$, while
$\mfU^{\ast,+}_t$ is only $h$-infinitely divisible if we merge the
marks $s \leq t-h$ since the information contained in the color is
also an information on the genealogy before time $t-h$ (see
\cite{ggr_GeneralBranching} for more on this). Different from
\eqref{e2204}, we look then for an i.i.d.\ decomposition of all the
$h$-tops of $\mfU_t^{\ast}$, to obtain a random concatenation of
i.i.d.\ elements in $\U(h)$ according to the depth-$h$ most recent
common ancestors using the \Levy{}-Khintchine representation:
\begin{align}
  \label{e2501}
  \lfloor \mfU_t^{\ast}\rfloor (h)=\mathop{{\bigsqcup}^h}_{i =1,\dots,M_t(h)}
  \mfU_i^{t}
\end{align}
taking $M_t(h)=\Pois (\bar \mfU^{0,\rm Kal}_{t-h})$ and pick
independently the i.i.d.\ sequence $(\mfU_i^{t})_{i \in \N}$
distributed according to $\varrho_h^t = \mcL[\mfU^{0,\rm Kal}_t]$,
recall Theorem~\ref{T:BRANCHING} (for the case of $\mfU_t$) and note
that $\mfU_t^{\mathrm{Palm}}$ is a Markov branching tree. Here we
obtain a $\Pois(\bar \mfU_{t-h}^{\mathrm{Kal},0})$-number of
independent elements of $\U(h)$ which are concatenated and which are
copies of the $h$-truncated $\mfU^{0,\mathrm{Kal}}_t$.

A similar result holds for $\wt \mfU^{\ast,+}_t$ where $\tilde{\cdot}$
means the marks $s$ are replaced by $s \vee (t-h)$ shifted back to
zero for \eqref{e2501}. Different from the backbone representation,
there is no nice interpretation of the splitting of the immigration
rate, nevertheless for mathematical purposes this conditional i.i.d.\
decomposition gives very useful information.

\subsubsection[Kolmogorov-Yaglom limits for
\texorpdfstring{$Q$}{Q}-process, Palm process]{Longtime behavior of
  Feller diffusion 4: Kolmogorov-Yaglom limits for
  \texorpdfstring{$Q$}{Q}-process, Palm process}
\label{sss.Kolmog}
Recall the scaling in \eqref{e1538} of $\U$-valued processes which we
denoted by $\breve{\mfU}$ with possible sub- and superscripts. We have
obtained in Theorem~\ref{T:KOLMOGOROVLIMIT} the KY-limit
$\breve\mfU^\infty_\infty$ of the Feller diffusion, the KY-limit
$\breve \mfU_\infty^\dagger$ of the $Q$-process
$(\breve \mfU^\dagger_t)_{t \ge 0}$, and the KY-limit
$\breve \mfU_\infty^\ast$ of $(\breve \mfU_t^\ast)_{t\ge 0}$. It is
well known that on the level of the total mass processes the limits
$\breve \mfU^\dagger_\infty$ and $\breve \mfU^\ast_\infty$ are equal
in distribution. As we have seen above this holds also on the level of
$\U$-valued random variables. We have obtained and identified the
generalized Yaglom limit of $\breve \mfU^\ast_t$ as $t \to \infty$ in
\eqref{e1404} by the backbone construction. Here we state the
existence of a generalized Yaglom limit of $\mfU^\ast$ by considering
the scaling limit, usually called KY-limit of this process and
represent it in terms of the $\mfU^\ast$ process.

Once we have the KY-limit $\mfU^\infty_\infty$ for the Feller
diffusion conditioned on surviving till time $t$ and then scaled to
$\breve{\mfU}_t$, the next task is to identify the $\U$-valued
limiting random variable and exhibit its difference compared to the
ones from $\breve{\mfU}^\dagger$ and $\breve{\mfU}^{\mathrm{Palm}}$.

In the case of the total mass part the relation is simple, we have the
\emph{size-biased exponential} and the \emph{exponential} as limit
laws. The genealogical part is more subtle as we see from
\eqref{e2509}. One approach to see the difference is the conditional
duality where we should look for the difference in the total mass path
which arises from the super-criticalities
\begin{align}
  \label{e2509}
  \wt a_T(s,x) \; \text{ resp.\ } \; b/\bar \mfu_s^\ast
\end{align}
in the two cases and which remain different after the scaling which
leads to two branching diffusions with branching at rate $b$ and
immigration in the one case and a non-linear super-criticality rate
$\wt a(s,x)$ in the new coordinates, recall Remark~\ref{r.1627}. In
particular that we get different super-criticality terms for our
operators in both cases.

In the following we denote by $\mcL^\mfu$ the law of a process with
the initial condition $\mfu$, also recall the scaling from
\eqref{e1538} where the scaled processes are denoted by the following
accent $\breve{{\,}}$.

\begin{theorem}[Kolmogorov-Yaglom limits]
  % \leavevmode\\
  The \label{TH.KOLMO} following Kolmogorov-Yaglom limits
  $\lim_{t \to \infty} \mathcal L^\mfu [\breve \mfU^\Delta_t] =
  \mathcal L [\breve \mfU^\Delta_\infty]$ exist for
  $\Delta \in \{\dagger, \mathrm{Palm},*\}$ and are independent of the
  initial condition $\mfu$. Furthermore and we have
  \begin{align}
    \label{e1880}
    \mcL[\breve \mfU^\dagger_\infty]
    = \mcL[\breve \mfU_\infty^{\mathrm{Palm}}]
    = \mcL[\breve \mfU^\ast_\infty].
  \end{align}
  We have the following identification of the above $\U$-valued
  KY-limits:
  \begin{align}
    \label{e2541b}
    \mcL [\breve \mfU^\dagger_\infty] & = \mcL^\ntree[\mfU^\dagger_1],
  \end{align}
  while for the original process $\mfU$ we have, recalling notation $\mfU^T$ from
  Theorem~\ref{T:KOLMOGOROVLIMIT}, that,
  \begin{align}
    \label{e2541a}
    \mcL [\breve{\mfU}_\infty^\infty] & = \mcL^\ntree[\mfU^1_1].
  \end{align}
  In fact we can strengthen the above to pathwise statements on
  $\mcL[(\breve{\mfU}^\Delta_{at})_{a \in (0,1]}]$ for $t \to \infty$.
\end{theorem}

We see that the KY-limit for $\mfU_T^T$ has not such a nice
mathematical structure as $\mfU^\dagger$.

\subsection{Results 3: Genealogies for spatial case and continuum
  random tree}
\label{ss.extension}

We now discuss first genealogies in \emph{spatial} processes
(Theorems~\ref{T.SRWALK} and~\ref{T.CLUMP}) which is the ultimate goal
of this project but for which the previous eleven theorems are the
basis. Second we look at processes including all fossils, i.e.\ all
individuals \emph{ever alive before time $t$}. This object is
established in Theorem~\ref{T.AGING}. In Theorem~\ref{T.CRT} we give
the relation of this process to the celebrated \emph{continuum random
  tree} from \cite{Ald1991a,LeGall93}. This intends to clarify the
connection with the existing literature on \emph{labeled} trees.

\subsubsection{Genealogies of spatial processes: super random walk}
\label{sss.genalspat}
In the previous section we have described a non-spatial model, in
particular we do not cover for example branching random walk,
\emph{super-random walk} or the Dawson-Watanabe process. We provide
now the framework to model genealogies of the current population, if
this population is structured, i.e.\ \emph{distributed in geographic
  space} denoted by $G$.

The mechanism of the Feller diffusion has then to be augmented by a
\emph{migration} mechanism for individuals which may follow a random
walk as in a branching random walk or in the continuum mass limit,
i.e.\ the super random walk, its limiting object a \emph{mass flow}.
This means we have to lift our branching operator from $\U$ to $\U^G$
and we have to add to the generator a new term for the mark evolution,
which is here induced by migration of individuals.

We focus mainly on \emph{super random walk} we recall next, later we
comment on other spatial models in Remark~\ref{r.DW}. Therefore we
assume $G$ to be a countable abelian group. Here we have a countably
infinite or finite geographic space $G$ where in the former we
typically consider populations with \emph{infinite} (but \emph{locally
  finite}) total mass.

This means we now want to pass from the genealogy associated with
$\dx Y_t=\sqrt{bY_t} \, \dx w_t$ to the one associated with strong
solution of the system of SDE's:
\begin{align}
  \label{e937}
  (Y_t)=(y_\xi(t))_{\xi \in G},
\end{align}
\begin{align}
  \label{e938}
  \dx y_\xi(t) = c \sum_{\xi' \in G} \;
  a(\xi,\xi')(y_{\xi'}(t)-y_\xi(t))\dx t + \sqrt{by_\xi
  (t)} \; \dx w_\xi(t), \; \xi \in G,
\end{align}
with $((w_\xi(t))_{t \ge 0})_{\xi \in G}$ an independent collection of
standard Brownian motions, $a$ is a transition probability kernel on
$G$ describing in the underlying \emph{individual based} model the
jump probability $a(\xi,\xi')$ from $\xi'$ to $\xi$ and therefore the
flow from $\xi'$ into $\xi$, $c>0$ and
$Y_0 \in E \subseteq [0,\infty)^G$. We will assume here that $G$ is a
countable abelian group and $a(\cdot,\cdot)$ is homogeneous
$(a(\xi,\xi')=a(0,\xi' -\xi) \text{ for } \xi,\xi' \in G)$ and spans
$G$. Furthermore we define $\bar a$ by $\bar a(i,j)\coloneqq a(j,i)$
as the jump kernel of the underlying random walk of the migration in
the underlying individual based dual model.

If $|G| = +\infty$ then we have to restrict the $Y_0$ to a set
$E \subseteq [0,\infty)^G$, the so called \emph{Liggett-Spitzer
  space} defined by
\begin{align}
  \label{e.939}
  E \coloneqq \Bigl\{y \in [0,\infty)^G:   \sum_{\xi \in G} \; y_\xi
  \cdot \gamma_\xi <\infty\Bigr\},
\end{align}
where $\gamma=(\gamma_\xi)_{\xi \in G}$ satisfies: $\gamma >0$,
$\gamma$ is summable and with $a (\cdot,\cdot)$ being the migration
transition rate $(\gamma) a \le M \cdot \gamma$ for some
$M \in (0,\infty)$; see \cite{LS81,GLW05}. This guarantees that for
all times we get states which are locally finite and remain in $E$
a.s. For nice properties like the \emph{(generalized) Feller property}
(recall Remark~\ref{rem:gFUV} for the definition) one needs more
restrictions on the initial state namely consider $\wt E$ defined by
(see \cite{SS80})
\begin{align}
  \label{e.939t}
  \wt E \coloneqq \Bigl\{y \in [0,\infty)^G: \sum_{\xi \in G} \; y_\xi^2
  \cdot \gamma_\xi <\infty\Bigr\}.
\end{align}

We have to define below again a generalized Feller property on
non-locally compact state space $\U^G$ as we did in the case of
$\U$-valued process. If we have an initial distribution which is
translation invariant and satisfies $E[\bar x_\xi]< \infty$, then a.s.
all initial states are in the Liggett-Spitzer space respectively in
$\wt E$ if $E[y^2_\xi]< \infty$.

In order to treat the \emph{genealogy} of this process via ultrametric
measure spaces we have to augment our state and have to pass from the
state space $\U$ to another Polish space, the space of equivalence
classes of \emph{$G$-marked ultrametric measure spaces} $\U^G$, for
some \emph{geographic space $G$} which is typically some topological
abelian group $\Z^d$, $\R^d$ or alike. This object allows to describe
a population where individuals have a location in geographic space,
recall the paragraph in Section~\ref{sss.1616} on marked genealogies
and the space $\U^V$ where we now choose $V=G$. In this subsection we
shortly summarized what we need here about \emph{marked} metric
measure spaces.

Since we are interested in infinite geographic spaces, such as $\Z^d$,
we need to recall furthermore here the concept of a \emph{$G$-marked}
metric measure space $\U^G$ where the measure can have \emph{infinite}
mass $\nu(U \times G)$ and need only to be finite on bounded sets in
mark space. We then need polynomials on that space $\U^G$ including
marks and finally we need to define the \emph{migration operator} and
extend the operators we have to ones on the augmented state space. We
can build here on a couple of papers
\cite{DGP11,GSW,ggr_GeneralBranching} where these points have been
developed.

Once we have this framework we can characterize the $\U^G$-valued
super random walk process by a \emph{well-posed martingale problem},
establish a \emph{Feynman-Kac moment duality} with an \emph{enriched
  spatial coalescent} and describe the long time behavior of the
process as $t \to \infty$.

\paragraph{State space of $G$-marked genealogies}
In order to include in the concept of marked genealogies described by
$\U^G$ the possibility of \emph{infinite} populations, which is needed
for infinite respectively unbounded geographic space, we consider
\emph{finitely bounded measures} $\nu$, bounded on the population
restricted to finite subsets of $G$ (i.e.\ elements $(u,\xi)$ with
$u \in U$, $\xi \in A$, $|A |< \infty$). Now the equivalence classes
are formed w.r.t.\ the sequence of restrictions in the spaces
$(\U^{G_m})_{m \in \N}$ which are required to be \emph{each
  equivalent} in the sense specified earlier. Namely we consider
$G_n \uparrow G$ with $G_n$ bounded and consider the restrictions of
the population to $G_n$, i.e.\ replace $[\bar U \times G,r,\mu]$ by
$[U \times G_n,r \big \vert_{(U \times G_n)^2}, \mu \big \vert_{U
  \times G_n}]$.

Introducing a topology is more subtle since we leave typically
infinite total mass of the ``sampling'' measures on an infinite space
$G$. Therefore we again work with the approximation of $G$ with finite
geographic spaces $G_n$. The \emph{topology} can be introduced by
defining the convergence of sequences of elements $\mfu_k$, $k\in \N$
from $\U^G$ in this topology. We consider for each $n \in \N$ the
sequence $(\mfu_k^{(n)})_{k \in \N}$ of restrictions to $G_n$, for
which convergence is already defined. We require for a sequence in
$\U^G$ to converge, the convergence of all restrictions to the
$G_n$-populations. (See \cite{GSW} for details in particular that the
topology does \emph{not} depend on the choice of the
$(G_n)_{n \in \N}$.). The space of all elements of the form as in
\eqref{e940} is denoted again $\U^G$, equipped with the above
topology and leads to a Polish space.

As was pointed out above in \eqref{e.939} we need restrictions on the
initial state. Namely we consider $\mcE$ resp. $\wt \mcE$ given by
(recall \eqref{e939} and the sequel):
\begin{equation}
  \label{e4039}
  \mcE=\{\mfu \in \U^G |\bar \mfu \in E\},
\end{equation}
analogously $\wt \mcE$.

\paragraph{The martingale problem}
The \emph{domain for the operator} of our martingale problem is a
subspace of $\Pi^{G}$, the set of spatial polynomials which are given
by
\begin{align}
  \label{eq:3}
  \Phi^{\varphi,g}(\mfu) = \int_{(U\times G)^n}
  \varphi(\uuu)g(\underline{v})
  \,d\nu^{\otimes n}((\uuu, \underline{v})),
\end{align}
where $\varphi \in \mathcal{C}_b(\R^{\binom{n}{2}},\R)$ and $g$ is a
function on $G^n$ depending on finitely many points, i.e.\ have
bounded support.

In order to specify the operator we choose a \emph{domain} $\mcD$ in
$\Pi^G$ by assuming a more special form of the polynomial where $g$
and $\varphi$ are of as special form, but still such that we can
generate a law determining algebra. The point of this is that on
$\mcD$ we can specify the operator of the martingale in a simple
fashion.

First we fix a typical $g$ that we have in mind. We fix arbitrary
$(\xi_1,\dots,\xi_n) \in G^n$. In particular there can be $i\ne j$
with $\xi_i = \xi_j$. All the discussion in this paragraph will be
w.r.t.\ this fixed $n$-tuple. Let $\{\xi_1,\dots,\xi_m\}$ be some
ordered set of its distinct elements. For $\zeta \in G$ we let
$A_\zeta = \{i \in \{1,\dots,n\}: \xi_i =\zeta\}$ be the set of all
indices in $\{1,\dots,n\}$ at which the elements of the fixed
$n$-tuple are given by $\zeta$. This set is of course empty unless
$\zeta \in \{\xi_1,\dots,\xi_m\}$.

We assume that $g$ is of the form
\begin{align}
  \label{eq:23}
  g(v_1,\dots,v_n) = \ind{v_1=\xi_1}\cdot \dots \cdot \ind{v_n=\xi_n}
  = \prod_{i=1}^m g_{A_{\xi_i}} (\underline v|_{A_{\xi_i}}),
\end{align}
where $\underline v|_{A_{\xi_i}}$ is the projection of $\underline v$
to coordinates in $A_{\xi_i}$ and
$g_{A_{\xi_i}} (\underline v|_{A_{\xi_i}}) = \prod_{k \in A_{\xi_i}}
\ind{v_k = \xi_i}$. Next we assume that $\varphi$ is of the form
\begin{align}
  \label{eq:29}
  \varphi(\uuu) = \prod_{i=1}^m \varphi_{A_{\xi_i}}
  (\uuu|_{A_{\xi_i}}),
\end{align}
where $\uuu|_{A_{\xi_i}}$ is the sub-matrix of $\uuu$ with indices
projected to $A_{\xi_i}$. In case $\abs{A_{\xi_i}}=1$ the function
$\varphi_{A_{\xi_i}}$ is a constant.

For $\zeta \in G$ we set $\nu_\zeta = \ind{\zeta}\nu $ and define
\begin{align}
  \label{eq:26}
  \Phi^{\varphi,g}_{A_\zeta} =\int_{(\U \times G)^{\abs{A_\zeta}}}
  \varphi_{A_\zeta} (\uuu) g_{A_\zeta}(\underline v)
  \bigotimes_{i \in A_\zeta} \nu_{\xi_i}(d(\uuu,\underline v)).
\end{align}
This is of course $0$ if $\zeta \notin \{\xi_1,\dots,\xi_m\}$. With
these choices of $g$ and $\varphi$ we can write the polynomial from
\eqref{eq:3} in the form
\begin{align}
  \label{eq:27}
  \Phi^{\varphi,g} = \prod_{i=1}^m \Phi_{A_{\xi_i}}^{\varphi,g}.
\end{align}

The operator for the martingale problem has the form
\begin{align}
  \label{e943}
  \wt \Omega^\uparrow = \wt \Omega^{\uparrow, \mathrm{grow}} + \wt
  \Omega^{\uparrow, \mathrm{bran}} + \wt \Omega^{\uparrow, \mathrm{mig}}.
\end{align}
Here $\wt \Omega^{\uparrow, \mathrm{grow}}$ and
$\wt \Omega^{\uparrow, \mathrm{bran}}$ are extensions of the operators
$\Omega^{\uparrow, \mathrm{grow}}$ and
$\Omega^{\uparrow, \mathrm{bran}}$ on $\U$ to $\U^G$ the spatial case,
recall \eqref{tv5} and \eqref{mr3}. They act on the polynomials as
before namely just via $\varphi$ and leave $g$ untouched. This means
that $\wt\Omega^{\uparrow,\mathrm{bran}}$ and
$\wt\Omega^{\uparrow,\mathrm{grow}}$ have the form
\begin{align}
  \label{eq:24}
  \wt\Omega^{\uparrow,\mathrm{bran}} = \sum_{\xi \in G}
  \wt\Omega^{\uparrow,\mathrm{bran}}_{\xi}, \quad
  \wt\Omega^{\uparrow,\mathrm{grow}} = \sum_{\xi \in G}
  \wt\Omega^{\uparrow,\mathrm{grow}}_\xi,
\end{align}
where $\wt\Omega^{\uparrow,\mathrm{bran}}_{\xi}$ and
$\wt\Omega^{\uparrow,\mathrm{grow}}_\xi$ act as
$\Omega^{\uparrow,\mathrm{bran}}$ resp.\
$\Omega^{\uparrow,\mathrm{grow}}$ on the population at location $\xi$.
This means
\begin{align}
  \label{eq:25}
  \wt\Omega^{\uparrow,\mathrm{bran}}_{\xi} \Phi^{\varphi,g} =
  ( \Omega^{\uparrow,\mathrm{bran}} \Phi^{\varphi,g}_{A_\xi})
  \cdot\Phi^{\varphi,g}_{\{1,\dots,n\} \setminus A_\xi}.
\end{align}
The operator $\wt\Omega_\xi^{\uparrow,\mathrm{grow}}$ is defined in
the same way in terms of $\Omega^{\uparrow,\mathrm{grow}}$.

The operator $\wt \Omega^{\uparrow, {\mathrm{mig}}}$ is \emph{new} and
next explained in detail. Recall here the defining SDE of the total
mass process from above and in particular the migration term of this
equation. The evolution of the marks leads to a first order operator
(a drift term). The migration operator is defined on $\Pi^{G,+}$, the
positive elements of marked polynomials space $\Pi^{G}$ (recall
\eqref{e2506}) as follows:
\begin{align}
  \label{e944}
  \wt \Omega^{\uparrow, \mathrm{mig}} \; \Phi^{\varphi,g} = \sum_{\xi,\xi' \in G} \;
  \wt \Omega^{\uparrow, \mathrm{mig}}_{\xi,\xi'} \; \Phi^{\varphi,g},
\end{align}
where the summands correspond to the flow between $\xi'$ and $\xi$ as
we now describe: For $\xi,\xi' \in G$ we define
$\prescript{\xi,\xi'}{}{\Phi}_k^{\varphi,g}$ as the monomial
$\Phi^{\varphi,g}$ where $g$ is replaced by $g_k^{\xi,\xi'}$ with
\begin{align}
  \label{e946}
  g_k^{\xi,\xi'}\; (v_1, \dots, v_{k-1}, \xi,v_{k+1}, \dots, v_n)
  = g (v_1,\dots, v_{k-1}, \xi', v_{k+1}, \dots, v_n).
\end{align}
Then the operator for the $\xi'$-$\xi$ flow acts as follows
\begin{align}
  \label{eq:36}
  \wt\Omega^{\uparrow,\mathrm{mig}}_{\xi,\xi'} \Phi^{\varphi,g} =
  a(\xi,\xi')\sum_{k=1}^n (\prescript{\xi,\xi'}{}{\Phi}_k^{\varphi,g_k} -
  \Phi^{\varphi,g})
\end{align}
This follows the same way as the standard moment calculation for
measure-valued processes; see \cite{D93} Section~4.7.

We see that $\wt\Omega^{\uparrow}$ maps the domain $\mcD$ into $\Pi^G$
and hence we have a linear operator on $\Pi^G$, such that we can use
it for a martingale problem.

Now we can calculate the operator in a way which allows to read off
the operators of $\bar\mfU$ and of $\wh\mfU$ conditioned on $\bar\mfU$
as we did in the non-spatial case in \eqref{eq:r.742.1}--\eqref{e746}.
The generator
$ \wt \Omega^{\uparrow, \mathrm{mig}} \; \Phi^{\varphi,g}$ acts on
$\Phi^{\varphi,g} = \bar \Phi^{\varphi,g} \wh \Phi^{\varphi,g}$ as
\begin{align}
  \label{eq:28}
  \wt \Omega_{\xi,\xi'}^{\uparrow,\mathrm{mig}} (\Phi^{\varphi,g}) = \wh
  \Phi^{\varphi,g} \cdot \bigl( \wt \Omega_{\xi,\xi'}^{\uparrow,\mathrm{mig,mass}}
  \bar \Phi^{\varphi,g}\bigr) + \bar
  \Phi^{\varphi,g} \cdot \bigl( \wt \Omega_{\xi,\xi'}^{\uparrow,\mathrm{mig,gen}}
  \wh \Phi^{\varphi,g}\bigr),
\end{align}
where we now have to define the operators for the mass and genealogy
parts. Define $\overline{\mfu}_\xi = \nu(U \times \{\xi\})$.

We have with $n$ denoting the degree of the monomial $\Phi$ the
following expression. The operator
$\wt\Omega^{\uparrow,\mathrm{mig,mass}}$ is essentially the operator
of the super random walk, i.e.\
\begin{align}
  \label{eq:30}
  \wt \Omega_{\xi,\xi'}^{\uparrow,\mathrm{mig,mass}} \bar
  \Phi^{\varphi,g} (\bar\mfu) = a(\xi,\xi')
  \Bigl(\prescript{\xi,\xi'}{}{\bar\Phi}_k^{\varphi,g}(\bar\mfu)
  - \bar\Phi^{\varphi,g} (\bar\mfu)\Bigr) \wh\Phi_k^{\varphi,g}.
\end{align}
For the generator part we have we have the following expression which
may take the value $+\infty$
\begin{align}
  \label{e945}
  \Big(\wt \Omega^{\uparrow, \mathrm{mig,gen}}_{\xi,\xi'} \; \Phi^{\varphi,g}
  \Big)(\mfu)= \sum_{k=1}^n \;
  \frac{\overline{\mfu}_{\xi'}}{\overline{\mfu}_\xi} \;
  a(\xi,\xi') \Big(\prescript{\xi,\xi'}{}{\wh\Phi}_k^{\varphi,g} \;
  (\mfu) - \wh\Phi^{\varphi,g} (\mfu) \Big).
\end{align}

We note that for $\bar \mfu_\xi=0$ the expression is still
well-defined. Namely the expression \eqref{e945} contains
$\bar \mfu_{\xi_1},\dots,\bar \mfu_{\xi_n}$ and hence there appears
the factor $\bar \mfu_\xi$ if $g$ is not equal to $0$ in $\xi$.

\paragraph{Feynman-Kac duality and conditional duality}

In the spatial case there is again a Feynman-Kac duality. At the same
time there is for the same type of conditional duality, but without
the Feynman-Kac term, which we introduce below. We discuss first the
Feynman-Kac duality.

The dual process in the spatial case is based on a simple pure jump
Markov process, namely the \emph{spatial coalescent}. This is a
process which takes values, in the \emph{$G$-marked partitions} of
$\{1, \dots, n \}, n \in \N$, i.e.\ every partition element gets a
\emph{location} in $G$. The dynamic of the non-spatial case is
modified by allowing the following transitions: a pair of partition
elements coalesces at rate $b$ during the time they spend both
together \emph{at the same location}, the marks of the
partition elements follow independent $a(\cdot,\cdot)$-random walks
till they coalesce and then the new partition element follows with its
mark one random walk.

Hence we have now a state, where locations are added to $(p,\uur^p)$
and is of the form:
\begin{align}
  \label{e1130}
  \big((p,\xi),\underline{\underline{r}}^p \big), \; \text{ with }
  \; \xi: p \mapsto G^{|p|}.
\end{align}
As corresponding state space for the dual process (\emph{distance
  matrix augmented spatial coalescent}) we choose $\K_G$ which we get
with denoting by $\bbS_G$ the set of $G$-marked partition elements and
put:
\begin{align}
  \label{e1131}
  \K_G=\bbS_G \times (\R_+)^{\binom{\N}{2}} \times \mathcal B_{\mathrm{fc}}.
\end{align}

The \emph{duality function} $H(\cdot,\cdot)$ is now given as follows.
Define first as an ingredient for every $\varphi$, $p$ and $\xi$ a
polynomial
\begin{align}
  \label{e1132}
  \begin{split}
    & H^{\varphi,g} : \U^G \times \K_G \longrightarrow \R \\
    & H^{\varphi,g} \Big(\mfu,\big(
    (p,\xi),(\underline{\underline{r}})^p \big) \Big) = \; \int_{U^n}
    \varphi \big((\underline{\underline{r}}^p +
    \underline{\underline{r}}')\big) {\mathop{\otimes}_{i=1}^{|p|}}
    \mu_{\xi_i}(\dx u_i) , \text{ where } \mu_{\xi_i}=\nu(\cdot \times
    \{\xi_i \})
  \end{split}
\end{align}
with $\mfu=[U,r',\mu]$, $\uur' \coloneqq (r(u_i,u_j))_{i,j}$,
$\underline{\xi}=(\xi_i)_{i=1,\dots,|p|}$ and $g:G^n \to \R$ given for
a fixed tuple $\xi \in G^n$ as parameter by:
\begin{align}
  \label{e1206}
  g(\underline{\xi}') =  g_{\underline{\xi}}(\underline{\xi}')=
  \prod\limits_{i=1}^n \; \ind{\xi'_i=\xi_i}, \underline{\xi}' \in G^n, n=|p|.
\end{align}
Now we augment the state $((p,\xi),\uur^p)$ by a further component
$(\varphi,g)$ and define
\begin{align}
  \label{eq:21}
  H(\cdot,(\cdot,(\varphi,g))) = H^{\varphi,g}(\cdot,\cdot).
\end{align}
By this procedure we obtain a duality function $H$ on
$\U_G\times (\bbK_G \times \mathcal B_{\mathrm{fc}}\times G^n)$ and
the new component $(\varphi,g)$ is constant in time.

The \emph{Feynman-Kac potential} on the state space does only depend
on the locations of partitions and not on $(\varphi,g)$ and is given
by the function
\begin{align}
  \label{e1133}
  ((p,\underline{\xi}),\uur) \mapsto b \cdot
  \sum_{\substack{{i,j=1} \\ {i \neq j}}}^{|p|} \indset{\{\xi_i=\xi_j\}}.
\end{align}
Note that the integral in \eqref{e1132} can be written as:
\begin{align}
  \label{e1202}
  \int_{(U \times G)^n} \varphi ((r(u_i,u_j))_{1 \le i < j \le n})
  g(\xi) \, \mu^{\otimes n} (\dx(u_1,\xi_1),
  \dots,\dx(u_n,\xi_n)).
\end{align}

The conditional duality for $(\wh\mfU_t (\bar\mfu))_{t\ge 0}$ for
given path $\bar\mfu$ of $(\bar\mfU_t)_{t\ge 0}$ will be defined for
a.s.\ all realizations. The duality function $H(\cdot,\cdot)$ from
above remains the same. For the \emph{conditional duality} the dual
process \emph{changes} and is now a time-inhomogeneous Markovian pure
jump process $\mfC_t(\bar\mfu)$, where the path $\bar\mfu$ is a
parameter. The rates are now time-inhomogeneous and are given by
$b/\overline{\mfu}_\xi$ for a coalescence event in $\xi$ and
$a(\xi,\xi')\overline{\mfu}_{\xi'}/\overline{\mfu}_\xi$ for migration
from $\xi$ to $\xi'$ with $\overline{\mfu}$ evaluated at time $T-t$ at
time $t$ and time horizon $T$ for the duality.

Here, because of the \emph{singularity} in the rate, we have to argue that
this jump process is well-defined for all times $t \in [0,T]$ and
actually no instantaneous transitions occur.

This amounts to showing that at the time where the individuals of the
coalescent sit in sites with a singularity at a time just prior to a
jump (of the coalescent) rapidly jump to sites without singularity
immediately beyond. This has been made precise and was shown in
Proposition~0.2 in \cite{DG03}.

\paragraph{Results on super random walk}
We can now precisely define the genealogical process of super random
walk, the $\U^G$-valued super random walk.

\begin{theorem}[$\U^G$-valued super random walk]
  \label{T.SRWALK}
  \leavevmode
  \begin{enumerate}[(a)]
  \item The
    $(\delta_\mfu,\wt\Omega^\uparrow_G,\Pi^G)$-martingale
    problem for $\mfu \in \mcE$ is well-posed and has a solution with
    continuous path defining a Markov process. This solution is a
    strong Markov and (generalized) Feller process for
    $\mfu \in \wt \mcE$. For general initial laws the solution of the
    \emph{local} martingale problem is given via \eqref{eq:3pnu}.

    The occupation measure
    $\bar\mfU=(\bar\mfU_t)_{t\ge 0} = (\mu_t(U_t \times \cdot))$ gives
    the unique weak solution of \eqref{e938}. The pure genealogy
    process $\wh\mfU (\bar\mfu)= (\wh\mfU_t(\bar\mfu))_{t\ge 0}$ is
    for a.s.\ all $\bar\mfu$ a time-inhomogeneous spatial
    genealogy-valued Fleming-Viot process (i.e.\ $\U_1^G$-valued) with
    local resampling rate given by $(b\bar\mfu_\xi(t)^{-1}$ at $\xi$ at time
    $t$ and migration rate of individuals from $\xi' \to \xi$ given by
    $a(\xi,\xi') \bar\mfu_{\xi'}/\bar\mfu_{\xi}$.
  \item The solution of the
    $(\delta_\mfu,\Omega^\uparrow_G,\Pi_{\mathrm{fin}}(C^1_b))$-martingale
    problem is in Feynman-Kac duality with the spatial augmented
    Kingman coalescent w.r.t.\ duality function $H$.
  \item The process $(\wh\mfU_t)_{t\ge 0}$ conditioned on the complete
    path $(\bar\mfU_t)_{t \ge 0}$ is in duality w.r.t.\ $H$ to the
    time-inhomogeneous spatial coalescent
    $(\mfC_t(\bar\mfu))_{t\ge 0}$.
  \end{enumerate}
\end{theorem}

The reader might have wondered whether the time-inhomogeneous
Fleming-Viot process appearing in the above theorem is well-defined, the
problem being that at certain times and sites rates $+\infty$ appear.
This is of course a point which needs some care but there are results
in the literature.

In \cite{DG03} such a situation was analyzed and a \emph{modified
  concept of the martingale problem introduced} for that. We have of
course the expression in our martingale problem as it stands diverging
terms. Therefore we must identify for test functions depending on a
finite number of sites and require for our modified martingale
property time intervals only which contain \emph{no singularity} at
these points. The point is that the complement of the set at time
points which are singularity free can be exhausted by collections of
singularity free closed sub-intervals of time (since the singularities
are a closed set of Lebesgue measure zero and the complement is an
open set with full Lebesgue measures). It needs to be proved that we
obtain a unique solution of this \emph{modified} martingale problem,
with continuous path. Then $\wh\mfU(\bar\mfu)$ is Markov process with
values in $\U_1^{G,\#}$.

This is done as in \cite{DG03} by defining approximations where on
where on very small intervals the process is frozen, i.e.\ our rates
which diverge are cut and give us standard processes which converge as
the cutting level is raised to $\infty$. We refer for details to the
literature and assume the well-posedness of the time-inhomogeneous
$\U_1^{G,\#}$-valued Fleming-Viot process here.

The next question is whether we have the \emph{generalized branching
  property} and the \emph{Cox point process representation} from the
\Levy{}-Khintchine formula analogous to Theorem~\ref{T:BRANCHING}
parts (a),(b).

This issue is addressed in \cite{ggr_GeneralBranching} respectively
\cite{infdiv} and answered to the \emph{positive}; for details we
refer the reader to these papers.

\begin{remark}[Longtime behavior]\label{r.longbeh}
  In this framework we can now also analyze the question, how the
  genealogies behave as $t \rightarrow \infty$. This depends very much
  on the kernel $a(\cdot,\cdot)$. If the symmetrized kernel
  $\wh a = \frac{1}{2}(a + \bar a)$ is recurrent then the process
  becomes \emph{locally extinct} and conditioned on local survival one
  has on each finite subset of $G$ a diverging family descending from
  a single founding father. In the transient case it is well known
  that the super random walk has a translation invariant ergodic
  equilibrium with mean $\theta$ for every $\theta \in [0,\infty)$.

  We obtain here also a stationary limiting genealogy with countably
  many such founding fathers whose descendants are in distance
  $+\infty$. To make the latter precise some reformulation is needed,
  in particular one passes from $r(\cdot,\cdot)$ to the ultrametric
  $(1-e^{-r(\cdot,\cdot)})$ which maps $0,\infty]$ onto $[0,1]$
  one-to-one. We can not work out details in this paper.
  %\blue{and
  %  refer the reader to \cite{gmuk}} where this problem is treated in
  % more detail for logistic super random walk.
  The case of the genealogical $G$-indexed Fleming-Viot process is
  treated in \cite{GSW} in complete generality and for $G=\Z$ in great
  detail.
\end{remark}

\begin{remark}[Dawson-Watanabe process: genealogical version]
  \label{r.DW}
  The treatment of the Dawson-Watanabe process involves as a further
  limit the \emph{spatial continuum limit} $\ve \Z^d \to \R^d$, (with
  $\ve \rightarrow 0)$, where we face the fact that in the duality
  relation the joint occupation times of the path degenerate in
  $d \ge 2$ and do \emph{not} lead to a Feynman-Kac duality with a
  spatial coalescent for a stochastic $\U$-valued dynamic, due to the
  \emph{lack of uniform integrability} of the \emph{exponential term}.
  This results in the necessity to work with a different argument here
  to obtain the uniqueness. We cannot carry out the details for
  $d \ge 2$ in this paper since this requires new techniques and a
  different formulation of the martingale problem.

  For $d=1$ the analogous limit for the $\U^\Z$-valued interacting
  Fleming-Viot process instead of the $\U^\Z$-valued super branching
  random walk is treated in \cite{GSW}. For branching in $d=1$ we can
  again work with the Feynman-Kac duality and obtain a well-posed
  martingale problem. We use this $\U^\R$-valued super process below
  to analyze the asymptotics of the super random walk on $\Z$.
\end{remark}

\paragraph{Application to long time behavior}
To get a better impression of what is behind the
Remark~\ref{r.longbeh} above we at least apply our techniques and
approach from the previous section for the non-spatial case and the
conditioned on survival process to study the long time behavior of
the super random walk in a specific case. We look at the regime where
the migration mechanism is \emph{strongly recurrent}, for example on
$\Z$, and where it is known (see \cite{DF88,DG96,DG03}) that the super
random walks forms \emph{clumps} of large mass on a \emph{thin set in
  space}. Here we can show now that such clumps have \emph{marked
  genealogies} for which we can give an explicit asymptotic
description as $t \to \infty$ and this description is fairly explicit.

We consider as an example the case of the \emph{super random walk on
  $\Z$} with a symmetric kernel $a(\cdot,\cdot)$, which is in the
domain of normal attraction of Brownian motion. We assume that the
random initial state is having translation invariant ergodic states
with mean $\theta$ for its total masses and all initial distances are
put equal to $0$, w.l.o.g. Then the total mass process goes locally to
extinction by forming rare, i.e.\ spatially separating, clumps of
diverging height and volume in space, as is known from the literature
\cite{DF88}.

Our point here is to describe in more detail the \emph{genealogy} of
these clumps which turn out to be \emph{single ancestors clumps} as
$t\to \infty$. Indeed a key point is the fact that all $t$-tops of the
states decompose at time $t$ in \emph{independent identically
  distributed elements of $\U^\Z$} corresponding to \emph{marked
  depth-$t$ subfamilies}. We can study all these independent
subfamilies \emph{separately} and then concatenate to the full state.
In fact we can decompose into single ancestor independent subfamilies.
This is a consequence of the branching property; cf.\
\cite{infdiv,ggr_GeneralBranching}. We next exploit this in detail.

\medskip
\noindent
\textbf{{\emph{(1) The scaling}}} Note that the state at time $t$ is
the independent concatenation of the processes starting with the mass
at one site, \cite{infdiv}. These processes become extinct and survive
up to time $t$ only with a probability of order $t^{-1}$. Hence the
sites where the time-$0$ population has descendants at time $t$
becomes thinner and thinner. In volume of order $t$ we can expect a
Poisson number of such sites in the limit $t\to\infty$. But of course
the growing clump around such a surviving mass has also a geographic
structure which is of interest. To see both these aspects we need a
two-scale analysis with a coarser first scale to describe the origin
of the time-$t$ population and second scale describing the internal
structure of such a surviving family at time $t$.

Therefore we observe that with a \emph{first scale} we can describe
the surviving founding fathers if we consider the sites which have
surviving mass somewhere at time $t$ and let $t \to \infty$. Namely we
get a point process on $\Z$, denoted $\mfP_t$ which has the property
that if we scale space by $t^{-1}$ getting what we call here
$\wt \mfP_t$ then a simple limit theorem holds:
\begin{align}
  \label{e2704}
  \mcL[\wt \mfP_t] \xRightarrow{t \to \infty} \mcL [\wt \mfP_\infty].
\end{align}
Here the r.h.s.\ is a Poisson point process on $\R$ with intensity
measure $\theta \lambda$, where $\lambda$ is the Lebesgue measure on
$\R$ \cite{DF88}. Here $\theta$ is $E[\nu_0(U \times \{0\})]$, the
initial intensity of individuals. In order to see more details of the
genealogy in the asymptotic analysis we need a \emph{two} (space-time)
scales approach to see the finer structure.

Therefore we come now to the \emph{second scale}. With each point in
$\wt \mfP_\infty$ we can associate a $\R$-marked ultrametric measure
space which describes the genealogy of the \emph{clump} consisting of
the individuals descending from those individuals initially at this
point. We know from section~\ref{ss.longuval} that this clump
asymptotically is associated with the surviving founding father
corresponding to that point since only one $2t$-ball has almost all
mass. More precisely we note that we can consider for each point in
$\Z$ the $\U^\Z$-valued Feller diffusion associated at time $t$ with
the \emph{$t$-top} of the population \emph{initially} in a point
$z \in \Z$. This defines a $\Z$-marked ultrametric measure space at
time $t$, denoted
\begin{align}
  \label{e2712}
  \left(\mfU_t^z \right)_{t \ge 0}.
\end{align}
Formally this is the process from Theorem~\ref{T.SRWALK} starting as
entrance law with mass $0$ in the point $z$ evolving for times $t >0$
as given by the system in \eqref{e937}, \eqref{e938}.

Here the point is now that we want to know the marked genealogy of a
\emph{typical} individual drawn at random from the population in
$[-n,n]$ and then we let $n \to \infty$. If we start with a
translation invariant state then if we pick a typical individual and
look at the system from the point of view of this individual we look
at the system under the \emph{Palm measure} (we typically pick from
families with large population in the ball of reference, note here
that the different surviving families segregate asymptotically in the
sense that $(1-\ve)$ of their mass is in part of space at that point).
Therefore we should look at the clumps under the \emph{size-biased
  law} in view of the scaling result in \eqref{e2704}.

Now condition on \emph{survival forever} of this process or
alternatively \emph{size bias} by the total population size to get
processes
\begin{align}
  \label{e2717}
  \bigl(\mfU_t^{z,\dagger}\bigr)_{t \ge 0} \text{ resp.\ }
  \bigl(\mfU_t^{z,\mathrm{Palm}}\bigr)_{t \ge 0}.
\end{align}
Then scaling as in \eqref{e1538} we get processes
\begin{align}
  \label{e2721}
  \bigl(\breve \mfU_t^{z,\dagger}\bigr)_{t \ge 0}, \quad
  \bigl(\breve \mfU_t^{z,\mathrm{Palm}}\bigr)_{t \ge 0}.
\end{align}
We have proved (take the non-spatial process!) that we get limit
configurations for
\begin{align}
  \label{e2726}
  \pi_U \; \mfU^{z,\dagger_t} \, , \quad  \pi_U \; \mfU_t^{z,\mathrm{Palm}}
\end{align}
as $t \to \infty$ denoted
$\breve \mfU_\infty^{\dagger, \downarrow} = \breve
\mfU_\infty^{\mathrm{Palm}, \downarrow}$, where the $\downarrow$
indicates that we get the limit of the projection of the state in
$\U^Z$ on the genealogy i.e.\ on $\U$.

The question is now whether we get a limit if we consider in addition
the \emph{$\Z$-marked object} in $\U^\Z$ better viewed as $\R$-marked
to be able to scale. For that purpose we consider the \emph{scaling of
  the marks}, the \emph{masses at a site} and distances as above:
\begin{align}
  \label{e2732}
  x \to t^{-1/2} x, \;  x \in \Z, \quad \mu(\{i\}\times U) \to
  t^{-1/2} \mu(\{i\}\times U), \quad r \to t^{-1} r.
\end{align}
This gives for finite collections of marks (sites) as $t \to \infty$
(via a first and second moment calculation for the super random walk
which is standard based on Ito's formula) a tight
object on $\R$. If we want to view the scaled $\pi_U \mu$ as a measure
on $\R$, we have to compensate the growing number of points (by
$\sqrt{t}$) in a macroscopic set $ A \in \R$. Therefore we scale the
measure $\wt \mu_t$ at the r.h.s. above by an additional $1/\sqrt{t}$.
We obtain then the \emph{equivalence classes of $\R$-marked
  ultrametric measure spaces}:
\begin{align}
  \label{e2736}
  \bigl(\wt \mfU_t^{z,\dagger}\bigr)_{t \ge 0},
\end{align}
which represent the \emph{time-space scaled} clump.

\medskip
\noindent
\textbf{\emph{(2) Limiting object}} In the first scale the limiting
object in \eqref{e2704} is of simple structure whereas it is more
complicated in the second scale.

The first basic ingredient of the limiting process of
$(\tilde\mfU_t^{z,\dagger})_{t \ge 0}$ is the $\U^\R$-\emph{valued
  Dawson-Watanabe} process (or superprocess). This object gives the
genealogy corresponding to the classical measure-valued
Dawson-Watanabe process in $\R^d$ for $d=1$. The latter is the
\emph{continuous-space-limit} of the super random walk we introduced
in Section~\ref{sss.genalspat}. The continuous space limit scales
time, space and mass. Namely we take $\varepsilon \Z^d$-super random
walks and let for a time rescaled version $\varepsilon \downarrow 0$
so that we get a limiting measure on $\R^d$.

To get the existence of the genealogy-valued continuum space time
limit we can define the $\U^\R$-valued superprocess rigorously as a
\emph{functional} of the \emph{historical} Dawson-Watanabe process,
introduced of in \cite{DP91}, for which we have to show that it is a
Markov process.

Recall that the historical process associates with a branching
population of migrating individuals a measure on \cadlag{} path. Think
of Galton-Watson random walk and consider for every individual alive
at the present time, say $t$, its \emph{path of descent} through
space, a path following the individuals location backward then that of
the father etc. Then take the counting measure on these path which are
for convenience continued constant before time $0$ and after time $t$
(the present time). This defines a
$\mcM(D((-\infty,+\infty),G)$-valued process. In the diffusive scaling
of time and space of \emph{many individuals of small mass and rapid
  branching} the historical Dawson-Watanabe process $\mcH$ arises; for
a characterization as scaling limit or via martingale problem see
\cite{DP91} or \cite{D93}.

Let $\mcH=(\mcH_t)_{t\ge 0}$ be the \emph{historical}
\emph{Dawson-Watanabe process} and define the process
$\mfU^{\mathrm{hDW}} = (\mfU^{\mathrm{hDW}}_t)_{t\ge 0}$ as a
functional of $\mcH$ as follows. We define the set $U_t$ of
"individuals" as the set of paths in the support of $\mathcal{H}_t$
and we let $r_t(\iota,\iota')$, $\iota,\iota' \in U_t$ be the value
$2(t-T)$ where $T$ is the maximal time with the property that the
paths $\iota$ and $\iota'$ agree for all $s \le T$. If no such
$T$ exists we set $r_t(\iota,\iota') =2t$. The mark of $\iota \in U_t$
is the value of the path at time $t$, i.e.\ $\kappa_t(\iota)=\iota_t$.
For the measure we take $\mu_t = \mcH_t$. This defines an element
\begin{align}
  \label{eq:HDW}
  \mfU_t^{\mathrm{hDW}} = [U_t,r_t,\kappa_t,\mu_t] \in \U^\R.
\end{align}
This functional gives again a process on the state space $\U^\R$,
which turns out to be \emph{Markov} and which we refer to as the
$\U^{\R}$-\emph{valued Dawson-Watanabe process}.

Note that this construction does not work for super random walk
because $t-T$ is \emph{not} the genealogical distance of the
corresponding individuals.

The limiting process in \eqref{e2736} is then the \emph{spatial
  version $\mfU^{\ast, DW(\R)}$ of the $\U$-valued Evans process}, we
described in the non-spatial case in \eqref{e1737} based on
Theorem~\ref{TH.1061}, namely the \emph{$\U^\R$-valued Evans process}
rigorously constructed based on $\mfU^{\mathrm{hDW}}$ starting with
some mark $z \in \R$, which we denote by
\begin{align}
  \label{e4400}
  \mfU^{z,*,\mathrm{DW}(\R)}=(\mfU^{z,*,\mathrm{DW}(\R)}_t)_{t\ge 0}.
\end{align}
Namely in this process an immortal particle with $0$-mass is
performing Brownian motion on $\R$ and throwing off at rate $b$ an
$\U^\R$-valued Feller diffusion (defined above) starting from $0$ mass
at the current position of the immortal particle. Then observing the
concatenated states of all surviving immigrant families at the present
time $T$ gives the time $T$ state of $\mfU^{\ast, DW(\R)}$. The state
has a population consisting of a countable number of immigrant
families each of which has a compact support on their geographic
positions.

The state can also be represented by a backbone construction as
follows. First fix $T >0$ and consider a Brownian path $B^\ast$ on
$\R$ starting at time $0$ in the point $z$, furthermore at rate
$(T-s)^{-1}$ at time $s$ a further process splits off in $B^\ast(s)$,
which evolves independently beyond time $s$ up to time $T$. Namely at
time $s$ start the $\U^\R$-valued process, the $\U^\R$-valued Feller
process conditioned on survival till time $T$, which is the
conditioned version of the $\U^\R$-valued Dawson-Watanabe process
(conditioned to survive till time $T$ and starting in $B^\ast(s)$).
These processes are denoted by $(\mfU^{T,s}_r)_{r \in [s,t]}$ with
$0 \leq s < t< T$ are the continuum space versions of the
$\U^\Z$-valued super random walk on $\Z$ conditioned to survive till
time $T$ and marked in addition to the position by the \emph{color
  $s$}. Then we concatenate all these elements of $\U^\R$, as below
\eqref{e1009}, to obtain $\mfU_t^{T,\sqcup}$. Then modifying the
argument of Section~\ref{sss.backboncon} let $t \uparrow T$ to get
\begin{align}
  \label{e2745}
  \mfU^{z,T, \sqcup}_T \in \U^\R.
\end{align}

\medskip
\noindent
\textbf{\emph{(3) Super random walk on $\Z$: Asymptotic clumps
    genealogy}} Now we can obtain the asymptotic clump description in
the second finer scale and show that the Evans process and the
backbone construction coincide and give asymptotic cluster as follows.
For the following result recall \eqref{e2736}, \eqref{e4400},
\eqref{e2745}.

\begin{theorem}[Asymptotic clump genealogy of super random walk on
  $\Z$]
  \label{T.CLUMP}
  % \leavevmode\\
  We have
  \begin{align}
    \label{e2751}
    \mcL[\wt \mfU_T^{z,\dagger}] \Rightarrow \mcL[\mfU_1^{z,1,\sqcup}] =
    \mcL [\mfU_1^{z,*,\mathrm{DW}(\R)}] \;
    \text{ as  }  \; T \to \infty.
  \end{align}
\end{theorem}

In fact we can obtain here even a result on processes. Recall that
then we have a given time-horizon $T$ where we observe our system and
we want to represent the observed clump at this time $T$ with the help
of a Markov process in, say $s$ which evolves through the time
interval $[0,T]$. Then we let their time-horizon go to infinity. This
representing process will be the process $\mfU^{z,*,\mathrm{DW}(\R)}$
at a specific time $s=1$. That means that the system at time $T'>T$
has its own representation. However since the dynamics of the process
of representation does not depend in its dynamics on $T$ we can read
off the limits $T'$ and $T$ from the same process. In other words:
$(\wt\mfU_{tT}^\dagger)_{t\ge 0}$ converges in law to
$(\mfU^{z,*,\mathrm{DW}(\R)}_t)_{t\ge 0}$.

This way we have the \emph{asymptotic} description of the
\emph{genealogy of a surviving clump} as $\U^\Z$-valued object, by the
limiting $\U^\R$-\emph{valued Evan's process}, with a scale-$t$
genealogy marked with locations on $\sqrt{t}$-spatial scale, which
appears in the first scale $t$ only as object
$\overset{\approx}{\mfP}_\infty$ marked with one point. More precisely
in a time scale $t$ we have a Poisson point process and on these
points a $\U^\R$-valued Evans process starts from that point and has
at macroscopic time $s$ (i.e.\ $st$ in real time) of the property that
the population is supported on a set of the form $A_s \sqrt{t}$, where
$A_s$ is a random compact set in $\R$ marking the genealogy of the
$\U$-valued Evans process. This compactness of $A_s$ follows from the
corresponding compactness property of the Dawson-Watanabe process; see
\cite{D93}. Here we observe that from the immortal line a countable
number of Dawson-Watanabe processes split off at a sequence of points
$(s_n)_{n \in N}$, $s_n \uparrow t$ as $n\to\infty$. Each has a
compact support but we need that these supports shrink fast enough wo
still be contained in a bounded set. Note that the total mass produced
is bounded since its expectation is $bt$. We note that the claim
amounts the claim that the \emph{historical process} corresponding to
the Evans process charges a compact set at fixed times $t$. This has
been shown in \cite{DP91} in Theorem~8.10 under certain assumptions
applying to the Brownian motion case with immigration. This follows
from the property that this process is the Kallenberg tree, which is
the subfamily of an randomly chosen individual, which is then the new
origin in space, whose law is absolutely continuous w.r.t.\ the
original law.

\begin{remark}
  \label{r.4650}
  Note that in $d=2$ we have a Poisson field of ancestors with
  descendants at time $t$, which are now spread in
  scale-$\sqrt{t}$-distance so that we are in the range, where the
  different families can hit all macroscopic balls. In $d \geq 3$ we
  start having a diverging number of ancestors even on the scale
  $\sqrt{t}$ and we get countably many ancestors to contribute locally
  to the population. The key effect is the divergence of the rate of
  individuals creating at time $t$ a surviving form of the form
  $(t-s)^{-1}$. This subfamilies produce the local peaks in the
  population distribution. Every dimension has its own flavor here.
  The analysis would need the $\U^{\R^d}$-valued Dawson-Watanabe
  process, which we cannot construct in this paper, see
  Remark~\ref{r.DW}.
\end{remark}

\subsubsection{The fossil process of Feller diffusion and the
  continuum random tree (CRT)}
\label{sss.contrand}

The reader might wonder how all our results are related to the by now
classical theory of the \emph{continuum random tree}, short CRT. To
make this connection we formulate rigorously two results, however the
proof is kept quite short as this is not our main message. The
genealogy of all individuals \emph{ever} alive which is really a tree
(in the sense of an $\R$-tree) has been described for the genealogy of
the Feller diffusion by the so called \emph{continuum random tree}
(CRT) introduced originally by Aldous in \cite{Aldous90} and extended
by Le Gall in \cite{LeGall93}. This object is of course a random
variable and \emph{not} a stochastic process and it is the latter for
which the description of genealogies we propose here is fruitful.

Our goal is not to develop this theory of these special $\M$-valued
processes in detail, but rather to provide the link between the
``classical'' object CRT and the approach we develop here to obtain
the time evolution of genealogies, as stochastic processes therefore
not all details are provided here.

We will now indicate how the continuum random tree builds up from an
\emph{evolving process}, called $\mfU^{\mathrm{foss}}$ of \emph{random
  weighted $\R$-trees}, describing the population alive at \emph{some}
time before or at the current time $t \ge 0$ and which includes also
all the \emph{fossils} before time $t$ and the individuals alive at
time $t$ as the force actually driving the evolution. This process is
then considered as $t \rightarrow \infty$. This is developed in
\cite{GSW} for interacting Fleming-Viot processes (interacting via
migration) on countable groups and on the spatial continuum $\R$.

\paragraph{State space of fossil $\M$-valued process}
In order to include \emph{fossils} up to the present time $t$ we think
of every individual $\iota$ alive at a time $s \in [0,\infty)$ as the
new basic entity which is characterized by a pair $(s,\iota)$ for
$\iota$ an element of the set $U_s$ describing the population alive at
time $s$. As sampling measure we then take the \emph{occupation
  measure} and as genealogical distance still \emph{twice the time
  back to a common ancestor}. This we have to formalize now.

We denote by $\M$ the space of equivalence classes of \emph{metric
  measure spaces} and with $\M^\rho$ the ones with a \emph{root}. (The
root is under equivalence classes mapped in the root.) The space is
equipped with the Gromov-weak topology under which it is a Polish
space (\cite{GPW09}).

For our purpose here we introduce as state space
$\M^{\rho,+} \subseteq \M^\rho$ a class of special \emph{rooted}
metric measure spaces (for the latter see \cite{GPW09}) which
describes the fossils as well, by replacing the ultrametric measure
spaces $(U_s,r_s,\mu_s)$ we had so far by an object of the form
\begin{align}
  \label{ag2}
  (M_T, \wt r_T, \wt \mu_T),
\end{align}
with the following ingredients. First let $T \ge 0$ denotes the
present time and for $s \in [0,T]$ consider
$\mfU_s = [U_s,r_s,\mu_s] \in \U$ and \emph{assume} that we can fix
representatives of the equivalence classes for every $s$ such that
$s \mapsto \mu_s$ is measurable. We can then define
\begin{align}
  \label{ag3}
  M_T = \{(s,\iota) : s \in [0,T], \iota \in U_s\} \cup \{\rho\}, \;\;
\;\; \wt r_T \text{ metric on } M_T, \;\;
  \wt \mu_T = \int^T_0 \mu_s \dx s + \mu_T^{\mathrm{top}},
\end{align}
where $\rho$ is the root and $\mu_T^{\mathrm{top}}$ is a measure fully
supported on $U_T$. The \emph{distance} between individuals in $U_s$
and the root is given by $s$. To get a Polish state space we pass to a
\emph{stronger} topology and a concept of equivalence under which the
structure in \eqref{ag3} is preserved, this means that the path of
measures $(\mu_s)_{s \in \R}$ is preserved (recall due to the root the
$U_s$ in distance $s$ from the root are preserved).

We then need to show that this structure constitutes then again a
Polish space. We have here a subset of the space $\M^\rho$, which is
known to be Polish. Hence we need to argue that we have a closed
subset.

Since this state space is the topic of work specifically on the
ancestral web on the one hand and on the other hand the fossil process
(see \cite{GKW}, \cite{GSWfoss}), \emph{we only sketched here the main
  idea}.

\begin{remark}[State space]
  For \label{r.4837} that purpose we consider the measure
  $\wt \mu_s^{\mathrm{top}}$ by extension of $\mu_s^{\mathrm{top}}$ to
  $M$ and form expressions
  $\suml_{k \in \{1,2\}} \alpha_k \delta_{\mu^k}$, $\alpha_k \ge 0$.
  This so called bi-measure metric space then combines the structure
  of the
  $\{[M,r_s,\wt \mu_s^{\mathrm{top}}],[M,r_s,\wt \mu_T]\}$. If
  we use the additional property that $t \mapsto \mu_t$ is continuous,
  we can use here two-level metric measure spaces to obtain a Polish
  space $\M^{\rho,+}$ which contains the needed information. For the
  concept of bi-measure metric spaces see \cite{Mei18}.
\end{remark}

\paragraph{Martingale problem of fossil tree-valued Feller process}
We will define a new dynamic such that the restriction of that process
$(\wt \mfU_t^{\rm foss})_{t \ge 0}$ to the \emph{time-$s$ slice}
$[U_s,r_s,\mu_s]$ of the state denoted $\mfU^{\mathrm{foss}}_t$ with
$s \in [0,t]$ gives a \emph{version of our $\U$-valued Feller
  diffusion} for time $[0,t]$. We denote this projection by $\pi_s$.
This new process we can again describe as the solution to a well-posed
martingale problem, where an \emph{additional operator} describes how
current individuals turn into fossils. Also the sampling measure $\mu$
is split into the \emph{top part} supported on $U_T$ the individuals
currently alive at time $T$ and the remaining part on $[0,T)$,
\emph{the fossils}.

In order to introduce the dynamic via a \emph{martingale problem} on
$\M^{\rho,+}$ rigorously next, i.e., we need test functions and
operator as ingredients. We introduce now again \emph{polynomials} on
the state space as follows. We write
\begin{align}
  \label{e1178}
  \wt \mu = \mu^{\mathrm{foss}} + \mu^{\mathrm{top}}
\end{align}
and set for $\mfM \in \M^{\rho, +}$:
\begin{align}
  \label{e1179}
  \begin{split}
    \Phi (\mfM) = \int_{(M)^{n+m}}
    & \varphi \big((r(u_i,u_j))_{1 \le i < j \le n+m}\big) \\
    & (\mu^{\mathrm{foss}})^{\otimes n} \big(\dx(u_1, \ldots, u_n)\big)
    (\mu^{\mathrm{top}})^{\otimes m} \big(\dx(u_{n+1}, \ldots, u_{n+m})\big),
  \end{split}
\end{align}
where $\varphi \in C^1_b \big((\R_+)^{\binom{n+m}{2}}, \R\big)$. We
denote the polynomials with non-negative $\varphi$ by $\Pi^1_+$.

On these polynomials we now define the \emph{generator}. We extend the
$\Omega^{\uparrow, \mathrm{grow}}$, $\Omega^{\uparrow, \mathrm{bran}}$
now by letting it \emph{act only on the top} and in addition we need
the operator describing the \emph{creation of the new top} by time
passing. Formally this looks as follows.

The \emph{aging} is now described by a weighted gradient. Let
$I=\{1,\ldots, n\}$, $J=\{n+1,\ldots, n+m\}$. Then define the weighted
gradient as follows:
\begin{align}
  \label{e1181}
  \frac{\partial^w}{\partial \; \underline{\underline{r}}}= 2 \sum_{k
  \in J,l \in J} \; \frac{\partial}{\partial \; r_{k,l}} + \sum_{k
  \in I,l \in J} \; \frac{\partial}{\partial \; r_{k,l}}.
\end{align}
This takes into account that fossils do not age, but only the top.

Introduce the map $s_k$ acting on $(\mu^{\mathrm{foss}})^{\otimes n}$
as follows:
\begin{align}
  \label{e1180}
  \mu^{\mathrm{foss}} \otimes \dots \otimes \mu^{\mathrm{foss}} \longrightarrow
  \mu^{\mathrm{foss}} \otimes \dots \otimes \mu^{\mathrm{top}} \otimes
  \mu^{\mathrm{foss}} \dots \otimes \mu^{\mathrm{foss}},
\end{align}
with $\mu^{\mathrm{top}}$ replacing $\mu^{\mathrm{foss}}$ at the
$k$-th position for $k \in \{1,2,\ldots, n \}$. Then
$\Omega^{\uparrow, \mathrm{foss}}$ is defined on $\Pi^1$ by
\begin{align}
  \label{e1182}
  \begin{split}
    \Omega^{\uparrow, \mathrm{foss}} \Phi (\mfM) & = b \int_{M^{n+m}}
    \; \dx\bigl( (\mu^{\mathrm{foss}})^{\otimes n} \otimes
    (\mu^{\mathrm{top}})^{\otimes m} \bigr)
    \; \sum_{k,l=n+1}^{n+m} \; (\wh \theta_{k,l} \;  \varphi-\varphi)
    \quad \text{(branching)}\\
    & \quad + \sum^n_{k=1} \; \int_{M^{n+m}} \;
    \dx \bigl(s_k(\mu^{\mathrm{foss}})^{\otimes n}
    \otimes (\mu^{\mathrm{top}})^{\otimes m}\bigr) \, \varphi
    \quad \text{(top layer growth)}\\
    & \quad + \int_{M^{n+m}} \; \dx \bigl((\mu^{\mathrm{foss}})^{\otimes n}
    \otimes (\mu^{\mathrm{top}})^{\otimes m}\bigr) \,
    \big(\frac{\partial^w}{\partial \underline{\underline{r}}}\;
    \varphi \big) \quad \text{(distance growth).}
  \end{split}
\end{align}
This defines now all ingredients for a martingale problem on
$\M^{\rho,+}$.

\paragraph{Results}
We now have a well-defined process
$(\mfU^{\mathrm{foss}}_t)_{t \ge 0}$. Namely
\begin{theorem}[Well-posed fossil martingale problem]\label{T.AGING}
  % \leavevmode \\
  The
  $(\delta_\mfu, \Pi^1_+, \Omega^{\uparrow,
    \mathrm{foss}})$-martingale problem is well-posed for every
  $\mfu \in \M^{\rho,+}$. The resulting process
  $(\mfU_t^{\mathrm{foss}})_{t \ge 0}$ with values in $\M^{\rho,+}$
  is called the fossil $\M^\rho$-valued fossil Feller diffusion.
\end{theorem}
We want to relate this process we defined as a stochastic process
above with the \emph{continuum random tree} shortly CRT. We focus on
one surviving family which is the basic component. We obtain the
fossil process $\mfU^{\rm foss}$ for the $\U$-valued Feller diffusion
starting in the mass $\bar \mfu$ in the element with
$(\bar \mfu,\hat \mfu)$ and
$\hat \mfu=\mfe=[1,\underline{\underline{0}},\delta_1]$.

The CRT arises as scaling limit in various settings among which is an
object in the $\U$-valued Galton-Watson critical branching process is
essentially embedded, a result going back to Aldous \cite{Ald1991a}.
It can be obtained from an \emph{explicit construction} from the paths
of an excursion of Brownian motion over the interval $[0,1]$; see
\cite{LeGall93}. First a \emph{metric space} is constructed from the
excursion of (better 2 $\cdot$ excursion) standard Brownian motion,
which is often called the CRT. This can be extended to a metric
\emph{measure} space which then allows to consider the equivalence
class, which would give an element in $\M$. We choose the measure
induced by the local time on a level; see \cite{PWak13}. Namely we
want the measure to be such that the process of the mass of the
population in distance $\leq 2t$ from the root has as function of $t$
the form $t \mapsto \int^t_0 \bar\mfU_s\, \dx s$. This provides a
random object of the form in \eqref{ag2}. We denote by
$[\textrm{CRT}] \in \M$ the equivalence class of the metric finite
measure space constructed from the Brownian excursions.

From the process $\mfU^{\mathrm{foss}}$ the CRT arises as the
$\M$-valued Kolmogorov-Yaglom limit as $t \rightarrow \infty$ limit.
We get the following:
\begin{theorem}[CRT as Kolmogorov-Yaglom limit of fossil process]
  % \leavevmode\\
  The \label{T.CRT} fossil process $(\mfU^{\mathrm{foss}}_t)_{t \ge 0}$
  has the properties,
  \begin{align}
    \label{ag3b}
    \mcL [\pi_{s} \circ \mfU^{\mathrm{foss}}_t] = \mcL[\mfU_s],
    \quad \forall \; s \le t,
  \end{align}
  and with initial state as described above:
  \begin{align}
    \label{ag3c}
    \mcL [\mfU^{\mathrm{foss}}_t] \xRightarrow{t\to\infty} \mcL
    \big[[\textnormal{CRT}]\big].
  \end{align}
\end{theorem}

The strength of the description by evolving ultrametric measure spaces
is to be able to handle the dynamical aspects, while the strength of
the CRT-embedding in a Brownian motion is the analysis of the static
aspects of the \emph{final full genealogical tree} up to extinction.
The analysis of the process based on the CRT adapted to the spatial
version of the Feller process, as super random walk or the
Dawson-Watanabe process is less easy to handle (recall here the
concept of the \emph{Brownian snake} for the latter see \cite{LG99}).

\section[Proofs of Theorems~\ref{T:DUALITY},~\ref{PROP.1007}]{Proofs
  of Theorems~\ref{T:DUALITY} and \ref{PROP.1007}: Dualities}
\label{sec:duality}
In this section we collect in Section~\ref{s.duality} all the
arguments needed to establish the duality relations we claimed in
Section~\ref{sec:results} and we extend in Section~\ref{ss.condualy}
these duality to the processes $\mfU^\dagger, \mfU^{\mathrm{Palm}}$
and $\mfU^\ast$.

\subsection{Feynman-Kac duality and conditioned duality for
  \texorpdfstring{$\U$}{U}-valued Feller}
\label{s.duality}

For a rigorous proof of the FK-duality we need now a more formal
definition for the dual dynamic, namely the dual is characterized via
a martingale problem. For a function $G: \bbK \to \R$ depending on
finitely many coordinates only. Define
\begin{align}
  \label{tv32}
  L^{\downarrow,\mathrm{coal}} G(p,\dr')
  & = b \sum_{\pi, \pi' \in p} \left(G(\kappa_p(\pi,\pi'),\dr') -
    G(p,\dr') \right),\\
  \label{tv32a}
  L^{\downarrow,\mathrm{grow}} G(p,\dr')
  & = \sum_{i \nsim_p j} \frac{\partial}{\partial r_{ij}'} G(p,\dr')
\end{align}
for $p\in \bbS$ and $\dr' \in [0,\infty)^{\binom{\N}{2}}$ and its sum
\begin{align}
  \label{tv33}
  L^{\downarrow,\mathrm{K}} =
  L^{\downarrow,\mathrm{grow}} +
  L^{\downarrow,\mathrm{coal}}.
\end{align}
Define the sets of test functions
\begin{align}
  \label{eq:testfunc:down}
  \mcG^{\downarrow} = \{ H^{\varphi}(\mfu, \cdot):\ \mfu \in \bbU, \varphi
  \in C_b^1(\R^{\binom{\N}{2}}) \text{ dep.\ on finitely many coord.}\}.
\end{align}

\begin{lemma}
  \label{lem:K:ex}
  Let $\mu \in \mcM_1(\bbK)$. The enriched Kingman coalescent (see
  also page 809 of \cite{GPWmp13}) is a solution of the
  $(\mu,
  L^{\downarrow,\mathrm{K}}, \cdot, \mcG^{\downarrow})$
  martingale problem.
\end{lemma}

\begin{proof}%[Proof of Lemma~\ref{lem:K:ex}]
  \label{pr.lem:k:ex}
  In \cite{GPWmp13} the enriched Kingman coalescent is defined as
  the solution of the $(\delta_k,\Omega^\downarrow,\mcG^{1,0})$
  martingale problem, where
  $\Omega^\downarrow =
  L^{\downarrow,\mathrm{K}}$ and
  \begin{align}
    \label{tv34}
    \mcG^{1,0} = \{ G \in b\mcB(\bbK): G(\cdot,\dr') \in C(\bbS) \,
    \forall \, \dr', \, \sum_{i\nsim_p j} \frac{\partial}{\partial
    r_{ij}'} G(p,\dr') \text{ exists}\}.
  \end{align}
  One may easily check that $\mcG^\downarrow \subset \mcG^{1,0}$,
  since the former elements only depend on finitely many coordinates.
  Thus, it is clear that any solution to the
  $(\delta_k,\Omega^\downarrow,\mcG^{1,0})$ martingale problem is also
  a solution to the
  $(\delta_k,
  L^{\downarrow,\mathrm{K}}, \cdot, \mcG^{\downarrow})$
  martingale problem.
\end{proof}

Analogously to \eqref{eq:testfunc:down} we set
\begin{align}
  \label{tv35}
  \mcH & = \{ H^{\varphi}(\cdot, \cdot):\ \varphi \in
         C_b^1([0,\infty)^{\binom{\N}{2}}) \text{ dep.\ on finitely many coord.}\}
         \quad \text{ and} \\
  \label{tv35b}
  \mcG^{\uparrow}
       & = \{ H^{\varphi}(\cdot, (p,\dr')):\ (p,\dr') \in \bbK, \varphi \in
         C_b^1([0,\infty)^{\binom{\N}{2}}) \text{ dep.\ on finitely many coord.}\}.
\end{align}

\begin{proof}[Proof of Theorem~\ref{T:DUALITY}]
  \label{pr.T:DUALITY}
  We know that both processes exist by Proposition~\ref{cor:ex} and
  Lemma~\ref{lem:K:ex}. We follow the proof of Proposition~4.1 in
  \cite{GPWmp13} and use Corollary~4.13 of Chapter~4 in \cite{EK86}.
  We need to show (4.52) of Theorem~4.11 of Chapter~4 in \cite{EK86},
  that is for $H=H^{\varphi}(\cdot,(p,\dr')) \in \mcG^{\uparrow}$ we
  find $G:\U \to \R$ such that
  \begin{align}
    \label{eq:4.52}
    H(\mfU_t)- H(\mfU_0) - \int_0^tG(\mfU_s)\, \dx s
  \end{align}
  is a martingale. Additionally, such $G$ has to satisfy
  \begin{align}
    \label{eq:4.42}
    G(\mfu) =
  L^{\downarrow,\mathrm{K}} H^{\varphi}(\mfu, \cdot) (p,\dr') +
    b \binom{|p|}{2} H^{\varphi}(\mfu, (p,\dr'))
  \end{align}
  and finally the assumptions (4.50) and (4.51) of Theorem~4.11 of
  Chapter~4 in \cite{EK86} need to hold. The latter two claims hold,
  since $|p_t|$ is decreasing and so $\binom{|p_t|}{2}$ is and we have
  moment bounds on $\bar{\mfU}_t$ as in Lemma~\ref{lem:Feller}.

  First, let us show \eqref{eq:4.52}. Consider
  $H=H^\varphi(\cdot,(p,\dr')) \in \mcG^\uparrow$ for certain fixed
  $(p,\dr')\in \bbK$ with $p= (\pi_1, \dotsc,\pi_n)$. Consider the
  bijective permutation $\sigma: \N \to \N$, (only depending on $p$),
  \begin{align}
    \label{tv36}
    \sigma: \;
    \begin{cases}
      \sigma (\min \pi_i) = i, &  i =1,\dots,n \\
      \sigma|_{\N \setminus \{\min \pi_1, \dotsc, \min \pi_n\}} &
      \text{ increasing}
    \end{cases}
  \end{align}
  and set
  \begin{align}
    \label{tv37}
    \tilde{p} = \sigma_*(p).
  \end{align}
  That means
  $\sigma_*(p) = (\tilde{\pi}_1,\dots, \tilde{\pi}_n) \in \bbS$ is a
  partition with the same number of partition elements as $p$ and such
  that
  $i \in \tilde{\pi}_k : \Leftrightarrow \sigma^{-1}(i) \in \pi_k$. We
  also define for $\dr' \in [0,\infty)^{\binom{\N}{2}}$:
  \begin{align}
    \label{tv38}
    \left(\sigma_* (\dr')_{ij}\right)_{1\le i < j}
    = \left( r_{\sigma^{-1}(i), \sigma^{-1}(j)}' \right)_{1\le i < j}.
  \end{align}
  Then,
  \begin{align}
    \label{tv39}
    H^\varphi(\mfu,(p,\dr')) = \int \mu^{\otimes n}(\dx u_1, \dots,
    \dx u_n)\,
    \varphi ( \dr^{\sigma_* p} (u_1,\dotsc, u_n) + \sigma_* \dr').
  \end{align}
  In particular $H$ can be written as a polynomial of order $n = |p|$
  and with $\varphi$ manipulated as in the previous line (remember
  that $p$ and $\dr'$ are fixed). Actually $H^\varphi(\cdot,(p,\dr'))$
  is in the domain $\Pi(\mathcal C_b^1)$ of $\Omega^{\uparrow}$ and by
  Proposition~\ref{cor:ex} we know that
  \begin{align}
    \label{tv40}
    \bigl(H^\varphi (\mfU_t,(p,\dr')) - H^\varphi (\mfU_0,(p,\dr')) - \int_0^t
    \Omega^{\uparrow} H^\varphi (\mfU_s,(p,\dr'))\, \dx s\bigr)_{t \ge 0}
  \end{align}
  is a martingale. This shows \eqref{eq:4.52}.

  Next, show \eqref{eq:4.42} separately for both parts of the
  generator for $p= (p_1,\dotsc, p_n)$ and $\varphi$ depending on
  finitely many coordinates only:
  \begin{align}
    \label{tv41}
    \Omega^{\uparrow, \mathrm{grow}} H^\varphi (\mfu,(p,\dr'))
    & = 2 \sum_{1\le i \le j \le n} \int \mu^{\otimes n}(\dx u_1,
      \dots, \dx u_n) \frac{\partial}{\partial r(u_i,u_j)} \varphi \left(
      \dr^{\sigma_* p} (u_1,\dotsc, u_n) + \sigma_* \dr' \right) \\
    & =  2 \sum_{1\le k < l, k\nsim_p l} \int \mu^{\otimes n}(\dx
      \underline{u}_p) \frac{\partial}{\partial r_{kl}} \varphi ( \dr^p (
      \underline{u}_p) +  \dr') \label{tv41b}\\
    & = 2 \sum_{k \nsim_p l} \int \mu^{\otimes |p|}(\dx \underline{u}_p)
      \frac{\partial}{\partial r_{kl}'} \varphi \left(
      \dr^p(\underline{u}_p) +  \dr' \right) \label{tv41c}\\
    & =
  L^{\downarrow,\mathrm{grow}} H^{\varphi} (\mfu,\cdot) (p,\dr').\label{tv41d}
  \end{align}
  Additionally, using (including a formal addition)
  \begin{align}
    \label{tv42}
    \hat{\theta}_{k,l} (x_1,\dotsc, x_n) = (y_1,\dotsc,y_n),\ \text{ with }
    y_i = \ind{i\neq l} x_i + \ind{i=l} x_k,
  \end{align}
  we get
  \begin{align}
    \label{tv43}
    & \Omega^{\uparrow, \mathrm{bran}} H^\varphi (\mfu,(p,\dr')) =
      \frac{2b}{\bar{\mfu}} \sum_{1\le k < l \le n} \int \mu^{\otimes
      n}(\dx \underline{u}) \, \varphi\circ \theta_{k,l} \bigl(
      \dr^{\sigma_* p}(u_1,\dotsc,u_n) + \sigma_* \dr' \bigr) \\
    \label{tv43b}
    & \quad =\frac{2b}{\bar{\mfu}} \sum_{1\le k < l \le n} \int
      \mu^{\otimes n}(\dx \underline{u}) \, \varphi \bigl(  \dr^{\sigma_*
      p}(\hat{\theta}_{k,l}(u_1,\dotsc,u_n)) + \sigma_* \dr' \bigr) \\
    \label{tv43c}
    &\quad = b \sum_{\pi \neq \pi' \in p} \int \mu^{\otimes (|p| -1)}(\dx
      \underline{u}_{\kappa_p(\pi,\pi')})\, \varphi\bigl(
      \dr^{\kappa_p(\pi,\pi')}(\underline{u}_{\kappa_p(\pi,\pi')}) +
      \dr' \bigr) \\
    \label{tv43d}
    & \quad  =  b \sum_{\pi \neq \pi' \in p}
      \bigl(\int \mu^{\otimes (|p| -1)}(\dx u_{\kappa_p(\pi,\pi')}) \varphi
      \bigl(\dr^{\kappa_p(\pi,\pi')}(u_{\kappa_p(\pi,\pi')}) + \dr'\bigr)
      - H^{|p|, \varphi}(\mfu,  (p,\dr')) \bigr) \\
    \label{tv43e}
    & \qquad \qquad \qquad + b \binom{|p|}{2}
      H^{\varphi}\bigl(\mfu,(p,\dr')\bigr) \\
    \label{tv43f}
    & \quad =
  L^{\downarrow,\mathrm{coal}} H^{\varphi}(\mfu,\cdot) (p,\dr')
      + b \binom{|p|}{2} H^{\varphi}\bigl(\mfu,(p,\dr')\bigr).
  \end{align}
  Now, we can apply Corollary~4.13 of Chapter~4 in \cite{EK86} to
  obtain the proposition.
\end{proof}

\begin{proof}[Proof of Theorem~\ref{PROP.1007}]
  The part (a) follows from the duality of the $\U_1$-valued
  Fleming-Viot process in the time-inhomogeneous case; see
  \cite{Gl12}. Then the part (b) follows from Corollary~\ref{prop.834}
  and part (a).
\end{proof}

\subsection{Conditioned duality and Feynman-Kac duality for
  related processes}
\label{ss.condualy}
% For branching processes we also have other interesting dualities
% namely because its total mass is an autonomous Markov process the
% \emph{conditional} duality where we condition on the complete (i.e.
% for all $t \in [0,\infty)$) total mass process.

The conditional duality techniques extend also to more general forms
of branching. Of particular interest for us are the $Q$-process and
the Palm process of the $\U$-valued Feller diffusion or $\U^V$-valued
branching diffusion with immigration. In this section we obtain the
conditioned duality respectively the Feynman-Kac duality for the
critical and non-critical $\U$-valued Feller diffusions, for processes
$\mfU^\dagger$, $\mfU^{\mathrm{Palm}}$, and for the $\U^V$-valued
Feller diffusion with immigration.

\subsubsection{Conditioned duality for \texorpdfstring{$Q$}{Q} and
  Palm process}
\label{sss.dualu}
We start with an observation concerning the conditioning on survival.
We see in particular that the process $(\wh \mfU_t)_{t \ge 0}$
conditioned on the total mass process is not affected by the
conditioning on survival which only changes the probability of such
path in the condition. In term of generators we observe the following.

Write $\cdot$ for $\mathrm{Palm}$, $\dagger$ or $\ast$. Observe that
the drift affects only the total mass process but \textit{not} the
mechanism of the conditioned (on the total mass) process
$(\wh \mfU_t^\cdot)_{t \ge 0}$. The component process
$\wh \mfU^\cdot$ of the process $\mfU^\cdot$ is only affected when
we integrate the law of the process conditioned (on the total mass
process) to get its full law. Therefore the conditioned dual is only
affected via the change of the dynamic of the underlying process
$\bar \mfU^\cdot$ on which we condition $\wh \mfU^\cdot$, which
affects the coalescence rate in the dual process.

Recall the formula \eqref{e1372} for the generator of
$\mfU^{\mathrm{Palm}}$, $\mfU^\dagger$ or $\mfU^*$ acting on
polynomials. Then we see that if we condition on the total mass
process we have as a conditional dual process a coalescent with rate
$b/\bar \mfu_t$ at time $s$ with $t=T-s$ where $T$ is the time horizon
of the duality where $\bar \mfu = (\bar\mfu_t)_{t \ge 0}$ is a
realization of the rate $b$ Feller diffusion with immigration at rate
$b$. Recall that $\bar\mfu_t>0$ for $t >0$ and that $\bar \mfu_0$ may
be zero.

We have to guarantee here that the process exists throughout up to the
potential singularity at $t=0$ i.e.\ $s=T$ in the backward time. Here
this is no problem since such a singularity can only occur at time
$s=T$ if the forward total mass diffusion does \emph{not} start with a
positive mass term. Therefore for positive initial mass the
conditioned duality holds again for $\mfU^{\mathrm{Palm}}$
$\mfU^\dagger$ and $\mfU^*$. In case of a zero we obtain the Kingman
coalescent for infinite time as state at $s=T$ giving the unit element
$\mfe$. For this we need that
$\int_0^\varepsilon \bar \mfu_t^{-1} \,\dx t = + \infty$ for
$\varepsilon >0$. This was shown in Proposition~0.2 in \cite{DG03}
even in the spatial context.

\begin{corollary}[Conditioned duality for $\mfU^{\mathrm{Palm}}$,
  $\mfU^\dagger$ and $\mfU^\ast$]\label{cor.3764}
  % \leavevmode \\
  The conditional duality from \eqref{e1015a} holds for
  $\mfU^{\mathrm{Palm}}$, $\mfU^\dagger$ and $\mfU^\ast$ and for their
  entrance laws from $0$.
\end{corollary}

Recall Corollary~\ref{cor.3640} giving the dual identification of
$\varrho^t_h$ for the $\U$-valued Feller diffusion. The
conditional duality from the above corollary gives a good idea about
the form of the states of the processes $\mfU^{\mathrm{Palm}}$,
$\mfU^\dagger$ and $\mfU^\ast$. One might hope indeed that this gives
us some information on $\varrho^t_h$ in the
\emph{\Levy{}-Khintchine representation} of the state at time $t$ as
in \eqref{e3732}. For the conditional laws we obtain this
$\varrho^t_h$ as a mixture over laws of coalescent trees, where we
can proceed as in the case of the Feller diffusion above just using
different total mass path now, namely the ones generated by the
diffusion $\dx \mfu_t=b\dx t + \sqrt{b \bar \mfu_t} \, \dx w_t$. Therefore we
also obtain here the measure $\varrho^t_h$ for the state at time
$t$ in terms of the coalescent as we did for the Feller diffusion,
only the \emph{mixing measure} i.e.\ the law of the total mass path is
now \emph{different}.

\begin{corollary}[conditioned $\varrho_h^t$]
  The representation of $\varrho_h^t$ via the dual of
  Corollary~\ref{cor.3764} holds for $\mfU^{\mathrm{Palm}}$,
  $\mfU^\dagger$ and $\mfU^\ast$.
\end{corollary}

\subsubsection{Conditioned duality: Feller diffusion with
  immigration}\label{sss.fellimm}
Consider first the total mass process. In the case of a constant
immigration at rate $\varrho > 0$ the total mass process is the
solution of
\begin{align}
  \label{e862}
  \dx Z_t=\varrho\, \dx t + \sqrt{bZ_t} \, \dx w_t.
\end{align}
For the Feynman-Kac duality we write
$\dx Z_t=\varrho(1-Z_t) \, \dx t + \varrho Z_t \, \dx t + \sqrt{bZ_t}
\, \dx w_t$ and obtain as a dual process for the total mass process
the coalescent that we describe next.

Add a site $\ast$ to the system in addition to the site $0$ where the
original process $(Z_t)_{t \ge 0}$ is located. On the site $*$ the
process has the constant state $\varrho > 0$, i.e., on $\ast$ all
rates of change are zero. Then the dual system is a spatial coalescent
which starts with $n$ individuals at site $0$ and all rates at $\ast$
are zero but a partition element jumps from $0$ to $\ast$ at rate
$\varrho$. This coalescent is denoted by $\mathfrak{C}$ and its
entrance law started with countably many individuals by
$\mathfrak{C}^\infty$. Then the two processes are again Feynman-Kac
dual with Feynman-Kac potential from \eqref{e.ggr1} with $a=\varrho$.
This can be combined with super- and sub-critical terms.

Consider the spatial model with $N$-colonies and uniform migration
mechanism and branching at each site. In other words for
$V=\{0,1,\dots,N-1\}$, equipped with addition modulo $N$ as the group
operation we consider the $\U^V$-valued super random walk on
$V=\{0,\dots,N-1\}$. Consider the system starting in an exchangeable
initial law and let $\varrho$ be the limit of the empirical mean over
the $N$ components. Then observe the system at a typical site, say
site $0$. If the initial state is i.i.d.\ this is $\varrho=E[Z_t (i)]$
which we assume to be finite. In the limit $N \to \infty$ we obtain
for the masses at a typical site the so called \emph{McKean-Vlasov
  limit} the equation above. What is the limiting dynamic for the
genealogies at rate $\varrho$? How to define the genealogies?

There are two possibilities of interest only one corresponds to the
duality suggested above. If we have equation \eqref{e862} for the
total mass from a spatial model with a site of observation and an
outside world with a source of ancestors unrelated to our population
immigrate at some constant rate, then we obtain for the total mass a
drift $\varrho \, \dx t$. Once the population has immigrated it
evolves as in the $\U$-valued Feller diffusion. Immigrants (at time
$t$) have distance $2t$ to the normal population with ancestor at site
$0$. This is close to the duality in the spatial model, precisely it
is the limit of the spatial dual.

\subsubsection[Conditioned duality: Feller diffusion with immigration
from immortal line]{Conditioned duality: Feller diffusion with
  immigration from immortal line
  $\mfU^{\ast,+}$}\label{sss.fellimortal}

Here we want to connect to the conditional dual of the $\U^V$-valued
process with immigration from the immortal line, where through the
marks more information is available and the condition is more complex,
since we have now for every color a total mass path
$\underline{\bar\mfu}=\{(\bar \mfu_s(\ell)),\ell_{0 \le s \le t} \in
\N\}$ with $\ell$ being the color. Then denoting by
$P^{\underline{\bar\mfu}}$ the law of this collection instead of a
simple path we get the same formula.

We know that ignoring the colors, i.e.\ observing only $\mfU^\ast$ we
have the same process and dual as for $\mfU^\dagger$ or
$\mfU^{\mathrm{Palm}}$, however once we have a certain color we have
a partition element with a fixed final element and time to all
coalesce. In other words the conditioning allows to represent the
different subfamilies for a given time of immigration of the
forefather. Therefore a single subfamily corresponds to a coalescent
which has to coalesce at a fixed time and with a coalescence rate
given by the inverse of the mass of the corresponding excursion.

First we need some ingredients, namely the colored Feller diffusion
with immigration where each color has a mass evolving as Feller
diffusion entrance law starting at time $s$, the color and surviving
till the time horizon $T$,
\begin{align}
  \label{e3864}
  \mathfrak{D}= \{(\bar\mfu_t(s))_{t \ge 0} : s \in S(T)\}, \quad
  S(T) \text{ the set of colors.}
\end{align}
The set $S(T)$ will be generated considering Evans branching diffusion
with immigration from an immortal line, namely the immigration times
leading to a diffusion equipped with that time as color surviving
till at time $T$.

Continue with the dual process. We consider individuals marked with
colors from $(0,T)$. The individuals may move to a cemetery. Instead
of the total mass path we consider now a \emph{point process on
  $[0,\infty)^2$} coding color and its mass at the current time
horizon $T$. Then we want to condition on this object and define a
\emph{marked coalescent} where \emph{coalescence occurs within colors
  only} with time-inhomogeneous rates at time $t'$ given by
$(\bar \mfu_t(s))^{-1}$ with $\bar \mfu_t(s)$ the mass of color named
$s$ at time $t$ with $t=T-t'$, with $t'$ the running time of the
coalescent and $T$ the time horizon.

The mass of colors form a Feller diffusion with coefficient $b$ and
with super-criticality coefficient for the colors
$a_T(t,\bar \mfu_t(s))$.

The question now is how to start the coalescent. Here we consider a
finite number $n$ of individuals, where we place $n_1,n_2,\dots,n_j$
of them on the color $t_1,t_2,\dots,t_j$ (note we have
\emph{countably} many colors altogether with $T$ the only limit
point).

The partition elements with a given mark evolve as explained above
till they reach the birth time of the color when they jump to the
\emph{cemetery} merging with the immortal line due to the fact that
$\int_s^T (\bar\mfu_r)^{-1} \, \dx r = +\infty$.

Finally we need to introduce now the duality function $H$. As an
ingredient take a polynomial $H^{n,\varphi,g}$ on $\U^{(0,\infty)}$.
Recall now \eqref{e1130}-\eqref{e1206}, to see how to define marked
polynomials $H^{n,\varphi,g}$. Then use the relation \eqref{eq:21} to
define $H$. Then we are able to write down the conditional duality.

By piecing together the arguments in the above sections we obtain now
that we have again a conditioned duality relation. Namely conditioned
on the path of the collection in \eqref{e3864} we conclude that, the
process above is in duality with $\mfU^{\ast,+}$.
\begin{corollary}[Conditioned duality for $\mfU^{\ast,+}$]
  \label{cor.3892}
  % \leavevmode\\
  As a consequence of the $H$ duality we have
  \begin{align}
    \label{e3894}
    \E_{\wh \mfU_0^{\ast,+}} \left[ H^{n,\varphi,g} \left(\wh
    \mfU_T^{\ast,+} (\bar \mfu),\mfC_0^{T,(\ast,+)} (\bar \mfu)
    \right) \right] = \E_{\mfC_0^{T,(\ast,+)}} \left[H^{n,\varphi,g}
    \left(\wh \mfU_0^{\ast,+} (\bar \mfu),\mfC_T^{T,(\ast,+)} (\bar
    \mfu) \right) \right],
  \end{align}
  the expectations are for the processes for given path $\bar \mfu$.
\end{corollary}

\section[Proof of Theorem~\ref{THM:MGP:WELL-POSED}]{Proof of
  Theorem~\ref{THM:MGP:WELL-POSED}: Existence, uniqueness and path
  properties of the $\U$-valued Feller diffusion}
\label{sec:Fel:ex}

We prove separately existence with path properties (continuity of
paths) and the uniqueness with semigroup properties. Finally we prove
in that context also the (generalized) Feller property and the strong
Markov property as a consequence.

\subsection{Existence and properties}
\label{ss.existence}

We begin with preparation in Step~0 where we introduce some notation
on polynomials in the polar setting that will be used throughout this
section. To obtain the existence result, in Step~1 we will use a
\emph{particle approximation} and show \emph{tightness} of its laws;
see Proposition~\ref{lem:gloede}. One point of the general existence
result is here different compared to the well-known diffusion
approximation of the total mass process, even though also in the
latter case the diffusion coefficients are not bounded for large
population size. As one can see from the form of the operator
$\Omega^\uparrow$ (recall for instance \eqref{tv5} and \eqref{e748}),
at the points of zero and infinite mass the action of this operator
\emph{can produce infinite values}. For this reason we will start
analyzing the martingale problem on bounded test functions which
vanish at zero and infinite mass and behave in a particular way
approaching them if the mass approaches these values. Then in Step~2
in several consecutive results finishing with
Proposition~\ref{cor:ex}, we will show that the limiting points of the
particle approximation solve the martingale problem of
Theorem~\ref{THM:MGP:WELL-POSED}. Finally, in Step~3 we prove
continuity of paths of solutions of the martingale problem.

\paragraph{Step 0}
A particle approximation was considered in \cite{Gl12} in the polar
setting that we have recalled here in
Section~\ref{sec:polar-repr-stat}. The results can be used in our
setting as well. Recall in Remark~\ref{r.1207} the extension of the
operator $\Omega^{\uparrow}$ to (polar) polynomials from sets $\mcD_1$
and $\mcD_2$.

For \emph{polar polynomials} we define the following general notation
\begin{align}
  \label{eq:ppc1c2}
  \Pi(\bar \mcD,\wh \mcD) = \{\Phi \in \Pi: \mfu \mapsto \Phi (\mfu) =
  \bar \Phi (\bar \mfu) \wh \Phi(\hat \mfu), \bar \Phi \in
  \bar \mcD, \wh \Phi \in \wh \mcD\}.
\end{align}
Of course, here we implicitly assume that $\bar \mcD$ and $\wh \mcD$
are appropriate sets of functions for $\bar\Phi$ and $\wh\Phi$
respectively. That means $\bar \mcD$ must be a subset of real-valued
functions on $\R_+$ and $\wh \mcD$ must be a subset of polynomials on
$\U_1$, i.e.\ a subset of $\wh \Pi$ which was defined in
\eqref{eq:Pi-total}. Recall also that by definition for all
polynomials $\Phi \in \Pi$ (in any form) we have $\Phi (0)=0$.

One of our goals in this section is to prove the existence and
uniqueness of solutions of the $\Omega^{\uparrow}$ martingale problem
on $\Pi(\mathcal C_b^1)$. We will prove existence and uniqueness for several
sets of polar polynomials generalizing the setting from step to step.
In the following we write
\begin{align}
  \label{eq:tr33}
  \begin{split}
    \wh{\mcD}_b^1 & = \wh\Pi(\mathcal C_b^1), \\
    \bar\mcD^2 & = \{\bar \Phi \in C^2(\R_+,\R) : \exists c>0,m \in
    \N_{\ge2} \text{ s.th. } \bar{\Phi}(\bar{\mfu})
    \le c\bar{\mfu}^m\}, \\
    \bar\mcD_c^2 & = \{\bar \Phi \in \bar\mcD^2 : \supp \bar{\Phi}
    \text{ compact}\},\\
    \bar\mcD_{c!}^2 & = \{\bar \Phi \in \bar\mcD_c^2 : \supp
    \bar{\Phi} \subset (0,\infty) \}.
  \end{split}
\end{align}

In the following lemma we first address some regularity properties of
our test functions.
\begin{lemma}
  \label{L.mp}
  We have $\Pi(\bar\mcD^2_{c!}, \wh\mcD^1_b) \subseteq C_b (\U,\R)$
  with continuity w.r.t.\ Gromov weak topology. Furthermore, for any
  $\Phi \in \Pi(\bar\mcD^2_{c!}, \wh\mcD^1_b)$,
  $\Omega^{\uparrow} \Phi$ is a bounded and continuous function on
  $\U$.
\end{lemma}

\begin{proof}
  The \label{pr.L.mp} only issue is continuity at zero. A proof of the
  first part can be found in Lemma~2.4.13 in \cite{Gl12}. The second
  part follows similarly.
\end{proof}

\paragraph{Step 1}
In Section~3 of \cite{Gl12} for each $N\in \N$ a discrete state
(continuous time) $\U$-valued Galton-Watson process
$(\mfU^{(N)}_t)_{t\ge 0}$ is constructed which solves a particular
martingale problem. For the explicit choice of the domains and the
form of the operators we refer the reader to the original reference.

The following proposition is also merely a citation of results from
\cite{Gl12} combined with classical theory on Markov processes, recall
\eqref{eq:tr33}.

\begin{proposition}[Tightness]
  Consider \label{lem:gloede} $\mfu \in \U\setminus\{\ntree\}$. The
  family $\{(\mfU^{(N)}_t)_{t\ge 0} : N \in \N\}$ is tight and any
  limit point $\mfU$ is a solution to the
  $(\delta_{\mfu}, \Omega^{\uparrow}, \Pi(\bar\mcD^2_{c!},
  \wh\mcD^1_b))$ martingale problem, provided the initial conditions
  $\mfU^{(N)}_0$ converge to $\mfu$ in the Gromov weak topology.
\end{proposition}
\begin{proof}
  \label{pr.gromov}
  In Proposition~4.1.3 of \cite{Gl12} it is shown that the family
  $\{(\mfU^{(N)}_t)_{t\ge 0} : N \in \N\}$ is tight in the Gromov weak
  topology. This allows us to apply Theorem~5.1 of Chapter~4 in
  \cite{EK86}. The condition (5.1) in \cite{EK86} can be checked via
  Proposition~5.5.1 and Remark~5.5.3 in \cite{Gl12}.
\end{proof}

\begin{lemma}
  \label{lem:Feller}
  Consider $\mfu = (\bar\mfu,\hat\mfu) \in \U\setminus\{\ntree\}$ and let
  $\mfU$ be a solution to the
  $(\delta_{\mfu}, \Omega^{\uparrow}, \Pi(\bar\mcD^2_{c!},
  \wh\mcD^1_b))$ martingale problem. Then for any $t>0$, and
  $m \in \N$ there is a constant $c(t,m,\bar\mfu)$ such that
  \begin{align}
    \label{tv20}
    \E_\mfu[\bar{\mfU}_t^m] \le c(t,m,\bar\mfu)
    \quad \text{and} \quad
    \E_\mfu[\sup_{s\le t} \bar{\mfU}_s^{m}] \le c(t,m,\bar\mfu).
  \end{align}
\end{lemma}
\begin{proof}
  \label{pr.Fellito}
  Since the process $(\bar{\mfU}_t)_{t\ge 0}$ is an ordinary
  Feller-diffusion the assertion follows by Ito's lemma and Doob's
  inequality.
\end{proof}

\paragraph{Step 2}
Now, in several consecutive steps we will extend the existence result
to the needed wider class of test functions. First we consider the
case \emph{close to extinction}, i.e.\ when the total mass is close to
zero.
\begin{lemma}\label{lem:ex:epsweg}
  Let $\mfu \in \U$ and assume that $\mfU^{(N)}_0$ converges to $\mfu$
  in the Gromov weak topology. Then any limiting point of the family
  $\{(\mfU^{(N)}_t)_{t\ge 0} : N \in \N\}$ is a solution to the
  $(\delta_{\mfu}, \Omega^{\uparrow},\Pi(\bar\mcD^2_c,\wh\mcD^1_b))$
  martingale problem, provided the initial conditions $\mfU^{(N)}_0$
  converge to $\mfu$ in the Gromov weak topology.
\end{lemma}
\begin{proof}
  \label{pr.Phi:eps}
  Let $\Phi \in \Pi(\bar\mcD^2_c, \wh\mcD^1_b)$ and for $\eps \in (0,1)$
  % Then $\bar{\Phi} \in C^2(\R_{\ge 0}, \R)$ with compact support
  define
  $\bar{\Phi}_\eps = \bar{\Phi} \varrho_\eps \in \bar\mcD^2_{c!}$,
  i.e.\ with compact support in $(0,\infty)$, where we choose
  $\varrho_\eps \in C^\infty(\R_+,\R)$ with
  $\varrho_\eps|_{[0,\eps)}=0, \varrho_\eps|_{[2\eps, \infty)} = 1$.
  Now we use the bound $\bar{\Phi}(\bar{\mfu}) \le c\bar{\mfu}^m$ to
  obtain the estimates
  \begin{align}
    \label{eq:Phi:eps}
    |\bar{\Phi} (\bar{\mfu}) \wh\Phi (\wh{\mfu}) - \bar{\Phi}_\eps
    (\bar{\mfu}) \wh\Phi (\wh{\mfu}) | \le  \wh\Phi (\wh{\mfu}) c
    \eps^m \le c \|\varphi\| \eps^m
  \end{align}
  and for $\mfu \in \U\setminus\{\ntree\}$
  \begin{align}
    \label{eq:LPhi:eps}
    \begin{split}
      |\Omega^{\uparrow} \bar{\Phi} (\bar{\mfu}) \wh\Phi (\wh{\mfu})
      & - \Omega^{\uparrow} \bar{\Phi}_\eps (\bar{\mfu}) \wh\Phi (\wh{\mfu})| \\
      & \le | \Omega^{\mathrm{mass}} \bar{\Phi}(\bar{\mfu}) -
      \Omega^{\mathrm{mass}}\bar{\Phi}_\eps (\bar{\mfu}) | \,
      |\wh{\Phi}(\wh{\mfu})| + |\bar{\Phi}(\bar{\mfu}) -
      \bar{\Phi}_\eps(\bar{\mfu}) | \cdot
      |\Omega^{\mathrm{gen}}_{\bar{\mfu}} \wh{\Phi} (\wh{\mfu}) | \\
      & \le | \ind{\bar{\mfu} \le \eps} \frac{b}{2} \bar{\mfu}
      \partial_{\bar{\mfu}}^2 \bar{\Phi}(\cdot) | \cdot \| \varphi \|
      + | \ind{\bar{\mfu} \le \eps} \bar{\Phi}(\bar{\mfu}) | \cdot |
      \Omega^{\mathrm{gen}}_{\bar{\mfu}} \wh{\Phi} (\wh{\mfu}) | \\
      & \le \frac{b}{2} \eps \| \partial_{\cdot}^2 \bar{\Phi}(\cdot)
      \| \cdot \| \varphi \| + c \eps \ind{\bar{\mfu} \le \eps}
      |\bar{\mfu}|^{m-1} (2 \| \overline{\nabla} \varphi \| +
      \frac{2}{\bar{\mfu}} \| \varphi \| ) \\
      & \le \eps ( b \| \partial_{\cdot}^2 \bar{\Phi}(\cdot) \| \cdot
      \| \varphi \| + 2c \| \overline{\nabla} \varphi \| + 2c \| \varphi \| ).
    \end{split}
  \end{align}
  Since both terms on the left hand side of \eqref{eq:Phi:eps} and of
  \eqref{eq:LPhi:eps} are zero for $\bar\mfu =0$, both bounds are
  uniform in $\mfu \in \U$. For any limiting point $(\mfU_t)_{t\ge 0}$
  of $(\mfU^{(N)}_t)_{t\ge 0}$, any $\eps >0$, any
  $0\le t_1 \le \dots \le t_k \le s < t$ and any bounded measurable
  $h_i: \U \to \R$ by
  % $(\delta_\mfu,L^{\uparrow,\text{pol}}_N,\mcD_N)$ martingale problem
  % (the former)
  Proposition~\ref{lem:gloede} we have
  \begin{align}
    \label{eq:251}
    \E \Bigl[ \Bigl( \bar{\Phi}_\eps (\bar{\mfU}_t) \wh\Phi
    (\wh{\mfU}_t) - \bar{\Phi}_\eps (\bar{\mfU}_s) \wh\Phi
    (\wh{\mfU}_s) - \int_s^t \Omega^{\uparrow}\bar{\Phi}_\eps
    (\bar{\mfU}_z) \wh\Phi (\wh{\mfU}_z) \, \dx z \Bigr) \prod_{i=1}^k
    h_i(\mfU_{t_i}) \Bigr] = 0.
  \end{align}
  Using \eqref{eq:Phi:eps} and \eqref{eq:LPhi:eps} and $\eps \to 0$ we
  see that \eqref{eq:251} also holds for $\bar{\Phi}$ instead of
  $\bar{\Phi}_\eps$. That means $(\mfU_t)_{t\ge 0}$ is a solution to
  the latter martingale problem.
\end{proof}

Next, we extend the martingale problem to allow arbitrary positive
mass.

\begin{proposition}[Limit points are solutions to martingale problem]
  % \leavevmode\\
  For \label{prop:polar} each $\mfu\in \U$ any limit point
  $\mfU=(\mfU_t)_{t\ge 0}$ of the sequence
  $\{(\mfU^{(N)}_t)_{t\ge 0} : N \in \N\}$ is a solution to the
  $(\delta_{\mfu},
  \Omega^{\uparrow},\Pi(\bar\mcD^2,\wh\mcD_b^1))$-martingale problem,
  provided the initial conditions $\mfU^{(N)}_0$ converge to $\mfu$ in
  the Gromov weak topology.
\end{proposition}
\begin{proof}
  \label{pr.epsweg}
  Let $\Phi \in \Pi(\bar\mcD^2,\wh\mcD_b^1)$ and let
  $(\mfU _t)_{t\ge 0}$ be a limit point of solutions of
  $\{(\mfU^{(N)}_t)_{t\ge 0} : N \in \N\}$ with
  $\mfU^{(N)}_0 \to \mfu$ as $N\to \infty$ in the Gromov weak
  topology. For $n \in \N$ define
  $\bar{\Phi}_n = \bar{\Phi} \tilde{\varrho}_n \in \bar\mcD_c^2$,
  where we choose $\tilde{\varrho}_n \in C_c^\infty(\R_{\ge_0},\R)$
  with $\tilde{\varrho}_n|_{[n+1,\infty)}=0$ and
  $\tilde{\varrho}_n|_{[0,n]} = 1$.

  By Lemma~\ref{lem:ex:epsweg} for any $k \ge n+1$, any
  $0\le t_1 \le \dots \le t_k \le s < t$ and bounded measurable
  $h_i: \U \to \R$ we have
  \begin{align}
    \label{eq:1}
    \E \Bigl[ \Bigr(\bar{\Phi}_n (\bar{\mfU}_t) \wh\Phi (\wh{\mfU}_t) -
    \bar{\Phi}_n (\bar{\mfU}_s) \wh\Phi (\wh{\mfU}_s) - \int_s^t
    \Omega^{\uparrow}\bar{\Phi}_n (\bar{\mfU}_u) \wh\Phi
    (\wh{\mfU}_u) \, \dx u \Bigl) \prod_{i=1}^k h_i(\mfU_{t_i})
    \Bigr] = 0.
  \end{align}

  For $n \in \N$ we define the stopping times
  \begin{align}
    \label{tv21}
    \tau_{n} \coloneqq  \inf \{ t \ge 0: \bar{\mfU}_t \ge n \}.
  \end{align}
  By the optional stopping theorem and \eqref{eq:1} we obtain
  \begin{align}
    \label{tv22}
    \begin{split}
      \E \Bigl[ \Bigl( \bar{\Phi} (\bar{\mfU}_{t\wedge \tau_n})
      \wh\Phi (\wh{\mfU}_{t\wedge \tau_n}) & - \bar{\Phi}
      (\bar{\mfU}_{s\wedge \tau_n}) \wh\Phi (\wh{\mfU}_{s\wedge
        \tau_n}) \\ & - \int_s^t \Omega^{\uparrow}\bar{\Phi}
      (\bar{\mfU}_{u\wedge \tau_n}) \wh\Phi (\wh{\mfU}_{u\wedge
        \tau_n}) \, \dx u \Bigr) \prod_{i=1}^k h_i(\mfU_{t_i}) \Bigr] =
      0.
    \end{split}
  \end{align}
  Now we need to show that all expressions in the above display tend
  to the expected ones as $n\to \infty$. For the first term
  this follows, since
  \begin{align}
    \label{eq:2}
    \begin{split}
      \E & \Bigl[ \bar{\Phi} (\bar{\mfU}_{t\wedge \tau_n}) \wh\Phi
      (\wh{\mfU}_{t\wedge \tau_n}) \prod_{i=1}^k h_i(\mfU_{t_i})
      \Bigr] \\
      & = \E \Bigl[ \ind{t < \tau_n} (\bar{\Phi} (\bar{\mfU}_{t})
      \wh\Phi (\wh{\mfU}_{t}) \prod_{i=1}^k h_i(\mfU_{t_i}) \Bigr] +
      \E \Bigl[\ind{t \ge \tau_n} (\bar{\Phi} (n) \wh\Phi
      (\wh{\mfU}_{t\wedge \tau_n}) \prod_{i=1}^k h_i(\mfU_{t_i})
      \Bigr].
    \end{split}
  \end{align}
  Lemma~\ref{lem:Feller} implies $\ind{t < \tau_n} \nearrow 1$. Recall
  that $\wh \Phi \in \wh{\mcD}^1$ entails that there is a $c_1<\infty$
  such that
  \begin{align}
    \label{tv24}
    \| \wh \Phi (\cdot) \prod_{i=1}^k h_i(\cdot) \| \le c_1
  \end{align}
  and since $\bar{\Phi}(\bar{\mfu}) \le c \bar{\mfu}^m$ we can use
  dominated convergence for the first term in \eqref{eq:1}.

  For the second term using \eqref{tv24}, Markov inequality and
  Lemma~\ref{lem:Feller} we obtain
  \begin{align}
    \label{tv25}
    \begin{split}
      \E \Bigl[\ind{t \ge \tau_n} (\bar{\Phi} (n) \wh\Phi
      (\wh{\mfU}_{t\wedge \tau_n}) \prod_{i=1}^k h_i(\mfU_{t_i}) \Bigr]
      & \le c_1 c \E [ \ind{t\ge \tau_n} n^m ] \\
      & \le c_1 c n^m \P_{\bar{\mfu}} [ \sup_{s< t} \bar{\mfU}_s \ge n]
      \\
      & \le c_1 c n^m n^{-m-1}   \E_{\bar{\mfu}} [ \sup_{s\le t}
      \bar{\mfU}_s^{m+1}]\\
      & \le c_1 c n^{-1} c(m+1,t,\bar{\mfu}).
    \end{split}
  \end{align}
  As $n\to \infty$ the term on the right hand side goes to zero. The
  remaining term in \eqref{eq:1} can be treated similarly.
\end{proof}

\begin{proposition}[Existence of solution for $\Omega^\uparrow$]
  \label{cor:ex}
  For each $\mfu\in \U$ any limit point $\mfU=(\mfU_t)_{t\ge 0}$ of
  the sequence $\{(\mfU^{(N)}_t)_{t\ge 0} : N \in \N\}$ is a solution
  to the $(\delta_{\mfu}, \Omega^{\uparrow},\Pi(\mathcal C_b^1))$ martingale
  problem, provided the initial conditions $\mfU^{(N)}_0$ converge to
  $\mfu$ in the Gromov weak topology.
\end{proposition}
\begin{proof}\label{pr.cor:ex}
  Let $\Phi = \Phi^{n,\varphi} \in \Pi(\mathcal C_b^1)$ then, as we have seen
  in Remark~\ref{r.1207}, for $\mfu \in \U$ we have
  $\Phi(\mfu) = \bar{\Phi}(\bar{\mfu}) \wh {\Phi}(\wh{\mfu})$ with
  $\bar{\Phi}(\bar{\mfu}) = \bar{\mfu}^n$ and
  $\wh{\Phi} = \wh{\Phi}^{n,\varphi}$. In particular
  $\Phi \in \Pi(\bar\mcD^2,\wh\mcD^1_b)$. For such functions we have
  various formulas for the action of the operator $\Omega^\uparrow$.

  Thus, by Proposition~\ref{prop:polar}, any limit point
  $\mfU = (\mfU_t)_{t\ge 0}$ of $\{(\mfU^N_t)_{t\ge 0} : N \in \N\}$
  is a solution to the
  $(\delta_{\mfu}, \Omega^{\uparrow},\Pi(\bar\mcD^2,\wh\mcD_b^1))$
  martingale problem. In Section~\ref{sec:appr-solut-mart} we will see
  that we have a solution of the
  $(\Omega^\uparrow,\Pi(\mathcal C_b^1))$-martingale problem.

  Use this to calculate
  \begin{align}
    \label{tv27}
    \E & \Bigl[\Bigl(\Phi(\mfU_t) - \Phi( \mfU_s) - \int_s^t
         \Omega^{\uparrow} \Phi(\mfU_u) \, \dx u\Bigr) \prod_{i=1}^k
         h_i(\mfU_{t_i}) \Bigr] \\
    \label{tv27b}
       & = \E\Bigl[\Bigl(\bar{\Phi} (\bar{\mfU}_t) \wh\Phi
         (\wh{\mfU}_t) - \bar{\Phi} (\bar{\mfU}_s) \wh\Phi
         (\wh{\mfU}_s) - \int_s^t \Omega^{\uparrow}\bar{\Phi}
         (\bar{\mfU}_u) \wh\Phi (\wh{\mfU}_u) \, \dx u \Bigr)
         \prod_{i=1}^k h_i(\mfU_{t_i})\Bigr] = 0.
  \end{align}
  Thus, from Section~\ref{sec:appr-solut-mart} we know that
  $(\mfU_t)_{t\ge 0}$ is also a solution to the
  $(\delta_\mfu,\Omega^{\uparrow}, \Pi(\mathcal C_b^1))$-martingale problem.
\end{proof}

\paragraph{Step 3} Now we prove that there is a version with almost
surely continuous paths.

\begin{lemma}[Continuous version]\label{lem:continuous}
  There exists a version of the $\U$-valued Feller diffusion $\mfU$
  with paths in $C([0,\infty), \U)$ almost surely.
\end{lemma}
\begin{proof}
  \label{pr.lem:continuous}
  Recall the definition of the Gromov-Prohorov metric $d_{\GP}$ in
  \eqref{eq:31} and consider an approximating particle system
  $\mfU^{(N)}$ as given in Proposition~\ref{cor:ex}. For
  $(\mfu_t)_{t\ge 0} \in \U^{\R_+}$ and $(x_t)_{t\ge 0} \in \R^{\R_+}$
  we define the functionals
  \begin{align}
    \label{tv28}
    J_{d_{\GP}}((\mfu_t)_{t\ge 0})
    & = \int_0^\infty e^{-u} \bigl( 1\wedge \sup_{s\le u}
      d_{\GP}(\mfu_s, \mfu_{s-}) \bigr) \, \dx u, \\
    \label{tv28b}
    J_{|\cdot|}((x_t)_{t\ge 0})
    & = \int_0^\infty e^{-u} \bigl(1\wedge \sup_{s\le u} |x_s-x_{s-}|
      \bigr) \, \dx u.
  \end{align}
  Using polynomials of order $1$, it is clear that for the total mass
  processes we have $\bar{\mfU}^{(N)} \Rightarrow \bar{\mfU}$, where
  $\bar{\mfU}$ is the classical Feller-diffusion which has continuous
  paths. Using Theorem~10.2 of Chapter~3 in \cite{EK86} we obtain
  \begin{align}
    \label{tv29}
    J_{|\cdot|}(\bar{\mfU}^{(N)}) \Rightarrow 0
    \quad \text{ as } N \to \infty.
  \end{align}
  Using the inequality
  \begin{align}
    \label{tv30}
    d_{\GP}(\mfu,\mfv) \le |\bar{\mfu}-\bar{\mfv}|
    + d_{\GP} (\hat \mfu,\wh \mfv).
  \end{align}
  we have $J_{d_{\GP}}(\mfU^{(N)}) \le J_{|\cdot|}(\bar{\mfU}^{(N)}) +
  J_{d_{\GP}} (\wh U^{(N)})$. Since we have the approximation for
  $\U_1$-valued Fleming-Viot models we obtain
  \begin{align}
    \label{tv31}
    J_{d_{\GP}}(\mfU^{(N)}) \Rightarrow 0 \quad \text{as } N\to \infty.
  \end{align}
  By Theorem~10.2 of Chapter~3 in \cite{EK86} the last inequality
  implies that $\mfU$ has a continuous version.
\end{proof}

\begin{remark}
  \label{r.5522}
  In the non-critical case we construct a solution of the martingale
  problem by conditioning on the complete total mass path
  $\bar \mfu\coloneqq (\bar \mfu_t)_{t\ge 0}$ and then running a
  time-inhomogeneous Fleming-Viot process at rate $b\bar \mfu^{-1}_t$
  to obtain first $(\wh \mfU_t(\bar \mfu))_{t \geq 0}$ for every path
  of the solution of the total mass path martingale problem, cf.\
  corresponding discussion on page~\pageref{p:skew}. The solution of
  our martingale problem is then given by
  $(\bar \mfU_t \wh \mfU_t(\bar \mfU))_{t \ge 0}$ and is obtained by
  averaging over the law of
  $\bar \mfU\coloneqq (\bar \mfU_t)_{t\ge 0}$, the non-critical Feller
  diffusion on $\R_+$.
\end{remark}

\begin{remark}[Second order operators]
  \label{r.4418}
  In Proposition~4.10 in \cite{DGP12} it is shown that the resampling
  operator is a second order operator for evolutions on $\U_1$. We can
  use this to conclude here that $(\wh \mfU_t)_{t \geq 0}$ has
  continuous paths as does the Feller diffusion since the generator is
  second order. Next we come to the $\U$-valued process and use
  formula \eqref{eq:r.742.1} and the information on the
  $\Omega^{\mathrm{mass}}$ and $\Omega^{\mathrm{res}}$ operators to
  conclude that the operator $\Omega^{\mathrm{bran}}$ is a second
  order operator. Then using the $\U$-valued version (instead of
  $\U_1$-valued) of Proposition~4.5 in \cite{DGP12} it can be shown
  that solutions of second order martingale problems have continuous
  paths and obtain continuity of paths of the solution of the
  martingale problem.
\end{remark}

\subsection{Uniqueness of Feller martingale problem on
  \texorpdfstring{$\U$}{U}}
\label{ss.uniqumar}

The Feynman-Kac duality relation for $\mfU^{\mathrm{Fel}}$ allows to
deduce \emph{uniqueness} of the $\Omega^{\uparrow}$-martingale problem
and the \emph{Feller property} of the solution, which we do in two
lemmata and their proofs.
\begin{lemma}
  \label{cor:unique}
  For any $\P_0 \in \mcM_1(\bbU)$ the local
  $(\P_0,\Omega^{\uparrow},\Pi(\mathcal C_b^1))$-martingale problem for the
  $\U$-valued Feller diffusion has a unique solution.
\end{lemma}
\begin{proof}\label{pr.twosol}
  WE consider first fixed initial states, i.e.\ $\P_0=\delta_{\mfu}$
  for some $\mfu \in \U$. Let $\mfU$ and $\mfU'$ be two solutions with
  the same initial distribution $\P_0$, i.e.\ under our assumption the
  same initial point. We base the duality now on a function $\varphi$
  which depends on $ m $ variables and include $ m $ in the notation.
  For $\Phi = \Phi^{m,\varphi} \in \Pi(\mathcal C_b^1)$ with $m \in \N$ define
  $p = (\{1\},\{2\},\dotsc, \{m-1\}, \{m,m+1,m+2,\dotsc \})$,
  $\dr' \equiv 0$ in Theorem~\ref{T:DUALITY}, to obtain that
  \begin{align}
    \label{tv44}
    \E_{\P_0} \left[ H^{m,\varphi} (\mfU_t,(p,0)) \right] = \E_{(p,0)}
    \left[ H^{|p_t|, \varphi}(\mfu, (p_t,\dr_t')) e^{\int_0^t
    \binom{|p_s|}{2} \, ds} \right] = \E_{\P_0} \left[ H^{m,\varphi}
    (\mfU_t',(p,0)) \right].
  \end{align}
  On the other hand
  $H^{m,\varphi}(\mfU_t,(p,0)) = \Phi^{m,\varphi}(\mfU_t)$ and since
  the algebra generated by the class $\Pi(\mathcal C_b^1)$ is separating
  (Lemma~\ref{lem:Pi:separating} and the fact that the moments of the
  total masses exist for all $t \ge 0$ and satisfy \eqref{tv4} as is
  well known and follows by a moment calculation \cite{D93},
  Chapter~4.7.) we know that $\mcL[\mfU_t] = \mcL[\mfU_t']$ for any
  $t \ge 0$, which gives uniqueness of the one-dimensional marginals
  implying the \emph{uniqueness of the martingale problem} by a result
  of Stroock and Varadhan; see Theorem~5.1.2 in \cite{D93}.

  Let $P_u$ denote the solution of the martingale problem with initial
  law $\delta_\mfu$. If we have a general $\P_0$ then we characterize
  the solution of the corresponding martingale problem as a solution
  to a local martingale problem which is given by
  $\int P_\mfu \P_0(d\mfu)$, so that for the same initial law we have
  the same solutions if we know that the solution starting in $\P_0$
  must have this form.
\end{proof}

The next point is to obtain the Feller property and the strong Markov
property from that. Here the key point is that for the duality
function $H$ the set of functions $\{H(\cdot,\mfK) : \mfK \in \bbK\}$
is law determining for the forward evolution and we can therefore use
the \emph{duality} to prove the following result.

\begin{lemma}[Feller property]
  The \label{cor:Feller} semigroup associated with the
  $(\P_0, \Omega^{\uparrow}, \Pi(\mathcal C_b^1))$-martingale problem is a
  Feller semigroup in the sense that the mapping
  \begin{align}
    \label{e.cor:Feller}
    \mcM_1(\U) \; \rightarrow \mcM_1(\U), \quad
    \pi \mapsto \int \pi (\dx \mfu) \P(\mfU_t \in \cdot \mid \mfU_0 =
    \mfu) \; \text{ for all $\pi \in \mcM_1(\U)$}.
  \end{align}
  is continuous.
\end{lemma}
\begin{proof}
  \label{pr.pin}
  We need to show that for $\pi,\pi_1,\pi_2,\dots \in \mcM_1(\U)$ with
  $\pi_n \Rightarrow \pi$ we have
  \begin{align}
    \label{tv45}
    \P_{\pi_n} (\mfU_t \in \cdot \, ) \Rightarrow \P_{\pi} (\mfU_t \in
    \cdot \,)\; \text{ weakly on $\mathcal{M}_1(\U)$}.
  \end{align}
  It suffices to consider the convergence determining class
  $\wt{\mcM}$ given in Lemma~\ref{lem:Pi:separating}. Recall that all
  moments of the mass process are finite for every $t$ and every
  initial distribution $\pi$, $\pi_n$ can be approximated by
  truncation in the weak topology by elements from $\wt{\mcM}$, so
  that we obtain the claim in the general case. The measure
  $\P_\pi (\mfU_t \in \cdot\,)$ is actually in the set $\wt\mcM$ by
  Lemma~\ref{lem:Feller}. Using the duality of Theorem~\ref{T:DUALITY}
  we have
  \begin{align}
    \label{tv46}
    \begin{split}
      \E_{\pi_n}[H(\mfU_t, (p_0,\dr_0'))]
      & = \int_{\U} \pi_n (\dx \mfu) \, \E_\mfu [H(\mfU_t, (p_0,\dr_0'))] \\
      & = \int_{\U} \pi_n (\dx \mfu) \, \tilde{\E}_{(p_0,\dr_0')}
      [H(\mfu, (p_t,\dr_t'))e^{\int_0^t
    \binom{p_s}{2} \, ds}] \\
      & \xrightarrow{n\to\infty} \int_{\U} \pi (\dx \mfu) \,
      \tilde{\E}_{(p_0,\dr_0')}
      [H(\mfu,(p_t,\dr_t'))e^{\int_0^t
    \binom{p_s}{2} \, ds}] \\
      & = \int_{\U} \pi (\dx \mfu) \, \E_\mfu
      [H(\mfU_t,(p_0,\dr_0'))e^{\int_0^t
    \binom{p_s}{2} \, ds}] =
      \E_{\pi}[H(\mfU_t,(p_0,\dr_0'))].
    \end{split}
  \end{align}
  Convergence in the next to last step holds since the function
  $\mfu \mapsto \tilde{\E}_{(p_0,\dr_0')} [H(\mfu, (p_t,\dr_t'))]$ is
  continuous by the dominated convergence theorem.
\end{proof}

We can now use the continuity of path and the generalized Feller
property to obtain the strong Markov property by approximating a
stopping time by ones with countably many values.

\begin{corollary}[Strong Markov property]
  The process $\mfU$ satisfies the strong Markov property.
\end{corollary}

\begin{remark}
  \label{r.tv45}
  Recall that we also know that $\E_\mfu[\Phi(\mfU_t)] \rightarrow
  \Phi(\mfu)$ as $t \rightarrow 0$ for any $\Phi \in \Pi(\mathcal C_b^1)$ (see
  below \eqref{tv45}). Since $\U$ is not locally compact, the previous
  result does not suffice to deduce the strong continuity of the
  semigroup as in Chapter~16 of \cite{Kall02}. In fact it is not
  possible to obtain uniform bounds on $\E_\mfu[F(\mfU_t)-F(\mfu)]$ on
  $\mathcal C_b(\U)$ unless $F \in \Pi(\mathcal C_b)$.
\end{remark}

\begin{remark}[Uniqueness, Feller non-critical case]
  In \label{r.5638} the non-critical case we just have an additional
  term in the potential, see \eqref{e.ggr1}, and the uniqueness of the
  martingale problem and the Feller property follow exactly along
  those lines above.
\end{remark}

\section[Proofs of Theorems~\ref{T:BRANCHING},~\ref{TH.MARTU}]{Proofs of
  Theorems~\ref{T:BRANCHING},~\ref{TH.MARTU}: branching property, Cox
  representation, conditioning}\label{sec:branching}

Here we prove the structural properties of the $\U$-valued Feller
diffusion.

\subsection{Proof of Theorem~\ref{T:BRANCHING}: Markov branching
  property and Cox cluster representation}\label{ss.th3}

We begin with the proof of the lemma on existence of the limiting
forest of $\U$-valued Yule trees.
\begin{proof}[Proof of Lemma~\ref{l.e1408}]
  \label{p.l.e1408}
  We have to show \emph{tightness} and then \emph{convergence} of
  $\mfY_s^{(t)}$ as $s \uparrow t$. We mark (for the elementary Yule
  process this is easy) the individuals with their time of birth,
  i.e.\ consider states $[U_t \times [0,t], r_t,\nu_t]$. Then we
  decompose the population in two parts the one with colors $\leq s$
  and the other consisting of individuals with colors $>s$, that is,
  we consider
  \begin{align}
    \label{e4443}
    \begin{split}
      U^i_t & = \supp\; \nu_t^i, \; i=1,2, \text{ where} \\
      \nu^1_t & = \nu_t |_{U \times [0,s]} \\
      \nu^2_t & = \nu_t |_{U \times (s,t]}.
    \end{split}
  \end{align}
  The metrics on $U^i_t$, $i=1,2$, are given by corresponding
  restrictions of the metric $r_t$.

  For tightness we use the standard tightness criterion for marked
  metric measure spaces in \cite{DGP11} extended to finite measures,
  see Section~B.1 in \cite{infdiv}.

  The sequence is tight, since first of all the diameter is bounded by
  $t$ and second the total mass at time $t-h$ is stochastically
  bounded in $h \in [0,t]$ even though the total rate in
  $u \in [0,t-u]$ as $u \uparrow t$ logarithmically resulting in
  countably many branches splitting off. However, since the expected
  population mass produced by the descendants is $t-u$ upon survival
  the mass becomes sufficiently small, namely the total expected mass
  production rate is $1$ over this time interval so that we get a
  finite mass in the limit.

  The final step is to check the ``no dust'' condition. We need the
  (smallest) number of ancestors making up at least fraction
  $(1-\varepsilon)$ of the total mass. But in fact the number of
  ancestors time $\varepsilon$-back in $\mfY_s^{(t)}$ is finite and
  stochastically bounded since the rate of splitting in the Yule tree
  at time $s$ increases as $s \uparrow t$ and the total rate up to
  $t-\delta$ is finite for $\delta < 0$.

  Next we have to show convergence. Since the $\U^V$-valued Feller
  diffusion has continuous paths we only need to see that as
  $s \uparrow t$ the contribution of the population with colors $s'$
  for $s' \in (s,t)$ converges to the zero tree. This is true since
  the total mass of that contribution goes to zero.
\end{proof}

In the sequel we will need the following property of the dual
dynamics. Recall the notation introduced in \eqref{tv32} --
\eqref{eq:testfunc:down}.
\begin{lemma}
  Let \label{lem:r':growth} $K = (p_t,\dr'_t)_{t\ge 0}$ be a solution
  to the
  $(\delta_{(p,0)},L^{\downarrow,\mathrm{K}},\mcG^{\downarrow})$-martingale
  problem started in $(p,0)$. Then, $(\dr'_t)_{ij} = 2t$ for all
  $1\le i < j$ with $p_t(i) \neq p_t(j)$.
\end{lemma}
\begin{proof}%[Proof of Lemma~\ref{lem:r':growth}]
  \label{pr.len:r':growth}
  Let $1\le i < j$ with $p_t(i) \neq p_t(j)$, which implies
  $p_s(i)\neq p_s(j), \, s\le t$. Consider $G(p,\dr') = f(\dr_{ij}')$
  for $f \in C_b^1(\R_{\ge 0})$. Then
  \begin{align}
    \label{tv47}
    L^{\downarrow,\mathrm{grow}} G(p,\dr') = \indset{\{p(i) \neq p(j)\}}
    f'(\dr_{ij}') , \quad
    L^{\downarrow,\mathrm{coal}} G(p,\dr') = 0
  \end{align}
  and therefore,
  \begin{align}
    \label{tv48}
    f(r_{ij}'(t))-f(0) - \int_0^t 2f'(r_{ij}'(s)) \, \dx s
  \end{align}
  is a martingale, which implies $r_{ij}'(t) = 2t$.
\end{proof}

\begin{proof}[Proof of Theorem~\ref{T:BRANCHING}, (a) branching property]
  \label{pr.T:BRANCHING}
  Fix $\mfu \in \U$. Let
  $Q_t(\mfu,\cdot) = \P_\mfu(\mfU_t \in \cdot)$, $t\ge 0$ be the
  semigroup related to $\mfU$. By Proposition~2.8 in \cite{infdiv} we
  have that truncated polynomials are separating. Then by Theorem~4.5.
  in Chapter~3 in \cite{EK86} we know that we need to show that for
  any $t,h\ge 0$, $\mfu_i =[U_i,r_i,\mu_i] \in \U(h)$, $i=1,2$ and
  integrable $\Phi = \Phi^{m,\varphi} \in \Pi(\mathcal C_b)$:
  \begin{align}
    \label{eq:branching:1}
    Q_t(\mfu_1 \sqcup^h \mfu_2, \Phi_{t+h}) = Q_t(\mfu_1,\Phi_{t+h}) +
    Q_t(\mfu_2,\Phi_{t+h}).
  \end{align}
  Integrability of $\Phi$ follows from the martingale problem.

  Using Theorem~\ref{T:DUALITY} on the Feynman-Kac duality, for
  $p= \{\{1\},\{2\}, \dotsc, \{m-1\}, \{m,m+1,\dotsc\}\}$, $\dr' = 0$
  and $\mu_{12} = \mu_1 + \mu_2$ we can write
  \begin{align}
    \label{tv49}
    Q_t(\mfu_1 \sqcup^h \mfu_2, \Phi_{t+h}^{m,\varphi} )
    & = \E_{\mfu_1\sqcup^h \mfu_2} [\Phi_{t+h}^{m,\varphi} (\mfU_t)] \\
    & = \E_{p,\dr'} \Bigl[H^{|p_t|,\varphi_{t+h}} (\mfu_1 \sqcup^h
      \mfu_2,(p_t,\dr'_t)) \exp(\int_0^t \binom{|p_t|}{2} \, \dx s)
      \Bigr] \label{tv49b}\\
    & = \E_{p,\dr'} \Bigl[ \int (\mu_1+\mu_2)^{\otimes |p_t|} (\dx
      \underline{u}_p) \varphi_{t+h}(\dr^{p_t}(\underline{u}_{p_t}) + \dr'_t)
      \exp(\int_0^t \binom{|p_t|}{2} \, \dx s)\Bigr]. \label{tv49c}
  \end{align}
  In the case that in $\underline{u}_{p_t}$ individuals from both
  $U_1$ and $U_2$ are drawn, say $u_1$ and $u_2$ for simplicity, then
  $r(u_1,u_2) > 2h$ by construction of $\mfu_1 \sqcup^h \mfu_2$. On
  the other hand, $p_t(1) \neq p_t(2)$ if we draw $1$ and $2$ from
  $ U_1,U_2 $ respectively. Thus, $r'_t(\pi_1,\pi_2) = 2t$ by
  Lemma~\ref{lem:r':growth} and we obtain:
  \begin{align}
    \label{tv50}
    \varphi_{t+h}( (\dr(\underline{u}_{p_t}))^{p_t} + \dr'_t) = 0.
  \end{align}
  This directly allows to deduce \eqref{eq:branching:1} by calculating
  backwards from \eqref{tv49c}.
\end{proof}

\begin{proof}[Proof of Theorem~\ref{T:BRANCHING}, (b)
  \Levy{}-Khintchine representation] Next we have to show the
  \emph{Markov branching property} and to identify the
  \emph{ingredients} of the \emph{\Levy{}-Khintchine representation},
  which means for each time $t$ and $h \in (0,t)$ we need to identify
  $m_h$ and $\varrho^t_h$, and show that $m_h$ is infinitely
  divisible and characterized by its \Levy{} measure. The first refers
  to the process of total masses and then we only need to know that
  given the mass, the tree structure fits with $\varrho^t_h$.

  \medskip
  \noindent
  \textbf{(1)} We first argue that we have the \emph{Markov branching
    property}. Let $M^{t}_{h} = \#_h(\mfU_t)$ be the number of
  disjoint balls of radius $2h$ in the ultrametric space $\mfU_t$.
  Recall that for the total mass process $(\bar{\mfU}_t)_{t\ge 0}$ we
  have the representation of the Laplace transform by the
  \emph{$\log$-Laplace functional} (see \cite{D93}, Section 4), namely
  for $x \ge 0$ and $\lambda>0$:
  \begin{align}
    \label{e.Lap:Fel:mass}
    \E_x[e^{-\lambda \bar{\mfU}_t} ] = \exp \left( -x u_t(\lambda)
    \right), \text{ where } u_t(\lambda) \text{ solves: } \;
    \frac{\partial u_t(\lambda)}{\partial t}=-\frac{b}{2} \;
    u_t^2(\lambda), \; u_0(\lambda)=\lambda.
  \end{align}
  In particular we have $u_t(\lambda) = 2\lambda/(2+bt\lambda)$. By
  combining Theorem~1.37 and Theorem~1.44 from \cite{infdiv} and using
  Proposition~\ref{thm:bransemig-brantree:1},
  Proposition~\ref{l.totmass:semigroup}, Lemma~\ref{l.Grey} and
  Proposition~\ref{p.MBT.balls.new} for $e(h) = u_h(\infty) = 2/(b h)$ we
  obtain
  \begin{align}
    \label{e4263}
    \begin{split}
      \E_\mfu [ \exp(- \Phi_h(\mfU_t))]
      & = \int Q_{t-h}(\mfu,\dx \mfw) \exp \Bigl( -\bar{\mfw}e(h)
      \int_{\U(h)\setminus \{\ntree\}} \varrho_h^t(\dx \mfv) (1-
      e^{-\Phi(\mfv)}) \Bigr) \\
      & = \int Q_{t-h}(\mfu,\dx \mfw) \E\biggl[\exp\Bigl(-
      \Phi_h(\bigsqcup_{i=1}^{N(\bar{\mfw},h)} \mfV_i) \Bigr)
      \biggr],
    \end{split}
  \end{align}
  where \emph{$\mfV_i$ are i.i.d.\ drawn according to
    $\varrho_h^t$} and \emph{independent} of the \emph{(random)
    number of summands}
  $N(\bar{\mfw},h) = \Pois(\bar{\mfw}e(h))$. Therefore we have
  the Markov branching property.

  \medskip
  \noindent
  \textbf{(2)} Now we identify $m_h$ the law of the random variable
  $\bar \mfw$. The measure $m_h$ involves information about the number
  of $2h$-balls and hence involves a lot of information about the tree
  structure. In particular, \eqref{e4263} means that conditionally on
  $\mfU_{t-h}$, the number of balls of radius $2h$ in $\mfU_t$ is a
  Poisson variable with parameter $u_h(\infty)\bar{\mfU}_{t-h}$,
  denoted by $M^{t}_{h}$.

  Let $0<h'<h$. Then for $\mfV_i$ drawn according to
  $\varrho_h^t$, $Z_i = \#_{h'}(\mfV_i), \, i \in \N $ is an
  i.i.d.\ collection of positive integers and
  \begin{align}
    \label{e4265}
    M^{t}_{h'} = \sum_{i=1}^{M^{t}_{h}} Z_i.
  \end{align}
  This can be translated into an equation for the generating functions
  with $q \in (0,1)$:
  \begin{align}
    \label{e.Poi:eq}
    \int Q_{t-h'}(\mfu,\dx \mfw) \E [ q^{\text{Pois}(\bar{\mfw}
    u_{h'}(\infty)) } ] = \int Q_{t-h}(\mfu,\dx \mfw)
    \E[q^{\sum_{i=1}^{\text{Pois}(\bar{\mfw} u_h(\infty))} Z_i)}].
  \end{align}
  Using \eqref{e.Lap:Fel:mass} we obtain for the left hand side:
  \begin{align}
    \label{e4266}
    \begin{split}
      \int Q_{t-h'}(\mfu,\dx \mfw) \E [ q^{\Pois(\bar{\mfw}
        u_{h'}(\infty)) }]
      &= \int Q_{t-h'}(\mfu,\dx \mfw) \exp( - \bar{\mfw}
      u_{h'}(\infty) (1-q) ) \\
      & = \exp \left( - \bar{\mfu} u_{h'}((1-q)u_{h'}(\infty))
      \right).
    \end{split}
  \end{align}
  Similarly we evaluate the right hand side with
  $g_{h,h'}:[0,1] \to \R$ the generating function of $Z_1$:
  \begin{align}
    \label{e4267}
    \int Q_{t-h}(\mfu,\dx \mfw) \E[ q^{\sum_{i=1}^{\text{Pois}(\bar{\mfw}
    u_h(\infty))} Z_i} ] = \exp( - \bar{\mfu} u_h((1-h(q))u_h(\infty))).
  \end{align}
  Inserting this into \eqref{e.Poi:eq} we get
  \begin{align}
    \label{e.tr37}
    u_{t-h}((1-g(q))u_h(\infty)) = u_{t-h'}((1-q)u_{h'}(\infty)),
  \end{align}
  which due to the dynamical system structure of the $u_h$ is
  independent of $t$ and can be written as
  \begin{align}
    \label{e.tr38}
    g_{h,h'}(q) = 1 - \frac{1}{u_h(\infty)}
    u_{h-h'}((1-q)u_{h'}(\infty)).
  \end{align}

  We get the total mass of each of the leaves in the trunk as follows.
  Clearly, the corresponding random variables are i.i.d.~and one of
  them, say $Y$, equals the total mass of an ultrametric space chosen
  according to $\varrho_h^t$. This can be calculated in general with
  the help of Proposition~\ref{thm:bransemig-brantree:1}. Note that
  $\varrho_h^t$ is the normalized version of $\hat\varrho_h^t$ from
  that proposition, i.e.\ $\varrho_h^t = e(h)^{-1}\hat\varrho_h^t$.
  For $\gamma >0$ we have
  \begin{align}
    \label{e4268}
    \E [e^{-\gamma Y}]
    & = \lim_{n\to \infty} (e(h))^{-1} \int_{\bar{\mfu}>0}
      nQ_h(\frac{1}{n}\mfe,\dx \mfu) \, e^{-\gamma \bar{\mfu}} \\
    & = \lim_{n\to \infty} (e(h))^{-1} \int_{x>0}
      n\bar{Q}_h(n^{-1},\dx x) \, e^{-\gamma x}  \qquad (\text{by
      Proposition~\ref{l.totmass:semigroup}}) \label{e4268b}\\
    & = (u_h(\infty))^{-1} \lim_{n\to \infty} \Bigl[ e^{-\frac{1}{n}
      u_h(\gamma)} - e^{-\frac{1}{n}u_h(\infty)} \Bigr] \qquad
      (\text{by Lemma~\ref{l.Grey}})\label{e4268c}\\
    & = 1- \frac{u_h(\gamma)}{u_h(\infty)}. \label{e.tr40}
  \end{align}

  We now specialize to our case. In the case of the $\U$-valued
  Feller diffusion we have
  \begin{align}
    \label{e4269}
    u_h(\gamma) = 2\gamma/(2+bh\gamma),
    \quad
    u_h(\infty) = 2/(bh).
  \end{align}
  Inserting this in \eqref{e.tr38} gives:
  \begin{align}
    \label{e4270}
    g_{h,h'}(q) = 1- h \frac{(1-q)/h'}{1+(1-q)(h-h')/h'} =
    \sum_{k=1}^\infty q^k \frac{h'}{h} \Bigl(\frac{h-h'}{h}\Bigr)^{k-1},
  \end{align}
  which is a geometric distribution with parameter $(h-h')/h$. That is
  \begin{align}
    \label{e4271}
    Z_1 \text{ is geometrically distributed with parameter }(h-h')/h .
  \end{align}
  By Proposition~\ref{p.MBT.balls} we know that
  $(M^{t}_{h'})_{h'\in [t-h,t)}$ is a Markov process, moreover a
  branching process. We have also shown that the marginal
  distributions are geometric conditionally on $M^{t}_{h} = 1$.

  Now we show that a Yule process on $\N_0$ denoted
  $(X_s^{(h)})_{s \in [0,h)}$ with $X_0^{(h)} =1$ and jump rate
  $2(h-s)^{-1}$ at time $s \in [0,h)$ has the same marginal
  distribution.

  Let $f(s,q) = \E_1 [q^{X_s^{(h)}}]$ and observe that $f(s,q)$ must
  be a solution of the following Kolmogorov backward equation for
  $s \in (0,h)$ and $q\in [0,1)$:
  \begin{align}
    \label{e4272}
    \partial_s f(s,q) = \frac{2}{h-s} q(q-1) \partial_q
    f(s,q),\ f(0,q) = q,\, f(s,1) = 1.
  \end{align}
  The only solution of this equation is $f(s,q) = g_s(q)$. Therefore
  we have shown that the Yule process $X^{(h)}$ and
  $(M^{t}_{h'})_{h'\in [t-h,t)}$ have the same law conditionally
  on $M^{t}_{h} = 1$.

  It remains to identify the law of the total mass of each of the
  leaves in the trunk. Using \eqref{e.tr40} the Laplace transform of
  the mass equals
  \begin{align}
    \label{e4273}
    1- \frac{u_h(\gamma)}{u_h(\infty)}
    =\frac{2}{2+\gamma b h}.
  \end{align}
  This correspond to an exponential distribution with parameter $bh$.
  It suffices to reconstruct the $h$-trunk uniquely which also gives
  the whole state as limit $h \uparrow t$ by Proposition~2.25 in
  \cite{infdiv}.

  In the non-critical case the differential equation analogous to
  \eqref{e.Lap:Fel:mass} is given by
  \begin{align}
    \label{eq:unoncrit}
    \frac{\partial u_t(\gamma)}{\partial t} = -au_t(\gamma) +
    \frac{b}2u_t^2(\gamma), \quad u_t(0) = \gamma, \quad \text{
    with $b >0$ and $a \ne 0$}.
  \end{align}
  The solution and its limit for $\gamma\to\infty$ is given by
  \begin{align}
    \label{e4274}
    u_t(\gamma) = \frac{2a\gamma }{2a e^{-at} + b\gamma(1-e^{-at})}, \quad
    u_t(\infty) = \frac{2a}{b(1-e^{-at})}.
  \end{align}
  In the limit $a\to 0$ these expressions coincide with the
  corresponding expressions in the critical case. We can use
  \eqref{e.tr38} to determine the law of a family descending from one
  individual between time $t-h < t-h'$:
  \begin{align}
    \label{e4275}
    g_{h,h'}(q)
    & = 1 - a^{-1} b(1-e^{-ah})
      \frac{a(1-q)u_{h'}(\infty)}{e^{a(h-h')}(a-b(1-q)u_{h'}(\infty))
      + b(1-q) u_{h'}(\infty) } \\
    \label{e4275b}
    & = 1 - a^{-1} b(1-e^{-ah})
      \frac{a(1-q)a(b(1-e^{-at}))^{-1}}{e^{a(h-h')}
      (a-(1-q)a(1-e^{-ah'})^{-1}) + (1-q) a(1-e^{-ah'})^{-1} } \\
    \label{e4275c}
    & = \frac{e^{-ah} - e^{-a(h-h')} }{e^{-ah}-1} \sum_{k\ge 0}
      q^{k+1} \Bigl( \frac{e^{-a(h-h')}-1}{e^{-ah}-1} \Bigr)^k.
  \end{align}
  This is a geometric distribution for any $a\in \R \setminus \{0\}$
  and can be extended to $a=0$ by a limit, which gives the result in
  the critical case. Using a PDE approach as in the critical case we
  get for fixed $h\in (0,t]$ and $s= h-h' \in [0,h)$:
  \begin{align}
    \label{e4276}
    g_{h,h-s} (q) = \frac{e^{-ah}-e^{-as}}{e^{-ah}-1} \sum_{k\ge 0}
    q^{k+1} \Bigl( \frac{e^{-as}-1}{e^{-ah}-1} \Bigr)^k.
  \end{align}
  We compare this with a Feller process $(X_s)_{s\in [0,h)}$ started
  in $X_0 = 1$ and with generator
  \begin{align}
    \label{e4277}
    A_s f(n) = \alpha(s) n (f(n+1)-f(n)), \; s \in [0,h).
  \end{align}
  We compute the backward PDE for $f(s,q) = \E_1[q^{X_s}]$ for
  $q \in [0,1]$ and $s \in [0,h)$ and obtain:
  \begin{align}
    \label{e4278}
    \partial_s f(s,q) = \alpha(s) q(q-1) \partial_q f(s,q).
  \end{align}
  Setting $f(s,q) = g_{h,h-s}(q)$ allows to obtain the rate
  \begin{align}
    \label{e4279}
    \alpha (s) = \frac{a e^{-as}}{e^{-as}-e^{-ah}}, \; s \in [0,h).
  \end{align}
  We want to show that we have captured the metric structure. In order
  to achieve this we can consider the \emph{number of $2h'$-balls} in
  $\mfU_h$ under the law $P_0$, i.e.\ the excursion law and compare
  this with the corresponding number in $\mfY_h$, call these
  $M^{\mfY,h}_{h'}$ resp. $M_{h'}^{\mfU,h}$. We have seen in point 2
  that these are equal in law for fixed $h'$. What we need is the path
  in $h'$ and its law to be equal. This follows from the fact (shown
  above) that both sides are Markov processes.

  Putting things together, we have identified the \emph{process} of
  ball numbers, namely:
  \begin{align}
    \label{e4855}
    \begin{split}
      & \text{the normalized law } (u_h(\infty))^{-1}
      \hat\varrho_h^t=\varrho_h^t \text{ has realizations with
        the same metric
        structure} \\
      & \text{as a Yule tree with splitting rate } \alpha(s), \; s \in
      [0,h).
    \end{split}
  \end{align}
  We have to determine now the leaf law. To calculate the
  i.i.d.\ masses at the leaves of the trunk we use \eqref{e.tr40} and
  obtain, if we call $Y$ one such mass:
  \begin{align}
    \label{e4280}
    \E[e^{-\gamma Y}] = 1-\frac{u_h(\gamma)}{u_h(\infty)}
    = 1-\frac{\gamma b(1-e^{-ah})}{2a e^{-at} + b\gamma(1-e^{-at})}
    = \Bigl(1+ \gamma \frac{b}{2a}(e^{ah}-1) \Bigr)^{-1}.
  \end{align}
  Thus, the distribution of $Y$ is the exponential distribution with
  parameter $\frac{b}{2a}(e^{ah}-1) \in (0,\infty)$. By \eqref{e4273}
  the critical case the distribution is $\Exp(bh/2)$, i.e.\ the
  entrance law at time $h$ conditioned to survive up to time $h$. This
  identifies now $m_h$ as claimed. We have however already obtained
  more information in particular on the ultrametric structure.

  \medskip
  \noindent
  \textbf{(3)} Next we turn to $\varrho_h^t$. Now, according to
  \eqref{e4855}, we have to identify only the law of an
  $(h-h')$-subfamily by including the mass distribution into the
  picture, which we obtain from the corresponding Yule tree as
  $\mfY^{(h)}$ (as limit) by using Lemma~\ref{l.e1408} and
  Proposition~\ref{cor.1382}.

  We identify $\varrho_h^t$ as the law of $\mfY^{(h)}$ from
  \eqref{e1408}. Recall the definition of trunks in \eqref{def:trunk}.
  We will have to look only at the $h'$-\emph{trunk} of $\mfU_t$
  (which determine the state as we saw above) and at $\mfY^t_{t-h'}$,
  i.e.\ consider the evaluation with polynomials and prove:
  \begin{align}
    \label{e4281}
    \E[\Phi^{n,\varphi}(\lceil\mfU_t\rceil(h'))] =
    \E[\Phi^{n,\varphi}(\lceil\mfY^t_{t-h'}\rceil)] \quad \text{ for all }
    h' \in (0,h].
  \end{align}
  We know this identity from combining \eqref{e4855} and
  \eqref{e4268}--\eqref{e4269}. We can decompose $\mfY^t_{t-h'}$ in
  $2(h'-h)$ subfamilies so that a single such family given for
  $h' \uparrow t$ a realization of $\varrho_{h'}^t$, which then is
  actually equal in law to the object $\mfY_h^h$ giving the claim.
\end{proof}

\begin{proof}[Proof of Proposition~\ref{cor.1382}]\label{pr.5155}
  Here we claim that the Yule tree at time $t-\varepsilon$
  approximates for $\varepsilon \downarrow 0$ the time-$t$ marginal of
  the entrance law. This was already proved in part (b) of the proof
  above.
\end{proof}

\subsection{Proof of Theorem~\ref{TH.MARTU}}
\label{ss.pr.thmartu}

\begin{proof}[Proof of Theorem~\ref{TH.MARTU} (a): Conditioned
  process] We have observed in Proposition~\ref{prop.tvF} that the total
  mass of the $\U$-valued Feller diffusion is an autonomous Markov
  process. Hence, we can condition the \emph{original} process on the
  event of survival up to time $T$, which is measurable w.r.t.\ the
  total mass process. We obtain (using polar decomposition) the pure
  genealogy part driven by the total mass path conditioned to survive
  till time $T$. Since we can extend the domain of the generator
  $\Omega^{\uparrow}$ to bounded twice differentiable functions in
  $\bar\mfu$, we obtain a suitable process for the
  supercriticallity $a_T(\bar\mfu,t)$ which then solves a well-posed
  martingale problem on $\R_+$.

  We will show that we can use the \emph{conditioned total mass
    process} to solve the \emph{conditional martingale problem} which
  is the one specified in part (b) of Theorem~\ref{TH.MARTU}. This is
  similar to Corollary~\ref{prop.834}, but now the specified
  $\R_+$-valued diffusion replaces the unconditioned total mass
  process, i.e.\ here we keep the coefficient $b$ in the operator and
  introduce a super-criticality term $a_T(s,\bar \mfu_s)$ which acts
  only on the total mass process.

  Now we introduce the conditioning non-extinction at time $T$. We can
  use an abstract structure to characterize the law of
  $\mfU^T=(\mfU_t^T)_{t \in [0,T]}$. To this end, using polar
  decomposition of the (unconditioned) $\U$-valued Feller diffusion
  $\mfU=(\mfU_t)_{t\ge 0}$ we factorize mass and genealogy and write
  $\mfU_t=\bar \mfU_t \wh \mfU_t$ identifying it with an element from
  $\R_+ \times \U_1$. After extinction the process is identified with
  $\ntree=(0,\mfe)$. Then denoting by $\bar P$ the law of the
  $\R_+$-valued Feller diffusion $\bar\mfU=(\bar \mfU_t)_{t\ge 0}$ for
  some fixed initial state and by $\wh P$ the law of
  $\wh\mfU=(\wh \mfU_t)_{t \ge 0}$ we have
  \begin{align}
    \label{e3418}
    \wh P = \int \bar P (\dx\bar\mfu) \wh P^{\bar \mfu}.
  \end{align}
  Here, for a realization $\bar\mfu=(\bar\mfu_t)_{t\ge 0}$ of
  $\bar\mfU$, $\wh P^{\bar \mfu}$ is a regular version of
  $\wh P \bigl(\, \cdot \, |\bar\mfU = (\bar \mfu_t)_{t \ge 0}\bigr)$.

  We denote by $\bar P_T^{\mathrm{cond}}$ the law $\bar P$ on
  $C([0,\infty),\R_+)$ conditioned on $\bar \mfU_T>0$, and by
  $\wh P^{\bar \mfu}_T$
  %the kernels from
  %$C([0,\infty),\R_+)$ to $C(C([0,\infty), \R_+),\U_1)$ which are
  the solution of the (corresponding w.r.t.\ $\bar\mfu$)
  \emph{conditioned $\U_1$-valued martingale problem}. For
  $\bar P_T^{\mathrm{cond}}$ almost surely each $\bar \mfu$ the
  solution $\wh P^{\bar \mfu}_T$ is the law of a time-inhomogeneous
  $\U_1$-valued Fleming-Viot process which is known to exist and is
  \emph{uniquely determined} by the specified martingale
  problem. % whose law we call $\wh P^{\bar \mfu}$.

  Now, denoting by $P_T^{\mathrm{cond}}$ the law $P$ conditioned on
  $\bar\mfU_T>0$, we have for given $T>0$:
  \begin{align}
    \label{e3425}
    P_T^{\mathrm{cond}}=\bar P_T^{\mathrm{cond}} \otimes \wh
    P_T^{\bar \mfu,\mathrm{cond}} \quad
    \text{ for } \bar P \; \text{a.s. all } \mfu=(\bar \mfu_t)_{t \ge 0},
  \end{align}
  where
  \begin{align}
    \label{e4330}
    \wh P_T^{\bar \mfu,\mathrm{cond}}
    = \wh P^{\bar \mfu}   \text{ restricted to }
    \bar \mfu \text{ with } \bar \mfu_t>0 \text{ for } t \in [0,T].
  \end{align}
  This can be explicitly verified because $\mfu=(\bar \mfu_t)_{t \ge
    0}$ is a realization of an autonomous process.
  % Therefore
  % $P_t^{\mathrm{cond}}= \bar P^{\mathrm{cond}}_t \otimes \wh
  % P^{\bar\mfu, \mathrm{cond}}_t$.
  Therefore $P_T^{\mathrm{cond}}$ is the law of a time-inhomogeneous
  Markov process. Its restriction to paths on $[0,T]$ is the law of
  the conditioned process $\mfU^T=(\mfU^T_t)_{t\in [0,T]}$.
\end{proof}

\begin{remark}
  \label{r.3998}
  In \cite{LN68} it is shown that the conditioned Galton-Watson
  process on $\N$ is Markovian, its conditional transition
  probabilities are calculated and the limiting process is identified.
  Thus, an alternative strategy of the proof of Theorem~\ref{TH.MARTU}
  (a) is via approximation by $\U$-valued Galton-Watson processes,
  i.e.\ by adapting the convergence result by Lamperti-Ney to the
  $\U$-valued setting. One would need to show that rescalings of
  \emph{$\U$-valued Galton-Watson processes converge towards the
    $\U$-valued Feller diffusion} which is a time-inhomogeneous
  state-dependent branching process.
\end{remark}

\begin{proof}[Proof of Theorem~\ref{TH.MARTU} (b): Martingale problem]
  We need to prove that the $\U$-valued Feller diffusion whose
  total-mass process is an $\R_+$-valued diffusion with drift and
  diffusion coefficients given via \eqref{e984}, exists and is
  uniquely determined by the martingale problem. In particular, we
  have to allow \emph{time-inhomogeneous} and \emph{state-dependent}
  super-criticality coefficients. The existence and uniqueness of the
  corresponding total mass process are well-known. This is as a
  diffusion process, which was studied in \cite{LN68} and which we
  prove correcting an error that paper in
  Appendix~\ref{sec:comp-diff-coeff}. Part (b) of
  Theorem~\ref{TH.MARTU} is a consequence of the Lemma~\ref{l.1986}
  below.
\end{proof}

\begin{lemma}
  \label{l.1986}
  For all $\mfu \in \U$ and all $s,t \in \R_+$, $s < t$ the
  $(\Omega^{\uparrow,(a,b)}, \Pi^1, \delta_\mfu)$-martingale problem
  with coefficients $a$ and $b$ as in \eqref{e984} has a unique
  solution in the space $ C([s,t], \U)$.
\end{lemma}
\begin{proof}\label{pr.2254}
  To prove this lemma we have to extend
  Theorem~\ref{THM:MGP:WELL-POSED} to \emph{time-inhomogeneous
    coefficients}. The existence is not a big problem since we can do
  approximations by piecewise constant super-criticality rates.
  Uniqueness is more subtle because \emph{state-dependence} of
  super-criticality rate breaks the FK-duality and the conditioned
  duality and we have to proceed differently. More precisely, since we
  can construct the total mass process as a \emph{diffusion process}
  uniquely from the given parameters, the $\U$-valued solution will be
  constructed uniquely as a process driven by the total mass process.

  First, we have to argue that the time-inhomogeneous total mass
  process is a solution to a martingale problem on $\R_+$, which is
  \emph{well-posed}. The fact that this is a solution can be seen
  adapting corresponding arguments in the time-homogeneous setting;
  cf.\ \eqref{eq:r.742.1} with $\varphi=const$ and
  Proposition~\ref{prop.1183}. Well-posedness follows from standard
  $\R_+$-valued diffusion theory. First, by the Feynman-Kac duality,
  the $\R_+$-valued Feller diffusion is the solution of a well-posed
  martingale problem. Then we have to add the drift term and show that
  the solution is still unique. For this we use the Yamada-Watanabe
  criterion; see Appendix~\ref{sec:ikeda-watan-argum}.
  %For that we have to use Girsanov's
  %theorem.

  Next, we prove that the \emph{pure genealogy} part of the process
  which solves our martingale problem \emph{conditioned} on the
  complete total mass process \emph{must solve a martingale problem on
    $\U_1$} which is \emph{well-posed}. The well-posedness of the
  conditioned martingale problem on $\U_1$ follows via duality which
  we have established in Section~\ref{sss.dualu}, so that only the
  first point remains to be shown.

  For that we have to generalize Theorem~6.4.2. in \cite{Gl12} to
  account for the \emph{drift term} appearing in our dynamics. This is
  easily done using the general Theorem~8.1.4 in \cite{Gl12} on
  \emph{skew martingale problems}, where the issue is resolved within
  an abstract setup. The setup from \cite{Gl12} applies here. We first
  choose the state spaces of the two processes, i.e.\ the one we
  condition on and the conditioned process as
  \begin{align}
    \label{e5276}
      E_1=[0,\infty) \quad \text{and} \quad E_2=\U_1.
  \end{align}
  Then the operator $A$ in \cite{Gl12} is the one of the martingale
  problem of our $\R_+$-valued diffusion, $C$ is the operator of the
  time-inhomogeneous Fleming-Viot process and the multiplication
  operator $B$ is given by multiplication with $\bar \mfu^{-1}$.

  An additional problem here is the case where we start with initial
  mass $\bar \mfu = 0$. In this case we have to construct the solution
  as an \emph{entrance law}. The diffusion coefficients from
  \eqref{e984} can be extended to mass $0$. The corresponding process
  on $\R_+$ has a unique solution, since the drift term is Lipschitz.
  We have to argue that this holds also for the \emph{$\U$-valued
    processes}. Here we have to study only the process
  \emph{conditioned} on the \emph{autonomous process of total masses}
  to show its convergence.

  We observe that the operator of the conditional martingale problem
  is not affected and the effect of the conditioning and the changing
  initial mass sits entirely in the path $\bar \mfu$. Using duality we
  see that because of the \emph{divergence} of
  $\int_{(0,\varepsilon]} \bar \mfu^{-1}_s \, \dx s$ the corresponding
  dual is a coalescent with a divergent accumulated coalescence rate.
  This dual in backward time is a coalescent hitting the trivial
  partition before any positive time and therefore due to diverging
  rates it converges to the \emph{zero tree} $\ntree$ as time goes to
  $0$. To establish the divergence we note that this is known for
  branching and branching with drift; see e.g.\ Proposition~0.2 in
  \cite{DG03}.
\end{proof}

\begin{proof}[Proof of Theorem~\ref{TH.MARTU} (c): Relation to
  entrance law] Finally we have to relate the excursion law of the
  Feller diffusion on $\U$ to $\mfU^T$. First we look at
  $\mfU^T_{(\ve,\mfe)}$ and its behavior as $\ve \downarrow 0$. We
  claim the processes $\mfU^T$ satisfy that $\mfU^T_{(0,\mfe)}$ is the
  limit of the processes $\mfU^T_{(\ve,\mfe)}$ as $\ve \downarrow 0$.
  This hold due to the fact that $\bar a_T(\cdot,\cdot)$ and the
  volatility, $\bar b_T(s,x)=bx,s \in [0,T], x \in [0,\infty)$ are
  continuous functions on the state space so that $\bar \mfU^T$ has
  the Feller property. Furthermore the path of $\bar \mfU^T$ is
  strictly positive on $(0,T]$ for all $t \in [0,T]$ with starting
  points $(\ve,\mfe)$. This means that we can use the representation
  of $\wh \mfU^T$ conditioned on $\bar \mfU^T$ as time-inhomogeneous
  Fleming-Viot process, that the process $\wh \mfU^T$ is Feller for
  every fixed mass path starting with non-zero mass a.s. Since the
  total mass path starting in $0$ mass is not integrable and
  $\wh \mfU^T $ converges to the element $\mfe$, we have that in fact
  the process $\wh \mfU^T$ has the derived continuity property. Using
  the polar representation we see that indeed $\mfU^T$ is Feller on
  $\U$.

  We note next that for $\ve >0$,
  $\mcL [(\mfU^T_t)_{t\in [0,T]}| \mfU^T_0 = (\ve \cdot\mfe)]$, up to
  a factor which converges to $1$ as $\ve \downarrow 0$, coincides
  with the law $P_{\ve \cdot \mfe}$ restricted to paths on $[0,T]$
  with $\bar \mfU_T >0$ and normalized by $\ve$. This is true by the
  definition of $\mfU^T$ and the asymptotics of $P(\bar \mfU_T >0)$.
  Second we know the $\ve \downarrow 0$ convergence result for
  $\bar \mfU^T$. The strategy is to use for the claim the conditioned
  duality to obtain the convergence of the process $\wh \mfU^T$
  conditioned on $\bar \mfU^T$. Recall that $\bar \mfU^T_t > 0$ for
  $t \in (0,T]$.

  With above two groups of properties we now have to prove that indeed
  $P^{\mathrm{prob}}_{\ntree;0,T}=\mcL[\mfU^{T,\rm entr}]$. This is
  follows from the choice of the topology as follows.

  For $\ve,\delta >0$ consider the measure
  $\prescript{\delta \mkern-5mu}{}{\bar P}_\ve \coloneqq \bar P_\ve
  (\cdot \cap \{ \sup\limits_{t \geq 0} \bar \mfU_t \geq \delta \})$.
  Then the result on the excursion law for $\R$-valued Feller tells us
  that the following limits exist
  \begin{equation}\label{e6765}
    \wlim_{\delta \downarrow 0}\Bigl(\wlim_{\ve \downarrow 0}
    \bigl[\ve^{-1} \cdot
    (\prescript{\delta \mkern-5mu}{}{\bar{P}}_{\ve;0,T})\bigr]\Bigr) = \bar
    P_{\ntree;0,T}.
  \end{equation}
  Furthermore we have as a consequence of the definition of $\mfU^T$
  that:
  \begin{equation}\label{e6769}
    \prescript{\delta \mkern-5mu}{}{\bar{P}}_{\ntree;0,T}
    (\cdot \cap \{ \bar \mfU_T > 0\})
    = \prescript{\delta \mkern-5mu}{}\mcL
    \bigl[(\bar \mfU^{T,\rm entr}_t)_{t \in [0,T]}\bigr].
  \end{equation}

  Denote by $\mfU^{\rm exc}$ the realization of $P_{\ntree,0;T}$. As a
  consequence of \eqref{e6769} and the form of the martingale problems
  for $\mfU^T$ and $\mfU^{\rm exc}$ evaluated on functions of
  $\wh \mfU$, $\wh \mfU^{\rm exc}$ it follows that the conditional
  martingale problem for $\wh \mfU^T$, $\wh \mfU^{\rm exc}$ (given
  $\bar \mfU^T$ resp.\ $\bar \mfU^{\rm ecx}$) are the same, see
  \eqref{e3181}-\eqref{e1081}. Thus, we have also that the processes
  $\prescript{\delta \mkern-5mu}{}\mcL[(\wh \mfU^T_t)_{t \in [0,T]}]$
  and
  $\prescript{\delta \mkern-5mu}{}\mcL[(\wh \mfU^{\rm exc})_{t \in
    [0,T]}]$ agree with
  $\prescript{\delta \mkern-5mu}{}{\wh P}_{\ntree;0,T}(\bar \mfu)$,
  $\bar \mfu$ - a.s.\ on $\{\bar \mfU^T > 0 \}$. This means $\mfU^T$
  and $\mfU^{\rm exc}$ are equal in
  $\prescript{\delta \mkern-5mu}{}\mcL$ for every $\delta >0$. Since
  the entrance law
  $\prescript{\delta \mkern-5mu}{}\mcL[(\mfU^{T,\rm entr}_t)_{t \in
    [0,T]}]$ converges to $\mcL [(\mfU^T_t)_{t \in [0,T]}]$ as
  $\delta \downarrow 0$ we are done.
\end{proof}b

\begin{proof}[Proof of Theorem~\ref{TH.MARTU} (d): Relation to
  Fleming-Viot] From the generator \eqref{e748} and the conditional
  duality in Section~\ref{ss.condualy} the claim follows adapting the
  duality relation for the $\U_1$-valued Fleming-Viot diffusion; see
  \cite{GPWmp13}.
\end{proof}

\section[Proofs of Theorems~\ref{T:KOLMOGOROVLIMIT},~\ref{T.1205},~\ref{TH.1061},~\ref{TH.IDENTKALL},~\ref{T:BACKBONE},~\ref{TH.KOLMO}]{Proofs of Theorems~\ref{T:KOLMOGOROVLIMIT},~\ref{T.1205},~\ref{TH.1061},~\ref{TH.IDENTKALL},~\ref{T:BACKBONE},~\ref{TH.KOLMO}: \\
  Conditioned processes and Kolmogorov-Yaglom limits}
\label{sec:Kolmogorovlimit}

Here we collect the proofs of statements concerning the $\U$-valued
processes which are conditioned to \emph{survive for long time} namely
to \emph{survive forever}, the $h$-transformed version, or the Evans
process in various representations. Furthermore we consider their
descriptions via $\U^V$- and $\U$-valued martingale problems which are
of different flavors than the Feller diffusion $\mfU$ itself.

% We prove the seven theorems and the propositions in eight respective
% subsections in their order.

\subsection{Proof of Theorem~\ref{T:KOLMOGOROVLIMIT}}
\label{ss.prtkol}

\begin{proof}[Proof of Theorem~\ref{T:KOLMOGOROVLIMIT}~(a)]
  It is well known that the $\R$-valued Feller diffusion conditioned
  to survive till time $T$ converges as $T \to \infty$ to the
  \emph{Feller diffusion} with \emph{immigration} at rate $b$ in
  $C_b([0,\infty),\R)$; see \cite{Lamb07}. Thus, the same holds for
  our total mass process $\bar \mfU^T$. Based on this we can show that
  the solution to the \emph{conditional} martingale problem converges
  to the one we obtained in Lemma~\ref{c.1310}.

  Below we first show \emph{tightness} of the laws of the process
  $\mfU^T$ on $[0,t]$ for $T \to \infty$ and then the
  \emph{convergence}. We use here the standard tightness criteria in
  $\mathcal C([0,t],\U)$. To this end, we separate masses and
  genealogies and show convergence of the branching rates on every
  time interval $[0,t]$.

  For the tightness of the laws of the process we first prove the
  \emph{compact containment property}. Consider a fixed time $t<T$
  with diverging time horizon $T$. By well known results on the
  $\R_+$-valued Feller diffusion on bounded time intervals the
  \emph{total mass} is stochastically bounded and hence compactly
  contained. Note that \emph{distances} on $[0,t]$ are bounded
  independently of $T$ by $2t+r_0(\cdot,\cdot)$. For further
  information on the \emph{genealogy}, namely that it remains
  dust-free for fixed time $t$ as we take the limit $T\to \infty$ we
  use the \emph{conditional duality} in the case of $\mfU^T$ see
  Section~\ref{sss.dualu}.

  To see that the number of ancestors time $\varepsilon$ back is
  bounded stochastically for all $T$ we observe the following facts.
  The \emph{coalescence rate} of the conditioned dual is bounded in
  $T$ for a \emph{fixed} path of total mass, because as
  $T \to \infty$, $b/\bar \mfu_s$ are bounded in $T$ for fixed total
  mass path both from above and below as long as $s \in [\delta,t]$
  for some $0<\delta < t < \infty$. Hence we have the \emph{compact
    containment} for fixed time $t$ for such times $s$. As we let
  $\delta \to 0$ we have divergent coalescence rates and the state of
  the genealogy converges to the unit tree $\mfe$. Hence we have
  compact containment for fixed $s$ on the time interval $[0,t]$ for
  every $t < \infty$.

  Next we have to give the compact containment property for the
  complete $\U_1$-valued \emph{path} $(\wh \mfU_t^T)_{t \in [0,s]}$ in
  $[0,s]$ for all $T > 0$ and some arbitrary $s > 0$. For this purpose
  we use the path properties of the $\U$-valued Feller diffusion which
  allows to bound the number of ancestors at some depth $\ve > 0$ over
  a bounded time interval where the total mass path is bounded from
  above and hence the resampling rate driving the process $\wh \mfU^T$
  conditioned on $\bar \mfU^T$ is bounded from below. Hence we need to
  see that in a $\U_1$-valued Fleming-Viot process the number of
  $\ve$-ancestors is stochastically bounded from above independently
  of the resampling rate $d \geq d^\ast$ with $d^\ast>0$, i.e.\ the
  states are in $\U_{\mathrm{comp}}$. This is well-known; see
  \cite{GPWmp13}.

  In order to show convergence we first note that it is well known
  that as $T \to \infty$ the total mass process converges to the
  Feller diffusion with immigration at rate $b$; see Section~4 in
  \cite{Lamb07}.

  To show the convergence of the genealogy $\wh \mfU^T$ in path space
  note that the coefficients in the \emph{operator} depend on $T$ via
  the total mass process only. More precisely the resampling rates are
  given by $b(\bar\mfu(t))^{-1}$ at time $t$. These varying rates can
  be viewed as a time change. This time change should converge in
  $\norm{\cdot}_\infty$-norm as a function of time. Therefore we can
  work with the convergence of the total mass process in \emph{path
    space} and the convergence of the solution to the claimed
  conditioned martingale problem for a fixed path; recall
  \eqref{e3425} and \eqref{e4330}. In order to combine both facts we
  need \emph{uniform continuity} of the conditioned law in the set of
  all total mass paths.

  The continuity of the conditional process for a given total mass
  path is evident. In order to obtain continuity \emph{uniformly} in
  the total mass paths we note that the operator of the martingale
  problem has a coefficient continuous in $\bar \mfu$, if the path is
  above some $\varepsilon > 0$. Hence, we have to take care of small
  values of the total mass process. We note that the process
  $\bar \mfU$ does not hit $0$ and is therefore bounded away from zero
  in any \emph{bounded} time interval in $(0,\infty)$. Since on the
  other hand the $\U_1$-part of the state converges to $\mfe$ as
  $t \downarrow 0$, we have the uniform continuity.
\end{proof}

We continue by first proving (c) and then (b).

\begin{proof}[Proof of Theorem~\ref{T:KOLMOGOROVLIMIT}~(c)]
  We prove first \emph{tightness} of the l.h.s.\ of \eqref{tv17} and
  then show \emph{convergence} of the ``\emph{moments}'' to finally
  conclude convergence in law.

  \medskip
  \noindent
  \emph{Tightness:} We need to check the following three points
  (i)-(iii) according to the standard tightness criteria on $\U_1$
  resp.\ $\U$ (from \cite{GPW09,Gl12}), see Section~B.1 in
  \cite{infdiv} for details.

  \medskip
  \noindent
  \textbf{(i)} Tightness of \emph{masses} follows by KY-limit
  theorem for the total mass process, which is a Feller diffusion and
  states convergence of the law of the scaled mass to an exponential
  distribution.

  \medskip
  \noindent
  \textbf{(ii)} Tightness of \emph{distances} follows from the
  construction since (after the scaling) distances are bounded by
  $1+t^{-1} r_0$.

  \medskip
  \noindent
  \textbf{(iii)} Tightness of \emph{modulus of mass distribution}
  requires more work. Here we have to control the \emph{number of
    ancestors} which contribute to at least \emph{fraction
    $1- \varepsilon$} of the total population size.

  To this end, we can make use of the stochastic representation of the
  state at any fixed time $s$ via the concatenation of independent
  sub-families from the \Levy{}-Khintchine representation of
  \eqref{e1402} and \eqref{e1403} in combination with the KY-limit law
  for the respective conditioned total mass process. More precisely,
  we take the depth $th$ for any fixed $0<h<1$ and consider the family
  decomposition in $2ht$-balls and the corresponding masses. We need
  to show that the fraction $(1-\ve)$ of the total mass is contained
  in a finite number of the largest balls uniformly as $t \to \infty$.
  To see this we argue as follows. According to the \Levy{}-Khintchine
  representation we have a Poisson distributed number of such
  $2ht$-balls and by \eqref{e1402} the parameter of the distribution
  is given as
  \begin{align}
    \label{e3504}
   2(bt(1-h))^{-1} Y_{t-ht}.
  \end{align}
  As $t \to \infty$ this converges according to the KY-limit law. The
  limiting distribution is an exponential distribution with parameter
  $1$. This completes the proof of (iii).

  \medskip
  \noindent
  \emph{Convergence:} The first step to conclude the argument is to
  identify the ``moments'', i.e.\ functionals of sampled finite
  subtrees, of the limit in a tangible way. The strategy is to obtain
  for all $h>0$ the information on the family decomposition at depth
  $h$ after the rescaling. The corresponding scaled masses can be
  identified as well as their distances.

  We note that from \eqref{e3504} we can identify the Cox measure in
  the CPP-representation in the limit as $\Exp(1)$. Next we have to
  consider the corresponding laws on scaled random genealogies in
  $\U$, more precisely $\breve{\mfU}_t$ on $\U(h)^\sqcup$. Recall the
  notation $\Phi_h^{m,\varphi}$ and $\varphi_h$ from
  \eqref{eq:pol:trunc}. Note that, as $t\to\infty$ (and denoting by
  $\sim$ asymptotic equivalence for $t\to\infty$) we have using the
  asymptotics of the extinction probability beyond time $t$
  \begin{align}
    \label{e1904}
    \E\bigl[\Phi_h^{m,\varphi}(\breve{\mfU}_{t}) | \bar{\mfU}_t >0 \bigr] =
    \frac{1}{\P(\bar{\mfU}_t>0)} \E \bigl[ \Phi_h^{m,\varphi}(\breve{\mfU}_t)
    \ind{\bar\mfU_t>0} \bigr] \sim \frac{bt}{2} \,
    \E \bigl[\Phi_h^{m,\varphi}(\breve{\mfU}_t) \bigr].
  \end{align}

  We next want to use duality \eqref{tv12} to rewrite the r.h.s.\ of
  \eqref{e1904}. For the version of $\varphi$ acting on distances
  rescaled by the factor $t^{-1}$ we write $\varphi^{(t)}$, i.e.\ we
  set $\varphi^{(t)}(\dr)=\varphi\bigl(t^{-1} \,\dr\,\bigr)$.
  Furthermore we use notation $\varphi_h$ from \eqref{eq:pol:trunc}.
  Note that when $\varphi$ and therefore $\varphi_h^{(t)}$ depends on
  $m$ coordinates it is enough to consider the dual coalescent
  starting with $m$ partition elements enriched with a metric on $m$
  points. We write $[m]$ to denote the partition of $\{1,\dots,m\}$
  into singletons. Recall in \eqref{tv11} and \eqref{tv12} the duality
  function $H$ evaluated in $([m],\dr',\varphi)$ for a function
  $\varphi$. With the notation from there we can write the r.h.s.\ of
  \eqref{e1904} using the Feynman-Kac duality as
  \begin{align}
    \label{e1904c}
    \begin{split}
      \frac{bt}{2} \, & \E_{([m],\dr')} \biggl[ \int t^{-1}
      \mu_0^{\otimes |p_t|} (\dx \underline{u}_{p_t})\,
      \varphi_h^{(t)} \left(\dr^{p_t}(\underline{u}_{p_t}) + \dr'
      \right) \exp \biggl(b \int_0^t \binom{\abs{p_s}}{2} \, \dx s
      \biggr)
      \biggr]\\
      & \eqqcolon bt \, V_m(t,\mfu_0, \varphi_h).
    \end{split}
  \end{align}
  Now we have on the r.h.s.\ the enriched coalescent evaluated with a
  scaled $\varphi$, but a \emph{reweighting} of the path by the
  exponential functional takes place here.

  The r.h.s.\ can now be calculated since in the first factor we have
  the coalescent (effectively scaled in its distance matrix via the
  evaluation) and as to the second term note that
  $\binom{\abs{p_s}}{2}$ is the rate of the exponential waiting time
  in backwards time, namely at time $s$ for the coalescent for a jump
  downwards by $1$. We can therefore write down the density for the
  successive waiting times for the coalescence events and the
  corresponding contribution of the exponential term at times before
  this jump and since the last jump namely $e^{b\binom{\abs{p_s}}2}$.
  The contribution is then asymptotically for $t \to \infty$ given by
  the $(n-1)$ fold integral from $s_{n-1}$ to $t$,$s_{n-2}$ to
  $s_{n-1},\dots, s_2$ to $s_1$ if there are $(n-1)$ jumps, which is
  the leading term. We observe that the distances are twice the
  coalescence time for two individuals in the coalescents cut at $2t$
  and then rescaled by $t^{-1}$. These explicit expression converges
  as $t \to \infty$ namely to the coalescence time in the $\varphi$
  and the time integral to the joint occupation time of all pairs.
  However we scale $\varphi$ by rescaling distances by $t^{-1}$. Hence
  asymptotic contributions arise on the event where we have the
  coalescences at some time of order $t$. Therefore (recall on the
  event of extinction the r.h.s.\ is $0$), and defining $C_m$ by
  choosing $L(\underline{1})=1$:
  \begin{align}
    \label{e1949}
    V_m(t,\mfu_0,\varphi) \sim C_m (bt)^{-1} L(\varphi) \text{ as } t
    \to \infty.
  \end{align}

  In order to be able to obtain $L(\varphi)$ we have to consider
  $\varphi$ of a specific form, in fact we may choose $\varphi^{(t)}$
  such that
  $\varphi^{(t)}(\dr)=\varphi\bigl(\,\underline{\underline{(r/t)}}\,\bigr)$
  and we may use $\varphi \in C_b([0,1]^n)$ here. The contributions to
  $C_m$ arise, as we saw on the event where all pairs do only coalesce
  at times $a_i \cdot t$ for some $a_i>0$ for $i=1,\ldots,\binom m 2$
  and here the $a_i$ have to be chosen such that they arise from
  successive coalescences. Hence we have to integrate over all
  possibilities, more precisely over the possibilities for $a_i$,
  inserting the probability density for these events which are based
  on i.i.d. exponential clocks. However a clock ringing makes other
  clocks redundant, namely coalescing with another one turns the own
  clocks into inactive. We get therefore in \eqref{e1949} for
  $\varphi$ invariant under permutation (note however we can w.l.o.g.
  assume this) of the sampled individuals
  \begin{align}
    \label{e3527}
    C_m = \left(\frac{1}{2}\right)^{m-1} b^{m-1} \; m!
  \end{align}
  and (integration is w.r.t.\ Lebesgue measure)
  \begin{align}
    \label{e5481}
    \begin{split}
      L(\varphi)= \int_{[0,1]} \cdots \int_{[0,1]}
      \varphi(\underline{\underline{a}}) \, \dx a_1 \ldots \dx a_m, \;
      & \text{ with } \underline{\underline{a}} = (a_{i,j})_{1 \leq i< j \leq m},  \\
      & \text{ and} \quad a_{i,j}=a_{j-i} \quad \text{for given} \quad
      a_1,a_2,\dots,a_{m-1}.
    \end{split}
  \end{align}

  The final point is to show that \emph{size-biased law} of the limit
  of the scaled $\U$-valued Feller diffusion conditioned to survive at
  time $t$ and observed at time $t$, equals the
  \emph{quasi-equilibrium} of $\breve\mfU^\dagger$ (see \eqref{e977}),
  which we know equals the one of $\breve\mfU^{\mathrm{Palm}}$; see
  Corollary~\ref{c.1609} for this fact. This means that we have to
  identify the limit as the claimed object, by showing the proper
  relation of the moments of the two objects, the limit of the scaled
  and conditioned to survive at time $t$ original process size-biased
  and the one conditioned to survive forever then taken in its long
  time limit $\U^\dagger$. Then in particular all finite subspaces
  generated by a sample of points from the population have
  \emph{different} laws the first has to be \emph{size-biased} to be
  equal to the other. This relation between $\U^\dagger$ and
  $\mfU^{\mathrm{Palm}}$ we have established in Corollary~\ref{c.1609}
  and we explain at the end of the proof of part (b) below how to
  obtain the claim.
\end{proof}

\begin{proof}[Proof of Theorem~\ref{T:KOLMOGOROVLIMIT}~(b)]
  Note that $\mfU^\dagger$ appears as solution of the martingale
  problem in Lemma~\ref{c.1310} which implies the compact containment
  condition on $[0,S]$ for every $S$. Together with the
  \emph{convergence of the coefficients} of the operator to the ones
  of the claimed operator, from which we have to conclude that the
  weak limit points are \emph{solutions} of the martingale problem.
  This follows since the compensators of the martingale problem
  converge, see (a), from the general theory; see e.g.\
  Lemma~5.1 in Chapter~4 in \cite{EK86}. Here we have to observe that
  the super-criticality enters only in the evolution of the total mass
  term in the generator so that the term $\bar \mfu^{-1}$ in the
  generator is compensated by the total mass terms and the only point
  here is that the generator maps polynomials not in bounded functions
  as required in the lemma. However we can use the extended form of
  the generator and consider the operator on $\mcD_1$; see
  Remark~\ref{r.1207} and Proposition~\ref{prop.1183}.

  Namely on the set $\mcD_1$ in \eqref{e1212} the operator maps into
  bounded functions. Then we see that the converging coefficients let
  the operators acting on $\mcD_1$ converge to the limit operator
  acting on $\mcD_1$. Now Lemma~5.1 in Chapter~4 in \cite{EK86} is
  applicable since these functions are still separating and hence we
  get the \emph{weak convergence} of the laws to a solution of the
  martingale problem on $\mcD_1$. Next observe that it solves also the
  $\Pi$-martingale problem as we see by approximation of
  $\Phi \in \Pi$ by elements of $\mcD_1$; see
  Section~\ref{sec:appr-solut-mart}. Therefore the limit is our
  process $\mfU^\dagger$ is identical to the one on $\mcD_1\cup \Pi$
  and hence is our process $\mfU^\dagger$.

  It remains to show the claim on the long time behavior of the limit
  dynamic. We have to show first the \emph{tightness} of its
  $t$-marginals as $t \to \infty$ and then the \emph{convergence}. We
  shall see now how to relate these two parts such that we can make
  use of the calculations already done.

  For the study of the behavior of polynomials we rewrite the
  expectation by absorbing the size-bias term into the polynomial by
  extending the $\varphi$ constant to a function of $(n+1)$-variables.
  Then the converging argument works just the final expression changes
  as claimed. This implies the convergence of all polynomials to a
  limit which is the size-biased law of the limit random variable we
  derived in \eqref{e1949}-\eqref{e5481} above. Therefore we have
  convergence to the claimed limit and we have the claimed relation
  between the two different limits.
\end{proof}

\begin{proof}[Proof of Lemma~\ref{c.1310}]\label{pr.lemma.c1310}
  Since existence was obtained above it remains to show uniqueness.
  Again, we can work with the conditioned martingale problem to get
  uniqueness from the uniqueness of the Fleming-Viot and the
  $\R_+$-valued diffusion. The details are similar to those from
  Section~\ref{ss.pr.thmartu}.
\end{proof}

\subsection{Proofs of Propositions~\ref{pr.palm1},~\ref{pr.palm2}}
\label{ss.prp1p2}

\begin{proof}[Proof of Proposition~\ref{pr.palm1}]
  \label{pr.4188}
  Let $ \mcS_t^h $ be the semigroup of the $h$-transform (with $h$
  satisfying certain conditions) which is given by
  \begin{align}
    \label{h-transf}
    \mcS_t^h (\Phi^{n,\varphi}(\mfu))
    = \frac{1}{h(\mfu)} \mcS_t\bigl(\Phi^{n,\varphi} h\bigr) (\mfu).
  \end{align}
  Then to get the generator we need to compute
  \begin{align}
    \label{h-transf-a}
      \frac{\dx}{\dx t} \mcS_t^h \Phi^{n,\varphi}(\mfu) =
      \frac{1}{h(\mfu)} \Omega^\uparrow \mcS_t
      \bigl(\Phi^{n,\varphi}h\bigr)(\mfv) \vert_{\mfv=\mfu}
       = \frac{1}{h(\mfu)} \mcS_t \Omega^{\uparrow}
      \bigl(\Phi^{n,\varphi}h\bigr)(\mfv) \vert_{\mfv=\mfu}
  \end{align}
  at $t=0$.

  Apply this now to $h: \U\setminus \{\ntree\} \to (0,\infty)$,
  $h(\mfu) = \bar\mfu$. The computations are analogous to those in
  \eqref{eq:r.742.1}--\eqref{eq:12c} but note that here we have to
  work with the object $\Phi^{n,\varphi} (\mfv) h(\mfv)$ which we
  write as $\bar\Phi(\bar\mfv) \wh\Phi^{n,\varphi}(\hat\mfv)$ with
  $\bar \Phi (\bar\mfv) = \bar\mfv^{n+1}$. In particular we have here
  $\Phi^{n,\varphi} (\mfv) = \bar\mfv^n \wh\Phi^{n,\varphi}
  (\hat\mfv)$. We obtain
  \begin{align}
    \label{h-transf-b}
    \begin{split}
      \Omega^{\uparrow} \bigl(\Phi^{n,\varphi}h\bigr)(\mfv)
      & = \wh \Phi ^{n,\varphi}(\hat\mfv) \Omega^{\mathrm{mass}} \bar \Phi
      (\bar\mfv) + \bar \Phi (\bar\mfv) \bigl(\frac{b}{\bar\mfv}
      \Omega^{\uparrow, \mathrm{res}} \wh \Phi^{n,\varphi}(\wh\mfv) +
      \Omega^{\uparrow, \mathrm{grow}} \wh \Phi^{n,\varphi}(\wh\mfv) \bigr) \\
      & = \wh \Phi ^{n,\varphi}(\hat\mfv) \frac{b\mfv}2 (n+1)n
      \bar\mfv^{n-1} + \bar \Phi (\bar\mfv) \bigl(\frac{b}{\bar\mfv}
      \Omega^{\uparrow, \mathrm{res}} \wh \Phi^{n,\varphi}(\wh\mfv) +
      \Omega^{\uparrow, \mathrm{grow}} \wh \Phi^{n,\varphi}(\wh\mfv) \bigr)\\
      & = \frac{(n+1)n}2 b \Phi ^{n,\varphi}(\mfv) + b \bar\mfv^{n}
      \Omega^{\uparrow, \mathrm{res}} \wh \Phi^{n,\varphi}(\hat\mfv) +
      \bar\mfv \Omega^{\uparrow, \mathrm{grow}}
      \Phi^{n,\varphi}(\mfv).
    \end{split}
  \end{align}
  Now, by \eqref{mr3aexta} we have
  \begin{align}
    \label{h-transf-c}
    b \bar\mfv^{n} \Omega^{\uparrow, \mathrm{res}} \wh
    \Phi^{n,\varphi}(\hat\mfv) = \bar\mfv \Omega^{\uparrow,
    \mathrm{bran}} \Phi^{n,\varphi}(\mfv) - b \frac{n(n-1)}2
    \Phi^{n,\varphi}(\mfv).
  \end{align}
  Plugging this in the last line of \eqref{h-transf-b} and simplifying,
  we obtain
  \begin{align}
    \label{h-transf-ba}
    \Omega^{\uparrow} \bigl(\Phi^{n,\varphi}h\bigr)(\mfv)
    = n b \Phi^{n,\varphi}(\mfv) +
    \bar\mfv \Omega^{\uparrow,
    \mathrm{bran}} \Phi^{n,\varphi}(\mfv)
    + \bar\mfv \Omega^{\uparrow, \mathrm{grow}}
    \Phi^{n,\varphi}(\mfv).
  \end{align}
  Finally, plugging this in \eqref{h-transf-a}, evaluating it with
  $\mfv=\mfu$ and $t=0$ we arrive at \eqref{e1372}.
  % \begin{align}
    % \label{h-transf-bb}
    % \Omega^{\uparrow,\mathrm{Palm}} \; \Phi^{n,\varphi} (\mfu)
    % = \frac{n b}{\bar\mfu} \Phi^{n,\varphi}(\mfu) +
    % \Omega^{\uparrow, \mathrm{bran}} \Phi^{n,\varphi}(\mfu)
    % + \Omega^{\uparrow, \mathrm{grow}} \Phi^{n,\varphi}(\mfu).
  % \end{align}
\end{proof}

\begin{proof}[Proof of Proposition~\ref{pr.palm2}]
  This is standard since $C([0,\infty), \U)$ is a Polish space, since
  $\U$ is a Polish space, \cite{EK86}.
\end{proof}

\subsection{Proof of Theorem~\ref{T.1205}}
\label{ss.proth67}
\begin{proof}[Proof of Theorem~\ref{T.1205}]
  Here we have to generalize the classical Kallenberg decomposition of
  the Palm law of the $\R_+$-valued Feller diffusion to the
  $\U$-valued case. The key is the \Levy{}-Khintchine formula again,
  but the $\U$-valued one. For detail we refer to \cite{infdiv} and to
  Section~\ref{branchmarkcox} here, where the formula is recalled. The
  first observation is that the \emph{Cox measure} is identical for
  both cases due to the result on the total mass process, therefore we
  have to focus on showing that the law $\varrho_h^t$ on
  $\U(h)^\sqcup$ fits. We show that we get the concatenation of the
  terms of the Feller diffusion, i.e.\ a Poisson number of entrance
  laws from $\ntree$ surviving up to the current time, and of the term
  given by the entrance law from state $0$ of the entrance law of the
  Palm process for the $\U$-valued Feller diffusion.

  The \Levy{}-Khintchine representation gives the $h$-top of the state
  of the $\U$-valued Feller diffusion as concatenation of i.i.d.\
  trees where the number of summands is Cox-distributed with
  Cox-measure $\mcL [\bar \mfU_{t-h}]$. Then it is a general fact, see
  \cite{AGK19}, that the size-biased distribution has the size-biased
  Cox measure and size-biased Poisson numbers $(=1+ \Pois(\lambda))$
  of elements in the concatenation and the additional \emph{special
    summand} has the size-biased distribution. Therefore we need the
  $h$-transformed $\U$-valued Feller process started at the zero-tree
  as the independent additional part as claimed.
\end{proof}

\subsection{Proofs of Theorems~\ref{TH.1061},~\ref{TH.IDENTKALL}:
  \texorpdfstring{$\U^V$}{UV} and \texorpdfstring{$\U$}{U}-valued
  Feller with immigration}
\label{ss.pr.th1061}

To start with we have to derive the \emph{generator} from the
description of the dynamic of the $\U$-valued version of Evans' tree
described in Section~\ref{sss.1616}, a result we had stated there in
Corollary~\ref{l.1913} and for which Lemma~\ref{l.0850} was the basis.

\begin{proof}[Proof of Lemma~\ref{l.0850}]\label{pr.5607}
  We will calculate based on our description of the mechanism for
  every process following our description (below we shall show the
  existence of such a process) the second term in \eqref{e1068}, by
  considering the \emph{effect on the polynomial}, given an excursion
  of the total mass from the immortal line starts at some time point
  $s$ in $[t,t+\Delta]$ that is in $\bar\mfU^{\ast,+}$ a new color $s$
  starts evolving as a color-$s$ diffusion, which exists thanks to
  results by Evans (Theorems~2.7-2.9 in \cite{Evans93}), and survives
  until time $t$.

  We focus on the part $\wh \Phi^{n,\varphi,g}$ of the polynomial,
  which is the new part here, and we determine the \emph{intensity} in
  time at which this excursion occurs. Note the intensities and laws
  of such excursions are measures on $C ([0,\infty), \U^V)$, i.e.\ on
  path space.

  Namely we want to argue that we can calculate here again as for the
  Feller diffusion in \eqref{eq:r.742.1} as follows:
  \begin{align}
    \label{e2774}
    \Omega^\uparrow(\bar \Phi \wh \Phi)=\left(\Omega^{\mathrm{mass}} \bar
    \Phi \right) \wh \Phi +\bar \Phi \wh \Omega^{\mathrm{gen}} \wh \Phi,
  \end{align}
  where $\Omega^{\mathrm{mass}}$ is the operator of the mass process
  which is the diffusion from \eqref{e1116} and
  $\wh \Omega^{\mathrm{gen}}$ is the operator acting on the functions
  describing $\wh \mfU^{\ast, +}$. In other words we have again the
  \emph{product rule}. This we explained earlier below \eqref{e1212}
  for the Feller diffusion, and which works here completely analog for
  the evolution of the colors $s \leq t$ and we have to handle the
  new incoming ones. In other words we have next to calculate the
  immigration operator using the construction via the sliding
  concatenation of processes in \eqref{e1963}.

  What we need is that the creation of mass of color $s$ acts on the
  product of $\bar \Phi \; \wh \Phi$ according to the product rule. We
  note here furthermore that the immigration increases the \emph{mass
    of a color $s$} at time $s$, that is its sole effect with two
  components changing the total mass and the relative weights of
  colors. By an approximation with Galton-Watson processes (see
  Remark~\ref{r.697}, where this is described and put there $c=b$) a
  generator calculation shows this property, cf.\ \cite{Gl12}.

  The next fact needed is that the total mass process changes
  \emph{autonomously} as \emph{Markov process}, namely as a diffusion
  process given by the solution of the SDE \eqref{e1116} and similarly
  the projection on the mark space as autonomously evolving
  measure-valued Markov process. Therefore we can calculate the action
  of the generator on functions $\Phi^{n,\varphi,g}(\mfu)$ as in
  \eqref{1076}. We obtain from the first order term in \eqref{e1116}
  of the mass process a term (this can be seen using a particle
  approximation similarly to \eqref{eq:r.742.1})
  \begin{align}
    \label{eq1220a}
    b n \bar\mfu^{n-1} \wh\Phi(\hat\mfu) = n \frac{b}{\bar\mfu} \Phi(\mfu)
  \end{align}
  and from the second order term
  \begin{align}
    \label{eq1220b}
    \frac12 b n(n-1) \bar\mfu^{n-1} \wh\Phi(\hat\mfu) = b
    \frac{n(n-1)}{2\bar\mfu} \Phi(\mfu).
  \end{align}

  It remains to obtain the action of the generator on the
  $\wh \mfU$-part, i.e.\ genealogy, giving the term
  $\bar \Phi \wh\Omega^{\uparrow, \ast} \wh \Phi$ here we need in
  particular the part of
  $\wh \Omega^{\uparrow,\ast}_{V,\mathrm{imm},t}$, which gives an
  influx of the color $s$ and therefore changes the relative frequency
  of the colors, which means that it only acts via an action on $g$ in
  $\wh \Phi$. A somewhat lengthy calculation allows to explicitly
  calculate the generator (compare \eqref{e2073} and \eqref{e2091r},
  \eqref{e2079} and \eqref{e2343} in the proof section). We arrive
  with $g \in C^1_b([0,\infty))$ at the formula \eqref{e1085a} below
  for the generator action.

  We will now develop a representation of the increment arising from
  the \emph{immigration} term, which then allows us to calculate the
  generator. We obtain contributions on a small time interval if an
  immigrant starts a population surviving for some time. We need the
  \emph{intensity} and the \emph{effect} of these increments arising.

  We begin calculating the \emph{intensity of successful immigration}.
  If we think of the immigration in the interval $[0,t]$ we can put
  mass $\ve$ in the beginning of an interval of length $\ve$ into the
  system and observe at time $t$. Now let $\ve \downarrow 0$ to get
  our process. Then we see that we obtain for the masses the excursion
  measure of Feller diffusion. Hence we have to consider here
  excursions from the zero mass which start between times $t$ and
  $t + \Delta$ and which survive until time $t + \Delta$. This means
  for $s \in [t,t+\Delta]$ there are \emph{excursions} starting from
  $0$ which last beyond time $t + \Delta$. We need an \emph{intensity
    measure} on $[t,t + \Delta]$ for these $s$-excursions (and later
  the \emph{effect} of the \emph{added concatenated element}).

  Start by the \emph{intensity} in the $\R_+$-valued object and denote
  by (this is the $\Lambda_s^b$ in \cite{PY82})
  \begin{align}
    \label{e2005}
    \begin{split}
    \bar P_{0;s}   & \text{ the \emph{excursion law} (from $0$) of
        the Feller diffusion on } \R_+ \text{ with parameter } b, \\
      & \text{ the excursion starting at time } s \text{ from } 0.
    \end{split}
  \end{align}
  Furthermore we define $\bar P_{0;s,t}$ as the probability measure on
  paths in $C(([s,t],[0,\infty))$ given via $\bar P_{0;s}$ by the
  restriction of the latter to paths visible in the interval $[s,t]$,
  More precisely we set
  \begin{align}
    \label{e1751}
    \bar P_{0;s,t}(\cdot)=\bar P_{0;s}(\cdot \cap \{\bar \mfu_t>0 \})
    /\bar P_{0;s}(\{\bar \mfu_t>0\}).
  \end{align}
  Next we consider the $\U_1$-valued part for which we need a
  generalization of the law introduced in \eqref{e1435}. We denote by
  $\wh P^{\bar \mfu}_{\mfe;s,t}(\,\cdot\,)$ the kernel on
  $C([s,t],[0,\infty)) \times C([s,t],\U_1)$ describing the law of the
  pure genealogy part in the interval $[s,t]$ conditioned on total
  mass path $\bar\mfu$. We observe that using the \emph{conditional
    duality} for this excursion specified in \eqref{e892}, for a given
  path of the total mass we have the conditional duality which
  determines uniquely a law on $\U^V_1$ and depends measurably on
  $\bar \mfu$. Therefore, the \emph{conditional duality} gives us the
  transition kernel generating the law on paths from $s$ to $t$ that
  we are looking for.

  Now we calculate the corresponding $\U$-valued object which arises
  from combination of the parts described above. Recall that the
  projection of $P_{\ntree;s;t}$ on $\R$ equals $\bar P_{0;s,t}$. It
  follows that the \emph{excursion law} on paths with values in
  $\R_+ \times \U_1$ running from $s$ to $t$ is of the form:
  \begin{align}
    \label{excQst}
    P_{\ntree;s,t} (\dx\bar\mfu,\dx\hat \mfu) = \bar P_{\ntree;s,t}
    (\dx \bar\mfu)
    \otimes \wh P^{\bar \mfu}_{\mfe;s,t} (\dx\hat \mfu).
  \end{align}

  Finally we need the \emph{intensity measure}
  $Q_{s,t + \Delta}(\cdot)$ of such an $\U$-excursion for
  $s \in [t,t + \Delta]$ contained in the general object, which we
  have specified in \eqref{e1377}. To this end, we consider the
  interval $[t,t+\Delta]$ and the colors in that interval. Then we
  have a random subset $I^t_\Delta \subseteq [t,t+\Delta]$ of points
  $s$ in which an excursion of the colors $s$ starts and reaches time
  $t+\Delta$. It is convenient to scale the sets to the interval
  $[0,1]$, that is to consider $\wt I_\Delta^t \subseteq [0,1]$ the
  set of points $s \in [0,1]$ such that $t + s \Delta \in I^t_\Delta$.
  For each $t$ we obtain a point process $\wt I_\Delta^t$ on $[0,1]$
  whose law is independent of $t$. We denote the generic point process
  with this law by $I_\Delta$. It is well-known that $I_\Delta$ is an
  inhomogeneous PPP on $[0,1]$ with \emph{intensity measure} (for the
  calculation see for example Section~\ref{sec:branching}, Proof of
  Theorem~\ref{T:BRANCHING} part (b) which gives this):
  \begin{align}
    \label{e2079}
    \frac{2}{(1-s) \Delta}.
  \end{align}

  Next we come to the \emph{effect} of the excursion on the
  polynomials. The state $\mfU^{\ast,+}_{t + \Delta}$ at \emph{time
    $t+\Delta$} has a \emph{$\Delta$-top} which can be written in the
  form
  \begin{align}
    \label{e2073}
    \lfloor \mfU_{t+\Delta}^{\ast,+}\rfloor (\Delta)=
    \lfloor \prescript{\leqslant t}{}\mfU'_{t+\Delta} \rfloor (\Delta)
    \sqcup^\Delta \bigl(\prescript{>t}{}\mfV_\Delta^{\ast,+}\bigr).
  \end{align}
  Here we denote by
  $\lfloor \prescript{\leqslant t}{}\mfU'_{t+\Delta}\rfloor (\Delta)$
  the population of colors $\leq t$ evolved further with the
  \emph{Feller dynamic} from their initial time up to time $t+\Delta$
  and by $\prescript{> t}{}\mfV_\Delta^{\ast,+}$ the population of
  colors $>t$ evolved up to time $t+\Delta$ further with the
  \emph{$(\ast,+)$-dynamic} from the time $t$ state. The part
  $^{> t}\mfV_\Delta^{\ast,+}$ in \eqref{e2073} arises in distribution
  as sliding concatenation of independent processes of populations
  with one color
  \begin{align}
    \label{e2083}
    \bigl\{\prescript{s}{}\mfV_{\Delta,s} : s \in I_\Delta \bigr\}.
  \end{align}

  Then the elements of the family in \eqref{e2083} are independent
  processes, their mark is $s \Delta$ and the genealogy part is a
  version of, forgetting the mark, the process in \eqref{e2087a}:
  \begin{align}
    \label{e2087}
    \prescript{s \Delta}{}\mfU_{\Delta,s\Delta}^{\Delta(1-s)}
  \end{align}
  marked for all times with one mark namely $s\Delta$. The sliding
  concatenation of those elements gives us $^{> 0}\mfV_\Delta$:
  \begin{align}
    \label{e2091}
    \prescript{> 0}{} \mfV_\Delta \stackrel{d}{=}
    \mathop{{\bigsqcup}^{\mathrm{sli}}}_{s \in I_\Delta}
    \enspace \prescript{s \Delta}{}\mfU_{\Delta,s\Delta}^{\Delta(1-s)}.
  \end{align}
  Now it is suitable to rewrite the equation \eqref{e2091} above as
  \begin{align}
    \label{e2091r}
    \prescript{> 0}{}\mfV_\Delta = \prescript{> 0\mkern-5mu}{}{\Bigl(\bar
    \mfV_\Delta,\wh\mfV_\Delta \Bigr)} \stackrel{d}{=}
    {\mathop{{\wt\bigsqcup}^{\mathrm{sli}}}_{s \in I_\Delta}} \;
    \prescript{s \Delta\mkern-5mu}{}{\Bigl(\bar\mfU_{\Delta,s\Delta}^{\Delta(1-s)},
    \wh\mfU_{\Delta,s\Delta}^{\Delta(1-s)}\Bigr)},
  \end{align}
  where the action of the ${\wt\bigsqcup}^{\mathrm{sli}}$ operator on
  the first component is addition of mass and the action on the second
  is ${\bigsqcup}^{\mathrm{sli}}$ as defined in \eqref{e1963}.

  In the spirit of the notation in Theorem~\ref{TH.MARTU} we denote by
  \begin{align}
    \label{e2087a}
    (\mfU_{t,t_0}^{T})_{t\in [t_0,T]}
  \end{align}
  the $\U$-valued Feller diffusion with starting time $t_0$
  conditioned to survive up to time $T$. In this notation the process
  in Theorem~\ref{TH.MARTU} would be written as
  $(\mfU_{t,0}^{T})_{t\in [0,T]}$.
  We also need the process where \emph{no new colors} appear after
  time $t$, call this $^{\leq t}\mfU^{\ast,+}=(^{\leq
    t}\mfU^{\ast,+}_s)_{s \ge t}$ which coincides with
  $\mfU^{\ast,+}_s)_{s \ge t}$ projected on the population with marks
  in $[0,t]$.

  Now we use this to calculate the effect on a polynomial of running
  time by $\Delta$ and adding new immigrants further using the above
  representation, which is denoted and given by, using independent
  copies in the concatenation below, where $\tau_t$ shift colors by
  $t$:
  \begin{align}
    \label{e-wtE}
    E (t,t+\Delta; \Phi^{n,\varphi,g}) = \E\bigl[\Phi^{n,\varphi,g}
    (^{\leq t}\mfU^{\ast,+}_{t+\Delta} \sqcup^{t,\Delta} (\tau_t \circ ({}^{>
    0}\mfV_\Delta))) - \Phi^{n,\varphi,g} (\mfU^{\ast,+}_t)\bigr].
  \end{align}

  Here, the $\sqcup^{t,\Delta}$-concatenation is a modification of the
  concatenation from \eqref{e.tr47} which now takes colors into
  account. More precisely, the modification concerns the distances of
  elements of $^{\leq t}\mfU^{\ast,+}_{t+\Delta}$ and
  $(\tau_t \circ ({}^{> 0}\mfV_\Delta))$ between each other. Let
  $(u_1,s)$ for $s \le t$ be an element of a representative of
  $^{\leq t}\mfU^{\ast,+}_{t+\Delta}$ and let $(u_2,s')$ for
  $s' \in [t,t+\Delta]$ be an element of a representative of
  $(\tau_t \circ ({}^{> 0}\mfV_\Delta))$, then their distance in
  the $\sqcup^{t,\Delta}$-concatenation at time $t+\Delta$ is given by
  \begin{align}
    \label{grx53dt}
    2 (t+\Delta -s).
  \end{align}

  Next we express the \emph{difference quotient} for the expectation
  of the polynomial by the r.h.s.\ of \eqref{e-wtE} in terms of the
  mass and genealogical quantities on the r.h.s.\ of \eqref{e2091r}.

  Return to \eqref{e2091r} and analyze the r.h.s. One ingredient we
  need is $\{\mathcal L (M_{t,s}) : t\ge s\}$ the entrance law of the
  Feller diffusion from state $0$ at time $s$ observed at time $t$. We
  want to condition this to survive up to time $T$. Then the
  corresponding conditioned entrance law is denoted by
  $\{\mathcal L (M_{t,s}^{T}): t\ge s\}$ and for $T=t$ we have
  \begin{align}
    \label{e2343}
    \mathcal L (M_{t,s}^t) = \Exp\bigl((t-s)\bigr).
  \end{align}

  This can be obtained by adapting (3.3) in \cite{LN68}. (Note that
  the formula there contains a typo: in the case $x=0$ the factor
  $t^2$ should rather be $t^{-2}$.) To construct the process we use
  that we have the colored Feller diffusion without immigration by
  simply giving each individual one and the same mark and we have to
  add now the effect of immigration.

  We know that excursions starting at time $s$ and surviving up to
  time $t$ have the intensity $2(t-s)^{-1}$; see \eqref{e2079}.
  Therefore we can now calculate the expected effects using
  \eqref{e-wtE} and \eqref{e2091r}. For the calculation it is more
  convenient, to use the time-homogeneous formulation. Namely we
  consider \eqref{e2532}.

  Observe that in $\prescript{\leq t}{}{\mfU}^{\ast,+}_{t + \Delta}$ the colors
  are just inherited otherwise no change occurs while in
  $\prescript{> 0}{}{V}_\Delta$ new colors immigrate and no old colors are
  there, hence the $\leq t$ population appears only in the first part
  of the concatenation. Therefore the difference between $\leq t$
  populations and $>t$ populations are sitting in the different parts
  of the concatenation.

  In \eqref{e-wtE} the difference arises from both terms of the
  concatenation and we get the $\frac{\partial}{\partial t}$ term as
  well. We get the following three terms. Let $\Omega^\uparrow_V$
  denote the $\U^V$-valued Feller diffusion with inheritable marks in
  $V$. Let furthermore the immigration operator be written formally as
  in \eqref{e1085a}, then we get using first \eqref{e2532}, then
  inserting \eqref{e1068}, and replacing with \eqref{e2070}:
  \begin{align}
    \label{e5841}
    \Psi' (t) \Phi^{n,\varphi,g}(\mfu) + \Psi (t) \Omega_V^{\uparrow,+}
    \Phi^{n,\varphi,g}(\mfu) + \frac{bn}{\bar\mfu}
    \Psi (t) \left[\Phi^{n,\varphi,\wt g}(\mfu)  + \Psi (t)
    \Phi^{n,\varphi,\wt g}(\mfu)\right],
  \end{align}
  where $\wt g=\sum_{i=1}^n \; g_i$ recall \eqref{e1085a}.

  In order to make sense out of the formal expression concerning the
  immigration operator we recall Remark~\ref{r.2530}, where we saw we
  should treat the explicit time and marks in $V$ together without
  product form and choose $g(t,v)$ as in the statements of the lemma.
\end{proof}

\begin{proof}[Proof of Theorem~\ref{TH.1061}]\label{pro.5348}
  We now have to prove the \emph{well-posedness} of the martingale
  problem, first the existence and then the uniqueness.

  We observe for \emph{existence} that the original description of the
  dynamics allows to construct the finite dimensional distributions of
  the stochastic process based on a Poisson point process and
  independent copies of conditioned on survival marked
  $\U^{[0,\infty)}$-valued Feller diffusions. This gives the state at
  time $t$ by the sliding concatenations (see \eqref{e1963}) and
  determines the potential transition kernels. We showed above that
  the resulting object would have to solve the martingale problem with
  the operator we derived in the proof of Lemma~\ref{pr.5607} above
  where we are starting in the state
  $[\{1\} \times \{0\}, \underline{\underline{0}},\delta_{(1,0)}]$.

  We have to show that the constructed state indeed defines a
  transition kernel, i.e.\ satisfies the Chapman-Kolmogorov equations.
  Here, the first ingredient is that the measure-valued, i.e.\
  $\mathcal{M}_{\mathrm{fin}}([0,\infty))$-valued Evans process exists
  indeed as a process with a.s.\ continuous path.

  This can be derived from Evans theorem, Theorems~2.7.-2.9. in
  \cite{Evans93}. Here to apply the results one needs to set up a
  motion which allows to distinguish the populations entering at time
  $s$ from the immortal line. One way is: we let the mark move to at a
  speed depending on $s$, which is strictly decreasing but remains
  positive. This can be recoded by a one-to-one mapping to match our
  process.

  The second ingredient is that conditioned on that measure-valued
  process the $\U_1^{[0,\infty)}$-valued process
  $(\wh\mfU_t^T)_{t\in [0,T]}$ exists as time-inhomogeneous
  Fleming-Viot process as we have shown before. This defines the
  transition kernel in the zero element. Then we argue that we can
  ``glue two pieces together'' to get the kernel with the general
  starting point and hence the general finite dimensional
  distribution.

  The above claim follows when we show that we can concatenate the
  special evolution with an initial state in $\U$ of the form we allow
  here. This means given the evolution starting in $\mfe$ and an
  element $\mfu' \in \U_{\mathrm{imm}}$ which are enriched by the
  color $s \leq 0$ a state we call $\mfu$ we want to obtain the state
  of the process at time $t$ starting in $\mfu$. This means we have to
  glue together at time $t$ the process we constructed above together
  with a $\U$-valued Feller diffusion starting in $\mfu'$ where colors
  are just inherited, and the process with colors $s \leq 0$, denoted
  $(\wt \mfU_t^{\ast,+})_{t \geq 0}$. The latter is easily constructed
  using the branching property, since every color is just attached to
  all descendants in the $\U$-valued Feller diffusion starting in the
  mass from that color. Then, as a candidate for the solution we
  consider
  \begin{align}
    \label{e6014}
    \wt \mfU_t^{\ast,+} \vee \mfU^{\ast,+}_t.
  \end{align}
  Here, the operation $\vee$ means the basic set is the disjoint union
  of $\wt U$ and $U$, the metric $\wt r \vee r$ is an extension to the
  joint union with the property that it coincides with $\wt r$ on
  $\wt U \times \wt U$, with $r$ on $U \times U$, and on
  $\wt U \times U$ it is twice the color difference. Note, that this
  is a certain extension of $\sqcup^t$ to a \emph{marked} genealogy.

  We have to check the \emph{Chapman-Kolmogorov equations}. The
  process $\mfU^{*,+}$ which we obtain by conditioning on the Markov
  measure-valued process $\bar\mfU^{*,+}$ is itself a Markovian time
  inhomogeneous $\U_1^{[0,\infty)}$-valued Fleming-Viot process. We
  have to construct this process construct for a given measure-valued
  path to conclude the existence proof. To this end, we have to
  construct a collection of independent time-inhomogeneous
  Fleming-Viot process each of which is $\U^{\{s\}}_1$-valued for some
  $s \in [0,\infty)$ and has as resampling rate
  $\nu_t(\{U_t \times \{s\})$ where $\nu_t$ comes from a realization
  of a path of $\bar \mfU^{\ast,+}$. The resulting process is then a
  $\U_1^{\{s\}}$-valued strong Markov process starting in the element
  $\mfe$ of $\U_1$. This process is well-defined for each
  $s \in [0,\infty)$. We define the process at a time $t$ as sliding
  concatenation over $s \in \mcI_t$ for the given path
  $(\mcI_t)_{t \geq 0}$, yielding altogether a
  $\U_1^{[0,\infty)}$-valued process, which is by the independence of
  the components a Markov process with continuous paths. The subtlety
  here is that for $\varepsilon >0$,
  $\mcI_t \setminus [t-\varepsilon,t]$ is a finite set, but $t$ is a
  real limit point and $|\mcI_t|=\infty$. Fix a specific time $T$
  first. In order to use the infinite concatenation as definition at
  time $T$ we need that the path of the Fleming-Viot process of a
  fixed color is continuous at $T$ and the masses and diameters of
  colors close to $T$ converge to zero. The total mass at time $T$ is
  finite. Therefore we obtain a limit of the infinite concatenation.

  Next we have to construct the whole $\U_1^{[0,\infty)}$-valued
  \emph{path} for all times at once. We use the fact that the
  measure-valued paths are continuous to argue that they are
  equi-continuous on the time points $t \in \mathbbm{Q} \cap [0,T]$
  for every $T$ and the masses are uniformly concentrated on finitely
  many colors. Then we can conclude that the $\U^{[0,\infty)}$-valued
  paths have this equi-continuity property. This follows from the fact
  that if we consider the process arising by sliding concatenation
  over $\mcI_{(t-\ve)^+}$ instead of $\mcI_t$ we obtain a uniform
  approximation on $\ve > 0$. Then by standard arguments we can define
  the path for all times as an a.s.\ continuous one.

  Now we have a conditioned process and we have seen already that we
  can use \cite{Evans93} for getting the measure-valued case.

  It remains to show \emph{uniqueness}. So far we have used duality at
  this point directly or via a conditional martingale problem where we
  need the dual for a time-inhomogeneous Fleming-Viot process. The new
  element now is that the process is one of \emph{marked} genealogies
  and has in its generator an additional \emph{immigration term} from
  a time-inhomogeneous source. We try to construct such a conditional
  duality nevertheless by using a richer process in which our process
  can be embedded and a dual can be constructed to then obtain
  uniqueness automatically.

  We can augment every solution of our process by special sites
  $\{\ast,\dagger\}$, where $\ast$ carries constant in time the
  element $\mfe = [\{1\},\underline{\underline{0}},\delta_1]$ marked
  with color $t$ at time $t$, and at time $\delta$ the element $\mfe$
  is the ancestor of every other color $s$ with $s$ larger than
  $\delta$. On $\dagger$ we construct the process that we later want
  to read off as $\mfU^{\ast,+}$. Here an observation is that we might
  view this as a spatial system with critical branching in one
  component of space. In the other special site with emigration to the
  first component and with super-critical branching evolution which we
  describe more precisely below.

  We consider a $\U^V$-valued process where
  \begin{align}
    \label{e3725}
    V= \{\dagger,\ast\}\times [0,\infty).
  \end{align}
  On $\ast$ we start in $\mfe$ and we run a super-critical $\U$-valued
  dynamic with super-criticality rate $b$ but the growth operator of
  distances is turned off for two individuals at $\ast$ (otherwise we
  are not dust-free). There the color changes deterministically, i.e.\
  grows with speed $1$. Furthermore, starting from time $0$ mass
  migrates at rate $b$ at time $t$ to \emph{site $(\dagger,t)$} and
  from this site it follows the Feller dynamic on $\dagger$. Hence on
  $\dagger$ we have a critical rate $b$, $\U$-valued Feller dynamics
  where the marks in $[0,\infty)$ are \emph{inherited} such that the
  mass on $\dagger$ with mark in $[0,\infty)$ increases in mean at
  constant rate $b$ with time. There are two points to prove.

  \medskip\noindent
  (1) For every color $t$ this process on $\dagger$ is well-posed as
  $\U^{\{t\}}$-valued Feller diffusion in its $\wh \mfU^{\{t\}}$ part,
  which also connect with each other once considered together as we
  shall see.

  \medskip\noindent
  (2) The process of the colors $\mcI_t$ on $[0,\infty)$ which have
  positive mass at time $t$ is uniquely determined in law by the
  generator.

  \medskip

  The projection of the above process on the marks
  $\{\dagger\} \times [0,\infty)$ can be mapped on $[0,\infty)$
  without loss of information and we obtain a process which is by
  inspection a version of $\mfU^{\ast,+}$ solving the martingale
  problem. We have to characterize this process via its martingale
  problem and prove its \emph{uniqueness} via \emph{duality}. More
  precisely, the process on $\ast$ is an autonomous $\U$-valued
  process, which is equal to $\mfe$ marked with $t$ at time $t$ whose
  uniqueness we must establish and for the uniqueness of the
  $\dagger$-component we have to use duality arguments. A subtlety
  here is that only \emph{countably many colors} are present at time
  $t$, which have non zero weight in the time $t$-population, but
  these colors are \emph{random}. Most suitable is therefore a
  conditional duality where we condition on the measure valued process
  on $\{\dagger\} \times [0,\infty)$ so that this set of colors
  becomes deterministic.

  For this we need the uniqueness of the measure valued process that
  is in particular of the $\mcI=(\mcI_t)_{t\ge 0}$. We use Evan's
  uniqueness result. This process is a special case of processes
  constructed by Evans \cite{Evans93} as super process on general
  geographic space. We would have here a super process on $(\R_+)^2$.
  We then need that solutions to our martingale problem must be such
  super processes to conclude uniqueness.

  We have to see that \cite{Evans93} is applicable. We consider
  $E=[0,\infty)^2$ and the motion process which is deterministic
  $(a,b) \to (a,b+t)$. Put $a=0$ and $b=0$ and start the processes.
  The immortal throws of type $(s,s)$ at time $s$ at rate $b$. Project
  the measure at time $t$ the component
  $\{\dagger\} \times [0,\infty)$ on $[0,\infty)$ to obtain
  $\bar \mfU^{\ast,+}(\cdot)$, a measure on $[0,\infty)$. This process
  is characterized uniquely by its $\log$-Laplace equation as is proved
  in \cite{Evans93}. We know that it solves our martingale problem,
  see \cite{D93} in Section 6.1, where it is proved in Theorem 6.13
  that a solution to the $\log$-Laplace equation solves a martingale
  problem and vie versus, which allow to conclude with Ito-calculus
  that it solves our martingale problem. In other words
  $(\bar \mfU_t^{\ast,+})_{t \geq 0}$ solves an autonomous martingale
  problem, but was given originally by the $\log$-Laplace equation.

  Since we want to prove uniqueness of the martingale problem via
  duality this raises first the question, what is the form of the
  \textit{duality function}. Here we use:
  \begin{align}
  \label{e6193}
    H(\mfU,\mfC)=\int_{(U \times V)^n} \; \varphi
    \left(r(u_i,u_j)\right)_{1 \leq i<j\leq n} g
    \left(t,(v_i)_{i=1,\cdots, |\mfC|}\right)
    \nu^{\otimes n}\bigl(\dx\underline{(u,v)}\bigr),
  \end{align}
  where $n=|\mfC|$ and the function $g$ satisfy the conditions posed
  in \eqref{e2636}.

  Therefore the process $\mfU^{\ast,+}$ has as conditioned dual a
  ``spatial'' coalescent starting on site $\dagger$ in color
  $s \in \mcI_t$, where $\mcI_t$ is the set of colors with
  $\bar \mfU^{\ast,+}_t(\{s\}) >0$, where particles with color $s$
  jump from site $(\dagger,s)$ to $(\ast,s)$ at time $s$, coalesce at
  every color as usual, i.e.\ only the same color can coalesce on the
  ``sites'' $(\dagger,s), s \in [0,\infty)$. This process will in fact
  \emph{coalesce by time $s$}. Namely the rate is
  $b/\bar \mfu_r(\{s\})$ at the backward time $r=t-s$. Recall the
  non-integrability of the rates in $s$ at $t$, which means we
  coalesce before reaching $t$. Once particles have reached $\ast$
  they instantaneously coalesce with other colors.

  This now proves uniqueness, if we can show that the duality is
  implied by the \emph{generator criterion}. What is new here
  (compared to the proof in Section~\ref{sec:duality}) is the color
  structure, the spatial structure $\{\ast,\dagger\}$ and the
  immigration operator. The duality for the spatial model is given in
  Section~\ref{s.proofext}, the color structure is as colors are
  inherited immediate, remains as issue the starting $0$ mass.

  Here we use that the generator criterion implies the uniqueness of
  the process starting in positive mass and then reading of from the
  duality relation that this converges to a \emph{limiting duality
    relation} letting the mass going to zero and this is giving the
  uniqueness.
\end{proof}
\begin{proof}[Proof of Proposition~\ref{prop.2034}]\label{pr.6355}
  Here we observe that we deal with the same martingale problem as for
  $\mfU^\dagger$, hence the same argument carries over here for
  uniqueness.
\end{proof}
\begin{proof}[Proof of Theorem~\ref{TH.IDENTKALL}]
  \label{pr.5380}
  Here we use the above theorem, that $\mfU^{\ast,+}$ is uniquely
  determined by the martingale problem. Then the claim follows from
  the fact that the operators have the \emph{same action} on functions
  \emph{not} depending on the colors so that $\mfU^\ast$ solves the
  same martingale problem as $\mfU^\dagger, \mfU^{\mathrm{Palm}}$,
  which by well-posedness will agree.
\end{proof}

\subsection{Proofs of Proposition~\ref{l.branch} and
  Theorem~\ref{T:BACKBONE}}
\label{ss.profbran}

We construct here first the ingredients needed for the backbone
construction before we come to the actual proof of
Theorem~\ref{T:BACKBONE}.

% \subsubsection{Preparation: Proof of
% Proposition~\ref{l.branch}}\label{ss.prepcon}

\begin{proof}[Proof of Proposition~\ref{l.branch}]
  \label{pr.l.branch}
  (a) Recall the construction of $\mfU^{\ast,+}_t$ by sliding
  concatenation in Remark~\ref{rem:eva-sl-con}. The construction of
  the process $(\mfV^{t,+}_r)_{r \leq s}$ for $s < t$ is simpler since
  we concatenate here for $r<t$ always a finite random number of
  copies of a Feller diffusion conditioned to survive until time $t$.
  Recall the process $\mfU^T$, which we use here for $T=t$ and which
  we have constructed and characterized by a martingale problem. It is
  then easy to explicitly construct the process given a path of the
  measure-valued process $(\bar \mfV_r^{t,+})_{r \leq s}$ of the
  colored masses.

  Hence the existence of a solution follows again by the construction
  we gave via the IPP and the sliding concatenation of independent
  pieces of the processes $(\mfU^t_r)_{r \in [s,t]}$ needed, we skip
  the standard details here.

  A bit more subtle is the \emph{uniqueness}, which we must base on
  \emph{conditional duality}, conditioning on
  $(\bar \mfV_r^{t,+})_{r \leq s}$, i.e.\ on the whole collection of the
  mass paths of the various immigrating masses marked by the points of
  the IPP. As dual we take the time-inhomogeneous coalescent,
  coalescing only on mark $s$ at rate $b$ times the inverse color $s$
  mass. This runs until time $s$ and then the color changes to $0$.
  Partition elements with color $0$ coalesce instantaneously. Because
  of the non-integrability of the coalescence rates at $s$ this means
  that color $s$ partition elements coalesce into one element before
  time $s$.

  This gives us the uniqueness of the genealogical (i.e.\ the
  $\wh{\cdot}$\,) part of the process conditioned on the collection of
  masses.

  In order to close the argument we have to show first that the
  measure valued process must be \emph{atomic}. Note next that then
  the number of colors at time $r<t$ is finite due to the bounded
  immigration rate up to time $t$ and show that the evolution of the
  collection of the atomic measure valued process is uniquely
  determined by the martingale problem. For that property we need that
  the atoms evolve independently, since the operator is the sum over
  the color $s$ operators. Then we need next the uniqueness of the
  single color of the collection of mass processes on time intervals
  $(0,r]$, $r<t$, but now conditioned on survival till time $t$, which
  we know from classical SDE results for the process $\bar \mfU^T$.
  Therefore we need the \emph{atomicity} and the \emph{independence of
    atom evolutions}.

  For the total mass process we see from the martingale problem that
  the points where the part with bounded variation of a color $s$
  starts increasing form a Poisson point process with rate
  $2/(b(t-r))$.

  To get the independence we consider test functions of the form
  $\exp(-\Phi)$ for some positive polynomials and as test functions on
  the colors namely linear combinations of indicators to conclude that
  the expectation factorizes into the contribution of the different
  colors, if we condition on $I_t$. Here we use Section~1.4 in
  \cite{infdiv} together with a fact on martingale problems first
  devised by Kurtz and extended to $\U$-valued processes in
  Theorem~2.8 in \cite{ggr_GeneralBranching}. This allows also to
  argue that the populations of different colors evolve independently.
  Furthermore we know already from Theorem~\ref{TH.MARTU} that the
  single color evolution is uniquely determined by the martingale
  problem for $\mfU^T$. This concludes the argument.

  \medskip
  \noindent
  (b) Here we have to show that the sequences as $t \uparrow h$ of
  solutions are \emph{tight} and \emph{converge} to a limit as
  $t \uparrow h$. For the tightness and convergence we have to deal
  with incoming immigrant populations arriving at times $s$ close to
  $h$ as there are before time $s$ only finitely many, each behaving
  as a process with continuous path while between $s$ and $h$ we have
  countably many so that we have to control their total mass as
  $s \uparrow h$ to show their contribution can be made arbitrarily
  small. The descendant population of a time $s$ immigrant evolves
  autonomously according to the $\U^V$-valued diffusion where
  immigration is turned off after time $h-s$ and converges to a limit
  state as we approach $t$ by the continuity of path. Therefore to
  close the argument we use a coupling argument, where we control the
  sub-populations of those immigrating after time $s$ with simpler
  ones and to study the effect of additional immigrants at the
  \emph{diverging} rate.

  We have to show that the random population with marks in $[s,h)$ is
  \emph{tight} in $s$ and \emph{converges to the zero-tree} as
  $s \uparrow h$. Furthermore we have to show that for every fixed
  $s<h$ the population at time $t$ with marks less than $s$ has a
  limit as $t \uparrow h$. Both facts together give the claim.

  The first point, the tightness, will follow from the construction as
  concatenation of Feller entrance laws conditioned to survive beyond
  time $t$ starting at $s$ and is related to the second, the
  convergence requires to bound the mass of a sum of Feller diffusions
  starting at some time between $s$ and $h$ and surviving till time
  $h$. The number of individuals in $[s,h-\delta]$ surviving till time
  $t$ grows to infinity as $\delta \uparrow 0$, but the contribution
  of mass is bounded by
  $\mathit{const} \cdot\int_0^\delta \ve \; \log |\ve|d \ve$; see also
  \eqref{e1711}. By choosing $\delta$ suitably we can then get that
  the contribution in mass goes to zero as $s \uparrow h$. This means
  we have convergence to a state in $\U$ as $t \uparrow h$; see
  \eqref{e1711}.
\end{proof}

% \subsubsection{Proof of Theorem~\ref{T:BACKBONE}}\label{sss.tback}
\begin{proof}[Proof of Theorem~\ref{T:BACKBONE}]
  (a) First we have to identify $\mcL[\mfU_t^{\rm Palm}]$ defined as
  size-biased law and then characterized by a martingale problem as
  the one arising here from the construction of $\mfV$ as
  $\mcL [\mfV_t^t]$. We know this for the total mass process see
  \cite{Lamb07} and also \cite{Lamb02} in other words
  $\bar \mfU^\ast_t =\bar \mfV^t_t$ in law. From the fact that the
  Laplace transform agree we can in fact read of that also
  $\mfV^{t,+}_t=\mfU_t^{t,+}$, since the former can be decomposed in
  the independent masses corresponding to the jump sizes of the IPP
  process with intensity $2(t-s)^{-1} ds$ and can hence be written in
  a specific form and on the other hand the process $\bar \mfU^\ast$
  can be written as arising from a Poisson point process with
  intensity $b \cdot ds$ of Feller diffusions containing those as
  component surviving till time $t$ a procedure, which is
  \emph{thinning} a PPP in an inhomogeneous way. This is the limit at
  starting on $\ve \N_0$ Feller processes at rate $b$ taking those
  surviving at time $t$ and starting at mass $\ve$ with
  $\ve \downarrow 0$. Recall~\eqref{e1264} for the asymptotic formula
  for survival. These two representations agree as has been shown.
  Namely they agree with the one induced by the color decomposition of
  the latter after removing the colors.

  We need here however more than
  $\bar \mfU_t^{t,+}= \bar \mfV_t^{t,+}$ in order to have sufficient
  information on genealogies, namely that the path
  $\{(\mfU_r^{\ast,+}(s))_{r \in [s,t]}, s \in \mcI_t\}$ equal in law
  $\{(\mfV_r^{t,+}(s))_{r \in [s,t]}, s \in \mcJ_t\}$, where $\mcI_t$
  and $\mcJ_t$ are the sets in mark space which carry an atom at time
  $t$. For that we need that pruning in $\mfU^{\ast,+}$ all colors
  that do not reach the time horizon $t$ is exactly resulting in the
  IPP with the intensity $(t-s)^{-1}$. This means that we are back to
  a property of $(\bar \mfU_t^\ast)_{t \geq 0}$, which follows from
  the explicit construction of the Evans process with rate $b$
  immigrations from an immortal line. This is known, as we saw above.

  Next we have to \emph{verify equality for the genealogy part}, i.e.\
  we have to show that the conditional laws of
  $\pi_{\mcI_t}(\wh \mfU^{\ast,+}_t)$ and $\wh \mfV^{t,+}_t$, both
  conditioned on $(\bar \mfU^{\ast,+}_r)_{r \in [0,t]}$, where
  $\pi_{\mcI_t}$ is the projection on the population with colors
  $\mcI_t$, respectively $(\bar \mfV_r^{t,+})_{r \in [0,t]}$ agree if
  we use the same path for $\bar \mfU$ and $\bar \mfV$. We therefore
  couple the total mass path and show that both descriptions first
  $\U$-valued Feller diffusions split off at rate $b $ and second
  copies of $\mfU$ at rate $2(t-s)^{-1}$ split off result in the same
  law at time $T$.

  We consider first the genealogies of the population of one color. We
  have seen in the proof of Theorems~\ref{TH.1061},~\ref{TH.IDENTKALL}
  that the genealogy for given total measure path can be given in
  terms of a certain $\U_1$-valued coalescent processes whose
  parameters are uniquely determined by the total mass process.
  Therefore we have to show now that given that the
  $(\bar\mfU^{*,+}(s))_{r \in [s,t]}$ and $(\bar \mfV_r^{t,+}(s))$
  agree, so that we can use a coupling of the pair in
  $[0,\infty) \times \U_1$ by actually choosing them equal. Then the
  processes $\wh \mfU^{\ast,+}(s)$ and $\wh \mfV^{t,+}(s)$ have the
  same law for all $s \in \mcI_t=\mcJ_t$. Namely the duality is a
  consequence of the calculation we did earlier below \eqref{e2005}
  calculating the generator of $\mfU^{\ast,+}$, which showed that
  given the path of one $\bar \mfU^{\ast,+}(s)$ (which is a marked
  version of $(\bar \mfU^{t-s}_r)_{0 \leq r \leq s}$) the conditional
  law of $\U_1$-valued part for the population of this color is a
  time-inhomogeneous Fleming-Viot process the resampling parameters
  depending on the total mass path of the various colors. This is
  however the same for the $\mfV^{T,+}$ process. Now we have to
  concatenate (sliding concatenation) the different color populations,
  according to the same rule. Hence we have the same time-$t$ marginal
  distributions and hence this holds in particular for the projections
  on $\U$. Then we use that the solutions of the two martingale
  problems (which are well-posed) are actually equal, according to a
  result of Ethier and Kurtz in Theorem~4.2 (a) in Chapter 4 of
  \cite{EK86}.

  \medskip
  \noindent
  (b) As for instance in the proof of
  Theorem~\ref{T:KOLMOGOROVLIMIT}(c) we proceed by showing first
  \emph{tightness} and second \emph{convergence} via \emph{convergence
    of moments}.

  The \emph{tightness} is again verified by checking the three
  conditions which guarantee tightness in the weak topology w.r.t.\ to
  the Gromov weak topology.

  \medskip
  \noindent
  \textbf{(i)} The tightness of \emph{total masses} follows by using
  the SDE for the total mass, which gives immediately the tightness,
  since the rescaled mass is bounded in expectation.

  \medskip
  \noindent
  \textbf{(ii)} The tightness of \emph{distances} by the
  representation via the concatenation of the surviving Feller
  diffusions coming of at rate $b$ with survival probability
  $2(b(t-s))^{-1}$ at time $s$ back and the ones surviving ones
  forming an inhomogeneous Poisson point process with intensity
  $2(t-s)^{-1}ds$ with surviving mass of size of order $t-s$. The
  distances are scaled by $t$. This means they are in macroscopic
  scale bounded by $1+t^{-1} r_0(\cdot,\cdot)$.

  \medskip
  \noindent
  \textbf{(iii)} Finally we have to bound the \emph{modulus of mass
    distribution}. This means we have to bound the number of ancestors
  which account for $(1-\varepsilon)$ of the total mass. Here we
  recall that the masses of the surviving family at time $s$ back from
  $t$ can be controlled by Kolmogorov's limit law. The ancestors are
  going off from the spine with rate $2(t-s)^{-1}$ for
  $s\in (0,t)$. Hence:

  We need that this gives the size-biased limit from before, which is
  now obvious. As $\delta \to 0$ we have
  \begin{align}
    \label{e1711}
    \E[\# \{\text{ind.~with descendants at $t$, born in } [s,s+\delta)]
    = \int_s^{s+\delta} \frac{2}{t-u} \, \dx u \sim
    \frac{2\delta}{t-s-\delta}.
  \end{align}
  Then a Borel-Cantelli argument gives that there are only finitely
  many ancestors in $[0,t-\varepsilon)$ for any finite
  $\varepsilon>0$, where of course as $\varepsilon \to 0$ this number
  diverges. Taking now into account the masses of the time-$t$
  population going back to ancestors immigrated in $(t-\varepsilon,t)$
  which has expectation $t-\varepsilon$ gives the needed property.

  \medskip

  Next we prove the \emph{convergence} by showing that the expectations
  of polynomials of the rescaled process converge.

  \medskip
  \noindent
  \emph{Observations:} We can calculate the Laplace transform as
  product of two Laplace transforms if we consider the concatenation
  of the two sub-families and use the additivity of the
  \emph{truncated} polynomials. We have already treated the part
  corresponding to the copy of the original process, but now there is
  no conditioning and this part does not contribute in the scale. The
  other part is the one that corresponds to the entrance law which we
  treat with the observation we make next.

  We can calculate the moments of the original $\U$-valued Feller
  diffusion via the FK-duality from which we obtain those for the size
  biased law. Namely the size-biased $n$-th moment correspond to a
  coalescent with $(n+1)$-individuals of which one is \emph{not}
  considered in the distance matrix but in the FK-functional.

  \medskip
  \noindent
  (c) The equation \eqref{e1098} is immediate from part (a). Relation
  \eqref{e1404} follows from the previous and part (b).
\end{proof}

\subsection{Proof of Theorem~\ref{TH.KOLMO}}\label{ss.prthkolmo}

It follows from the identity of $\mfU^\dagger$, the Palm process
$\mfU^{\mathrm{Palm}}$, the Evans process $\mfU^\ast$ and in
connection with Theorem~\ref{T:KOLMOGOROVLIMIT}, that we have to treat
at most \emph{two} cases and we know already from
Theorem~\ref{T:KOLMOGOROVLIMIT} that in fact we do have \emph{two}
cases. We have to show the convergence in the case of $\mfU^\dagger$,
$\mfU^{\mathrm{Palm}},\mfU^\ast$ and in the scaled $\mfU^T$,
furthermore we have to identify the limit in both cases different from
the situation in Theorem~\ref{T:KOLMOGOROVLIMIT}. This is since we now
have to identify the limit \emph{processes} as certain specific
$\U$-valued diffusions. Of course we do that by showing convergence of
the scaled processes and identify the limit processes and then get the
claim as Corollaries. We have already the result for the
one-dimensional marginal distributions. Therefore we need now the
\emph{f.d.d.-convergence} and the \emph{tightness in path space}.
There are two possible strategies to proceed.

In both cases we may work with the \emph{time-space} Feynman-Kac
\emph{duality} to show in the $\mfU^{\mathrm{Palm}}$ case the
\emph{convergence}, the \emph{tightness} of the f.d.d.'s we have
essentially done in the proof of Theorem~\ref{T:KOLMOGOROVLIMIT} for
the path space convergence we need however the \emph{compact
  containment} in \emph{path space}, finally we have to compare the
dual expressions with the ones we obtain for the claimed limit.

Alternatively for convergence (and that is what we follow up on) we
might work directly with the \emph{generators}, better operators of
the martingale problem, of the rescaled process and show their
convergence. For the latter we have to deal with the fact that the
resampling operator involves the term $\bar \mfU^{-1}$, which is
unbounded and in fact diverges as we approach the initial point. This
raises technical questions. Otherwise it is easy to see that the
generators \emph{converge pointwise}. We have therefore to consider
functions as in the case of the Feller-diffusions functions, which are
zero at the zero-tree and at $\infty$-mass and use the extension of
the operators as in Remark~\ref{r.1207}. Then we can use again the
pointwise convergence of the coefficients in the operator, to obtain
the claim. This is clear in the case of $\mfU^{\mathrm{Palm}}$ and it
remains to look at $\mfU^T$, recall \eqref{e984} to see that this
expression scales. This gives f.d.d.\ convergence of the scaled
processes and hence in particular our claim, which is about its
marginal law.

To obtain the stronger path convergence we need in addition to the
generator convergence to establish compact containment which has as
nontrivial point the others are handled by inspection, the
\emph{uniform in time dustfree condition}. Here we use that the total
mass path is tight, a classical result and then we can conclude using
that the genealogy part is time-inhomogeneous Fleming-Viot with bounds
from below at the resampling rate.

\section[Proofs of
Theorems~\ref{T.SRWALK},~\ref{T.CLUMP},~\ref{T.AGING},~\ref{T.CRT}]{Proofs
  of
  Theorems~\ref{T.SRWALK},~\ref{T.CLUMP},~\ref{T.AGING},~\ref{T.CRT}:
  the extensions}
\label{s.proofext}

The two extensions are treated separately, since they require very
different frameworks.

\subsection{Proofs of Theorems~\ref{T.SRWALK},~\ref{T.CLUMP}:
  \texorpdfstring{$\U^G$}{UG}-valued super random walk}
\label{ss.treeval}

First, we have to show \emph{existence and uniqueness} for the
solution of the martingale problem including its properties we have
claimed in Theorem~\ref{T.SRWALK}. Second, we prove the
\emph{application} to the long time behavior in Theorem~\ref{T.CLUMP}.

\subsubsection{Martingale problem and proof of Theorem~\ref{T.SRWALK}}
% We have to prove Theorem~\ref{T.SRWALK} first.

\paragraph{\emph{Existence}}
The \emph{existence} of solutions of the martingale problem for a
spatial model on $\U^G$ is shown first in \cite{GSW} for Fleming-Viot
models. We give here the basic steps for the branching case.

\paragraph{(1)}
The first step is to work on \emph{finite} geographic spaces
$G_n \uparrow G$ with finite $G_n$ which are abelian groups embedded
in $G$ such that the random walks on the finite spaces converge to the
one on the infinite space. This is achieved by using on $G_n$ a random
walk, which, if $G_n$ is a subgroup of $G$, suppresses all jumps
leading out of $G_n$. If $G=\Z^d$ we can use $G_n =\{-n,\dots,n\}^d$
and addition modulo $2n$ for the addition on $G_n$. Next we show using
the duality that the solutions for the geographic spaces $G_n$
converge to a solution of the martingale problem for $G$. Since the
dual is a Markov pure jump process this is a standard procedure. The
details on the approximation of the model on infinite geographic space
with finite geographic spaces is explained in \cite{GSW} and therefore
we are short here.

Now the argument continues with the \emph{approximation by individual}
based models on \emph{finite} $G_n$. The \emph{independent branching
  processes} on each site converge to i.i.d.\ $\U^{G_n}$ Feller
diffusions as we know. The limit of individual based models on finite
$G_n$ uses the result for the non-spatial case and applies it to each
component to get an i.i.d.\ system of evolving components
corresponding to what we call a $\U^{G_n}$-valued Feller diffusion. On
the other hand it is well known that the configuration
$(\bar\mfu_\xi)_{\xi \in G_n}$ under the \emph{pure migration} process
converges to the deterministic mean flow given by the SDE specified by
the drift term. What about the genealogy? Since the distances between
two individuals at positions $i$ and $j$ at two fixed sites are
changing only due to the flow since we evaluate for the time $t$ state
two individuals located at positions $i'$ and $j'$ with the respective
individual sampling measures at time $0$ at positions $i'$ and $j'$
given by the reweighted original measures namely
$\nu(\cdot,i')\otimes \nu(\cdot,j')a_t(i,i')a_t(j,j')$. The migration
dynamics acting on $\U$ above is therefore converging to the mass flow
in geographic space moving the time-$0$ genealogical state via its
measure $\nu$ only and hence we have convergence.

Next we impose on these i.i.d.\ evolutions of components the
\emph{interactions}, i.e.\ the \emph{migration} of individuals. The
\emph{combination of the two mechanisms} has to be shown to converge
to the $\U^{G_n}$-valued solution of our martingale problem. Here we
use the \emph{Trotter formula} as follows.

Consider time intervals of length $h/2$ for some $h>0$. Alternating we
apply during these time intervals one of the two dynamics. We note
that for each $d$ these $h$-indexed dynamics we have a duality to the
corresponding dual process applying alternatively migration resp.\
coalescence. One can see due to the simple nature of this dual that
the Trotter formula holds, that is, the $h$-approximations converge to
the process, which implies then that this holds for the original
process.

This gives the existence of a solution of the martingale problem for
finite $G_n$.

\paragraph{(2)}
The second step is (cf.\ \cite{GSW}) an \emph{approximation} of the
dynamic on infinite geographic spaces by a suitable dynamic on finite
subsets $G_n \uparrow G$. This follows immediately from the
convergence of the dual process as $n \to \infty$, due to the
convergence of each of the finitely many random walks.

\begin{remark}
  Alternatively \label{r.6757} we can use the existence of the measure
  valued super random walk. First construct the process given the
  population size process which is classical, as $\U_1^G$-valued
  time-inhomogeneous Fleming-Viot process, which is no problem. Since
  we have a duality we can construct and prove on $G_n$ existence of
  time-inhomogeneous process via piecewise constant approximations,
  only in time and space the rate is now varying. Then via the
  conditional martingale problem to combine the population size
  process with conditional $\U_1^G$-valued processes to obtain the
  pair $(\bar\mfU_t,\wh\mfU_t)_{t\ge 0}$ from which the desired
  process $(\mfU_t)_{t\ge 0}$ is now a functional, recall here our
  detailed explanation of the skew martingale problem in a paragraph
  below \eqref{eq:coneq}. Here we can use arguments similar to those
  used in \cite{DG03} and \cite{Gl12}.
\end{remark}

\paragraph{\emph{Uniqueness}}
We focus therefore here on the \emph{uniqueness} which is based on the
Feynman-Kac duality. This follows from a general statement, which is
given in \cite{EK86}. This requires to verify the duality only based
on the \emph{operator relation} for duality, relating the operator
from the forward martingale problem respectively the one solved by the
dual process on state spaces, denoted by $E$ and $\wt E$ and defined
in \eqref{e.939} resp.\ \eqref{e.939t}, namely
\begin{align}
  \label{eq:34}
  (G_X H(\cdot,y)(x) = (G_Y H(x,\cdot))(y) + V(y) H(x,y), \; \; x \in
  E, \, y \in \wt E,
\end{align}
with $G_X$ and $G_Y$ being the generators of the $X$ resp.\ $Y$
processes and $V$ the potential on $\wt E$. The forward operator we
have calculated in \eqref{eq:3}-\eqref{e945}.

Consider next the dual process and its operator. The dual process is a
pure jump process with deterministic motion (the growth of the
distances in the distance matrix) which can be read off from the rules
of the dynamic right away. Recall here \eqref{tv32}-\eqref{tv33}
formulas in the non-spatial case which have just to be lifted from
$\Phi^{n,\varphi}$ to $\Phi^{n,\varphi,g}$ by acting with the
coalescence operator on $\varphi$ as before but \emph{also} now on $g$
by identifying two variables which correspond to the merging partition
element, recall \eqref{e2581}-\eqref{e2589}.

Therefore the \emph{dual operator} is given by the following operator
acting on a bounded continuous function F of the state which is twice
continuously differentiable as a function of the distance matrix. The
dual operator consists of the non-spatial operators, coalescence and
distance growth, lifted to the spatial case, i.e.\ acting on
partitions at the same site in the polynomial as before and leave $g$
untouched.

The operator has as \emph{new} terms the ones from the
\emph{migration} jumps of the locations of the partitions. The latter
is acting as pure jump generator with rate $a(i,j)$ independently for
a jump from $i$ to $j$ for the $k$-th mark for each of the marks of
the current partition elements.

The calculations for the generator relation we have already given for
the branching part, since it is local we did for the $\U$-valued
Feller diffusion and our explicit lifting of this operator to the
spatial case before Theorem~\ref{T.SRWALK}. We only have to still deal
with the new, the \emph{migration} operator
$\wt \Omega^{\uparrow,\mathrm{mig}}$ to the one of the distance matrix
enriched spatial coalescent. This operator is dual to the mass flow
forward operator; see \cite{GSW}. We know the FK-duality for the
branching operator already and hence we have checked the generator
criterion and have a Feynman-Kac duality and hence also
\emph{uniqueness}. This concludes the proof of Theorem~\ref{T.SRWALK}.

\paragraph{\emph{Skew representation}}

To complete the proof of Theorem~\ref{T.SRWALK}(a) we need to argue
why $\bar\mfU$, $\wh\mfU$ have the claimed characterization as super
random walk on $G$ respectively $\wh\mfU$ conditioned on a realization
of $\bar\mfU$ as a $\U_1^G$-valued time inhomogeneous Fleming-Viot
process. We note that both for $\bar\mfU$ and $\wh\mfU$ the branching
part of the operators lifted to $\U^G$ as explained in front of
Theorem~\ref{T.SRWALK} allow to conclude from our results in
\eqref{eq:r.742.1}--\eqref{e746} that the second order part of the
operator of the operators of $\bar\mfU$ and $\wh\mfU$ are as claimed.
Here we can use of course simply the generator calculations for the
branching mechanism if we lift them from $\U$ to $\U^G$ to again get
the defined diffusion term. However, we have to complement this also
with the calculation for the migration operators. Note that the
calculations for the branching and migration parts are separate
matters. The needed calculations we did in
\eqref{eq:28}--\eqref{e945}.

Once we have shown that the operators of $\bar\mfU$ are the ones of
the super random walk and the ones of $\wh\mfU$ the ones of the
time-inhomogeneous Fleming-Viot process with the rates as specified in
the theorem, the uniqueness property of the super random walk and the
time inhomogeneous Fleming-Viot martingale problems gives then the
claim.

\paragraph{\emph{Duality}} In order to prove Theorem~\ref{T.SRWALK}(b)
we use again the criterion in \eqref{eq:34}, where we have to check
the operator relations both for $\bar\mfU$, $\wh\mfU$ conditioned on
$\bar \mfU$ and the two respective dual processes, the spatial
$a$-coalescent respective the time-inhomogeneous spatial coalescent
with the rates as specified in the theorem. We note that the
calculation for the two parts of the operator, branching and migration
versus coalescence and migration in the dual are separate matters.
Hence we have to check here only the migration part, showing the
\emph{duality} of the \emph{mass flow} of particles moving according
to $\bar a(\cdot,\cdot)$ to the \emph{independent $a(\cdot,\cdot)$
  random walks} of the dual individuals.

\paragraph{\emph{Feller property, strong Markov property}}

Here we use that it suffices to show that
$\nu \mapsto \E_\nu[H(\cdot,\cdot)]$ is continuous in the weak
topology for measures in $\mcM_1(E)$. Since the polynomials are
separating on $\wt{\mathcal M}$ and the coalescent is non-increasing
in size and the exponential in the Feynman-Kac term is bounded in time
$t \in [0,T]$ for every $T<\infty$, the dual expectations converge if
the initial measures converge and they as well as their limits are
supported on $\wt{\mathcal{M}}$. If our initial state is in $\wt E$ we
can approximate it by truncation with states in $\wt{\mathcal{M}}$ and
argue as before on $\U$. This proves the generalized Feller property.

Furthermore, since the polynomials $\{H(\cdot,y) : y \in \wt E\}$ are
measure determining we see that again as a consequence of the
generalized Feller property we have the \emph{strong} Markov property
as observed in the non-spatial case.

\subsubsection{Application and proof of Theorem~\ref{T.CLUMP}}

We have defined the spatial Evans process on the space $\R$ as a
functional of a collection of $\U^\R$-valued processes which we
obtained as functionals of the historical Dawson-Watanabe process,
which we call historical process in the sequel. Such a construction is
not possible for the super random walk and we have to work with
martingale problems. Therefore we first need to show that the
$\U^\R$-valued Dawson-Watanabe process \emph{exists} as the
\emph{unique} solution of a martingale problem so that we can conclude
the \emph{Markov property} of the functional of the historical process
we used as definition of the limit process. Second we have to prove
the convergence result.

\medskip
\noindent
\textbf{(1)} The \emph{historical process} is rigorously defined via a
$\log$-Laplace equation (\cite{DP91}). However this is equivalent to a
martingale problem for polynomials of degree $1$ and specifying the
increasing process (see Section~12.3.3 in \cite{D93}) which we want to
use. Therefore we have to get from this characterization an equivalent
martingale problem for the polynomials.

Here the polynomial test functions are based on test functions on
path, we call $g$, which is evaluating the path at $m$ time points
which looks as follows. The function $g$ is now a function on path.
Let $v \in D [((-\infty,\infty),\R)$ and let
\begin{equation}\label{e7801}
  g(v_1,\dots,v_n)=\prod^m_{i=1} g_i(v_i), \text{  with  } g_i(v)=
  g_i((v(t^i_1),\dots,v(t^i_m))),\; m \in \N.
\end{equation}

More precisely for the \emph{time-homogeneous} set up of the path
process (i.e.\ the time-space process) we add the \emph{explicit time}
coordinate. We consider for some $n \in \N $ and
$0 \leq t_1^{(i)}<t_2^{(i)}< \dots < t^{(i)}_{m(i)}< \infty$,
$i=1,\dots,n$
\begin{equation}
  \label{e1322}
  \wh g^{\uuxi}(t,v) = \prod^n_{i=1} \; \prod^{m(i)}_{k=1} \wh
  g^{\uuxi}_{i,k} \Bigl(t,v(t \wedge t_k^{(i)})\Bigr),
\end{equation}
with
$\wh g^{\uuxi}_{i,k}(t,v) = \Psi(t) \; \Psi^{\uuxi}_{i,k}(t)
g^{\uuxi}_{i,k}(v)$ and $\Psi(t)$ and
$\Psi_{i,k}^{\uuxi} \in C^1_b(\R, \R)$ as the functions to generate
polynomials.

We begin by writing down the \emph{operator} of the \emph{historical
  process} acting on the spatial monomials, there is the \emph{passing
  of explicit time coordinate} in the path, the \emph{branching} part
and the \emph{migration} part including the passing of explicit path
time (which will induce the growth operator!), which we denote by
$\Omega^{\ast, \mathrm{time}}, \Omega^{\ast,\mathrm{bran}}$ and
$\Omega^{\ast,\mathrm{mig}}$.

Recall the generator of the motion process of a single individual
$(Y(t))_{t \geq 0}$ was called $A$. For the process $Y$ the
time-space process $(t,Y(t))_{t \geq 0}$ then has generator
$\wt A = \frac{\partial}{\partial t}+A$. The corresponding
\emph{path process generator} $\wh A$ acts (see Section 12.2.2 in
\cite{D93}) on $ g $ of the form \eqref{e1322} for
$t_k \leq s < t_{k+1}$:
\begin{equation}
  \label{e1326}
  \wh A \wh g(s,v)= \prod_{\ell=1}^k \wh g_\ell \left(s,v(s \wedge
    t_\ell)\right) \wt A \Bigl(\prod_{\ell=k+1}^m \wh g_\ell
    \big(s,v(s)\big)\Bigr) \; \text{ and gives $0$ for $ s>t_m$.}
\end{equation}
This operator specifies a \emph{well-posed martingale problem} on the
spaces $ D([0,\infty),\R \times D(\R,E)) $ (Section 12.2.2 in
\cite{D93}).

We obtain for the operator of the martingale problem the formulas:
\begin{align}
  \label{e7667}
  \Omega^{\ast,\mathrm{time}} \Psi \Phi^{n,g} & =\Psi' \Phi^{n,g},\\[+1.5ex]
  \label{e7670}
  \Omega^{\ast,\mathrm{bran}} \Psi \Phi^{n,g} & =2 \sum_{1 \leq k<\ell\leq n}
  \Psi \Phi^{n,\wt \theta_{k,\ell} \circ g},\\
  \label{e7673}
  \Omega^{\ast,\mathrm{mig}} \Psi \Phi^{n,g} & =\Psi \sum_{k=1}^n \Phi^{n,g_k},
\end{align}
where now $g_k=A^\ast_k g$, $k=1,\dots,n$ and the operator $A^\ast_k$
is acting on the $k$-th variable of $g$. This operator $A^\ast_i$ is
defined as follows.

We have for each sampled marked individual the action of the path
process generator $\wh A$ but now acting on the corresponding factor
$\wh g^k$:
\begin{equation}
  \label{e.1157}
  \Omega^{\uparrow, \mathrm{anc}} \Phi^{\varphi, \wh g} = \sum_{k=1}^n \;
  \Phi^{\varphi, A^\ast_k \wh g} \; , \text{  where  } A^\ast_k \wh g
  \eqqcolon \Big(\prod^n_{\substack{\ell=1 \\ \ell \neq k}} \; \wh g^\ell
  \Big) \wh A  \wh g^k.
\end{equation}

To continue the argument we have to argue how the expressions above
follow from the version of the martingale problem in Theorem~12.3.3.1
in \cite{D93}, which gives the operator on \emph{(degree
  $1$)-monomials} to describe the drift and specifies the
\emph{increasing process} to describe the martingale part.

This step to replace the specification of the increasing processes of
the martingales by compensator terms of nonlinear functions is needed
often and uses the continuous martingale representation theorem and
Ito's formula.

Alternatively we can use a version of Theorem~6.1.3 in \cite{D93},
which deduces the martingale problem in the formulation we use here
from the $\log$-Laplace equation directly.

We have to prove that our functional \emph{solves} the $\U^\R$-valued
\emph{martingale problem for polynomials}, which we spelled out above
in \eqref{e943}--\eqref{e946}. To see this, two points are crucial.

The \emph{first point} is that the states of the historical process on
the geographic space $\R$ are concentrated on paths so that for any
two sampled paths there is a $T \ge 0$ so that they agree for
$s \le T$ or they are not identical in any (positive length) interval
contained in $(0,t)$. The \emph{second point} is that the part of the
paths before time $t$ do not change anymore from time $t$ on and only
their mass can change by branching. Why do these two properties hold?

The second point is immediate from the generator of the path process
generator which acts only on the functions of the path value we
observe at or after time $t$. For the first point we have to use the
fact that the functional giving $T$ grow with time at speed $1$.

We next have to let the operator of the historical process act on test
functions which depend only on the time-$t$ location and on the
functional giving $T$ which grows at rate $1$ with the time and for
$T=t$ by the branching which duplicates a path in two independent
copies equal before and at time $t$ since the path evolves only at the
tip. This gives the generator terms quoted above.

One would think that the \emph{uniqueness} follows from the duality
with $\U^V$-valued \emph{delayed coalescing Brownian motions}, which
follow from the fact that the process is the continuum limit of
rescaled super random walks so that our limit below will be the unique
solution of the \emph{$\U^\R$-Dawson-Watanabe martingale problem}
which is in duality to \emph{delayed coalescing Brownian motions}.
However, for the uniqueness of the martingale problem we would need
here that the duality follows from the martingale problem more
specific from the \emph{duality criterion for the forward and backward
  operator}.

This is an \emph{open} problem which, is also not resolved for the
$\U^\R$-valued Fleming-Viot process; see \cite{GSW}. The reason is
there are some conceptual and some heavy technical problems, which are
in the focus of upcoming work \cite{GSWfoss}. As a consequence we need
to obtain the Markov property differently.

We observe that the historical process is the unique solution of the
$\log$-Laplace equation and is Markov. The future evolution of our
functional depends only on the present, since the evolution of the
historical process is uniquely determined from the data we have in the
\emph{$\U^\R$-valued process at time $t$}, namely the current location
and the functional $T$. Hence we have nevertheless the Markov
property.

\medskip
\noindent
\textbf{(2)} Next comes the proof of our \emph{convergence statement}
which itself consists of three steps. In the \emph{first} step we
focus on the scaling of the genealogy, then in the \emph{second} step
we take care of the behavior of the scaled spatial marks separately.
Finally in the \emph{third} step we bring the first pieces of
information together to prove joint convergence. Recall first the
notation given in \eqref{e2712}-\eqref{e2736} which is needed now.

\paragraph{Step 1}
The scaling behavior of the $\U$-valued process, that we obtain by
projection on the genealogy is known from our analysis of the
non-spatial case in Section~\ref{ss.longuval}. On the other hand, the
projection on the marks results in a \emph{measure-valued process},
i.e.\ in a $\mcM(\Z)$-valued process, a super random walk, which has
been studied in \cite{DF88}.

In particular we know that the single ancestor subfamily conditioned
to survive forever, evaluated at time $T$ is a random variable with
values in $\U(T)$ should be \emph{scaled by $T^{-1}$} in the
\emph{distances} and in the \emph{total mass}, to converge to a
limiting object in $\U$. Recall that we have identified in this object
in Theorem~\ref{TH.KOLMO} as $\mfU^\dagger_1$ which equals
$\mfU^\ast_1$ in law. Indeed if we project the claimed limit
$\mfU^{\dagger, \sqcup}_1$ onto the genealogy, i.e.\ from $\U^{\R}$ to
$\U$ we obtain $\mfU^\dagger_1$. This means if we condition on having
a surviving subfamily starting at site $0$, we obtain a limiting
genealogy in $\U$.

We also know that the set of those sites, where with single ancestor
subfamilies which survive till time $T$ are located and we have
started initially in a translation invariant law, will converge in law
to a Poisson point process on $\R$ if we \emph{scale space by $T$}.
Hence these different clumps sit in distance of order $t$ while the
population descending from the founders up to time $t$ sits
essentially on a smaller spatial scale and the subfamilies separate as
$t \to \infty$.

\paragraph{Step 2}
\emph{Second}, we project the state of $\wt U^{z,\dagger}_t$ from
$\U^{\R}$ onto $\mcM(\R^1)$. This is the classical super random walk
which is the size-biased and scaled and is given again by the Evans
process, recall that the result of Evans works for superprocesses in
general geographic spaces. Hence, in the limit $t\to\infty$, we want
to obtain an immortal particle following a Brownian motion in $\R$ and
which is marked with a \emph{color} and the explicit time coordinate
and throws off superprocesses on $\R$, marked with the current color
(which is inherited) and explicit time and then project this from
$\mcM(\R \times [0,\infty))$ on $\mcM (\R)$.

We first have to prove therefore that the space-mass scaled surviving
forever $\mcM(\R)$-valued super random walk converges to a limit and
to identify this limit as the Evans process for the Dawson-Watanabe
process. Here we consider the \emph{moment measures} and show that the
space-mass rescaled moment measures converge to ones of the claimed
limit.

In order to calculate the moment measures of spatial averages which
are mass rescaled we need in particular the spatial mixed moments of
the super random walk under this scaling. The moment measures of the
super random walk can be calculated via the \emph{Feynman-Kac duality}
and this generalizes to the $\U^\Z$-valued process.

The behavior of the $\U^\Z$-valued dual process, the spatial
coalescent enriched with distance-matrices has been studied
asymptotically in \cite{GSW} and been shown to converge to
distance-matrix enriched instantaneously \emph{coalescing Brownian
  motions}. We have to show here in contrast that the dual expectation
once we \emph{include} the \emph{exponential functional converges} to
the corresponding expression for delayed coalescing Brownian motions
where, because of the exponential term with the Feynman-Kac
functional, delayed coalescence is based on the \emph{joint collision
  measure} of the two Brownian paths. In distribution this collision
measure equals the local time of Brownian motion in $0$.

In order to get this we have to show the path converge to Brownian
motion, the exponential terms to the one of the Brownian collision
measure and coalescence to the delayed coalescence based on the
collision measure. The points to discuss are the last two. For that
purpose we take into account the reweighting by the Feynman-Kac
potential which puts weight on the paths which do not coalesce by time
$t$ of order $\exp(b\mfl^{(t)})$, where $\mfl^{(t)}$ is the joint
occupation time of pairs of paths up to time $t$. Next we need that
the joint collision measure converges in law to the local time of
Brownian motion. This can be found for example in \cite{Bor87}.

For the convergence of the whole dual expectation we have to make sure
that all exponential integrals are finite. This however is clear since
the potential is bounded in $t \leq T$, by $\binom{n}{k} bT$, so the
scaling produces finite values. Furthermore we need that the
asymptotic for $T \to \infty$ is in the spatial case similar to the
calculation we did proving the scaling result in the KY-limit
in the non-spatial model. This gives us then immediately the
convergence of the pure genealogy part $\wh\mfU$ by using a test
function $\varphi g$ and then put $g$ constant. For $\varphi$ constant
and $g$ varying we get the convergence of the measure-valued object
$\bar\mfU$ directly from the convergence of random walks to Brownian
motion and the convergence of the Feynman-Kac term to the local time.

\paragraph{Step 3}
In the \emph{last} step of the convergence proof we have to show that
the \emph{joint distribution} of measure-valued component $\bar \mfU$
and the $\wh \mfU$-component converges. We know that since $\bar \mfU$
and $\wh \mfU$ each in law in the scaling limit converge that indeed
the scaled joint distribution is \emph{tight}.

This means that we have convergence along suitable subsequences to an
$\U^\R$-valued process, where the laws of projection on $\U_1$
respectively $\mcM(\R)$ are identified.

Since we have the tightness of the projections of the process with
values on $\U^V$ to $\U$ respectively $\mcM(V)$ already, we have
tightness of the joint law and the remaining point is to identify the
limit points of the \emph{joint law} and show their uniqueness. We
study the joint law using as tool the \emph{marked trunk}; see
\eqref{def:trunk}. We know that the laws are equal for the trunk
already from the results on the $\U$-valued case and we have to lift
this to the \emph{marked} trunk. In order to obtain the marked
$h$-trunk we consider the marks at time $t$ in the disjunct open
$2h$-balls of the state of the process
$\wt \mfU_t^{z,\dagger}, \mfU_1^{\ast, DW(\R)}$ at time $t$ and note
that all these path of one of these $2h$-balls have a common piece of
path in times $[0,t-h]$, which will be the mark after being extended
constant before time $0$ and after time $t-h$.

Since the limiting object can be approximated by its marked $h$-trunks
as $h \downarrow 0$, it suffices to show that all \emph{marked}
$h$-trunks converge in law to the \emph{marked} $h$-trunk of the
claimed limit objects.

The marked $h$-trunks are in our cases however finite marked
ultrametric measure spaces, where the marks are the "truncated" path
as explained above. They are obtained from the present time $t$ state.
In the limit object the distances in the trunk are exactly given by
the time where the two path agree all the way back to time $0$, this
implies the joint convergence as soon as a single path of migration
converges by the law of large numbers giving the convergence of the
historical process. These paths however are fixed in law over any
finite time interval $[0,s]$ as soon as $t \ge s$.

The steps 1-3 prove the convergence claim. This concludes the proof of
Theorem~\ref{T.CLUMP}.

\subsection{Proofs of Theorems~\ref{T.AGING},~\ref{T.CRT}:
  \texorpdfstring{$\M$}{M}-valued processes with fossils and the CRT}
\label{ss.treevalfoss}

Again we have to establish existence and uniqueness of the martingale
problem and to then show the existence of the large time limit of the
solution which then has to be identified as the CRT, in the sense that
$[CRT]$ is viewed as an element of $\M^{\rho,+}$.

\begin{proof}[Proof of Theorem~\ref{T.AGING} (existence and uniqueness)]
  \emph{Existence} will follow via approximation with an individual
  based model with diverging populations size as before, no new ideas
  are needed, details are suppressed.

  The \emph{uniqueness} is based on a Feynman-Kac duality with a
  \emph{time-space coalescent} with respect to the function
  $H(\mathfrak{m},\mathfrak{c})=\Phi^{\mathfrak{c}}(\mathfrak{m})$,
  where $\mathfrak{c}$ is the state of the dual process, which is the
  following object.

  The time-space coalescent is a coalescent which has two states for a
  partition element \emph{active} and \emph{frozen}. One fixes time
  points $0 \le t_0 < t_1< \dots < t_n < T$ with $T$ the present time.
  Then the system starts with active partition elements, say $n_0$
  many with the usual dynamic from time $t$ on backwards, but at times
  $T-t_1,T-t_2,\dots$ we have associated further partition elements
  say $n_1,n_2$ etc., where $n_i$-many which become active at times
  $t-t_i$ respectively and participate from then on in the usual
  dynamic with time running backward. This will be the dual particle
  system.

  With this system we obtain the \emph{space-time duality relation}
  which is in form equal to the statement which relates the finite
  dimensional distributions of the $\U$-valued Feller diffusion to the
  augmented time-space coalescent, except that in the forward
  expression we have the law of the fossil process at time $T$ (i.e.
  its time $T$ marginal law) and the state is in the time-variable
  evaluated at the times $0 \leq t_0< t_1< \dots t_n<T$ . As before if
  this duality relation follows from the generator criterion for
  duality, we have uniqueness and the proof is complete. The
  calculation that the criterion is satisfied is essentially the same
  calculation as before and not repeated here.
 \end{proof}

\begin{proof}[Proof of Theorem~\ref{T.CRT}]
  \label{pr.T.CRT}
  We show first that a limit for $t \rightarrow \infty$ exists in law,
  by showing \emph{tightness} of the laws and then convergence by
  showing the expected values of polynomials \emph{converge}. Then
  second we have to identify the limit as the CRT.

  For the \emph{tightness} we use the standard tightness criterion and
  observe that all distances are bounded by $2t$ and the total mass is
  stochastically bounded. To get the last point of the tightness
  criterion we use the FK-dual which guarantees that we have a finite
  number of ancestors back at times $t-\ve$, where an upper bound is
  given by the entrance law of the time-space Kingman coalescent for a
  fixed time horizon.

  For the \emph{convergence} we consider now the
  dual expressions for the expectations of polynomials as the time
  horizon $t \to \infty$, where the dual expression is the
  asymptotically given by the time-space coalescents at time
  $t_1<t_2<\cdots < t_n<t$ as $t \to \infty$ and using the fact that
  \begin{equation}
    \label{e7724}
    \int_0^t \bar \mfU_s \dx s \to \int_0^\infty \bar \mfU_s \dx s.
  \end{equation}
\end{proof}

Finally we have to show the limit is given by $[CRT]$. We have to deal
with the convergence of the total mass process and the $\U_1$-valued
part. For definition of the CRT from a Brownian excursion on $[0,1]$
we refer to \cite{LeGall93}, see also \cite{PWak13}.

In order to obtain that the distance matrices of samples from the CRT
agree with the $t \to \infty$ object of our $\M$-valued fossil Feller
diffusion we work with the polynomials and the fact that $\bar\mfU$
becomes extinct, so that after a random time $\bar T_{\mathrm{ext}}$
the fossil process is constant, namely equal to the $\ntree$-element.
Therefor polynomials converge as $t\to\infty$ to a limit. This limit
has to be identified as CRT.

At this point we use the results on the convergence of the individual
based models. For the CRT this goes back to Aldous \cite{Ald1991a} and
the result for our dynamic this is a simple extension of what was done
in \cite{GPW09,Gl12}. For the individual based model the two objects
are identical by inspection. Therefore the two limit objects are
identical. We omit the standard details for this extension.

\appendix
% \appendixpage % http://www.tex.ac.uk/cgi-bin/texfaq2html?label=appendix
% \addappheadtotoc %

\section{Computation of diffusion coefficients}
\label{sec:comp-diff-coeff}

We know from \cite{LN68} that the Feller diffusion conditioned to
survive until some fixed time $T>0$ is again a Markov process which is
time-inhomogeneous. In fact it is a diffusion process. Hence the
process has local characteristics which we want to calculate for
general coefficient $b$ and for arbitrary time horizon $T$. In
\cite{LN68} the case $b=1$ and $T=1$ was considered with an
\emph{error} in the calculation of the volatility which we correct
here.

We carry out the calculations using the Laplace transform of the
conditioned process. Throughout this section we denote by
$(P_{s,t}(x,\cdot))_{0 \le s \le t}$ the family of transition kernels
of the (unconditioned) $\R_+$-valued Feller diffusion. We denote by
$(P_{s,t}^T (x,\cdot))_{0 \le s < t \le T}$ the family of transition
kernels of the $\R_+$-valued Feller diffusion conditioned to survive
up to time $T$. For fixed $0\le s < t \le T$ we need to calculate the
local characteristics of $P_{s,t}^T (x, \cdot)_{0 \le s < t \le T}$
for $x>0$. We denote the corresponding Laplace transform by
\begin{align}
  \label{eq:lTxst}
  L^T(x,s,t;\lambda) \coloneqq \int e^{-\lambda y} P_{s,t}^T (x,\dx
  y), \quad \lambda \ge 0.
\end{align}

First note, that we have the following elementary identities
\begin{align}
  \label{e1276}
  \int P^T_{s,t} (x,\dx y)(y-x)= - \frac{\partial}{\partial \lambda}
  \left(e^{\lambda x} L^T(x,s,t;\lambda)\right) \mid_{\lambda=0},
\end{align}
\begin{align}
  \label{e1278}
  \int P_{s,t}^T (x,\dx y)(y-x)^2=
  \frac{\partial^2}{\partial \lambda^2} \left(e^{\lambda x}
  L^T(x,s,t,\lambda)\right) \mid_{\lambda=0}.
\end{align}
Dividing by $(t-s)$ and taking $t \downarrow s$ we will get the
infinitesimal drift and variance of the conditioned process at time
$s$ is in state $x$. The rest of this section is devoted to the
computations.

\medskip
The Laplace transform of the $\R_+$-valued Feller diffusion at time
$t$ starting in $x$ is given by
\begin{align}
  \label{e1260}
  L(x,t;\lambda) = \exp \left(- \frac{2x\lambda}{2+bt\lambda} \right),
  \quad \lambda \ge 0.
\end{align}

From that we obtain the probability that the Feller diffusion survives
until time $t$ as
\begin{align}
  \label{e1264}
  1- \lim_{\lambda \to \infty} L(x,t;\lambda) = 1 - \exp \left(-
  \frac{2x}{bt} \right).
\end{align}
If we denote by $(Z_t)_{t \ge 0}$ the $\R_+$-valued Feller diffusion
starting with positive initial conditions then for $x>0$ the Laplace
transform $L^T(x,s,t;\lambda)$ satisfies
\begin{align}
  \label{e1268}
  \begin{split}
    L^T(s,x,t; \lambda)
    & = \E \left[\exp (-\lambda Z_t) |Z_s=x,Z_T >0 \right] \\
    & = \Bigl(1- \exp \Bigl( - \frac{2x}{b(T-s)} \Bigr) \Bigr)^{-1} \\
    & \qquad \cdot \biggl(
    \exp\Bigl(-\frac{2x\lambda}{2+(t-s)b\lambda}\Bigr) -
    \exp\Bigl(-\frac{2x(T-t)\lambda b+4}{(T-t)(t-s)\lambda b^2 +
      2(T-s) b}\Bigr) \biggr).
  \end{split}
\end{align}
This follows by the next calculation
\begin{align}
  \label{e1274}
  \begin{split}
    \int P_{s,t} (x,\dx y)e^{-\lambda y}
    & \cdot \left(1-P_{t,T} (y,0)\right)\\
    & = \int P_{s,t} (x,\dx y)e^{-\lambda y} \;
    \Bigl(1-\exp\Bigl(- \frac{2y}{b(T-t)}\Bigr)\Bigr)\\
    & = \int P_{s,t} (x,\dx y)e^{-\lambda y} - \int P_{s,t} (x,\dx y)
    \exp\Bigl(-y(\lambda + \frac{2}{b(T-t)}\Bigr)\Bigr)\\
    & = L(x,t-s;\lambda) - L\Bigl(x,t-s;\lambda +
    \frac{2}{b(T-t)}\Bigr).
  \end{split}
\end{align}
Using \eqref{e1260} and simplifying the obtained expression one easily
arrives at \eqref{e1268}.

Now we want to compute the diffusion coefficients using the Laplace
transform. Denoting by $P^T_{s,s+h}(x,\dx y)$ the transition density
corresponding to the Laplace transform $L^T(s,x,s+h; \cdot)$ the
infinitesimal drift is given by
\begin{align}
  \label{eq:diffq-a}
  \wt a_T(s,x)
  = \lim_{h\to 0} \frac{1}{h} \int (y-x) P^T_{s,s+h}(x,\dx y)
  = - \lim_{h \to 0} \frac{1}{h} \frac{\partial}{\partial\lambda}
  \bigl(e^{\lambda x}
  L^T(s,x,s+h; \lambda)\bigr) \Big\vert_{\lambda=0}.
\end{align}
and the infinitesimal variance by
\begin{align}
  \label{eq:diffq-b}
  \wt b_T(s,x)
  = \lim_{h\to 0} \frac{1}{h} \int (y-x)^2 P^T_{s,s+h}(x,\dx y)
  = \lim_{h \to 0} \frac{1}{h} \frac{\partial^2}{\partial\lambda^2}
  \bigl(e^{\lambda x} L^T(s,x,s+h; \lambda)\bigr) \Big\vert_{\lambda=0}.
\end{align}

For fixed $T,s,h$ with $0 \le s \le s+h \le T$ we define functions
$f_1$ and $f_2$ by
\begin{align}
  \label{eq:f-function1}
  f_1(\lambda)
  & = \frac{2\lambda}{2+hb\lambda} \\
  \intertext{and}
  \label{eq:f-function2}
  f_2(\lambda)
  & = \frac{2(T-s-h)\lambda b+4}{(T-s-h)h\lambda
    b^2 + 2(T-s)b}.
\end{align}
Then, we have
\begin{align}
  \label{eq:LT-viaf1f2}
  e^{\lambda x} L^T(s,x,t;\lambda) =
  \frac{1}{1-\exp(-\tfrac{2x}{b(T-s)})}e^{\lambda
  x}(\exp(-xf_1(\lambda)) - \exp(-xf_2(\lambda))),
\end{align}
where the first factor on the right hand side does not depend on
$\lambda$. We need to compute the first and second derivatives of
functions of the form
\begin{align}
  \label{eq:8}
  \lambda \mapsto e^{x\lambda} e^{-xf(\lambda)}
\end{align}
at $\lambda=0$. The first is given by
\begin{align}
  \label{eq:8d1}
  \frac{\partial}{\partial\lambda} (e^{x\lambda}
  e^{-xf(\lambda)})\Big\vert_{\lambda=0}  = e^{x\lambda}
  e^{-xf(\lambda)} (x-xf'(\lambda))\Big\vert_{\lambda=0} =
  e^{-xf(0)} (x-xf'(0)),
\end{align}
and the second by
\begin{align}
  \label{eq:8d2}
  \begin{split}
    \frac{\partial^2}{\partial\lambda^2} (e^{x\lambda}
    e^{-xf(\lambda)})\Big\vert_{\lambda=0} & = e^{x\lambda}
    e^{-xf(\lambda)} ((x-xf'(\lambda))^2
    -xf''(\lambda))\Big\vert_{\lambda=0} \\
    & = e^{-xf(0)} ((x-xf'(0))^2-xf''(0)).
  \end{split}
\end{align}

For $f_1$ from \eqref{eq:f-function1} we have $f_1(0) = 0$ and
\begin{align}
  \label{eq:f-function1-d1}
  f'_1(\lambda) & = \frac{4}{(2+h b\lambda)^2}, \quad
                  f'_1(0) =1\\
  \label{eq:f-function1-d2}
  f''_1(\lambda) & = -\frac{8hb}{(2+hb\lambda)^3}, \quad
                   f''_1(0)=-hb.
\end{align}
For $f_2$ from \eqref{eq:f-function2} we have
$f_2(0) = \tfrac{2}{b(T-s)}$ and
\begin{align}
  \label{eq:f-function2-d1}
  f'_2(\lambda)
  & = \frac{4(T-s-h)^2}{((T-s-h)h\lambda
    b + 2(T-s))^2},
    \quad f'_2(0) = \frac{(T-s-h)^2}{(T-s)^2} = 1-
    \frac{2h}{T-s} + \frac{h^2}{(T-s)^2}\\
  \label{eq:f-function2-d2}
  f''_2(\lambda)
  & = - \frac{8(T-s-h)^3h b}{((T-s-h)h\lambda
    b + 2(T-s))^3} , \quad
    f''_2(0) = - \frac{(T-s-h)^3h b}{(T-s)^3}.
\end{align}
It follows
\begin{align}
  \label{eq:9}
  \frac{\partial}{\partial\lambda} \Bigl(e^{\lambda x}
  \exp\Bigl(-x f_1(\lambda)\Bigr) \Big\rvert_{\lambda=0} = 0
\end{align}
and
\begin{align}
  \label{eq:911}
  \frac{\partial}{\partial\lambda} \Bigl(e^{\lambda x}
  \exp\Bigl(-x f_2(\lambda)\Bigr) \Big\rvert_{\lambda=0}
  = \exp\Bigl(-\frac{2x}{(T-s)b}\Bigr)
  \Bigl(\frac{2hx}{T-s} -\frac{xh^2}{(T-s)^2}\Bigr).
\end{align}
Using next \eqref{eq:diffq-a} we obtain from the above
\begin{align}
  \label{eq:diff-a-res}
  \begin{split}
    \wt a_T(s,x)
    & = \frac{1}{ 1- \exp \left( - \frac{2x}{(T-s)b}
      \right)} \lim_{h \to 0} \frac{1}{h} \exp\Bigl(-\frac{2x}{(T-s)b}\Bigr)
    \Bigl(\frac{2hx}{T-s} -\frac{xh^2}{(T-s)^2}\Bigr) \\
    & = \frac{2x}{(T-s)}\frac{1}{\exp\Bigl(\frac{2x}{(T-s) b}\Bigr)-1}.
  \end{split}
\end{align}

It remains to compute $b_T(s,x)$. With the above preparations we have
\begin{align}
  \label{eq:10}
  \frac{\partial^2}{\partial\lambda^2} \Bigl(e^{\lambda x}
  \exp\Bigl(-x f_1(\lambda)\Bigr)\Bigr)
  \Big\rvert_{\lambda=0} = xhb
\end{align}
and
\begin{multline}
  \label{eq:5}
  \frac{\partial^2}{\partial\lambda^2} \Bigl(e^{\lambda x}
  \exp\Bigl(-x f_2(\lambda)\Bigr)\Bigr)
  \Big\rvert_{\lambda=0} \\ = \exp\Bigl(-\frac{2x}{(T-s)b}\Bigr)
  \Bigl(\Bigl(\frac{2hx}{T-s} -\frac{xh^2}{(T-s)^2}\Bigr)^2
  + x \frac{(T-s-h)^3 h b}{(T-s)^3} \Bigr).
\end{multline}
Hence:
\begin{align}
  \label{eq:13}
  \begin{split}
    \wt b_T(s,x)
    & = \lim_{h\to 0} \frac{1}{h}   \frac{\partial^2}{\partial\lambda^2}
    \bigl(e^{\lambda x} L^T(s,x,s+h; \lambda)\bigr) \Big\vert_{\lambda=0}\\
    & = \frac{1}{1-\exp \left( - \frac{2x}{(T-s)b} \right)} \\
    & \quad \cdot \lim_{h\to 0} \frac{1}{h} \Bigl(xhb -
    \exp\Bigl(-\frac{2x}{b(T-s)}\Bigr) \Bigl(\Bigl(\frac{2hx}{T-s} -
    \frac{h^2x}{(T-s)^2} \Bigr)^2 + x \frac{(T-s-h)^3 h b}{(T-s)^3}
    \Bigr)\Bigr) \\
    & = \frac{1}{1-\exp \left( - \frac{2x}{(T-s)b} \right)} \Bigl(xb -
    xb \exp\Bigl(-\frac{2x}{(T-s)b}\Bigr) \Bigr) \\
    & = xb.
  \end{split}
\end{align}

\section{Facts for Markov branching trees}
\label{app.facts}

In this section we collect some facts about $\U$-valued $t$-Markov
branching trees ($t$-MBT). These are $\U$-valued random variables
$\mfU$ whose $t$-tops $\lfloor \mfU \rfloor (t)$ have infinitely
divisible laws with \Levy{}-Khintchine representation \eqref{ag2inf}
and \Levy{} measure of the form \eqref{e.tr12}.

We denote by $D(I,E)$ the set of \cadlag{} paths on an interval
$I\subset \R$ with values in a measurable space $E$. The set $D(I,E)$
is equipped with the classical Skorohod topology; see Chapter~3 in
\cite{EK86} for details.

We begin by studying the number of open balls of certain diameter
$t-s$ in a $t$-MBT. If we do this for fixed $t$ and vary $s$ we obtain
a \emph{ball counting process} with values in $\N_0$. For
$\mfu = [U,r,\mu] \in \U$ and $h >0$ we denote by $\anz_h(\mfu)$ the
number of open $2h$-balls in the metric space $(U,r)$.

The following result states that the process of the number of balls
indexed by the decreasing radius is a \emph{Markov branching process
  on $\N_0$}.

\begin{proposition}
  Consider \label{p.MBT.balls.new} the critical $\U$-valued rate $b$
  Feller diffusion. Then for $t >0$ and $s \in [0,t)$ the number
  $(M_{s,t})_{s \in [0,t)}$ of $2(t-s)$-balls is a continuous time
    branching process uniquely determined by the marginal laws at time
    $s \in [0,t)$, which conditionally on $\bar\mfu_{t-s}$ are
    \begin{align}
      \label{eq:44}
      \Pois\Bigl(\frac{2 \bar\mfu_{t-s}}{b(t-s)}\Bigr)
    \end{align}
    distributed.
\end{proposition}
\begin{proof}
  The assertion follows using the property of Galton-Watson processes
  which we formulated as Lemma~\ref{FellDiffDecomp} and the connection
  of such processes with Feller diffusion. %xxx
\end{proof}

The general case of Proposition~\ref{p.MBT.balls.new} formulated below
as a conjecture will be proved in \cite{ggr_MBT}.

\begin{conjecture}[Number of covering balls of varying radius]
  \label{p.MBT.balls}
  Assume that $\mfU$ is a $t$-Markov branching tree with almost surely
  finite measure. Let
  \begin{align}
    \label{ball1}
    M_s\coloneqq M_s(\mfU) \coloneqq  \anz_{t-s}(\mfU),\;  s \in [0,t)
  \end{align}
  be the (random) number of open $2(t-s)$-balls in the metric space
  from the metric measure space $\mfU$. Then
  $(M_s(\mfU))_{s\in [0,t)}$ is a non decreasing Markov branching
  process with values in $D([0,t],\N_0)$. The state $0$ is an
  absorbing state of this process and for $k,\ell \in \N_0$ and
  $0\leq s < s' < t$ the transition probabilities are given by
  \begin{multline}
    \label{grx59}
    \P( M_{s'} = k+\ell | M_s = k) = \\
    \int_{(\U(t-s))^k} (\varrho^t_{t-s})^{\otimes
      k}(\dx\mfu_1,\dots,\dx\mfu_k) \, \1\{ \anz_{t-s'}(\mfu_1) + \dots +
    \anz_{t-s'}(\mfu_k) = k+\ell\}.
  \end{multline}
  Here, $\varrho_{t-s}^t$ is as in Definition~\ref{d.mbt} and
  Remark~\ref{r.1394}. The initial distribution $\P(M_0 \in \cdot)$ is
  a mixed Poisson distribution with mixing measure
  $\int m_0(\dx \mfv) \ind{\bar{\mfv} \in \cdot}$.
\end{conjecture}

Next we reformulate and refine some general results from Section~4 in
\cite{infdiv} which have a special form in the branching context.
These results are concerned with some properties of limit points of
the \emph{approximate excursion law}
\begin{align}
  \label{e.tr39}
  \hat\varrho_h^{h,(n)}(\cdot) \coloneqq  \1\{\cdot \neq \ntree\}\,
  nQ_h(\frac{1}{n}\mfe, \cdot).
\end{align}

\begin{proposition}[Excursion law and $\hat\varrho_h^t$]
  \label{thm:bransemig-brantree:1}
  Let $(Q_t)_{t\geq 0}$ be a Feller semigroup that has the branching
  property. Then for all $0 < h \le t$ there exists
  $\hat\varrho_h^t \in \mcM^\#(\U(h)^\sqcup \setminus \{\ntree\})$
  s.t.\ for all $\mfu \in \U$ and $\Phi \in \Pi_+$ we have
  \begin{align}
    \label{e1325}
    Q_t(\mfu, e^{-\Phi_h(\cdot)}) = \int Q_{t-h}(\mfu,\dx \mfv) \,
    \exp\left( - \bar{\mfv} \int \hat\varrho_h^t(\dx \mfw)
    (1-e^{-\Phi_h(\mfw)}) \right).
  \end{align}
  Furthermore, the measure $\hat\varrho_h^t$ does not depend on $t$ an
  is the boundedly weak limit of the measures $\hat\varrho_h^{h,(n)}$
  defined in \eqref{e.tr39}.
\end{proposition}
\begin{proof}
  \label{p.branstree}
  Let $\mfU_t$ be a realization of a random variable with law
  $Q_t(\mfu,\cdot)$. We calculate the Laplace transform of $\mfU_t$
  for $t \geq 0$ at depth $h \in (0,t]$. We have
  \begin{align}
    \label{e1324}
    \E_\mfu [\exp(- \Phi_h(\mfU_t))]
    &  = \int Q_t (\mfu, \dx \mfv) \, \exp(-\Phi_h(\mfv)) \\
    \label{e1324b}
    & = \int Q_{t-h}(\mfu, \dx \mfw) \int Q_h(\mfw,\dx \mfv)
      \, \exp(-\Phi_h(\mfv)) \\
     \label{e1324c}
    & =  \int Q_{t-h}(\mfu,\dx \mfw) \left( \int
      Q_h(\frac{1}{n} \mfe, \dx \mfv) \, \exp
      (-\Phi_h(\mfv)) \right)^{n\bar{\mfw}} \qquad
      (\text{Lemma~\ref{l.Lapsemigroup}})\\
    \label{e1324d}
    & =  \int Q_{t-h}(\mfu,\dx \mfw) \left( 1 -\frac{1}{n}
      \int nQ_h(\frac{1}{n} \mfe, \dx \mfv) \,\left( 1- \exp
      (-\Phi_h(\mfv))\right) \right)^{n\bar{\mfw}}.
  % & = \int Q_{t-h}(\mfu,d \mfw) \left( Q_h(\frac{1}{n}
  %   \mfe, \{ \cdot = 0 \}) + \int Q_h(\frac{1}{n} \mfe,
  %   d \mfv)\, \1(\mfv \neq 0) \exp (-\Phi_h(\mfv))
  % \right)^{n\bar{\mfw}}.
  \end{align}
  Again by Lemma~\ref{l.Lapsemigroup} we have
  $Q_h(\mfe,\cdot) = Q_h(\frac{1}{n}\mfe,\cdot)^{*n}$, so that the
  property (4.7) in \cite{infdiv} holds and we can use results on
  excursion laws of Section~4 from \cite{infdiv}. In particular, by
  Lemma~4.3 in \cite{infdiv} and $n \to \infty$ in the last line of
  the above display we deduce that
  \begin{align}
    \label{e.tr104}
    \E_\mfu [\exp(- \Phi_h(\mfU_t))] = \int Q_{t-h}(\mfu,\dx \mfw)
    \exp\left(- \bar{\mfw} \int \hat\varrho_h^t(\dx \mfv) \,
    (1-e^{-\Phi_h(\mfv)}) \right),
  \end{align}
  for a certain measure
  $\hat\varrho_h^t \in \mcM^\#(\U(h)^\sqcup\setminus\{\ntree\})$ with
  $\int \hat\varrho_h^t(\dx \mfu) \, (1\wedge \bar{\mfu}) < \infty$.
  We have seen that in the case where the branching rate $b$ is
  constant in time we have $\hat\varrho^t_h=\hat\varrho^{t'}_h$ for
  all $t' \geq t \geq h>0$. Thus, we have
  $\hat\varrho_h^t=\hat\varrho^h_h$ for $t \geq h >0$. Finally, again
  by Lemma~4.3 in \cite{infdiv}, this measure $\hat\varrho_h^h$ is the
  weak limit of $nQ_h(\frac{1}{n}\mfe,\cdot)$.
\end{proof}

\begin{lemma}[\cite{Grey:74}]
  \label{l.Grey}
  If $(\bar{Q}_t)_{t\geq 0}$ is a branching semigroup on $[0,\infty)$
  with branching mechanism $\Psi$ and satisfying the Feller property,
  then for all $t \geq 0$ the following assertions hold:
  \begin{enumerate}[(a)]
  \item \label{i:excursion}
    \begin{align}
      \label{e.D.e(t)} e(t) \coloneqq
      \lim_{n\to \infty} [n\bar{Q}_t(n^{-1}, \{\cdot \neq 0\})] =
      \psi_t(\infty) \in (0, \infty] \, ,
    \end{align}
    where the value $\infty$ is not attained if $\Psi(\theta) > 0$ and
    $\int_\theta^\infty \dx \xi / \Psi(\xi) < \infty$ for large enough
    $\theta$.
  \item \label{i:exp:totmass} Furthermore:
    \begin{align}
      \label{ae1322}
      \lim_{n\to \infty} n \int \bar{Q}_t(n^{-1}, \dx z) z =  \int
      \bar{Q}_t(1, \dx z) \, z = \exp( -t \Psi'(0)),
    \end{align}
    where
    $\Psi'(0) = - \alpha -\int_0^\infty x(1-e^{-x})\, \Pi (\dx x) \in
    [-\infty, \infty)$.
  \end{enumerate}
\end{lemma}
\begin{proof}
  \label{pr.1962}
  The first equality of \eqref{i:exp:totmass} is given by the
  branching property and the second equality is shown below the proof
  of Theorem~1 in \cite{Grey:74}. To show \eqref{i:excursion}, a proof
  similar to Lemma~4.4 in \cite{infdiv} shows the existence of the
  limit and that $e(t) = - \log \bar{Q}_t(\mfe,\mfu = 0)$. But the
  latter is the extinction probability after time $t$ and Theorem~1 in
  \cite{Grey:74} and its preceding lines show the rest of the
  statement.
\end{proof}

\begin{lemma}[Laplace transform and branching property]
  \label{l.Lapsemigroup}
  Suppose $(Q_t)_{t\geq 0}$ is a Feller semigroup with the branching
  property. Then
  \begin{enumerate}
  \item for any $\Phi \in \Pi_+, t\geq 0$ and $\mfu \in \U$
    \begin{align}
      \label{e2108131928}
      \int Q_t(\mfu,\dx \mfv) e^{-\Phi_{t}(\mfv)}
      = \Bigl(\int Q_t(\mfe,\dx \mfv) e^{-\Phi_{t}(\mfv)}\Bigr)^{\bar{\mfu}}.
    \end{align}
  \item for any $\Phi \in \Pi_+$, $t\geq 0$ and $\mfu \in \U$
    \begin{align}
      \label{e.Lapsemigroup}
      \int Q_t(\mfu,\dx \mfv) \Phi_t(\mfv) = \bar{\mfu}
      \int Q_t(\mfe,\dx \mfv) \Phi_t(\mfv),
    \end{align}
    as long as the expressions involved are finite.
  \end{enumerate}
\end{lemma}

\begin{proof}\label{pr1851}
  The second claim follows from the first by differentiation, i.e.\
  consider $\lambda \Phi_t$ for $\lambda \ge 0$ instead of $\Phi_t$ in
  \eqref{e2108131928}, differentiate both sides with respect to
  $\lambda$ and evaluate the resulting equality at $\lambda = 0$.

  For the first claim let $\mfu \in \bbU_f$ be an ultrametric measure
  space with finitely many points, that is we have
  $\mfu=[\{1,\dotsc,n\},r,\sum_{i=1}^np_i\delta_i]$ for some metric
  $r$ and weights $p_i$, $i=1,\dots,n$. We assume that $n\geq2$ and
  define
  $\alpha\coloneqq \min\{r(x,y):x,y\in \{1,\dotsc,n\},x\neq y\}>0$.
  Furthermore, for $i=1,\dots,n$ we define
  \begin{align}
    \label{rg91}
    \mfp_i\coloneqq \mfp^{(p_i)}\coloneqq [\{i\}, 0, p_i \delta_{i}].
  \end{align}
  Then, for any $t\in[0,\alpha]$, by the branching property we have
  \begin{align}
    \label{e.Lap:branch}
    \int Q_t(\mfu,\dx \mfv) e^{-\Phi_t(\mfv)}
    % &=(Q_t(\mfp_1,\cdot) \ast^h \dotsm \ast^h
    % Q_t(\mfp_n,\cdot)) (e^{-\Phi_{t}})\\
    % &= \int Q_t(\mfp_1,d \mfv_1) \int Q_t(\mfp_2,d \mfv_2) \dotsm
    % \int Q_t(\mfp_n, d \mfv_n) \
    % e^{-\Phi_{t}(\bigsqcup_{i=1}^m \mfv_i) } \\
      & = \prod_{i=1}^n \int  Q_t(\mfp_i,\dx \mfv_i)
        e^{-\Phi_{t}( \mfv_i) }\,.
  \end{align}
  Now, assume that $n=2$ and write $\mfu_\alpha$ to indicate the
  dependence on $\alpha$. Then, \eqref{e.Lap:branch} becomes
  \begin{align}
    \label{e.Lap:branch:2}
    \int Q_t(\mfu_\alpha,\dx \mfv) e^{-\Phi_t(\mfv)} =
    \int Q_t(\mfp_1,\dx \mfv) e^{-\Phi_t(\mfv)}
    \int Q_t(\mfp_2,\dx \mfv) e^{-\Phi_t(\mfv)}.
  \end{align}
  On the other hand we know that
  \begin{align}
    \label{rg92}
    \lim_{\alpha \to 0} \mfu_\alpha = \mfp^{(p_1+p_2)}=\mfp^{\bar{\mfu}},
  \end{align}
  in the Gromov-Prohorov topology. Since $(Q_t)_{t\geq 0}$ is a
  Feller semigroup we obtain
  \begin{align}
    \label{rg93}
    \int Q_t(\mfp^{(p_1+p_2)},\dx \mfv) e^{-\Phi_{t}(\mfv)}
    & = \lim_{\alpha\to 0} \int Q_t(\mfu_\alpha,\dx \mfv)
      e^{-\Phi_{t}(\mfv)}  \\
    & = \int Q_t(\mfp_1,\dx \mfv) e^{-\Phi_t(\mfv)}
    \int Q_t(\mfp_2,\dx \mfv) e^{-\Phi_t(\mfv)}.
      \label{rg93b}
  \end{align}
  For fixed $t>0$, this is a functional equation in the parameter $p$:
  $p \mapsto f_t(p) = Q_t(\mfp^{(p)},e^{-\Phi_{t}})$, i.e.\ we have
  $f_t(p_1+p_2) = f_t(p_1)f_t(p_2)$, $p_1,p_2\in(0,\infty)$. By the
  Feller property we also know that $p\mapsto f_t(p)$ is continuous
  and so we obtain the well-known solution
  \begin{align}
    \label{e2108131929}
    \int Q_t(\mfp^{(p)},\dx \mfv) e^{-\Phi_{t}(\mfv)}
    = f_t(p) = (f_t(1))^p = \Bigl(\int Q_t(\mfp^{(1)},\dx \mfv)
    e^{-\Phi_{t}(\mfv)} \Bigr)^p.
  \end{align}

  Using \eqref{e.Lap:branch:2}, we can extend this to the case $n>2$
  and obtain
  \begin{align}
    \label{e.Lap:branch:3}
    \int Q_t(\mfu,\dx \mfv) e^{-\Phi_{t}(\mfv)} =
    \Bigl(\int Q_t(\mfe,\dx \mfv) e^{-\Phi_{t}(\mfv)}\Bigr)^{(p_1 + \cdots + p_n)} =
    \Bigl(\int Q_t(\mfe,\dx \mfv) e^{-\Phi_{t}(\mfv)}\Bigr)^{\bar\mfu}.
  \end{align}
  where we write $\mfe=\mfp^{(1)}$. Taking into account
  \eqref{e2108131929} to cover the case $n=1$, we have now proved
  \eqref{e2108131928} for finite um-space.

  For the extension to general um-spaces, note that any um-space can
  be approximated in the Gromov-weak topology by a sequence of finite
  um-spaces (see \cite{Gl12}, Proposition 2.3.13, or \cite{GPWmp13},
  Proposition 5.6 for the normalized case; the extension to general
  mm-spaces is immediate). Let $\mfu\in\bbU$ and assume that
  $\mfu_n\in\bbU_f$, $n\in\bbN$, are finite um-spaces such that
  $\mfu_n\to\mfu$. Note that this implies
  $\bar{\mfu}_n^2=\langle1,\nu^{2,\mfu_n}\rangle\to\langle1,\nu^{2,\mfu}\rangle
  =\bar{\mfu}^2$ and thus $\bar{\mfu}_n\to\bar{\mfu}$. Therefore we
  obtain
  \begin{align}
    \label{rg94}
    \begin{split}
      \int Q_t(\mfu,\dx \mfv) e^{-\Phi_t(\mfv)}
      & =\int Q_t(\lim_{n\to\infty}\mfu_n, \dx \mfv) e^{-\Phi_t(\mfv)} \\
      & =\lim_{n\to\infty} \int Q_t(\mfu_n,\dx \mfv) e^{-\Phi_t(\mfv)} \\
      & =\lim_{n\to\infty} \Bigl(\int Q_t(\mfe,\dx \mfv)
      e^{-\Phi_t(\mfv)}\Bigr)^{\bar{\mfu}_n}
      = \Bigl(\int Q_t(\mfe,\dx \mfv) e^{-\Phi_t(\mfv)}\Bigr)^{\bar{\mfu}}.
    \end{split}
  \end{align}
  In the second equality we use the Feller property of $Q_t$, in the
  third equality we use that we have already proved the result for
  finite um-spaces.
\end{proof}

\begin{proposition}
  \label{l.totmass:semigroup}
  Suppose $(Q_t)_{t\geq 0}$ is a Feller semigroup and has the
  branching property. Then there exists a Feller semigroup
  $(\bar{Q}_t)_{t\geq0}$ on $C_b([0,\infty))$ with the branching
  property such that:
  \begin{align}
    \label{e.totalmass.semigroup}
    \int Q_t(\mfu, \dx\mfv) \, \1\{\bar{\mfv} \in \dx y\}
    = \bar{Q}_t(x, \dx y)
  \end{align}
  as measures on $[0,\infty)$ for all $\mfu \in \U$ with
  $\bar{\mfu} = x \in [0,\infty)$. Moreover, $\bar{Q}_t(x,\cdot)$ is
  an infinitely divisible distribution on $[0,\infty)$ for any $t>0$
  and $x >0$.
\end{proposition}
\begin{proof}
  \label{p.totsem}
  Let $x \in [0,\infty)$. Define
  \begin{align}
    \label{e865}
    \bar{Q}_t (x,\dx y ) = \int Q_t(x \cdot \mfe, \dx\mfv)
    \1\{\bar{\mfv} \in \dx y\}.
  \end{align}
  By Lemma~\eqref{l.Lapsemigroup} for $\lambda >0$:
  \begin{align}
    \label{e.tr7}
    \begin{split}
      \int \bar{Q}_t (x,\dx y) e^{-\lambda y} = \int Q_t(x \cdot \mfe,
      \dx\mfv) \1\{\bar{\mfv} \in \dx y\} e^{-\lambda y} & = \left( \int
        Q_t(\mfe, \dx\mfv) \1\{\bar{\mfv} \in \dx y\}
        e^{-\lambda y} \right)^x \\
      & = \int Q_t(\mfu, \dx\mfv) \1\{\bar{\mfv} \in \dx y\} e^{-\lambda y},
    \end{split}
  \end{align}
  for any $\mfu \in \U$ with $\bar{\mfu} = x$. That shows
  \eqref{e.totalmass.semigroup}. Next $(\bar{Q}_t)_{t\geq0}$ is a
  Markov semigroup since,
  \begin{align}
    \label{e1921}
    \int \bar{Q}_t(x,\dx y) \int \bar{Q}_s(y,\dx z)
    & = \int Q_t(x\cdot \mfe, \dx\mfv)\, \1(\bar{\mfv} \in \dx y )
      \int Q_s(y\cdot \mfe, \dx\mfw) \, \1\{\bar{\mfw} \in \dx z\} \\
    \label{e1921b}
    & \stackrel{\eqref{e.tr7}}{=} \int Q_t( x \cdot \mfe, \dx\mfv)
      \int Q_s(\nu, \dx\mfw) \, \1\{\bar{\mfw} \in \dx z\}\\
    \label{e1921c}
    & = \int Q_{t+s}(x \cdot \mfe, \dx\mfw) \, \1(\bar{\mfw} \in
      \dx z) = \int \bar{Q}_{t+s}(x,\dx z).
  \end{align}
  The Feller property follows from the corresponding property of
  $(Q_t)_{t\geq 0}$. It remains to verify the branching property;
  therefore let $\mfu_1,\mfu_2 \in \U$ with $\bar{\mfu}_1 = x_1$,
  $\bar{\mfu}_2 = x_2$. Then
  \begin{align}
    \label{e1320}
    \int & \bar{Q}_t(x_1, \dx y_1)  \int \bar{Q}_t(x_2,\dx y_2) \,
           \1\{y_1 + y_2 \in \dx z\} \\
    \label{e1320b}
         & = \int Q_t(\mfu_1, \dx\mfv_1) \int Q_t(\mfu_2, \dx\mfv_2)
           \, \1\{\bar{\mfv}_1 + \bar{\mfv}_2 \in \dx z\} \\
    \label{e1320c}
         & = \int Q_t(\mfu_1, \dx\mfv_1) \int Q_t(\mfu_2, \dx\mfv_2)
           \, \1(\overline{\mfv_1 \sqcup \mfv_2} \in \dx z)\\
    \label{e1320d}
         & \stackrel{\text{bran.~prop}}{=} \int Q_t(\mfu_1 \sqcup
           \mfu_2, \dx\mfv) \, \1(\bar{\mfv} \in \dx z ) = \int
           \bar{Q}_t(x_1+x_2,\dx y) \1(y \in \dx z) \\
    \label{e1320e}
         & = \bar{Q}_t(x_1+x_2,\dx z).
  \end{align}
  The last claim is clear, since marginal distributions of branching
  processes are infinitely divisible.
\end{proof}

\section{Infinite divisibility and Markov branching trees}
\label{sec:infin-divis-xxx}

In this section we show that the distributions of (generalized)
$t$-branching trees are infinitely divisible and identify the
corresponding \Levy{} measures. Here, for $t\in (0,\infty)$ extending
Definition~\ref{d.mbt}, we say that an $\U(t)^{\sqcup}$-valued random
variable $\mfU$ is a \emph{(generalized) $t$-branching tree} if for
every $h \in (0,t]$ the $h$-top $\lfloor \mfU\rfloor (h)$ can be
written in the form
\begin{align}
  \label{n.rgx7}
  \lfloor \mfU\rfloor (h) \coloneqq  \bigsqcup_{\mfu \in N} \mfu
  \in \U(h)^\sqcup.
\end{align}
Here, for $h \in (0,t]$, $N$ is PPP on $\U$ arising as follows:
\begin{itemize}
\item $m_h^t$ is an infinitely divisible law on $\U(t-h)^\sqcup$ with
  the \Levy{} measure $\lambda^{m_h^t}$,
\item $\varrho_h^t$ is a kernel on $\U(t-h)^\sqcup \times \mathcal
  M^\# (\U(h)\setminus \{\ntree\})$, $\mfv \mapsto \varrho_h^t
  (\mfv,\cdot)$,
\item first $\mfv$ is drawn according to $m_h^t$, then
  $N=N(\varrho_h^t(\mfv,\cdot))$ is PPP on $\U(h)^\sqcup$ with
  intensity measure $\varrho_h^t(\mfv,\cdot)$.
\end{itemize}

We start by proving infinite divisibility of general $t$-branching
trees.

\begin{proposition}
  \label{l.ac7289}
  Any $t$-branching tree is $t$-infinitely divisible.
\end{proposition}
\begin{proof}
  \label{pr.7343}
  Let $\mfU$ be a $t$-branching tree, $n\in \N$ and $h \in (0,t]$.
  Since $m_h^t$ is infinitely divisible there is a law
  $m_h^{t,(n)} \in \mcM_1(\U(t-h)^{\sqcup})$ so that for
  $\U(t-h)^{\sqcup}$-valued random variables $\mfV$ and i.i.d.\
  $\mfV^{i,n}$, $i=1,\dots,n$ with $\mcL (\mfV) = m_h^t$ respectively
  $\mcL(\mfV^{1,n}) = m_h^{t,(n)}$ we have
  \begin{align}
    \label{n.rgx5}
    \mfV \eqd \mfV^{1,n} \sqcup^{(t-h)} \dots \sqcup^{(t-h)}
    \mfV^{n,n}.
  \end{align}
  Using this representation and the description after \eqref{n.rgx7}
  for all non-negative $\Phi \in \Pi$ we obtain
  \begin{align}
    \label{n.rgx6}
    \begin{split}
      \E[\exp (-\Phi_{\hu}(\mfU) ) ]
      & = \int_{\U(t-h)^{\sqcup}} m_h^t (\dx \mfv) \, \exp
      \Bigl(- \int_{\U(h)} \Bigl(1-e^{-\Phi_{\hu}(\mfu)}\Bigr) \,
      \varrho_{h}^{t}(\mfv,\dx \mfu) \Bigr) \\
      & = \int_{\U(t-h)^{\sqcup}} m_h^{t,(n)} (\dx \mfv^{1,n}) \cdots
      \int_{\U(t-h)^{\sqcup}} m_h^{t,(n)} (\dx \mfv^{n,n}) \\
      & \qquad
      \exp \Bigl(-\int_{\U(h)} \Bigl(1-e^{-\Phi_{\hu}(\mfu)}\Bigr) \,
      \varrho_{h}^{t}(\mfv^{1,n} \sqcup \dots \sqcup \mfv^{n,n} % (t-h)
      ,\dx \mfu) \Bigr) \\
      & = \int_{\U(t-h)^{\sqcup}} m_h^{t,(n)} (\dx \mfv^{1,n})
      \cdots \int_{\U(t-h)^{\sqcup}} m_h^{t,(n)} (\dx \mfv^{n,n}) \\
      & \qquad \exp \Bigl(-\int_{\U(h)} \Bigl(1-e^{-\Phi_{\hu}(\mfu)}\Bigr) \,
      (\varrho_{h}^{t}(\mfv^{1,n},\dx \mfu)+\dots +
      \varrho_{h}^{t}(\mfv^{n,n},\dx \mfu)) \Bigr)\\
      & = \Bigl( \int_{\U(t-h)^{\sqcup}} m_h^{t,(n)} (\dx \mfv) \,
      \exp\Bigl(- \int_{\U(h)} \Bigl(1-e^{-\Phi_{\hu}(\mfu)}\Bigr) \,
      \varrho_{h}^{t}(\mfv,\dx \mfu) \Bigr) \Bigr)^n.
    \end{split}
  \end{align}
  Comparing the first and the last line of the above display we see
  that in the last line we have $n$-th power of the Laplace transform
  of a $\U(h)^\sqcup$-valued random variable which itself fits the
  description after \eqref{n.rgx7} with $m_h^{t}$ replaced by
  $m_h^{t,(n)}$. This completes the proof.
\end{proof}

In the next result we identify the \Levy{} measure of a (generalized)
$t$-branching tree. The result is a generalized version of the formula
for the \Levy{} measure that we have claimed in Remark~\ref{r.1394}
for specific $\U$-valued random variables which is an MBT.

\begin{proposition}
  \label{pr.infdivMBT}
  The \Levy{} measure $\Lambda_h^\mfU$ on of a (generalized)
  $t$-branching tree $\mfU$ on $\U(t)^\sqcup$ is of the form
  \begin{align}
    \label{n.e.tilde.lambda}
   \Lambda_h^\mfU(\dx \mfu)
    = \int_{\U(t-h)^\sqcup\setminus \{\ntree\}} \lambda^{m_h^t}(\dx
    \mfv) \, \E_N\Bigl[\1\bigl\{\bigsqcup_{\mfw \in N(\varrho_h^t(\mfv,\cdot))}
    \mfw \in \dx \mfu \bigr\}\Bigr],
  \end{align}
  where we use the notation of the description after \eqref{n.rgx7}
  and expectation is w.r.t.\ PPP $N$.
\end{proposition}
\begin{proof}
  \label{pr.ac7329}
  Let $\mfU$ be a $t$-branching tree and let $\mfV$ be a
  $t$-infinitely divisible $\U(t)^\sqcup$-valued random variable whose
  \Levy{} measure $\Lambda_h^\mfV(\dx \mfu)$ is given by the r.h.s.\
  of \eqref{n.e.tilde.lambda}.

  We have to verify that
  $\E[\exp(-\Phi_h(\mfV))] = \E[\exp(-\Phi_h(\mfU))]$ for all
  non-negative $\Phi \in \Pi$. Then by Theorem~1.30 in \cite{infdiv}
  it would follow that $\mfU \eqd \mfV$ and that in particular the
  \Levy{} measures agree. We have
  \begin{align}
    \label{n.rgx8}
    \begin{split}
      -\log \E[\exp(-\Phi_h(\mfV))] & \stackrel{\textnormal{(i)}}{=}
      \int_{\U(h)^\sqcup \setminus \{\ntree\}}
      \Lambda_h^\mfV(\dx \mfu) (1-e^{-\Phi_h(\mfu)}) \\
      & \stackrel{\textnormal{(ii)}}{=} \int_{\U(t-h)^\sqcup\setminus
        \{\ntree\}} \lambda^{m_h^t}(\dx \mfv) \, \E_N\biggl[ 1-
      \exp\Bigl( - \Phi_h \Bigl(\bigsqcup_{\mfw \in
        N(\varrho_h^t(\mfv,\cdot))} \mfw \Bigr)\Bigr) \biggr] \\
      & \stackrel{\textnormal{(iii)}}{=} \int_{\U(t-h)^\sqcup\setminus
        \{\ntree\}} \lambda^{m_h^t}(\dx \mfv) \, \biggl(1- \E_N\biggl[
      \exp \Bigl( - \sum_{\mfw \in N(\varrho_h^t(\mfv,\cdot))}
      \Phi_h(\mfw) \Bigr) \biggr]
      \biggr) \\
      & \stackrel{\textnormal{(iv)}}{=} \int_{\U(t-h)^\sqcup\setminus
        \{\ntree\}} \lambda^{m_h^t}(\dx \mfv) \, \Bigl(1- \exp\Bigl( -
      \int \varrho_h^t(\mfv,\dx \mfw) (1- e^{-\Phi_h(\mfw)}) \Bigr) \Bigr) \\
      & \stackrel{\textnormal{(v)}}{=} \int_{\U(t-h)^\sqcup\setminus
        \{\ntree\}} m_h^t(\dx \mfv) \, \exp \Bigl( - \int
      \varrho_h^t(\mfv,\dx \mfw)
      (1-e^{-\Phi_h(\mfw)}) \Bigr) \\
      & \stackrel{\textnormal{(vi)}}{=} - \log
      \E[\exp(-\Phi_h(\mfU))].
    \end{split}
  \end{align}
  Here, (i) follows by equation \eqref{ag2inf}; (ii) follows by
  \eqref{n.e.tilde.lambda}, Fubini and integration over $\dx \mfu$;
  (iii) by \eqref{e755}; (iv) by the usual Laplace transform formula
  for integrals over Poisson measures; see e.g.\ Lemma~12.2 in
  \cite{Kall02}; (v) this is the property that $\Lambda^{m_h^t}$ is
  the \Levy{} measure of $m^t_h$;
  (vi) is the first line of \eqref{n.rgx6}.
\end{proof}

\section{Approximation of solutions of
  \texorpdfstring{$\Omega^\uparrow$}{Omega-uparrow}-martingale
  problems}
\label{sec:appr-solut-mart}

We observe next that we can take every function on $\U$ of the form
$\mfu = (\bar\mfu,\hat\mfu) \mapsto
\bar\Phi(\bar\mfu)\wh\Phi(\hat\mfu)$, where $\bar\Phi$ is in $C^2$ and
$\wh\Phi$ is a polynomial induced by
$\varphi\in C^1_b([0,\infty)^{\binom{n}2})$.

We consider the set $\Pi^*$ consisting of functions (of the above
form) satisfying the following conditions: There are
$M_1,M_2,M_3,M_4 \in (0,\infty)$ and $n \in \N$ so that
\begin{enumerate}[(i)]
\item $\abs{\bar\Phi^{(i)}} \le M_1 \bar \mfu^n + M_2$ for
  $i=0,1,2$;
\item $\abs{\bar\Phi(\bar\mfu)/\bar\mfu} \le M_3$ for $\bar\mfu
  \le M_4$.
\end{enumerate}
Note that $\Pi^*$ contains the sets $\Pi(\mathcal C_b^1)$ and $\mathcal D_2$.

Now, for every $\Phi \in \Pi^*$ we can find functions
$(\Psi_k)_{k \in \N}$ with $\Psi_k \in \mathcal D_1$ so that
$\Psi_k \to \Phi$, $\Omega^\uparrow \Psi_k \to \Omega^\uparrow \Phi$
as $k \to \infty$ and
$\abs{\Omega^\uparrow \Phi - \Omega^\uparrow \Psi_k} \le \tilde M_1
\bar\mfu^n + \tilde M_2$ for suitable $\tilde M_1$ and $\tilde\M_2$
and all $k\in \N$.

By dominated convergence theorem it follows that a solution of the
$(\Omega^\uparrow,\mcD_1)$-martingale problem solves also the
$(\Omega^\uparrow,\Pi^*)$-martingale problem and in particular also
the $(\Omega^\uparrow,\mcD_2)$ and
$(\Omega^\uparrow,C_b^1)$-martingale problems.
% \end{remark}

\section{Yamada-Watanabe criterion}
\label{sec:ikeda-watan-argum}

Recall the process given in \eqref{e1011}-\eqref{e1336} and
\eqref{eq:FelBDiff-ab} for the total mass process. Let
\begin{align}
  \label{eq:40}
  \gamma(t,x) \alpha(t,x) = a(t,x) = \frac{2x(T-t)}{\exp(2x/(b(T-t)))
  -1}
\end{align}
with $a(t,0)\equiv b$, $a(T,x)\equiv b$, $\alpha(t,x) = \sqrt{bx}$,
$t\ge 0$, $x\ge 0$.

The function is continuous and bounded on $[0,T]\times [0,\infty)$.
Then
\begin{align}
  \label{eq:41}
  \gamma(t,x) = \sqrt{b}x^{-1/2}.
\end{align}
For each $\varepsilon >0$ the function $\gamma(t,\cdot)$ is bounded on
$[\varepsilon,\infty)$ uniformly in $t \in [0,T]$.

For a Feller diffusion $(X_t)_{t \in [0,T]}$ starting in
$\varepsilon>0$ we define the stopping time
\begin{align}
  \label{eq:42}
  T_\delta^\varepsilon = \inf\{t \in [0,T]: \gamma(t,X_t) \le
  \delta\}.
\end{align}
Note that we have $T_\delta^\varepsilon <\infty$ if  $0<\delta \le
\varepsilon$. Furthermore,
\begin{align}
  \label{eq:43}
  \int_0^t \gamma(s,X_s)\, \dx s <\infty \quad \text{for all $\varepsilon>0$}
\end{align}
and $T_\delta^\varepsilon \uparrow T_0^\varepsilon = T$, since by
Lemma~1.1 in \cite{DG03} the path cannot hit $0$. Then using the
Yamada-Watanabe criterion, see for instance pages 178-179 in
\cite{IkedaWatanabe1989}, the diffusion is uniquely determined as the
solution of the corresponding stochastic differential equation. To
obtain the solution starting from $0$ observe that the solution of
(4.7) in \cite{IkedaWatanabe1989} converges to a limiting SDE which
starts in $0$. Thus, the entrance law of the above diffusion from $0$
is uniquely determined.

% \bibliography{tvF}
% \bibliographystyle{alpha}

\newcommand{\etalchar}[1]{$^{#1}$}

\end{document}